\documentclass[reqno, 11pt, a4paper]{amsart}

\usepackage[utf8]{inputenc}
\usepackage[T1]{fontenc}

\usepackage{enumitem}
\usepackage{amsthm}
\usepackage{thmtools}
\usepackage{amssymb}
\usepackage{amsmath}
\usepackage{mathtools}
\usepackage{bm}
\usepackage{cite}

\usepackage{fnpct}
\usepackage{tikz-cd}
\usepackage{graphicx}
\usepackage[only, llbracket, rrbracket]{stmaryrd}
\usepackage{units} %diagonal fractions
\usepackage{xcolor}
\usepackage[pdfusetitle, bookmarks, colorlinks, allcolors=blue]{hyperref}
\usepackage{doi}
\usepackage{cprotect}

\newtheorem{theorem}{Theorem}[section]

\newtheorem{conjecture}[theorem]{Conjecture}
\newtheorem{corollary}[theorem]{Corollary}
\newtheorem{lemma}[theorem]{Lemma}

\theoremstyle{definition}

\newtheorem{definition}[theorem]{Definition}
\newtheorem{example}[theorem]{Example}
\newtheorem{proposition}[theorem]{Proposition}
\newtheorem{remark}[theorem]{Remark}

\numberwithin{equation}{section}

\newcommand{\define}[1]{\textit{#1}}

\let\OldParagraph\paragraph
\renewcommand{\paragraph}[1]{\smallskip\OldParagraph{\textbf{#1}}}

\usepackage{caption,subcaption}
\usepackage{cleveref}

\newcommand{\EndComma}{,}%{\,,}
\newcommand{\EndFullStop}{.}%{\,.}

\newcommand{\ComposedWith}{\circ}
\newcommand{\DoublePrime}{^{\prime \prime}}
\newcommand{\Inverse}{^{-1}}

\newcommand{\Prime}{^\prime}
\newcommand{\RestrictedTo}[1]{|_{ #1 }}
\renewcommand{\epsilon}{\varepsilon}
\renewcommand{\phi}{\varphi}

\newcommand{\Count}{\#}

\newcommand{\DirectProduct}{\prod}
\newcommand{\DisjointUnion}{\sqcup}

\newcommand{\Intersection}{\cap}

\newcommand{\Set}[1]{\left\{ #1 \right\}}
\newcommand{\SetCondition}[2]{\left\{ #1 \ : \  #2 \right\}}
\newcommand{\SetConditionBar}[2]{\left\{ #1 \ \middle| \  #2 \right\}}
\newcommand{\Union}{\cup}

\newcommand{\Conjugate}[1]{\overline{ #1 }}
\newcommand{\Complex}{\mathbb{C}}
\newcommand{\ComplexNumbers}{\mathbb{C}}

\newcommand{\ExponentialNumber}{e}

\newcommand{\ImaginaryNumber}{i}

\newcommand{\Infinity}{\infty}
\newcommand{\Integers}{\mathbb{Z}}
\newcommand{\IntervalClosed}[2]{\left[ #1 , #2 \right]}
\newcommand{\IntervalClosedOpen}[2]{\left[ #1 , #2 \right)}
\newcommand{\IntervalOpen}[2]{\left( #1 , #2 \right)}
\newcommand{\IntervalOpenClosed}[2]{\left( #1 , #2 \right]}
\newcommand{\KroneckerDelta}{\delta}
\newcommand{\Modulus}[1]{\left| #1 \right|}

\newcommand{\PiNumber}{\pi}

\newcommand{\RationalNumbers}{\mathbb{Q}}

\newcommand{\RealNumbers}{\mathbb{R}}

\newcommand{\Norm}[1]{\left\| #1 \right\|}

\newcommand{\wrt}[1]{\, \mathrm{d} { #1 } }

\newcommand{\Lie}{\operatorname{Lie}}

\newcommand{\LieBracket}[2]{\left[ #1 , #2 \right]}

\newcommand{\Automorphisms}{\operatorname{Aut}}

\newcommand{\Dimension}{\dim}
\newcommand{\DirectSum}{\oplus}

\newcommand{\Homomorphisms}{\operatorname{Hom}}

\newcommand{\InnerProduct}[2]{\langle #1 , #2 \rangle}

\newcommand{\Identity}{\mathrm{Id}}

\newcommand{\Projective}{\mathbb{P}}

\newcommand{\ArbitraryIndex}{\ast}
\newcommand{\Argument}{\cdot}

\newcommand{\FractionField}{\operatorname{Frac}}

\newcommand{\Inclusion}{\hookrightarrow}
\newcommand{\PullBack}{^\ast}
\newcommand{\PushForward}{_\ast}
\newcommand{\Tensor}{\otimes}

\newcommand{\Boundary}{\partial}
\newcommand{\Circle}{{S^1}}

\newcommand{\ConnectedSum}{\#}
\newcommand{\ContractibleLoopSpace}{\mathcal{L}}

\newcommand{\Disc}{D}
\newcommand{\FirstChernClass}{c_1}

\newcommand{\HomotopyGroup}{\pi}

\newcommand{\InfiniteComplexProjectiveSpace}{\Complex \Projective ^\infty}
\newcommand{\InfiniteSphere}{S^\infty}

\newcommand{\Derivative}{D}

\newcommand{\EvaluateAt}[1]{|_{ #1 }}
\newcommand{\ExteriorDerivative}{\mathrm{d}}

\newcommand{\SmoothFunctions}{C^\infty}

\newcommand{\Sphere}{S^2}
\newcommand{\HighDimensionalSphere}[1]{S^{ #1 }}
\newcommand{\TangentSpace}{T}

\newcommand{\Gradient}{\nabla}
\newcommand{\StableManifold}{W^s}
\newcommand{\UnstableManifold}{W^u}

\newcommand{\Energy}{E}

\newcommand{\HamiltonianVectorField}[1]{X_{ #1 }}

\newcommand{\NovikovRing}{\Lambda}
\newcommand{\QuantumProduct}{\ast}
\newcommand{\QuantumAction}{\ast}

\newcommand{\ModuliSpace}{\mathcal{M}}

\newcommand{\WithFilling}[1]{{\widetilde{ #1 }}}

\newcommand{\VerticalTangentSpace}{T^{\text{vert}}}

\newcommand{\CanonicalMap}{c}

\newcommand{\eqnt}{\text{eq}}
\newcommand{\eqModuliSpace}{\mathcal{M}^\eqnt}

\newcommand{\uformal}{\mathbf{u}}

\newcommand{\SpaceOfDerivations}{\mathcal{X}}

\newcommand{\Connection}{\nabla}

\newcommand{\Curvature}{\mathcal{R}}
\newcommand{\ShiftOperator}{\mathbb{S}}
\newcommand{\dbyd}[1]{\frac{d}{d#1}}
\newcommand{\dbydmult}[1]{\formalAdditionalCircle \frac{d}{d#1}}
\newcommand{\AlgebraHomomorphismSpace}{\operatorname{AlgHom}}

\newcommand{\Manifold}{M}
\newcommand{\SymplecticForm}{\omega}
\newcommand{\dimManifold}{n}
\newcommand{\ConvexCoordMap}{\psi}
\newcommand{\ContactManifold}{\Sigma}
\newcommand{\ContactForm}{\alpha}
\newcommand{\ReebVectorField}{X_\ContactForm}
\newcommand{\ReebPeriods}{\mathcal{R}}
\newcommand{\RadialCoord}{R}

\newcommand{\NovVariable}{q}
\newcommand{\DegreeTwoCoclass}{\alpha}
\newcommand{\secondDegreeTwoCoclass}{\beta}
\newcommand{\MorseFunction}{f}
\newcommand{\eqMorseFunction}{f^{\eqnt}}
\newcommand{\ACS}{J}
\newcommand{\eqACS}{J^\eqnt}
\newcommand{\DegreeTwoClass}{A}
\newcommand{\pointSphere}{p}

\newcommand{\Torus}{T}
\newcommand{\dimTorus}{k}
\newcommand{\Character}{\chi}
\newcommand{\Cocharacter}{\sigma}
\newcommand{\secondCocharacter}{{\sigma\Prime}}
\newcommand{\basisCocharacter}{\tau}
\newcommand{\LatticeCharacters}{\operatorname{Char}}
\newcommand{\LatticeCocharacters}{\operatorname{Cochar}}
\newcommand{\NonnegativeLatticeCocharacters}[1]{\operatorname{Cochar}^{\ge 0}_{#1}}
\newcommand{\TorusAction}{\rho}
\newcommand{\SymmetricAlgebra}{\operatorname{Sym}}
\newcommand{\Dual}{^\vee}
\newcommand{\ExtendedTorus}{\widehat{T}}
\newcommand{\Extended}[1]{\widehat{#1}}
\newcommand{\AdditionalCircle}{S^1_0}
\newcommand{\FixedPoint}{\mu}

\newcommand{\Projection}{\pi}

\newcommand{\UniversalSpace}{E}
\newcommand{\ClassifyingSpace}{B}
\newcommand{\formalTorus}{y}
\newcommand{\formalAdditionalCircle}{\Extended{\formalTorus}_0}
\newcommand{\morseUniversalSpace}{g}
\newcommand{\CoordwiseProductAction}{\varrho}

\newcommand{\eltTorus}{\mathbf{t}}
\newcommand{\eltUniversalSpace}{e}
\newcommand{\eltManifold}{m}
\newcommand{\eltContactManifold}{y}
\newcommand{\eltAdditionalCircle}{a}
\newcommand{\eltCircle}{\theta}
\newcommand{\eltSphere}{z}
\newcommand{\eltExtendedTorus}{\Extended{\eltTorus}}

\newcommand{\critManifold}{x}
\newcommand{\flowManifold}{\gamma}
\newcommand{\flowUniversalSpace}{v}
\newcommand{\FloerSolution}{u}
\newcommand{\critUniversalSpace}{\epsilon}
\newcommand{\critClassifyingSpace}{c}
\newcommand{\CriticalPointSet}[1]{\operatorname{Crit}( #1 )}
\newcommand{\eqCriticalPointSet}[1]{\operatorname{Crit}^\eqnt( #1 )}
\newcommand{\minUniversalSpace}{\min(\morseUniversalSpace)}

\newcommand{\Hamiltonian}{H}
\newcommand{\ActionHamiltonian}{K}
\newcommand{\CocharacterSlope}[1]{\kappa^{#1}}
\newcommand{\eqHamiltonian}{H^{\eqnt}}
\newcommand{\HamiltonianOrbit}{x}
\newcommand{\Slope}{\lambda}
\newcommand{\HamiltonianOrbitSet}{\mathcal{P}}
\newcommand{\FillingBasis}{\mathcal{B}}

\newcommand{\Cohomology}{H}
\newcommand{\Homology}{H}
\newcommand{\QuantumCohomology}{QH}
\newcommand{\QuantumNotation}{\mathcal{Q}}
\newcommand{\GradedCompletedTensorProduct}{\widehat{\otimes}}
\newcommand{\QuantumSeidel}{Q\mathcal{S}}
\newcommand{\FloerSeidel}{F\mathcal{S}}
\newcommand{\WeightedQuantumSeidel}{W\!Q\mathcal{S}}
\newcommand{\FloerCochainComplex}{FC}
\newcommand{\FloerCohomology}{FH}
\newcommand{\ContinuationMap}{\phi}
\newcommand{\SymplecticCohomology}{SH}
\newcommand{\PSSisomorphism}{\operatorname{PSS}}

\newcommand{\ClutchingBundle}[1]{E(#1)}
\newcommand{\Hemisphere}{\mathbb{D}}
\newcommand{\ClutchingProjection}{\pi}
\newcommand{\ClutchingBilinearForm}{\Omega}
\newcommand{\Pole}{z}
\newcommand{\ClutchingACS}{\mathbf{J}}
\newcommand{\eqClutchingACS}{\mathbf{J}^\eqnt}
\newcommand{\SphereACS}{j}
\newcommand{\ClutchingSection}{u}
\newcommand{\ClutchingFibreInclusion}{I}

\newcommand{\MomentMap}[1]{\mu_{#1}}
\newcommand{\TorusLieAlgebra}{\mathfrak{t}}
\newcommand{\DualTorusLieAlgebra}{\TorusLieAlgebra^{\ast}}
\newcommand{\MomentPolytope}{\Delta}
\newcommand{\pointDualTorusLieAlgebra}{y}
\newcommand{\InwardNormalVector}{e}
\newcommand{\numberFacets}{N}
\newcommand{\FacetBoundaryCondition}{\lambda}
\newcommand{\Facet}{F}
\newcommand{\Face}{F}
\newcommand{\InwardNormalFan}{\Sigma}
\newcommand{\ToricGenerator}{x}
\newcommand{\StanleyReisnerIdeal}{J_{\text{SR}}}
\newcommand{\LinearRelationsIdeal}{J_{\text{lin}}}
\newcommand{\Divisor}{D}
\newcommand{\edgeMomentPolytope}{e}
\newcommand{\edgeVectorMomentPolytope}{\gamma}
\newcommand{\vertexMomentPolytope}{v}
\newcommand{\QuantumStanleyReisnerIdeal}{J_{\text{qSR}}}
\newcommand{\NovikovExponentForDivisor}{d}
\newcommand{\ToricGeneratorWithNovikovWeighting}{z}
\newcommand{\LineBundle}{\mathcal{O}}
\newcommand{\Cone}{C}
\newcommand{\NeighbouringFacets}{N}
\newcommand{\Star}{\operatorname{Star}}

\newcommand{\eltComplexProjectiveSpace}{w}

\newcommand{\ToricLineBundle}{E}
\newcommand{\BaseToricLineBundle}{B}
\newcommand{\ProjectionToricLineBundle}{\pi}
\newcommand{\IndexOfMonotonicity}{\lambda}
\newcommand{\NegativeLineBundleNumber}{k}
\newcommand{\DivisorContributionToBundle}{m}
\newcommand{\SphereBundle}{S}

\newcommand{\PerturbedStableManifold}{\StableManifold_{\text{pert}}}
\newcommand{\FillingIntersectionMap}{w}

\title[Shift operators and connections on equivariant symplectic cohomology]{Shift operators and flat connections on equivariant symplectic cohomology}
\author{Todd Liebenschutz-Jones}
\date{\today}
\thanks{\textit{Correspondence}: \href{mailto:todd.liebenschutz-jones@maths.ox.ac.uk}{todd.liebenschutz-jones@maths.ox.ac.uk}}

\begin{document}

\maketitle
\begin{abstract}
    We construct shift operators on equivariant symplectic cohomology which generalise the shift operators on equivariant quantum cohomology in algebraic geometry.
    That is, given a Hamiltonian action of the torus $T$, we assign to a cocharacter of $T$ an endomorphism of $(S^1 \times T)$-equivariant Floer cohomology based on the equivariant Floer Seidel map.
    We prove the shift operator commutes with a connection.
    This connection is a multivariate version of Seidel's $q$-connection on $S^1$-equivariant Floer cohomology and generalises the Dubrovin connection on equivariant quantum cohomology.
    We prove that the connection is flat, which was conjectured by Seidel.
    As an application, we compute these algebraic structures for toric manifolds.
\end{abstract}

\section{Introduction}
\label{sec:introduction}

\subsection{Flat connections on \texorpdfstring{$\Circle$}{S1}-equivariant symplectic cohomology}
\label{sec:flat-connections-intro-s1-only}

The $\Circle$-equiv\-ari\-ant symplectic cohomology $\SymplecticCohomology^\ArbitraryIndex _{\Circle} (\Manifold)$ of a convex symplectic manifold $\Manifold$ is a $\Integers [\uformal]$-module invariant introduced by Viterbo \cite[Section~5]{viterbo_functors_1996}, and later developed by Seidel \cite[Section~8b]{seidel_biased_2007} and Bourgeois and Oancea \cite{bourgeois_s1-equivariant_2017} among others.
It builds on symplectic cohomology by incorporating the natural $\Circle$-action on the loop space $\ContractibleLoopSpace{\Manifold} = \Set{\text{contractible $\HamiltonianOrbit : \Circle \to \Manifold$}}$ given by $(\eltCircle \cdot \HamiltonianOrbit)(t) = \HamiltonianOrbit(t - \eltCircle)$ for $\eltCircle \in \Circle$.
This $\Circle$-action readily distinguishes constant and nonconstant Hamiltonian orbits by their stabilizer groups.
The constant orbits capture topological information about $\Manifold$ via a localisation theorem \cite[Theorem~1.1]{zhao_periodic_2019}, while the nonconstant orbits give rise to the \define{positive $\Circle$-equivariant symplectic cohomology} $\SymplecticCohomology^{\ArbitraryIndex, +}_{\Circle} (\Manifold)$ of $\Manifold$, which is an effective invariant for distinguishing contact structures \cite[Theorem~1.4]{bourgeois_s1-equivariant_2017, gutt_positive_2017}.
McLean and Ritter used $\Circle$-equivariant symplectic cohomology to deduce the cohomology of a crepant resolution of an isolated singularity by analysing the nonconstant orbits, establishing a new proof of the McKay correspondance \cite{mclean_mckay_2018}.

The $\Circle$-action breaks the definition of the pair-of-pants product.
Equipped with this product, symplectic cohomology is a graded-commutative, associative and unital algebra, and in fact has a full TQFT structure \cite[Section~(8a)]{ritter_topological_2013, seidel_biased_2007}.
Without the pair-of-pants product, $\SymplecticCohomology^\ArbitraryIndex _{\Circle} (\Manifold)$ has only its module structure, a poor offering compared to the assortment of algebraic structures available on symplectic cohomology.

To remedy this, Seidel defined the \define{$q$-connection}, which is an additive endomorphism 
    \begin{equation}
        \Gamma_q : \FloerCohomology^\ArbitraryIndex _{\Circle} (\Manifold, \Slope) \to \FloerCohomology^{\ArbitraryIndex + 2} _{\Circle} (\Manifold, \Slope)
    \end{equation}
of the $\Circle$-equivariant Floer cohomology of a Hamiltonian function with slope $\Slope$ \cite[Section~(2a)]{seidel_connections_2018}.
The endomorphism satisfies the Leibniz rule
    \begin{equation}
    \label{eqn:leibniz-rule-for-seidel-q-connection}
        \Gamma_q (f x) = f \Gamma_q (x) + \uformal (\partial_q f) x
    \end{equation}
for $f \in \NovikovRing \llbracket \uformal \rrbracket$ and $x \in \FloerCohomology^\ArbitraryIndex _{\Circle} (\Manifold, \Slope)$.
Under Seidel's assumptions, the Novikov ring $\NovikovRing$ is a ring of formal power series in $q$ and $\partial_q : \NovikovRing \llbracket \uformal \rrbracket \to \NovikovRing \llbracket \uformal \rrbracket$ is the operation which differentiates with respect to $q$.

Seidel's $q$-connection is an example of a \define{differential connection}, an algebraic structure that abstracts the algebraic properties of a connection on a vector bundle (\autoref{sec:differential-connections-definition}).
The existence of a map satisfying \eqref{eqn:leibniz-rule-for-seidel-q-connection} indicates that the $\NovikovRing \llbracket \uformal \rrbracket$-module $\FloerCohomology^\ArbitraryIndex _{\Circle} (\Manifold, \Slope)$ respects the differentiation operation $\uformal \partial_q : \NovikovRing \llbracket \uformal \rrbracket \to \NovikovRing \llbracket \uformal \rrbracket$ in the natural way.
As such, the map $\Gamma_q$ upgrades the underlying $\NovikovRing \llbracket \uformal \rrbracket$-module structure on $\FloerCohomology^\ArbitraryIndex _{\Circle} (\Manifold, \Slope)$.
Differential connections are useful when working with maps which preserve the connection.
For example, when we compute shift operators in \autoref{sec:toric-manifolds-general}, the shift operators preserve the connection and we can determine each shift operator by its value on a single element.
Compare this to how a linear map is determined by its values on a basis.

The operation $\uformal \partial_q$ on the $\Circle$-equivariant Floer cochain complex is not a cochain map, so a correction term is required when defining $\Gamma_q$.
We can understand this correction term by reducing to non-equivariant cohomology.
We have the following commutative diagram.
\begin{equation}
    \begin{tikzcd}[column sep=large]
        \FloerCohomology^\ArbitraryIndex _{\Circle} (\Manifold, \Slope)
        \arrow[r, "\Gamma_q"]
        \arrow[d, "\uformal \mapsto 0"]
        & \FloerCohomology^{\ArbitraryIndex + 2} _{\Circle} (\Manifold, \Slope)
        \arrow[d, "\uformal \mapsto 0"] \\
        \FloerCohomology^\ArbitraryIndex (\Manifold, \Slope)
        \arrow[r, "{q \Inverse [\SymplecticForm] \QuantumProduct}"]
        & \FloerCohomology^{\ArbitraryIndex + 2} (\Manifold, \Slope)
    \end{tikzcd}
\end{equation}
The quantum action $[\SymplecticForm] \QuantumProduct$ by the cohomology class $[\SymplecticForm] \in \Cohomology^2 (\Manifold; \RealNumbers)$ counts Floer solutions $\FloerSolution : \RealNumbers \times \Circle \to \Manifold$ which intersect a locally-finite cycle representing $[\SymplecticForm]$ at $\FloerSolution(0, 0)$.
The map $q \Inverse [\SymplecticForm] \QuantumProduct$ is morally similar to performing $\partial_q$, but `on the Floer solutions': for the operation $\partial_q$, we first multiply by the exponent of $q$ (which records the symplectic energy $[\SymplecticForm] (\DegreeTwoClass)$ of classes $\DegreeTwoClass \in \Homology_2 (\Manifold)$), and then we divide by $q$.

In this paper, we define a generalisation of Seidel's $q$-connection for any degree-2 cohomology class.

\begin{theorem}
\label{thm-intro:differential-connection-s1-only}
    For every $\DegreeTwoCoclass \in \Cohomology^2 (\Manifold; \Integers)$, there is a $\Integers [\uformal]$-module endomorphism
        \begin{equation}
        \label{eqn-intro:floer-connection-circle-only}
            \Connection_\DegreeTwoCoclass : \FloerCohomology^\ArbitraryIndex _{\Circle} (\Manifold, \Slope) \to \FloerCohomology^{\ArbitraryIndex + 2} _{\Circle} (\Manifold, \Slope)
        \end{equation}
     which satisfies the Leibniz rule
        \begin{equation}
        \label{eqn:leibniz-rule-for-my-connection-no-large-torus}
            \Connection_\DegreeTwoCoclass (f x) = f \Connection_\DegreeTwoCoclass (x) + \uformal \left( \dbyd{\DegreeTwoCoclass} f \right) x
        \end{equation}
    and makes the following diagram commute.
\begin{equation}
    \begin{tikzcd}[column sep=large]
        \FloerCohomology^\ArbitraryIndex _{\Circle} (\Manifold, \Slope)
        \arrow[r, "\Connection_\DegreeTwoCoclass"]
        \arrow[d, "\uformal \mapsto 0"]
        & \FloerCohomology^{\ArbitraryIndex + 2} _{\Circle} (\Manifold, \Slope)
        \arrow[d, "\uformal \mapsto 0"] \\
        \FloerCohomology^\ArbitraryIndex (\Manifold, \Slope)
        \arrow[r, "{\DegreeTwoCoclass \QuantumProduct}"]
        & \FloerCohomology^{\ArbitraryIndex + 2} (\Manifold, \Slope)
    \end{tikzcd}
\end{equation}
    These maps commute with continuation maps, and hence they induce $\Integers [\uformal]$-module endomorphisms
        \begin{equation}
            \Connection_\DegreeTwoCoclass : \SymplecticCohomology^\ArbitraryIndex _{\Circle} (\Manifold) \to \SymplecticCohomology^{\ArbitraryIndex + 2} _{\Circle} (\Manifold)
        \end{equation}
    on $\Circle$-equivariant symplectic cohomology which satisfy \eqref{eqn:leibniz-rule-for-my-connection-no-large-torus}.
\end{theorem}

Our setup differs from Seidel's construction in two important ways.
First, instead of Seidel's formal $\Integers [\uformal]$-module structure, our $\Integers [\uformal]$-module structure arises from a Morse cup product construction in the classifying space $\ClassifyingSpace \Circle = \InfiniteComplexProjectiveSpace$, as in our previous paper \cite[Section~4.2.2]{liebenschutz-jones_intertwining_2020}.
This uses the identification $\Integers [\uformal] = \Cohomology^\ArbitraryIndex (\ClassifyingSpace \Circle)$.
Second, our Novikov ring records classes $\DegreeTwoClass \in \Homology_2 (\Manifold; \Integers)$, unlike Seidel's Novikov ring which records only the symplectic energy $[\SymplecticForm](\DegreeTwoClass)$.
The operation $\dbyd{\DegreeTwoCoclass}$ is given by
    \begin{equation}
        \dbyd{\DegreeTwoCoclass}(\NovVariable^\DegreeTwoClass) = \DegreeTwoCoclass(\DegreeTwoClass) \NovVariable^\DegreeTwoClass.
    \end{equation}
Notice this operation does not change the exponent of $\NovVariable$, unlike normal differentiation.

For any $\DegreeTwoCoclass, \secondDegreeTwoCoclass \in \Cohomology^2 (\Manifold)$, the two operations $\dbyd{\DegreeTwoCoclass}$ and $\dbyd{\secondDegreeTwoCoclass}$ commute.
We show that the corresponding maps $\Connection_{\DegreeTwoCoclass}$ and $\Connection_{\secondDegreeTwoCoclass}$ also commute.
This condition is called \define{flatness}, continuing the analogy with connections on a vector bundle (see \autoref{eg:abstract-connection-on-vector-bundle-motivation}).
The result was anticipated by Seidel \cite[Section~(2a)]{seidel_connections_2018}, but is new.

\begin{theorem}
    [Flatness]
\label{thm-intro:flatness-of-differential-connection-s1-only}
    The maps $\Connection_{\DegreeTwoCoclass}$ and $\Connection_{\secondDegreeTwoCoclass}$ commute for any $\DegreeTwoCoclass, \secondDegreeTwoCoclass \in \Cohomology^2 (\Manifold)$.
\end{theorem}

Like Seidel's $q$-connection, our connection $\Connection$ can be viewed as a Floer-theoretic analogue of a Dubrovin connection on $\Circle$-equivariant quantum cohomology $\QuantumCohomology^\ArbitraryIndex _{\Circle} (\Manifold)$ which differentiates with respect to the Novikov variable.
Here, $\Circle$ acts trivially on $\Manifold$.
The connection is given by
    \begin{equation}
    \label{eqn-intro:quantum-connection-circle-only}
        \begin{aligned}
            &\Connection_{\DegreeTwoCoclass} : \QuantumCohomology^\ArbitraryIndex _{\Circle} (\Manifold) \to \QuantumCohomology^{\ArbitraryIndex + 2} _{\Circle} (\Manifold) \\
            &\Connection_{\DegreeTwoCoclass} (f x) = \uformal \left( \dbyd{\DegreeTwoCoclass} f \right) x + \DegreeTwoCoclass \QuantumProduct (f x)
        \end{aligned}
    \end{equation}
for $f \in \NovikovRing$ and $x \in \Cohomology^\ArbitraryIndex _{\Circle} (\Manifold)$, using the isomorphism $\QuantumCohomology^\ArbitraryIndex _{\Circle} (\Manifold) \cong \NovikovRing \Tensor \Cohomology^\ArbitraryIndex _{\Circle} (\Manifold)$.
For small slopes, our connection \eqref{eqn-intro:floer-connection-circle-only} agrees with the connection \eqref{eqn-intro:quantum-connection-circle-only} under the PSS isomorphism.
$\Circle$-equivariant quantum cohomology has many different connections, and all of these connections are flat by a general argument \cite{dubrovin_integrable_1992}.
The flatness corresponds to important geometric properties of the $\Circle$-equivariant quantum product---namely its graded-commutativity and associativity.

Like the connection \eqref{eqn-intro:quantum-connection-circle-only}, our connection on $\FloerCohomology^\ArbitraryIndex_{\Circle} (\Manifold, \Slope)$ is the sum of a formal differentiation operation $\uformal \dbyd{\DegreeTwoCoclass}$ with an operation $\DegreeTwoCoclass \QuantumAction$.
The $\Circle$-equivariant quantum action $\DegreeTwoCoclass \QuantumAction$ counts Floer solutions $\FloerSolution : \RealNumbers \times \Circle \to \Manifold$ which intersect with $\DegreeTwoCoclass$ at $\FloerSolution (0, t_0)$, where the value of $t_0 \in \Circle$ is determined by an equivariant construction on the classifying space $\ClassifyingSpace \Circle$.
An additional correction term $-\uformal \FillingIntersectionMap^\FillingBasis_\DegreeTwoCoclass$ is required because $\FloerCochainComplex^\ArbitraryIndex _{\Circle} (\Manifold, \Slope)$ does not have a canonical basis unlike quantum cohomology.
Roughly speaking, $- \FillingIntersectionMap^\FillingBasis_\DegreeTwoCoclass$ acts like a geometric version of the $\dbyd{\DegreeTwoCoclass}$ operator on the choices of cappings of the contractible Hamiltonian orbits that generate the cochain complex.
The two terms in \eqref{eqn-intro:quantum-connection-circle-only} are individually chain maps, whereas in the Floer case it is only the sum $\uformal \dbyd{\DegreeTwoCoclass} + {\DegreeTwoCoclass \QuantumAction} - \uformal \FillingIntersectionMap^\FillingBasis_\DegreeTwoCoclass$ which is a chain map on $\FloerCochainComplex^\ArbitraryIndex _{\Circle} (\Manifold, \Slope)$.
The operation $\DegreeTwoCoclass \QuantumAction$ corresponds to the correction term in Seidel's $q$-connection.

\subsection{Flat connections on \texorpdfstring{$\ExtendedTorus$}{T}-equivariant symplectic cohomology}

In this paper, we work with a Hamiltonian action $\TorusAction$ of the $\dimTorus$-dimensional torus $\Torus$ on our symplectic manifold $\Manifold$.
In \autoref{sec:toric-manifolds-general}, we work with toric manifolds, for which $\TorusAction$ is an effective Hamiltonian $\Torus$-action with $\dimTorus = \frac12 \Dimension \Manifold$, but our constructions here apply for any $\dimTorus$.
In particular, the case $\dimTorus = 0$ was described in \autoref{sec:flat-connections-intro-s1-only}.

The torus $\Torus$ acts pointwise on the loops in $\ContractibleLoopSpace{\Manifold}$ via $\TorusAction$.
Combined with the $\Circle$-action from \autoref{sec:flat-connections-intro-s1-only}, this yields a natural $\Circle \times \Torus$-action on $\ContractibleLoopSpace{\Manifold}$.
Denote a copy of $\Circle$ by $\AdditionalCircle$, set $\ExtendedTorus = \AdditionalCircle \times \Torus$ and denote by $\Extended{\TorusAction}$ the $\ExtendedTorus$-action on $\Manifold$ determined by $\Extended{\TorusAction}\RestrictedTo{\Torus} = \TorusAction$ and $\Extended{\TorusAction} \RestrictedTo{\AdditionalCircle} = \Identity_\Manifold$.
We define $\ExtendedTorus$-equivariant Floer cohomology $\FloerCohomology^\ArbitraryIndex _{\ExtendedTorus} (\Manifold, \Extended{\TorusAction}, \Slope)$ similarly to $\Circle$-equivariant Floer cohomology, combining Morse theory on the classifying space $\ClassifyingSpace \ExtendedTorus$ with Floer theory on $\Manifold$.
The $\ExtendedTorus$-equivariant symplectic cohomology of $\Manifold$ is the direct limit of $\FloerCohomology^\ArbitraryIndex _{\ExtendedTorus} (\Manifold, \Extended{\TorusAction}, \Slope)$ as $\Slope \to \infty$.
These are both $\Cohomology^\ArbitraryIndex (\ClassifyingSpace \ExtendedTorus)$-modules.
Associated to the projection $\ExtendedTorus \to \AdditionalCircle$ is a special element $\formalAdditionalCircle \in \Cohomology^\ArbitraryIndex (\ClassifyingSpace \ExtendedTorus)$, which has similar properties to $\uformal$ in \autoref{sec:flat-connections-intro-s1-only}.

\begin{theorem}
    [Flat connection]
\label{thm-intro:flat-differential-connection-torus}
    Let $\FixedPoint \in \Manifold$ be a fixed point of $\TorusAction$.
    For every $\DegreeTwoCoclass \in \Cohomology^2 (\Manifold; \Integers)$, there is a $\Cohomology^\ArbitraryIndex (\ClassifyingSpace \ExtendedTorus)$-module endomorphism
        \begin{equation}
        \label{eqn-intro:floer-connection-torus}
            \Connection^\FixedPoint_\DegreeTwoCoclass : \FloerCohomology^\ArbitraryIndex _{\ExtendedTorus} (\Manifold, \Extended{\TorusAction}, \Slope) \to \FloerCohomology^{\ArbitraryIndex + 2} _{\ExtendedTorus} (\Manifold, \Extended{\TorusAction}, \Slope)
        \end{equation}
     which satisfies the Leibniz rule
        \begin{equation}
        \label{eqn:leibniz-rule-for-my-connection-large-torus}
            \Connection^\FixedPoint_\DegreeTwoCoclass (f x) = f \Connection^\FixedPoint_\DegreeTwoCoclass (x) + \formalAdditionalCircle \left( \dbyd{\DegreeTwoCoclass} f \right) x
        \end{equation}
    and makes the following diagram commute.
        \begin{equation}
        \label{eqn:commutative-diagram-for-connection-to-identify-correction-term-torus}
            \begin{tikzcd}[column sep=large]
                \FloerCohomology^\ArbitraryIndex _{\ExtendedTorus} (\Manifold, \Extended{\TorusAction}, \Slope)
                \arrow[r, "\Connection^\FixedPoint_\DegreeTwoCoclass"]
                \arrow[d]
                & \FloerCohomology^{\ArbitraryIndex + 2} _{\ExtendedTorus} (\Manifold, \Extended{\TorusAction}, \Slope)
                \arrow[d] \\
                \FloerCohomology^\ArbitraryIndex (\Manifold, \Slope)
                \arrow[r, "{\DegreeTwoCoclass \QuantumProduct}"]
                & \FloerCohomology^{\ArbitraryIndex + 2} (\Manifold, \Slope)
            \end{tikzcd}
        \end{equation}
    In \eqref{eqn:commutative-diagram-for-connection-to-identify-correction-term-torus}, the vertical maps are induced by the restriction map $\Cohomology^\ArbitraryIndex (\ClassifyingSpace \ExtendedTorus) \to \Cohomology^\ArbitraryIndex (\textup{pt})$.
    The maps $\Connection^\FixedPoint_\DegreeTwoCoclass$ commute with continuation maps, so they induce maps on $\ExtendedTorus$-equivariant symplectic cohomology which satisfy \eqref{eqn:leibniz-rule-for-my-connection-large-torus}.
    Moreover, the maps $\Connection^\FixedPoint_\DegreeTwoCoclass$ and $\Connection^\FixedPoint_\secondDegreeTwoCoclass$ commute for any $\DegreeTwoCoclass, \secondDegreeTwoCoclass \in \Cohomology^2 (\Manifold)$, so $\Connection^\FixedPoint$ is flat.
\end{theorem}

The connection is given by the sum 
\begin{equation}
    \Connection^\FixedPoint _\DegreeTwoCoclass = \formalAdditionalCircle \dbyd{\DegreeTwoCoclass} + {\DegreeTwoCoclass^\FixedPoint \QuantumAction} - \formalAdditionalCircle \FillingIntersectionMap^\FillingBasis_\DegreeTwoCoclass.
\end{equation}
The fixed point $\FixedPoint$ is used to lift the class $\DegreeTwoCoclass$ to a class $\DegreeTwoCoclass^\FixedPoint \in \Cohomology^\ArbitraryIndex _{\ExtendedTorus} (\Manifold, \TorusAction)$.
Note that \autoref{thm-intro:differential-connection-s1-only} and \autoref{thm-intro:flatness-of-differential-connection-s1-only} are the $\dimTorus=0$ case of \autoref{thm-intro:flat-differential-connection-torus}, with $\FixedPoint$ any point of $\Manifold$.

\subsection{Equivariant Seidel maps and shift operators}

A \define{cocharacter} of the torus $\Torus$ is a group homomorphism $\Cocharacter : \Circle \to \Torus$.
The composition $\TorusAction \ComposedWith \Cocharacter$ is a Hamiltonian $\Circle$-action on $\Manifold$.
This $\Circle$-action naturally induces an automorphism of the loop space $\ContractibleLoopSpace{\Manifold} = \Set{x : \Circle \to \Manifold}$ given by $((\TorusAction \ComposedWith \Cocharacter)(\HamiltonianOrbit))(t) = \TorusAction_{\Cocharacter(t)} (\HamiltonianOrbit(t))$.
Seidel defined an isomorphism 
    \begin{equation}
    \label{eqn-intro:floer-seidel-map-non-equivariant-definition}
        \FloerSeidel (\Cocharacter, \FixedPoint) : \FloerCohomology^\ArbitraryIndex (\Manifold; \Hamiltonian) \to \FloerCohomology^{\ArbitraryIndex + |\Cocharacter, \FixedPoint|} (\Manifold; \Cocharacter \PullBack \Hamiltonian)
    \end{equation}
which corresponds to the pullback by this automorphism $(\TorusAction \ComposedWith \Cocharacter) : \ContractibleLoopSpace{\Manifold} \to \ContractibleLoopSpace{\Manifold}$ \cite{seidel_$_1997}.
We call this isomorphism the \define{Floer Seidel map}.
Seidel's construction applies to Hamiltonian $\Circle$-actions on closed symplectic manifolds.
Ritter extended the definition to \emph{linear} Hamiltonian $\Circle$-actions on convex symplectic manifolds \cite{ritter_floer_2014}.
The Floer Seidel map is an isomorphism of the underlying Floer cochain complexes, it preserves continuation maps, and it intertwines the pair-of-pants product.
Moreover, the Floer Seidel map satisfies
    \begin{equation}
    \label{eqn:non-equivariant-floer-seidel-map-composition-of-actions}
        \FloerSeidel (\Cocharacter, \FixedPoint) \ComposedWith \FloerSeidel (\secondCocharacter, \FixedPoint) = \FloerSeidel (\Cocharacter + \secondCocharacter, \FixedPoint)
    \end{equation}
for any two cocharacters $\Cocharacter, \secondCocharacter$.

We introduced the \define{$\Circle$-equivariant Floer Seidel map $\FloerSeidel_{\Circle} (\Cocharacter, \FixedPoint)$} in our previous work \cite{liebenschutz-jones_intertwining_2020}.
This map combines the identity map on $\ClassifyingSpace \Circle$ with the Floer Seidel map on $\Manifold$.
Like the non-equivariant Floer Seidel map, $\FloerSeidel_{\Circle} (\Cocharacter, \FixedPoint)$ is an isomorphism of the underlying equivariant Floer cochain complexes, it preserves continuation maps, and it satisfies \eqref{eqn:non-equivariant-floer-seidel-map-composition-of-actions}.
The $\Circle$-action on $\Manifold$ changes under the $\Circle$-equivariant Floer Seidel map, in a similar way to how the Hamiltonian $\Hamiltonian$ changes to the pullback Hamiltonian $\Cocharacter \PullBack \Hamiltonian$.

The $\Circle$-equivariant construction readily extends to the $\ExtendedTorus$-equivariant setup.
We assume the $\Circle$-action $\TorusAction \ComposedWith \Cocharacter : \Circle \times \Manifold \to \Manifold$ is linear for all cocharacters $\Cocharacter \in \LatticeCocharacters{\Torus}$.
We define the \define{$\ExtendedTorus$-equivariant Floer Seidel map}
    \begin{equation}
    \label{eqn-intro:torus-equivariant-floer-seidel-map-definition}
        \FloerSeidel_{\ExtendedTorus} (\Cocharacter, \FixedPoint) : \FloerCohomology^\ArbitraryIndex_{\ExtendedTorus} (\Manifold, \Extended{\TorusAction}, \Slope; \Hamiltonian) \to \FloerCohomology^{\ArbitraryIndex + |\Cocharacter, \FixedPoint|}_{\ExtendedTorus} (\Manifold, \Cocharacter \cdot \Extended{\TorusAction}, \Slope - \CocharacterSlope{\Cocharacter}; \Cocharacter \PullBack \Hamiltonian)
    \end{equation}
analogously to the $\Circle$-equivariant Floer Seidel map.
The $\ExtendedTorus$-action $\Cocharacter \cdot \Extended{\TorusAction}$ is determined by $(\Cocharacter \cdot \Extended{\TorusAction}) \RestrictedTo{\Torus} = \TorusAction$ and $(\Cocharacter \cdot \Extended{\TorusAction}) \RestrictedTo{\AdditionalCircle} = (\TorusAction \ComposedWith \Cocharacter) \Inverse$.
Let $\Extended{\Cocharacter} : \ExtendedTorus \to \ExtendedTorus$ be the automorphism $(\eltAdditionalCircle, \eltTorus) \mapsto (\eltAdditionalCircle, \eltTorus + \Cocharacter(\eltAdditionalCircle))$, so that $\Cocharacter \cdot \Extended{\TorusAction}$ is given by\footnote{
    The map $\LatticeCocharacters{\Torus} \to \Automorphisms(\ExtendedTorus)$ given by $\Cocharacter \mapsto \Extended{\Cocharacter}$ is a $(\LatticeCocharacters{\Torus})$-action on $\ExtendedTorus$.
    The induced $(\LatticeCocharacters{\Torus})$-action on maps $\Extended{\TorusAction} : \ExtendedTorus \to \Automorphisms(\Manifold)$ naturally precomposes with the \emph{inverse} of the $(\LatticeCocharacters{\Torus})$-action on $\ExtendedTorus$.
} $\Extended{\TorusAction} \ComposedWith \Extended{\Cocharacter} \Inverse$.
By the functorial properties of equivariant cohomology, associated to $\Extended{\Cocharacter}$ is an automorphism $(\ClassifyingSpace \Extended{\Cocharacter}) \PullBack : \Cohomology^\ArbitraryIndex (\ClassifyingSpace \ExtendedTorus) \to \Cohomology^\ArbitraryIndex (\ClassifyingSpace \ExtendedTorus)$ and an isomorphism
    \begin{equation}
    \label{eqn-intro:equivariant-cohomology-pullback-map-for-group-homomorphism}
        (\ClassifyingSpace \Extended{\Cocharacter}) \PullBack : \Cohomology^\ArbitraryIndex_{\ExtendedTorus} (\Manifold, \Cocharacter \cdot \Extended{\TorusAction}) \to \Cohomology^\ArbitraryIndex_{\ExtendedTorus} (\Manifold, \Extended{\TorusAction}).
    \end{equation}
The map \eqref{eqn-intro:equivariant-cohomology-pullback-map-for-group-homomorphism} is not an $\Cohomology^\ArbitraryIndex (\ClassifyingSpace \ExtendedTorus)$-module homomorphism and instead satisfies 
    \begin{equation}
    \label{eqn-intro:twisted-equation-for-equivariant-cohomology-pullback-map}
        (\ClassifyingSpace \Extended{\Cocharacter}) \PullBack (f x) = (\ClassifyingSpace \Extended{\Cocharacter}) \PullBack (f) \  (\ClassifyingSpace \Extended{\Cocharacter}) \PullBack (x)
    \end{equation}
for $f \in \Cohomology^\ArbitraryIndex (\ClassifyingSpace \ExtendedTorus)$ and $x \in \Cohomology^\ArbitraryIndex_{\ExtendedTorus} (\Manifold, \Cocharacter \cdot \Extended{\TorusAction})$.
There is a similar map $(\ClassifyingSpace \Extended{\Cocharacter}) \PullBack$ on $\ExtendedTorus$-equivariant Floer cohomology which satisfies \eqref{eqn-intro:twisted-equation-for-equivariant-cohomology-pullback-map}, and it `undoes' the change of $\ExtendedTorus$-action in \eqref{eqn-intro:torus-equivariant-floer-seidel-map-definition}.

If the slope $\CocharacterSlope{\Cocharacter} \ge 0$ of the linear action $\TorusAction \ComposedWith \Cocharacter$ is nonnegative, then the slope of the pullback Hamiltonian $\Cocharacter \PullBack \Hamiltonian$, given by $\Slope - \CocharacterSlope{\Cocharacter}$, is smaller than the slope $\Slope$ of $\Hamiltonian$.
Denote by $\NonnegativeLatticeCocharacters{\TorusAction}{\Torus}$ the set of cocharacters $\Cocharacter$ for which $\TorusAction \ComposedWith \Cocharacter$ is linear of nonnegative slope.
It is a commutative monoid inside the lattice of cocharacters $\LatticeCocharacters{\Torus}$.
For $\Cocharacter \in \NonnegativeLatticeCocharacters{\TorusAction}{\Torus}$, the continuation map $\ContinuationMap^{\CocharacterSlope{\Cocharacter}}$ is well-defined and `undoes' the change of slope.

Putting this together, the composition
    \begin{equation}
    \label{eqn-intro:shift-operator-definition-on-floer}
        \ShiftOperator^\FixedPoint_{\Cocharacter} = \ContinuationMap^{\CocharacterSlope{\Cocharacter}} \ComposedWith (\ClassifyingSpace \Extended{\Cocharacter}) \PullBack \ComposedWith \FloerSeidel_{\ExtendedTorus} (\Cocharacter, \FixedPoint) :
        \FloerCohomology^\ArbitraryIndex_{\ExtendedTorus} (\Manifold, \Extended{\TorusAction}, \Slope) \to 
        \FloerCohomology^{\ArbitraryIndex + |\Cocharacter, \FixedPoint|}_{\ExtendedTorus} (\Manifold, \Extended{\TorusAction}, \Slope)
    \end{equation}
is an endomorphism of $\ExtendedTorus$-equivariant Floer cohomology for all $\Cocharacter \in \NonnegativeLatticeCocharacters{\TorusAction}{\Torus}$.
The maps $\FloerSeidel_{\ExtendedTorus} (\Cocharacter, \FixedPoint)$ and $\ContinuationMap^{\CocharacterSlope{\Cocharacter}}$ are $\Cohomology^\ArbitraryIndex (\ClassifyingSpace \ExtendedTorus)$-module homomorphisms, but $(\ClassifyingSpace \Extended{\Cocharacter}) \PullBack$ satisfies \eqref{eqn-intro:twisted-equation-for-equivariant-cohomology-pullback-map}.
Therefore $\ShiftOperator^\FixedPoint _\Cocharacter$ is not a $\Cohomology^\ArbitraryIndex (\ClassifyingSpace \ExtendedTorus)$-module homomorphism, and instead satisfies
    \begin{equation}
    \label{eqn-intro:twisting-of-shift-operator-floer}
        \ShiftOperator^\FixedPoint _{\Cocharacter} (f x) = (\ClassifyingSpace \Extended{\Cocharacter}) \PullBack (f) \ \ShiftOperator^\FixedPoint _{\Cocharacter} (x)
    \end{equation}
for $f \in \Cohomology^\ArbitraryIndex (\ClassifyingSpace \ExtendedTorus)$.

\begin{theorem}
    [Flat shift operator]
    For every cocharacter $\Cocharacter$ in the commutative monoid $\NonnegativeLatticeCocharacters{\TorusAction}{\Torus}$, there is a $\NovikovRing$-module endomorphism
        \begin{equation}
        \label{eqn-intro:shift-operator-on-equivariant-floer-cohomology}
            \ShiftOperator^\FixedPoint_{\Cocharacter} :
            \FloerCohomology^\ArbitraryIndex_{\ExtendedTorus} (\Manifold, \Extended{\TorusAction}, \Slope) \to 
            \FloerCohomology^{\ArbitraryIndex + |\Cocharacter, \FixedPoint|}_{\ExtendedTorus} (\Manifold, \Extended{\TorusAction}, \Slope)
        \end{equation}
    which satisfies \eqref{eqn-intro:twisting-of-shift-operator-floer} and makes the following diagram commute.
        \begin{equation}
        \label{eqn:commutative-diagram-for-shift-operator-relating-to-nonequivariant}
            \begin{tikzcd}[column sep=huge]
                \FloerCohomology^\ArbitraryIndex _{\ExtendedTorus} (\Manifold, \Extended{\TorusAction}, \Slope)
                \arrow[r, "\ShiftOperator^\FixedPoint_{\Cocharacter}"]
                \arrow[d]
                & \FloerCohomology^{\ArbitraryIndex + |\Cocharacter, \FixedPoint|} _{\ExtendedTorus} (\Manifold, \Extended{\TorusAction}, \Slope)
                \arrow[d] \\
                \FloerCohomology^\ArbitraryIndex (\Manifold, \Slope)
                \arrow[r, "{\ContinuationMap^{\CocharacterSlope{\Cocharacter}} \ComposedWith \FloerSeidel(\Cocharacter, \FixedPoint)}"]
                & \FloerCohomology^{\ArbitraryIndex + |\Cocharacter, \FixedPoint|} (\Manifold, \Slope)
            \end{tikzcd}
        \end{equation}
    In \eqref{eqn:commutative-diagram-for-shift-operator-relating-to-nonequivariant}, the vertical maps are induced by the restriction map $\Cohomology^\ArbitraryIndex (\ClassifyingSpace \ExtendedTorus) \to \Cohomology^\ArbitraryIndex (\textup{pt})$ and the bottom map is the composition of a (non-equivariant) continuation map with \eqref{eqn-intro:floer-seidel-map-non-equivariant-definition}.
    Moreover, we have $\ShiftOperator^\FixedPoint _{\Cocharacter} \ShiftOperator^\FixedPoint _{\secondCocharacter} = \ShiftOperator^\FixedPoint _{\Cocharacter + \secondCocharacter}$ for any two cocharacters $\Cocharacter, \secondCocharacter \in \NonnegativeLatticeCocharacters{\TorusAction}{\Torus}$.
\end{theorem}

The map $\ShiftOperator^\FixedPoint$ is called a \define{shift operator}.
The automorphism $\Extended{\Cocharacter} : \ExtendedTorus \to \ExtendedTorus$ induces a `shift' in the lattice of characters of $\ExtendedTorus$.
This `shift' is represented in the algebra $\Cohomology^\ArbitraryIndex (\ClassifyingSpace \ExtendedTorus)$ by the isomorphism $(\ClassifyingSpace \Extended{\Cocharacter}) \PullBack$.
The operator $\ShiftOperator^\FixedPoint _\Cocharacter$ captures this `shift' in the $\Cohomology^\ArbitraryIndex (\ClassifyingSpace \ExtendedTorus)$-module $\FloerCohomology^\ArbitraryIndex_{\ExtendedTorus} (\Manifold, \Extended{\TorusAction}, \Slope)$ because it satisfies \eqref{eqn-intro:twisting-of-shift-operator-floer}.

The map \eqref{eqn-intro:shift-operator-on-equivariant-floer-cohomology} commutes with continuation maps, and hence induces a map $\ShiftOperator^\FixedPoint_{\Cocharacter}$ on $\ExtendedTorus$-equivariant symplectic cohomology $\SymplecticCohomology^\ArbitraryIndex_{\ExtendedTorus} (\Manifold, \Extended{\TorusAction})$ that satisfies \eqref{eqn-intro:twisting-of-shift-operator-floer}.
The continuation map in \eqref{eqn-intro:shift-operator-definition-on-floer} is absorbed by the continuation maps in the direct limit which defines $\SymplecticCohomology^\ArbitraryIndex_{\ExtendedTorus} (\Manifold, \Extended{\TorusAction})$.
This has two consequences for the shift operators on $\SymplecticCohomology^\ArbitraryIndex_{\ExtendedTorus} (\Manifold, \Extended{\TorusAction})$: they are isomorphisms and they are well-defined even for cocharacters $\Cocharacter$ of negative slope $\CocharacterSlope{\Cocharacter} < 0$.
The shift operator on $\FloerCohomology^\ArbitraryIndex_{\ExtendedTorus} (\Manifold, \Extended{\TorusAction}, \Slope)$ does not have these two properties in general.

\begin{corollary}
    [Shift operator representation]
    For every cocharacter $\Cocharacter \in \LatticeCocharacters{\Torus}$, we have a shift operator $\ShiftOperator^\FixedPoint_{\Cocharacter} : \SymplecticCohomology^\ArbitraryIndex_{\ExtendedTorus} (\Manifold, \Extended{\TorusAction}) \to \SymplecticCohomology^{\ArbitraryIndex + |\Cocharacter, \FixedPoint|}_{\ExtendedTorus} (\Manifold, \Extended{\TorusAction})$ which is a $\NovikovRing$-module automorphism satisfying \eqref{eqn-intro:twisting-of-shift-operator-floer}.
    This yields a representation
        \begin{equation}
            \ShiftOperator^\FixedPoint : \LatticeCocharacters{\Torus} \to \Automorphisms_\NovikovRing (\SymplecticCohomology^\ArbitraryIndex_{\ExtendedTorus} (\Manifold, \Extended{\TorusAction}))
        \end{equation}
    of the commutative group of cocharacters of $\Torus$.
\end{corollary}

We have two collections of endomorphisms of $\ExtendedTorus$-equivariant Floer cohomology which naturally augment the module structures: the connection $\Connection^\FixedPoint$ augments the $\NovikovRing$-module structure and the shift operator $\ShiftOperator^\FixedPoint$ augments the $\Cohomology^\ArbitraryIndex (\ClassifyingSpace \ExtendedTorus)$-module structure.
The two module structures commute, inducing a $\NovikovRing \Tensor \Cohomology^\ArbitraryIndex (\ClassifyingSpace \ExtendedTorus)$-module structure on $\FloerCohomology^\ArbitraryIndex _{\ExtendedTorus} (\Manifold, \Extended{\TorusAction}, \Slope)$.
Our main theorem states that the augmentations of these module structures also commute.

\begin{theorem}
    [Flatness]
\label{thm-intro:connection-and-shift-operators-commute-floer}
    The maps $\ShiftOperator^\FixedPoint_\Cocharacter$ and $\Connection^\FixedPoint_\DegreeTwoCoclass$ commute on $\FloerCohomology^\ArbitraryIndex_{\ExtendedTorus} (\Manifold, \Extended{\TorusAction}, \Slope)$ for any $\DegreeTwoCoclass \in \Cohomology^2 (\Manifold)$ and any $\Cocharacter \in \NonnegativeLatticeCocharacters{\TorusAction}{\Torus}$.
    This also holds on $\SymplecticCohomology^\ArbitraryIndex_{\ExtendedTorus} (\Manifold, \Extended{\TorusAction})$ for any $\Cocharacter \in \LatticeCocharacters{\Torus}$.
\end{theorem}

A key step in our proof of \autoref{thm-intro:connection-and-shift-operators-commute-floer} is a Morse homotopy argument which describes how $\FloerSeidel_{\ExtendedTorus} (\Cocharacter, \FixedPoint)$ and $\DegreeTwoCoclass^\FixedPoint \QuantumAction$ intertwine as operations.
This argument builds on our proof of the analogous \define{intertwining relation} in equivariant quantum cohomology, \autoref{thm:interwining-relation-introduction}, which describes how the equivariant quantum Seidel map and the equivariant quantum product $\DegreeTwoCoclass \QuantumAction$ intertwine \cite[Section~7.4]{liebenschutz-jones_intertwining_2020}.

\subsection{\texorpdfstring{$\ExtendedTorus$}{T}-equivariant quantum cohomology}

The quantum Seidel map $\QuantumSeidel(\Cocharacter, \FixedPoint) : \QuantumCohomology^\ArbitraryIndex (\Manifold) \to \QuantumCohomology^{\ArbitraryIndex + |\Cocharacter, \FixedPoint|} (\Manifold)$ was developed by Seidel \cite{seidel_$_1997}.
It counts pseudoholomorphic sections of a clutching bundle over $\Projective^1$ with fibre $\Manifold$.
It is compatible with the Floer Seidel map under PSS isomorphisms for closed $\Manifold$ \cite[Section~8]{seidel_$_1997}.
The analogous compatibility result for convex $\Manifold$ involves a continuation map to undo the change in slope \cite[page~1039]{ritter_floer_2014}, and as such $\QuantumSeidel(\Cocharacter, \FixedPoint)$ is not an isomorphism in general.
The quantum Seidel map intertwines the quantum product and satisfies
    \begin{equation}
    \label{eqn:non-equivariant-quantum-seidel-map-composition-of-actions}
        \QuantumSeidel (\Cocharacter, \FixedPoint) \ComposedWith \QuantumSeidel (\secondCocharacter, \FixedPoint) = \QuantumSeidel (\Cocharacter + \secondCocharacter, \FixedPoint)
    \end{equation}
for any two cocharacters $\Cocharacter, \secondCocharacter$.

Equivariant quantum Seidel maps were initially developed by Maulik and Okounkov in their study of quiver varieties \cite[Section~8]{maulik_quantum_2019}.
In the construction of the $\ExtendedTorus$-equivariant quantum Seidel map $\QuantumSeidel_{\ExtendedTorus} (\Cocharacter, \FixedPoint)$, the group $\AdditionalCircle$ acts by rotating $\Projective^1$ while $\Torus$ acts fibrewise on the clutching bundle.
With a technique called \emph{virtual localisation}, they proved the \define{intertwining relation}, \autoref{thm:interwining-relation-introduction}, which describes how the equivariant quantum Seidel map interacts with the equivariant quantum product.
They interpreted this relation as the flatness of a difference-differential connection \cite[Section~1.4]{maulik_quantum_2019}.
Braverman, Maulik and Okounkov used equivariant quantum Seidel maps to derive the equivariant quantum product on the Springer resolution \cite{braverman_quantum_2011}.
Iritani used equivariant quantum Seidel maps to describe shift operators on the big equivariant quantum cohomology of toric varieties and proved flatness in this setting \cite{iritani_shift_2017}.

\begin{theorem}
    [Intertwining relation]
\label{thm:interwining-relation-introduction}
    The equation
    \begin{equation}
        \label{eqn:quantum-intertwining-seidel-equivariant-introduction}
        \QuantumSeidel_{\ExtendedTorus}(\Cocharacter, \FixedPoint) (x \QuantumProduct \alpha^+)
        -
        \QuantumSeidel_{\ExtendedTorus}(\Cocharacter, \FixedPoint) (x) \QuantumProduct \alpha^-
        =
        \uformal \  \WeightedQuantumSeidel_{\ExtendedTorus}(\Cocharacter, \FixedPoint, \alpha)(x)
    \end{equation}
    holds for all $x \in \QuantumCohomology^\ArbitraryIndex_{\ExtendedTorus}(\Manifold, \Extended{\TorusAction})$.
    Here, $\alpha^+ \in \Cohomology^2_{\ExtendedTorus}(\Manifold, \Extended{\TorusAction})$ and $\alpha^- \in \Cohomology^2_{\ExtendedTorus} (\Manifold, \Cocharacter \cdot \Extended{\TorusAction})$ are two equivariant cohomology classes that are related via the clutching bundle and 
    \begin{equation}
        \label{eqn:error-term-description-introduction}
        \WeightedQuantumSeidel_{\ExtendedTorus}(\Cocharacter, \FixedPoint, \alpha) : \QuantumCohomology^\ArbitraryIndex_{\ExtendedTorus}(\Manifold, \Extended{\TorusAction}) \to \QuantumCohomology^{\ArbitraryIndex + |\Cocharacter, \FixedPoint|}_{\ExtendedTorus}(\Manifold, \Cocharacter \cdot \Extended{\TorusAction})
    \end{equation}
    is a weighted version of the $\ExtendedTorus$-equivariant quantum Seidel map defined in \eqref{eqn:weighted-equivariant-quantum-seidel-map-definition}.
\end{theorem}

The algebrogeometric technique of virtual localisation does not work in equivariant Floer cohomology.
In \cite{liebenschutz-jones_intertwining_2020}, we gave a new Morse-theoretic proof of the intertwining relation which does not use virtual localisation.
This proof gives the relation as the boundary of a 1-dimensional moduli space which is made up of a configuration on $\UniversalSpace \Circle = \InfiniteSphere$ and a configuration on $\Manifold$ which depends on the configuration on $\UniversalSpace \Circle$.
We wanted the configuration on $\Manifold$ to depend on an element of $\Circle$, but there is no globally-defined equivariant map $\UniversalSpace \Circle \to \Circle$.
A key insight in our proof was that a map $\UniversalSpace \Circle \to \Circle$ defined on an open dense $\Circle$-invariant \emph{proper} subset of $\UniversalSpace \Circle$ was sufficient.
This approach gives rise to an additional boundary component in which the configuration on $\UniversalSpace \Circle$ leaves the proper subset.
We use this insight in many of the moduli spaces in this paper (with $\UniversalSpace \ExtendedTorus$ replacing $\UniversalSpace \Circle$).
For example, in our construction of the operation $\DegreeTwoCoclass^\FixedPoint \QuantumAction$ on $\FloerCohomology^\ArbitraryIndex_{\ExtendedTorus} (\Manifold, \Extended{\TorusAction}, \Slope)$, the intersection of  $\DegreeTwoCoclass^\FixedPoint$ and the Floer solution $\FloerSolution : \RealNumbers \times \Circle \to \Manifold$ occurs at $\FloerSolution (0, t_0)$ for an element $t_0 \in \Circle$ determined by the configuration in $\UniversalSpace \ExtendedTorus$.

The $\ExtendedTorus$-equivariant quantum cohomology $\QuantumCohomology^\ArbitraryIndex _{\ExtendedTorus} (\Manifold, \Extended{\TorusAction})$ of $\Manifold$ is equipped with a connection $\Connection^\FixedPoint$ and a shift operator $\ShiftOperator^\FixedPoint$, given by the composition $\ShiftOperator^\FixedPoint _\Cocharacter = (\ClassifyingSpace \Extended{\Cocharacter}) \PullBack \ComposedWith \QuantumSeidel_{\ExtendedTorus} (\Cocharacter, \FixedPoint)$.
The flatness result that $\Connection^\FixedPoint_\DegreeTwoCoclass$ and $\ShiftOperator^\FixedPoint_\Cocharacter$ commute on $\QuantumCohomology^\ArbitraryIndex _{\ExtendedTorus} (\Manifold, \Extended{\TorusAction})$ is equivalent to the intertwining relation (see the proof of \autoref{thm:overall-flatness-of-difference-differential-connection-in-equivariant-quantum-cohomology}).
We prove the intertwining relation in $\ExtendedTorus$-equivariant quantum cohomology using a $\ExtendedTorus$-equivariant version of our Morse-theoretic proof.
Under $\ExtendedTorus$-equivariant PSS isomorphisms, which identify $\QuantumCohomology^\ArbitraryIndex _{\ExtendedTorus} (\Manifold, \Extended{\TorusAction})$ with $\FloerCohomology^\ArbitraryIndex _{\ExtendedTorus} (\Manifold, \Extended{\TorusAction}, \Slope)$ for small slopes $\Slope > 0$, the connections and shift operators on these different modules are identified.

Seidel conjectured that his $q$-connection corresponds to the equivariant quantum connection under the equivariant Floer Seidel map and the PSS isomorphism \cite[Remark~5.6]{seidel_connections_2018}.
We confirm the analogous result holds in our setup.
Let $\Cocharacter$ be a cocharacter which induces a $\Circle$-action of strictly positive slope $\CocharacterSlope{\Cocharacter} > 0$.
The $\ExtendedTorus$-equivariant Floer Seidel map is an isomorphism
    \begin{equation}
        \FloerSeidel_{\ExtendedTorus} (\Cocharacter, \FixedPoint) :
        \FloerCohomology^{\ArbitraryIndex}_{\ExtendedTorus} (\Manifold, \Extended{\TorusAction}, \CocharacterSlope{\Cocharacter} + \epsilon) \to
        \FloerCohomology^{\ArbitraryIndex + |\Cocharacter, \FixedPoint|}_{\ExtendedTorus} (\Manifold, \Cocharacter \cdot \Extended{\TorusAction}, \epsilon)
    \end{equation}
which preserves the connection $\Connection^\FixedPoint$.
For small $\epsilon > 0$, the right-hand side is isomorphic to $\QuantumCohomology^\ArbitraryIndex_{\ExtendedTorus} (\Manifold, \Cocharacter \cdot \Extended{\TorusAction})$ via the PSS isomorphism, and the PSS isomorphism preserves the connection also.
Therefore, moving to the $\Circle$-equivariant setup, the isomorphism
    \begin{equation}
        \PSSisomorphism \ComposedWith\FloerSeidel_{\Circle} (\Cocharacter, \FixedPoint) :
        \FloerCohomology^{\ArbitraryIndex}_{\Circle} (\Manifold, \CocharacterSlope{\Cocharacter} + \epsilon) \to
        \QuantumCohomology^{\ArbitraryIndex + |\Cocharacter, \FixedPoint|}_{\Circle} (\Manifold, \TorusAction \ComposedWith (-\Cocharacter))
    \end{equation}
preserves the connection.
The $\Circle$-action on $\Circle$-equivariant Floer cohomology rotates loops with no pointwise action on $\Manifold$, meanwhile the $\Circle$-action on $\Circle$-equivariant quantum cohomology is\footnote{
    For the Borel quotient in the construction of equivariant cohomology, we use the antidiagonal action \eqref{eqn:relation-on-borel-space} whereas Seidel used the diagonal action.
    The minus sign in $\TorusAction \ComposedWith (-\Cocharacter)$ is not present in Seidel's conjecture, but it arises merely as a consequence of this convention.
} $\TorusAction \ComposedWith (-\Cocharacter)$.

\subsection{Toric manifolds}

A \define{toric manifold} $\Manifold$ is a symplectic manifold equipped with an effective Hamiltonian $\Torus$-action with $\Dimension \Torus = \frac12 \Dimension \Manifold$.
It is determined up to equivariant symplectomorphism by its \define{moment polytope}, a convex polytope in $\RealNumbers^{\Dimension \Torus}$.
The combinatorial data describing the moment polytope can be used to compute many invariants of the corresponding toric manifold, including its quantum cohomology \cite[Proposition~5.2]{mcduff_topological_2006} and its $\Torus$-equivariant cohomology \cite[Theorem~12.4.14]{cox_toric_2011}.
We combine these two results to find a presentation for the $\ExtendedTorus$-equivariant quantum cohomology (\autoref{thm:equivariant-quantum-cohomology-presentation-closed-toric-manifold}).
We use the combinatorial data to compute the connection $\Connection^\FixedPoint$ and the shift operator $\ShiftOperator^\FixedPoint$ in this presentation (\autoref{sec:connection-computation-on-closed-toric-manifold}).

\begin{theorem}
    For toric manifolds, the shift operator $\ShiftOperator_{\Cocharacter}^\FixedPoint$ on $\QuantumCohomology^\ArbitraryIndex_{\ExtendedTorus}(\Manifold, \Extended{\TorusAction})$ is determined by its value at 1 via the analogues to \eqref{eqn-intro:twisting-of-shift-operator-floer} and \autoref{thm-intro:connection-and-shift-operators-commute-floer}.
    Both of these conditions may be expressed combinatorially using the moment polytope.
    Moreover, we provide a simple combinatorial recipe to compute $\ShiftOperator_{\Cocharacter}^\FixedPoint (1)$ from the polytope data.
\end{theorem}

Ritter showed that the moment polytope of a convex toric manifold $\Manifold$ gives rise to a neat presentation for $\SymplecticCohomology^\ArbitraryIndex (\Manifold)$ \cite[Theorem~1.5(2)]{ritter_circle-actions_2016}.
The $\ExtendedTorus$-equivariant version of Ritter's argument does not produce a presentation for $\ExtendedTorus$-equivariant symplectic cohomology, but we can nonetheless use it to deduce properties of $\SymplecticCohomology^\ArbitraryIndex _{\ExtendedTorus} (\Manifold, \Extended{\TorusAction})$.
In particular, for $\Manifold = \LineBundle_{\Projective^1}(-1)$, we deduce that there are elements $d_i \in \Cohomology^2 (\ClassifyingSpace \ExtendedTorus)$ for $i \in \Integers_{> 0}$ such that every element of $\SymplecticCohomology^\ArbitraryIndex _{\ExtendedTorus} (\Manifold, \Extended{\TorusAction})$ can be written in the form
    \begin{equation}
        \frac{x}{d_1 \cdots d_r}
    \end{equation}
for $x \in \QuantumCohomology^\ArbitraryIndex_{\ExtendedTorus} (\Manifold, \Extended{\TorusAction})$ and $r > 0$.
The elements $d_i$ arise as the determinants of the shift operators, and they are determined by the Reeb dynamics of the sphere bundle $\SphereBundle \LineBundle_{\Projective^1}(-1)$.
While it is beyond the scope of this paper to verify, we conjecture that this relationship between the determinants of the shift operators and the Reeb dynamics holds in general (\autoref{conjecture:significance-of-denominators}).

\subsection{Notation}

In \cite{liebenschutz-jones_intertwining_2020}, we only looked at $\Circle$-equivariant constructions, denoting $\Circle$-equivariant invariants with an $E$ ($E\QuantumSeidel$, $E\QuantumCohomology^\ArbitraryIndex$, etc.).
In this paper, we consider $\Circle$-equivariant, $\Torus$-equivariant and $\ExtendedTorus$-equivariant constructions, so we specify the group in this paper's notation ($\QuantumSeidel_{\Circle}$, $\QuantumCohomology^\ArbitraryIndex_{\Torus}$, etc.).

The Novikov ring in this paper records classes $\DegreeTwoClass \in \Homology_2 (\Manifold)$ (definition in \autoref{sec:novikov-ring}), whereas the Novikov ring in our previous paper recorded only $\FirstChernClass(\TangentSpace \Manifold, \SymplecticForm)(\DegreeTwoClass)$ and $[\SymplecticForm](\DegreeTwoClass)$ \cite[Section~3.2.6]{liebenschutz-jones_intertwining_2020}.
We use a fixed point $\FixedPoint$ to construct a `reference section' of the clutching bundle, replacing the \emph{lifted $\Circle$-actions} $\widetilde{\sigma}$ from our previous paper.
For the $q$-connection, Seidel assumed $\FirstChernClass(\TangentSpace \Manifold, \SymplecticForm) = 0$ and his Novikov ring recorded only $[\SymplecticForm](\DegreeTwoClass)$ \cite[Section~2a]{seidel_connections_2018}.

We introduce the manifold $\Manifold$ in \autoref{sec:assumptions-on-symplectic-manifold}, the torus $\Torus$ and the action $\TorusAction$ in \autoref{sec:torus-and-torus-action-on-manifold}, and the extended torus $\ExtendedTorus$ in \autoref{sec:extended-torus}.

\subsection{Outline}

\begin{description}
    \item[\autoref{sec:abstract-connections}]
    We introduce connections and shift operators as algebraic structures. 
    We discuss how the terminology is inspired by bundles.
    
    \item[\autoref{sec:equivariant-quantum-cohomology}]
    We list the assumptions on $\Manifold$ and $\TorusAction$.
    We construct the connection and shift operator on $\ExtendedTorus$-equivariant quantum cohomology.
    We prove flatness in \autoref{thm:overall-flatness-of-difference-differential-connection-in-equivariant-quantum-cohomology}.
    
    \item[\autoref{sec:equivariant-floer-cohomology-all-content}]
    We construct $\ExtendedTorus$-equivariant Floer cohomology together with its connection and shift operator.
    We prove flatness in \autoref{thm:flatness-of-difference-differential-connection-on-floer-cohomology}.
    
    \item[\autoref{sec:toric-manifolds-general}]
    We compute the $\ExtendedTorus$-equivariant quantum cohomology of a toric manifold, together with its connection and shift operator.
    We demonstrate the methods for $\Projective^2$ and $\LineBundle_{\Projective^1}(-1)$.
    
    \item[\autoref{app:topological-proofs}] We postpone a few topological proofs to the appendix.
\end{description}

\subsection{Acknowledgements}

    The author would like to thank his supervisor Alexander Ritter for his guidance, support and ideas. 
    The author wishes to thank Paul Seidel for his suggestions.
    The author thanks Nicholas Wilkins and Filip Živanović for many constructive discussions and their feedback. 
    The author thanks Jack Smith for useful conversations.
    The author gratefully acknowledges support from the EPSRC grant EP/N509711/1. 
    This work will form part of his DPhil thesis.
    
\section{Connections as an algebraic structure}
\label{sec:abstract-connections}

    \subsection{Differential connections}
    
        \label{sec:differential-connections-definition}

        This viewpoint on connections is taken from \cite[Chapter~1]{koszul_lectures_1986}.
        It abstracts the definition of connections on vector bundles (see \autoref{eg:abstract-connection-on-vector-bundle-motivation}).
    
        Let $k$ be an integral domain.
        Let $A$ be a unital, commutative and associative algebra over $k$.
        A \define{derivation} or \define{vector field} on $A$ is a $k$-linear map $X : A \to A$ which satisfies the Leibniz rule
            \begin{equation}
                \label{eqn:Leibniz-rule-for-derivations}
                X(ab) = (Xa)b + a(Xb) \EndFullStop
            \end{equation}
        Denote by $\SpaceOfDerivations(A)$ the set of all derivations on $A$.
        It has the structure of an $A$-module with post-multiplication and the structure of a $k$-Lie algebra with the commutator $\LieBracket{X}{Y} = XY - YX$.
        
        Let $P$ be a unitary\footnote{
            In a \define{unitary} $A$-module $P$, the relation $1_A \cdot p = p$ holds for all $p \in P$.
        } $A$-module.
        A \define{(differential) connection on $P$} or a \define{derivation law in $P$} is an $A$-module map $\Connection : \SpaceOfDerivations(A) \to \Homomorphisms_k(P, P)$ which satisfies the Leibniz rule
            \begin{equation}
                \label{eqn:Leibniz-rule-for-abstract-connections}
                \Connection_X(ap) = (Xa)p + a\Connection_X(p) \EndFullStop
            \end{equation}
        The \define{curvature} of $\Connection$ is the $k$-linear map $\Curvature^\Connection : \SpaceOfDerivations(A) \times \SpaceOfDerivations(A) \to \Homomorphisms_A(P, P)$ given by
            \begin{equation}
                \label{eqn:curvature-definition}
                \Curvature^\Connection_{X, Y} = \Connection_X \Connection_Y - \Connection_Y \Connection_X - \Connection_{\LieBracket{X}{Y}} \EndFullStop
            \end{equation}
        A connection is \define{flat} if its curvature vanishes.
        
        \begin{example}
            [Canonical connections]
            The $A$-module $A$ has a \define{canonical connection} given by $\Connection_{X}(a) = Xa$.
            The free $A$-module $P$ with specified basis $\SetCondition{p_i}{i \in I}$ has a \define{canonical connection} given by $\Connection_{X}(\sum_i a_i p_i) = \sum_i (Xa_i) p_i$.
        \end{example}
        
        \begin{example}
            [Polynomial rings]
        \label{eg:differential-connection-on-polynomial-ring}
            Let $A = k [t]$.
            The standard differentiation operation $\partial_t : t^n \mapsto n t^{n-1}$ generates $\SpaceOfDerivations (k [t])$.
        \end{example}
        
        \begin{example}
            [Differentiable manifolds]
            \label{eg:abstract-connection-on-vector-bundle-motivation}
            Let $E \to B$ be a smooth vector bundle over a smooth manifold $B$.
            Let $k = \RealNumbers$, let $A = \SmoothFunctions(B)$ be the smooth $\RealNumbers$-valued functions on $B$, and let $P = \SmoothFunctions(B, E)$ be the smooth sections of the vector bundle $E$.
            The \emph{vector fields on $A$} are precisely the smooth vector fields on $B$, so $\SpaceOfDerivations (A) = \SmoothFunctions(B, \TangentSpace B)$.
            Any connection on $E \to B$, in the conventional sense, is a \emph{connection in $P$}.
        \end{example}
        
        \begin{example}
            [Non-existence]
            Let $A = k [t]$ and $P = \nicefrac{k [t]}{(t)}$.
            The module $P$ has no connection, since \eqref{eqn:Leibniz-rule-for-abstract-connections} yields
                \begin{equation}
                    0 + (t) = \Connection_{\partial_t}(t\cdot(1 + (t))) = (\partial_t t)(1 + (t)) + t \Connection_{\partial_t}((1 + (t))) = 1 + (t) \ne 0 + (t)
                \end{equation}
            for any connection $\Connection$.
        \end{example}
        
        \begin{lemma}
            [Flatness sufficient on a generating set]
            \label{lem:flatness-sufficient-on-generating-set-differential-connection}
            Suppose the set of vector fields $\Set{X_\alpha}_{\alpha \in I}$ generates $\SpaceOfDerivations (A)$ as an $A$-module.
            If $\Curvature^\Connection_{X_\alpha, X_\beta} = 0$ for all $\alpha, \beta \in I$, then the connection $\Connection$ is flat.
        \end{lemma}
        
        \begin{proof}
            This follows immediately from the Leibniz rules.
        \end{proof}
        
        \begin{definition}
            [Partial connection]
        \label{def:partial-differential-connection}
            In some cases, it is undesirable to define a connection $\Connection_X$ for all $X \in \SpaceOfDerivations (A)$.
            Instead, we might define the connection $\Connection^\SpaceOfDerivations_X$ for all $X$ in a given $k$-submodule $\SpaceOfDerivations \subseteq \SpaceOfDerivations (A)$.
            The connection $\Connection^\SpaceOfDerivations$ must still be $A$-linear, whenever this condition makes sense.
            That is, for all $a \in A$ and $X \in \SpaceOfDerivations$, the relation $\Connection^\SpaceOfDerivations _{a X} = a\Connection^\SpaceOfDerivations _{X}$ holds whenever $a X \in \SpaceOfDerivations$.
            Equivalently, $\Connection^\SpaceOfDerivations$ is the restriction of an $A$-module map $A \cdot \SpaceOfDerivations \to \Homomorphisms_k(P, P)$ which satisfies \eqref{eqn:Leibniz-rule-for-abstract-connections}.
            The \define{curvature} of $\Connection^\SpaceOfDerivations$, defined exactly as in \eqref{eqn:curvature-definition}, is well-defined if $\SpaceOfDerivations$ is a $k$-Lie subalgebra of $\SpaceOfDerivations(A)$.
        \end{definition}
    
    \subsection{Difference connections}
    
        \label{sec:difference-connections-definition}
        
        Let $S$ be any set.
        We will call the elements of $S$ the \define{fundamental directions}.
        
        A \define{shift operator} on $A$ is an assignment to each fundamental direction $\sigma \in S$ an algebra homomorphism $E_\sigma : A \to A$ such that $E_\sigma$ and $E_\tau$ commute for any $\sigma, \tau \in S$.
        The corresponding \define{difference operator} is the map $X^E : S \to \Homomorphisms_k (A, A)$ which assigns to each fundamental direction $\sigma$ the $k$-linear map $X^E_\sigma = E_\sigma - \Identity_A$.
        It satisfies a \define{deformed Leibniz rule}
            \begin{equation}
                \label{eqn:deformed-Leibniz-rule-for-difference-operators}
                X^E_\sigma(ab) = (X^E_\sigma a) b + (E_\sigma a) (X^E_\sigma b) \EndFullStop
            \end{equation}
        Note this rule is symmetric in $a$ and $b$ because $A$ is commutative.
        
        A \define{difference connection} or \define{discrete connection} on $P$ is a map $\Connection : S \to \Homomorphisms_k (P, P)$ which satisfies the \define{deformed Leibniz rule}
            \begin{equation}
                \label{eqn:deformed-Leibniz-rule-for-difference-connection}
                \Connection_\sigma (ap) = (X^E_\sigma a) p + (E_\sigma a)\Connection_\sigma (p) \EndFullStop
            \end{equation}
        The \define{curvature $\Curvature^\Connection : S \times S \to \Homomorphisms_k(P,P)$} is given by $\Curvature^\Connection_{\sigma, \tau} = \Connection_\sigma \Connection_\tau - \Connection_\tau \Connection_\sigma$, and $\Connection$ is \define{flat} if the curvature vanishes.
        Unlike the curvature of a differential connection, the codomain of $\Curvature^\Connection$ is not the $A$-module endomorphisms of $P$.
        Instead, we have the relation
            \begin{equation}
                \label{eqn:curvature-relation-for-difference-connection}
                \Curvature^\Connection_{\sigma, \tau} (ap) = (E_\sigma  E_\tau a) \  \Curvature^\Connection_{\sigma, \tau} (p)
            \end{equation}
        for all fundamental directions $\sigma$ and $\tau$.
        
        An equivalent formulation of difference connections is via shift operators.
        A \define{shift operator} $\ShiftOperator$ on $P$ is an assignment to each fundamental direction $\sigma \in S$ a $k$-linear map $\ShiftOperator_\sigma : P \to P$ which satisfies $\ShiftOperator_\sigma(ap) = (E_\sigma a) \ \ShiftOperator_\sigma(p)$.
        The map $\Connection^\ShiftOperator = \ShiftOperator - \Identity_P$ is a difference connection.
        The flatness of $\Connection^\ShiftOperator$ is equivalent to the commutativity of the maps $\ShiftOperator_\sigma$ and $\ShiftOperator_{\sigma\Prime}$ for any $\sigma, \sigma\Prime \in S$.
        
        Flat shift operators may also be defined on commutative monoids\footnote{
            A \define{commutative monoid} $M$ is a set with a commutative associative binary operation which has a two-sided identity element.
            It does not have an inverse operation, unlike a group.
            $(\Integers^{\ge 0}, +, 0)$ is a commutative monoid.
        } $M$.
        A \define{shift operator} on $A$ for $M$ is a monoid map $E : M \to \AlgebraHomomorphismSpace_k (A, A)$.
        A \define{flat shift operator} $\ShiftOperator$ on $P$ for $M$ is an assignment to each $\sigma \in M$ a $k$-linear map $\ShiftOperator_\sigma : P \to P$ which satisfies $\ShiftOperator_\sigma(a p) = (E_\sigma a) \ \ShiftOperator_\sigma(p)$ and $\ShiftOperator_{\sigma \sigma\Prime} = \ShiftOperator_\sigma \ShiftOperator_{\sigma\Prime}$ for all $\sigma, \sigma\Prime \in M$.
        A flat shift operator $\ShiftOperator$ defined for a set $S$ may be extended to the \emph{commutative monoid} $\DirectSum_{\sigma \in S} \Integers^{\ge 0} \langle \sigma \rangle$ generated by $S$.
        Explicitly, the shift operator assigned to the product $\sigma_1 \cdots \sigma_m$ is $\ShiftOperator_{\sigma_1} \cdots \ShiftOperator_{\sigma_m}$.
        
        \begin{example}
            [Lattice bundles]
        \label{eg:difference-connection-on-lattice-vector-bundle-motivation}
            Let $S$ be a finite set of linearly independent vectors in $\RealNumbers^n$ and let $L = \bigoplus_{\sigma \in S} \Integers \sigma$ be the lattice generated by $S$.
            Let $V$ be any $k$-module.
            Set $A = \operatorname{Map}(L, k)$, the $k$-algebra of all $k$-valued functions on $L$, and $P = \operatorname{Map}(L, V)$.
            Define $E_\sigma (a) (l) = a(l + \sigma)$, so the function $E_\sigma (a)$ is the composition of $a$ with a shift in $L$ by $\sigma$.
            Similarly, define $\ShiftOperator_\sigma (p) (l) = p(l + \sigma)$.
            The resulting difference connection is flat.
        \end{example}
        
        \begin{remark}
            [Comparison of \autoref{eg:abstract-connection-on-vector-bundle-motivation} and \autoref{eg:difference-connection-on-lattice-vector-bundle-motivation}]
            Take the setup of \autoref{eg:difference-connection-on-lattice-vector-bundle-motivation}.
            The directional derivative of $a : \RealNumbers^n \to \RealNumbers$ at 0 in the direction $\sigma$ is the limit of $\frac{1}{t} (a(t \sigma) - a(0))$ as $t \to 0$.
            The difference operator $X^E_\sigma$ should be considered as pre-limit version of the directional derivative, where without loss of generality we have scaled so that $t = 1$.
        \end{remark}
        
    \subsection{Difference-differential connections}
    
        \label{sec:difference-differential-connections-definition}
    
        Let $k$, $A$, $P$ and $S$ be as before.
        Let $\SpaceOfDerivations$ be a
        $k$-module of vector fields and let $E$ be a shift operator on $A$ with respect to $S$.
        A \define{difference-differential connection} on $P$ is a map $\Connection : S \DisjointUnion \SpaceOfDerivations \to \Homomorphisms_k (P, P)$ which is a difference connection on $S$ and a partial differential connection on $\SpaceOfDerivations$.
        
        Suppose $\SpaceOfDerivations$ is a $k$-Lie algebra and $\LieBracket{X}{X^E_\sigma} = 0$ for all $X \in \SpaceOfDerivations$ and all $\sigma \in S$.
        The \define{curvature} of $\Connection$ is the map $\Curvature^\Connection : (S \DisjointUnion \SpaceOfDerivations)^2 \to \Homomorphisms_k(P, P)$ given by \eqref{eqn:curvature-definition}, interpreting $\sigma$ as $X^E_\sigma$ in the formula where appropriate.
        Since shift operators in different fundamental directions commute, the commutator $\LieBracket{X^E_\sigma}{X^E_\tau}$ is zero, so that $\Curvature^\Connection$ coincides with the curvature of a difference connection when restricted to $S \times S$.
        The connection $\Connection$ is \define{flat} if $\Curvature^\Connection$ vanishes.
        
        Difference-differential connections may be described using shift operators.
        In this context, a \define{difference-differential connection} is a pair $(\ShiftOperator, \Connection)$ consisting of a shift operator $\ShiftOperator$ and a differential connection $\Connection$.
        
        \begin{lemma}
            [Flatness sufficient on a generating set]
            \label{lem:flatness-sufficient-on-generating-set-difference-differential-connection}
            Suppose the shift operators on $A$ are monomorphisms and that the set of vector fields $\mathcal{B} = \Set{X_\alpha}_{\alpha \in I} \subseteq \SpaceOfDerivations$ generates $A \cdot \SpaceOfDerivations$ as an $A$-module.
            If $\Curvature^\Connection = 0$ on $(S \DisjointUnion \mathcal{B})^2$, then the connection $\Connection$ is flat.
        \end{lemma}
        
        \begin{proof}
            \autoref{lem:flatness-sufficient-on-generating-set-differential-connection} shows that $\Curvature^\Connection$ vanishes on $\SpaceOfDerivations^2$.
            By the anti-symmetry and additivity of $\Curvature^\Connection$, it is sufficient to show $\Curvature^\Connection_{a X_\alpha, \sigma} = 0$ for any $a \in A$ such that $a X_\alpha \in \SpaceOfDerivations$.
            We have
                \begin{equation}
                    0 = a \LieBracket{X_\alpha}{X^E_\sigma} - \LieBracket{a X_\alpha}{X^E_\sigma} = X^E_\sigma(a) \ (E_\sigma X_\alpha) = E_\sigma ((X^E_\sigma(a)) X_\alpha) \EndComma
                \end{equation}
            from which we deduce $(X^E_\sigma(a)) X_\alpha = 0$ holds because $E_\sigma$ is injective.
            Hence we derive
                \begin{equation}
                    \begin{aligned}
                        \Curvature^\Connection_{a X_\alpha, \sigma} 
                            &= \Connection_{a X_\alpha} \Connection_\sigma - \Connection_\sigma (a \Connection_{X_\alpha}) \\
                            &= \Connection_{(E_\sigma a) X_\alpha} \Connection_\sigma - X^E_\sigma(a)  \Connection_{X_\alpha} - (E_\sigma a) \Connection_\sigma \Connection_{X_\alpha} \\
                            &= (E_\sigma a) \LieBracket{\Connection_{X_\alpha}}{\Connection_\sigma} -  \Connection_{X^E_\sigma(a) X_\alpha} \\
                            &= 0 \EndFullStop
                    \end{aligned}
                \end{equation}
        \end{proof}
        
        \begin{remark}
            [Grading]
            Suppose $A$ and $P$ are $\Integers$-graded.
            The $A$-module of derivations is naturally graded by the degree of the (homogeneous) derivations.
            A \define{degree-$d$} differential connection is thus a differential connection which is degree-$d$ as an $A$-module map $\SpaceOfDerivations(A) \to \Homomorphisms_k(P, P)$.
            Note that shift operators on $A$ are degree-0 maps since they preserve 1, and hence difference operators are degree-0 also.
            A difference connection is \define{degree-$d_S$} if $\Connection_\sigma$ is a degree-$d_S(\sigma)$ map for all $\sigma \in S$, where $d_S : S \to \Integers$ is a given map.
            These definitions may be combined to get a \define{degree-$(d_S, d)$} difference-differential connection.
        \end{remark}

\section{Equivariant quantum cohomology}
\label{sec:equivariant-quantum-cohomology}

\subsection{Assumptions on symplectic manifold}
\label{sec:assumptions-on-symplectic-manifold}

    Let $(\Manifold, \SymplecticForm)$ be a $2\dimManifold$-dimensional symplectic manifold.
    Assume $\Manifold$ is nonempty and connected for convenience.
    
    We constrain the possible bubbling configurations of pseudoholomorphic spheres by imposing that $\Manifold$ is \define{nonnegatively monotone}.
    This means that the implication
        \begin{equation}
            \FirstChernClass(\TangentSpace \Manifold, \SymplecticForm)(\DegreeTwoClass) < 0 \implies \SymplecticForm(\DegreeTwoClass) \le 0
        \end{equation}
    holds for all $\DegreeTwoClass \in \HomotopyGroup_2(\Manifold)$.
    It is sufficient that $\Manifold$ is monotone or that $\FirstChernClass(\TangentSpace \Manifold, \SymplecticForm)$ is zero.
    
    A \define{convex structure} on $\Manifold$ is a pair $(\ConvexCoordMap, \ContactManifold)$, where $\ContactManifold$ is a closed contact manifold with contact form $\ContactForm$ and $\ConvexCoordMap : \IntervalClosedOpen{1}{\Infinity} \times \ContactManifold \to \Manifold$ is a diffeomorphism onto its image, such that $\Manifold \setminus \ConvexCoordMap(\IntervalClosedOpen{1}{\Infinity} \times \ContactManifold)$ is a relatively compact subset of $\Manifold$ and the equation
        \begin{equation}
            \label{eqn:symplectic-form-exact-on-convex-end}
            \ConvexCoordMap \PullBack \SymplecticForm = \ExteriorDerivative (\RadialCoord \ContactForm) 
        \end{equation}
    holds on $\IntervalClosedOpen{1}{\Infinity} \times \ContactManifold$.
    The \define{convex end} (the image of $\ConvexCoordMap$) is captured by the notation $\Set{\RadialCoord \ge 1}$ using the \define{radial coordinate} $R \in \IntervalClosedOpen{1}{\Infinity}$.
    Similarly, its complement is given by the notation $\Set{R<1}$.
    
    Equation \eqref{eqn:symplectic-form-exact-on-convex-end} states that the symplectic form $\SymplecticForm$ is exact on the convex end, however this does not imply that $\SymplecticForm$ is exact on all of $\Manifold$.
    We allow $\ContactManifold = \emptyset$ so that closed symplectic manifolds have a convex structure and are included in our analysis.
    
    The Reeb periods of the contact manifold $(\ContactManifold, \ContactForm)$ are the periods of the closed paths along the Reeb vector field $\ReebVectorField$. ($\ReebVectorField$ is determined by $\ExteriorDerivative \alpha(\ReebVectorField, \Argument) = 0$ and $\alpha(\ReebVectorField)=1$.)
    Denote by $\ReebPeriods$ the set of all Reeb periods (including 0 if $\ContactManifold$ is nonempty).
    We assume that $\Manifold$ has a convex structure $(\ConvexCoordMap, \ContactManifold)$ for which $\RealNumbers \setminus \ReebPeriods$ is unbounded.
    This unboundedness assumption ensures symplectic cohomology can be defined.
    
    Finally, to facilitate lifting cohomology classes to equivariant cohomology, we further assume that $\Manifold$ is \emph{simply connected}.
    Our work involves a torus $\Torus$ of dimension $\dimTorus$ acting on $\Manifold$ (the action is introduced in \autoref{sec:torus-and-torus-action-on-manifold}).
    For results with a trivial $\Torus$-action (or indeed with $\dimTorus = 0$), we do not need to make the assumption that $\Manifold$ is simply connected.
    In this case, the lifting is straightforward, and our results hold as stated without the simply connectedness assumption.
    
    Henceforth, $\Manifold$ will be a nonnegatively monotone and simply connected symplectic manifold with a convex structure $(\ConvexCoordMap, \ContactManifold)$ for which $\RealNumbers \setminus \ReebPeriods$ is unbounded.
    
\subsection{Novikov ring}
\label{sec:novikov-ring}

    The Novikov ring is an algebraic tool for recording the homology class of pseudoholomorphic spheres.
    We use the following variation.
    
    A degree-$l$ element of the Novikov ring $\NovikovRing$ is a formal sum of monomials $\NovVariable^\DegreeTwoClass$ with exponents $\DegreeTwoClass \in \Homology_2(\Manifold)$ and coefficients in $\Integers$, subject to the grading requirement that $2 \FirstChernClass(\TangentSpace\Manifold, \SymplecticForm)(\DegreeTwoClass) = l$ and the symplectic energy requirement that, for any $c \in \RealNumbers$, there are only finitely-many supported $\DegreeTwoClass$ with $\SymplecticForm(\DegreeTwoClass) \le c$.
    We use sum notation, even though elements can be infinite sums.
    The \define{Novikov ring} is the direct sum of its homogeneous parts.
    
\subsection{Morse cohomology}
\label{sec:morse-cohomology}

    Fix a Riemannian metric on $\Manifold$.
    Let $\MorseFunction : \Manifold \to \RealNumbers$ be a Morse-Smale function which increases in the radial direction, so that $\partial_\RadialCoord \MorseFunction(\ConvexCoordMap(\eltContactManifold, \RadialCoord)) > 0$ holds at infinity.
    Such Morse-Smale functions are \define{convex}.
    Let $\CriticalPointSet{\MorseFunction, \Manifold}$ denote the set of critical points of $\MorseFunction$.
    The \define{Morse index} $|\critManifold|$ of a critical point $\critManifold \in \CriticalPointSet{\MorseFunction, \Manifold}$ is the dimension of the maximal subspace of the tangent space at $\critManifold$ on which the Hessian of $\MorseFunction$ is negative definite.
    
    The \define{Morse cochain complex} is $\Integers \langle \CriticalPointSet{\MorseFunction, \Manifold} \rangle$, and the \define{Morse differential} counts the negative gradient trajectories between critical points.
    Explicitly, given two critical points $\critManifold^\pm \in \CriticalPointSet{\MorseFunction, \Manifold}$, we consider the moduli space
        \begin{equation}
            \ModuliSpace(\critManifold^-, \critManifold^+) = \SetConditionBar{\flowManifold : \RealNumbers \to \Manifold}{\begin{aligned}
                \partial_s(\flowManifold(s)) &= (- \Gradient \MorseFunction)_{\flowManifold(s)} \\
                \flowManifold(\pm \Infinity) &= \critManifold^\pm
            \end{aligned}}.
        \end{equation}
    It has dimension $|\critManifold^-| - |\critManifold^+|$ and an $\RealNumbers$-action given by $s$-translation.
    The differential is given by
        \begin{equation}
        \label{eqn:morse-differential}
            d \critManifold^+ =         \sum_{\substack{
                \critManifold^- \in \CriticalPointSet{\MorseFunction, \Manifold}
                \\ |\critManifold^-| = |\critManifold^+| + 1
            }} \Count \left( \frac{\ModuliSpace(\critManifold^-, \critManifold^+)}{\RealNumbers} \right) \ \critManifold^-,
        \end{equation}
    where $\Count (\Argument)$ denotes the signed count of a closed 0-dimensional manifold.
    We denote the \define{Morse cohomology} of $\Manifold$ by $\Cohomology^\ArbitraryIndex(\Manifold; \MorseFunction)$ to emphasise the fact that Morse cohomology recovers ordinary cohomology and does not depend on $\MorseFunction$.
    
    We suppress orientation information from our notation.
    A choice of orientation for each unstable manifold $\UnstableManifold(\critManifold)$ is nonetheless required to ensure the moduli spaces of trajectories may be consistently oriented.
    
    For convenience, a \define{flowline} is a map $\flowManifold : \RealNumbers \to \Manifold$ satisfying
        \begin{equation}
        \label{eqn:morse-flowline-gradient}
            \partial_s(\flowManifold(s)) = (- \Gradient \MorseFunction)_{\flowManifold(s)}
        \end{equation}
    with finite energy $\int \Norm{ \partial_s \flowManifold }^2 $.
    The finite energy requirement ensures that a flowline converges to critical points of $\MorseFunction$ as $s \to \pm \Infinity$.
    A map $\flowManifold^- : \IntervalOpenClosed{-\Infinity}{0} \to \RealNumbers$ satisfying \eqref{eqn:morse-flowline-gradient} with finite energy is a \define{half\,\textsuperscript{$-$} flowline} and a map $\flowManifold^+ : \IntervalClosedOpen{0}{\Infinity} \to \Manifold$ satisfying \eqref{eqn:morse-flowline-gradient} with finite energy is a \define{half\,\textsuperscript{$+$} flowline}.
    % The \, spaces are because the terms are in italics and the spacing looks odd.
    
\subsection{Quantum product}

    Let $\ACS$ be a regular $\SymplecticForm$-compatible almost complex structure on $\Manifold$, so that the moduli space $\ModuliSpace(\DegreeTwoClass; \ACS)$ of $\ACS$-holomorphic maps $u : \Projective^1 \to \Manifold$ with $u\PushForward [\Projective^1] = \DegreeTwoClass$ has dimension $2 \dimManifold + 2 \FirstChernClass(\TangentSpace\Manifold, \SymplecticForm)(\DegreeTwoClass)$, for any $\DegreeTwoClass \in \Homology_2(\Manifold)$.
    We moreover assume that $\ACS$ is \define{convex}, which means that $- \ExteriorDerivative \RadialCoord \ComposedWith \ACS = \RadialCoord \ContactForm$ holds at infinity, using the coordinates provided by $\ConvexCoordMap$.
    This compatibility with the convex structure allows us to prove compactness results for the moduli spaces $\ModuliSpace(\DegreeTwoClass; \ACS)$ using the maximum principle \cite[Appendix~D]{ritter_topological_2013}.
    
    The Morse cup product counts `Y'-shaped graphs, and the (Morse) quantum product is a deformation of the Morse cup product.
    We define it as follows.
    
    Let $\pointSphere^-_0, \pointSphere^+_1, \pointSphere^+_2 \in \Projective^1$ be fixed distinct points.
    A \define{deformed `Y'-shaped graph} is a quadruple $(\flowManifold^-_0, \flowManifold^+_1, \flowManifold^+_2, u)$ where $\flowManifold^\pm_i$ are half\textsuperscript{$\pm$} flowlines and $u : \Projective^1 \to \Manifold$ is a $\ACS$-holomorphic map that satisfies $u(\pointSphere^\pm_i) = \flowManifold^\pm_i (0)$.
    Given the critical points $\critManifold^-_0, \critManifold^+_1, \critManifold^+_2 \in \CriticalPointSet{\MorseFunction, \Manifold}$ and the class $\DegreeTwoClass \in \Homology_2(\Manifold)$, the moduli space  $\ModuliSpace(\critManifold^-_0, \critManifold^+_1, \critManifold^+_2, A)$ of deformed `Y'-shaped graphs that satisfy $\flowManifold^\pm_i(\pm \Infinity) = \critManifold^\pm_i$ and $u \in \ModuliSpace(\DegreeTwoClass; \ACS)$ is a smooth manifold of dimension
        \begin{equation}
        \label{eqn:dimension-of-deformed-y-shaped-flowline-moduli-space}
            |\critManifold^-_0| - |\critManifold^+_1| - |\critManifold^+_2| + 2 \FirstChernClass(\TangentSpace\Manifold, \SymplecticForm)(\DegreeTwoClass).
        \end{equation}
    
    Technically, we must use three $s$-dependent perturbations of the Morse function, one for each half flowline, to guarantee the moduli space is a smooth manifold.
    Such $s$-dependent perturbations may only depend on $s \in \RealNumbers$ in a compact interval.
    To keep our notation simple, we refer to flowlines that use a Morse function perturbed in this way as \define{perturbed} flowlines (though we keep the perturbations themselves implicit).
    
    \define{Quantum cohomology} $\QuantumCohomology^\ArbitraryIndex(\Manifold; \MorseFunction)$ is the cohomology of $\NovikovRing \Tensor \Integers \langle \CriticalPointSet{\MorseFunction, \Manifold} \rangle$ with the Morse differential \eqref{eqn:morse-differential}, so the $\NovikovRing$-module isomorphism $\QuantumCohomology^\ArbitraryIndex(\Manifold; \MorseFunction) \cong \NovikovRing \Tensor \Cohomology^\ArbitraryIndex(\Manifold; \MorseFunction)$ holds.
    Quantum cohomology is equipped with the \define{quantum product} $\QuantumProduct$, which is given by
        \begin{equation}
            \critManifold^+_1 \QuantumProduct \critManifold^+_2 = \sum_{\substack{
                \DegreeTwoClass \in \Homology_2(\Manifold) \\
                \critManifold^-_0 \in \CriticalPointSet{\MorseFunction, \Manifold} \\
                \eqref{eqn:dimension-of-deformed-y-shaped-flowline-moduli-space} = 0
            }} \Count \ModuliSpace(\critManifold^-_0, \critManifold^+_1, \critManifold^+_2, A) \ \NovVariable^\DegreeTwoClass \critManifold^-_0.
        \end{equation}
    The quantum product is graded-commutative, associative and unital.
    These facts are proved by standard homotopy arguments.
    
\subsection{Equivariant cohomology: the general case}
\label{sec:equivariant-cohomology-general}

    Let $G$ be a closed connected smooth Abelian Lie group ($G$ is always a torus in this paper).
    Let $\TorusAction : G \times X \to X$ be a continuous $G$-action on a topological space $X$.
    $G$-equivariant cohomology is a functorial invariant of the pair $(X, \TorusAction)$.
    It is constructed as follows.
    
    Let $\UniversalSpace G$ be a contractible space with a free $G$-action.
    The \define{classifying space} of $G$ is the quotient $\ClassifyingSpace G = \UniversalSpace G / G$ and the \define{universal bundle} of $G$ is the bundle $\UniversalSpace G \to \ClassifyingSpace G$.
    The cohomology $\Cohomology^\ArbitraryIndex(\ClassifyingSpace G)$ of the classifying space depends only on $G$.
    The \define{Borel homotopy quotient} $\UniversalSpace G \times_G X$ is the quotient of $\UniversalSpace G \times X$ by the relation
        \begin{equation}
        \label{eqn:relation-on-borel-space}
            (\eltUniversalSpace, g \cdot x) \sim (g \cdot \eltUniversalSpace, x).
        \end{equation}
    Equivalently, the Borel homotopy quotient can be obtained by taking the quotient of $\UniversalSpace G \times X$ by the antidiagonal action $g \cdot (\eltUniversalSpace, x) = (g \Inverse \cdot \eltUniversalSpace, g \cdot x)$.
    The \define{$G$-equivariant cohomology} of $(X, \rho)$ is the cohomology ring $\Cohomology^\ArbitraryIndex(\UniversalSpace G \times_G X)$, and it is denoted $\Cohomology^\ArbitraryIndex_G (X, \rho)$.
    The projection map $\UniversalSpace G \times_G X \to \ClassifyingSpace G$ induces a $\Cohomology^\ArbitraryIndex(\ClassifyingSpace G)$-algebra structure on $\Cohomology^\ArbitraryIndex_G (X, \rho)$.
    
    $G$-equivariant cohomology respects $G$-equivariant continuous maps between topological spaces with $G$-actions.
    Moreover, it is functorial with respect to group homomorphisms:
    let $\phi : G \to H$ be a Lie group homomorpism and let $\UniversalSpace \phi : \UniversalSpace G \to \UniversalSpace H$ be a continuous map\footnote{
        A map satisfying \eqref{eqn:equivariant-functoriality-map-relation} may be constructed as follows.
        Let $(\UniversalSpace G)\Prime$ and $\UniversalSpace H$ be arbitrary.
        Set $\UniversalSpace G = (\UniversalSpace G)\Prime \times \UniversalSpace H$, with the action $g \cdot (\eltUniversalSpace_G\Prime, \eltUniversalSpace_H) = (g \cdot \eltUniversalSpace_G\Prime, \phi(g) \cdot \eltUniversalSpace_H)$.
        The projection map $\UniversalSpace G \to \UniversalSpace H$ satisfies \eqref{eqn:equivariant-functoriality-map-relation}.
    } which satisfies 
        \begin{equation}
        \label{eqn:equivariant-functoriality-map-relation}
            (\UniversalSpace \phi) ( g \cdot \eltUniversalSpace) = \phi(g) \cdot (\UniversalSpace\phi) (\eltUniversalSpace).
        \end{equation}
    The induced map on the classifying spaces, $\ClassifyingSpace \phi : \ClassifyingSpace G \to \ClassifyingSpace H$, is well-defined and continuous, so it induces a map $(\ClassifyingSpace \phi)\PullBack : \Cohomology^\ArbitraryIndex (\ClassifyingSpace H) \to \Cohomology^\ArbitraryIndex (\ClassifyingSpace G)$ on cohomology.
    Suppose that $f : X \to Y$ is a continuous map which satisfies 
        \begin{equation}
        \label{eqn:equivariant-function-definition}
            f(g \cdot x) = \phi(g) \cdot f(x)
        \end{equation}
    with respect to a $G$-action $\TorusAction_X$ on $X$ and an $H$-action $\TorusAction_Y$ on $Y$.
    The map $(\UniversalSpace \phi, f) : \UniversalSpace G \times X \to \UniversalSpace H \times Y$ descends to the Borel homotopy quotients, and hence induces a map 
        \begin{equation}
        \label{eqn:functorial-map-on-equivariant-cohomology-general}
            (\UniversalSpace \phi, f)\PullBack : \Cohomology^\ArbitraryIndex_H (Y, \TorusAction_Y) \to \Cohomology^\ArbitraryIndex_G (X, \TorusAction_X).
        \end{equation}
    When $f$ is the identity map on $X = Y$, the map \eqref{eqn:functorial-map-on-equivariant-cohomology-general} is denoted $(\ClassifyingSpace \phi)\PullBack$, and when $\phi$ is the identity map on $G = H$, the map \eqref{eqn:functorial-map-on-equivariant-cohomology-general} is denoted $f \PullBack$.
    
    Associated to the fibre bundle $\UniversalSpace G \times_G X \to \ClassifyingSpace G$ with fibre $X$, there is a sequence
        \begin{equation}
        \label{eqn:sequence-on-cohomology-associated-to-borel-quotient}
            \Cohomology^\ArbitraryIndex (\ClassifyingSpace G) \to \Cohomology^\ArbitraryIndex_G (X, \TorusAction) \to \Cohomology^\ArbitraryIndex (X).
        \end{equation}
    The second map is induced by fibre inclusion, and it is well-defined on cohomology because $G$ and $\UniversalSpace G$ are path-connected.
    
\subsection{The torus}
\label{sec:torus-and-torus-action-on-manifold}

    Let $\Torus$ denote the $\dimTorus$-dimensional (real) torus, with $0 \le \dimTorus \le \dimManifold$.
    We always use the parameterisation $\Circle = \RealNumbers / \Integers$ for the circle.
    The \define{characters} of the torus are the homomorphisms $\Character : \Torus \to \Circle$, and the \define{cocharacters} are the homomorphisms $\Cocharacter : \Circle \to \Torus$.
    We denote the spaces of characters and cocharacters by $\LatticeCharacters{\Torus}$ and $\LatticeCocharacters{\Torus}$ respectively.
    Each space is a lattice isomorphic to $\Integers^\dimTorus$.
    There is a natural pairing $\LatticeCocharacters{\Torus} \times \LatticeCharacters{\Torus} \to \Integers$ which records the multiplicity of the composition $\Cocharacter \ComposedWith \Character : \Circle \to \Circle$; we denote it by $(\Cocharacter, \Character)$.
    
    A \define{basis} for $\Torus$ is a $\dimTorus$-tuple of cocharacters $\basisCocharacter = (\basisCocharacter_1, \ldots, \basisCocharacter_\dimTorus)$ such that $\basisCocharacter : (\Circle)^\dimTorus \to \Torus$ is an isomorphism.
    
    Let $\TorusAction : \Torus \times \Manifold \to \Manifold$ be a Hamiltonian $\Torus$-action on $\Manifold$, such that $\TorusAction \ComposedWith \Cocharacter$ is a linear Hamiltonian $\Circle$-action for all $\Cocharacter \in \LatticeCocharacters{\Torus}$.
    Recall that a Hamiltonian $\Circle$-action is \define{linear} if the corresponding Hamiltonian function is a linear function of the radial coordinate $\RadialCoord$ at infinity (i.e. for sufficiently large $\RadialCoord$).
    
    We assume the action $\TorusAction$ has a fixed point $\mu \in \Manifold$.
    The existence of a fixed point is guaranteed for $\dimTorus \le 1$ by \cite[Lemma~3.13 and Remark~3.14]{liebenschutz-jones_intertwining_2020}.
    The toric manifolds in \autoref{sec:toric-manifolds-general} have fixed points.
    
\subsection{Classifying space for the torus}
\label{sec:classifying-space-for-torus}

    A standard choice of $\UniversalSpace\Circle$ is the infinite sphere $\InfiniteSphere$, which is defined to be the limit of the odd-dimensional spheres $\Circle \Inclusion \HighDimensionalSphere{3} \Inclusion \cdots$ using embeddings as the unit spheres $\HighDimensionalSphere{2r-1} \subset \ComplexNumbers^r$.
    The standard rotation action on complex space restricts to the spheres, and induces a $\Circle$-action on $\InfiniteSphere$.
    The corresponding quotient $\InfiniteSphere / \Circle$ is the infinite complex projective space $\InfiniteComplexProjectiveSpace$.
    It follows that $\Cohomology^\ArbitraryIndex(\ClassifyingSpace \Circle) = \Cohomology^\ArbitraryIndex(\InfiniteComplexProjectiveSpace) = \Integers [\formalTorus]$ is a polynomial ring with $\formalTorus$ in degree 2.
    
    Given a character $\Character : \Torus \to \Circle$, we get a map $(\ClassifyingSpace \Character) \PullBack : \Cohomology^\ArbitraryIndex (\ClassifyingSpace \Circle) \to \Cohomology^\ArbitraryIndex (\ClassifyingSpace \Torus)$.
    Define $[\Character] = (\ClassifyingSpace \Character) \PullBack (y) \in \Cohomology^\ArbitraryIndex(\ClassifyingSpace\Torus)$.
    The map $\Character \mapsto [\Character]$ naturally extends to the symmetric algebra\footnote{
        The \define{symmetric algebra} $\SymmetricAlgebra(V)$ associated to the abelian group $V$ is $\Integers \DirectSum V \DirectSum \SymmetricAlgebra^2(V) \DirectSum \cdots$, where $\SymmetricAlgebra^2(V) = V \Tensor V / v\Tensor v\Prime \sim v\Prime \Tensor v$ corresponds to commutative products of two elements of $V$.
        Essentially, it is a way to write down a polynomial ring without choosing a basis for $V$.
    } $\SymmetricAlgebra (\LatticeCharacters{\Torus})$ and is an isomorphism of rings $\SymmetricAlgebra (\LatticeCharacters{\Torus}) \cong \Cohomology^\ArbitraryIndex (\ClassifyingSpace \Torus)$.
    
    More concretely, we can use a basis $\basisCocharacter$ for $\Torus$ to construct a universal bundle explicitly.
    Write $\Torus$ as a product $(\Circle)^\dimTorus$, using the basis $\basisCocharacter$.
    Set $\UniversalSpace \Torus = (\InfiniteSphere)^\dimTorus$, and equip it with the natural action coming from this product decomposition.
    If $\basisCocharacter\Dual_i$ is the character dual to $\basisCocharacter_i$ in the basis $\basisCocharacter$, and we set $\formalTorus_i = [\basisCocharacter\Dual_i]$, then we can write $\Cohomology^\ArbitraryIndex(\ClassifyingSpace \Torus) = \Integers [\formalTorus_1, \formalTorus_2, \ldots, \formalTorus_\dimTorus]$.
    The cohomology ring does not depend on the choice of $\basisCocharacter$ (\autoref{remark:morse-bott-construction-to-show-independence-of-torus-basis}).
    
    \begin{proposition}
    \label{prop:sequence-relating-cohomology-and-equivariant-cohomology-is-short-exact-sequence}
        The sequence \eqref{eqn:sequence-on-cohomology-associated-to-borel-quotient} is a short exact sequence in degree 2 for $(\Manifold, \TorusAction)$.
    \end{proposition}
    % This is so the lifting map in \autoref{sec:lifting-ordinary-to-equivariant-degree-two-coclasses} may be defined.
    
    \begin{proof}
        If the action $\TorusAction$ is trivial, or $\dimTorus = 0$ holds, the result is immediate.
        Otherwise, we use the assumption that $\Manifold$ is simply connected.
        Consider the following commutative diagram.
        The top row is the homotopy long exact sequence for the fibre bundle $\Manifold \to \UniversalSpace \Torus \times_{\Torus} \Manifold \to \ClassifyingSpace \Torus$, the bottom row is the homology version of \eqref{eqn:sequence-on-cohomology-associated-to-borel-quotient}, and the vertical maps are the Hurewicz homomorphisms.
            \begin{equation}
            \label{eqn:homotopy-exact-sequence-to-derive-homology-exact-sequence-on-borel-quotient}
                \begin{tikzcd}
                    \HomotopyGroup_3(\ClassifyingSpace \Torus) \arrow[r]
                    & \HomotopyGroup_2 (\Manifold) \arrow[r] \arrow[d]
                    & \HomotopyGroup_2 (\UniversalSpace \Torus \times_{\Torus} \arrow[r] \arrow[d] \Manifold)
                    & \HomotopyGroup_2 (\ClassifyingSpace \Torus) \arrow[r] \arrow[d]
                    & \HomotopyGroup_1 (\Manifold)
                    \\
                    & \Homology_2 (\Manifold) \arrow[r]
                    & \Homology_2 (\UniversalSpace \Torus \times_{\Torus} \Manifold) \arrow[r]
                    & \Homology_2 (\ClassifyingSpace \Torus)
                    &
                \end{tikzcd}
            \end{equation}
        The group $\HomotopyGroup_3(\ClassifyingSpace \Torus) = 0$ vanishes (this may be derived using the homotopy long exact sequence for the fibre bundle $\Circle \to \InfiniteSphere \to \InfiniteComplexProjectiveSpace$).
        The group $\HomotopyGroup_1 (\Manifold) = 0$ vanishes since $\Manifold$ is simply connected.
        Therefore the top row of \eqref{eqn:homotopy-exact-sequence-to-derive-homology-exact-sequence-on-borel-quotient} is a short exact sequence.
        Since $\HomotopyGroup_1 (\Manifold) = 0$ vanishes, the vertical Hurewicz homomorphisms are isomorphisms by the Hurewicz theorem.
        The proposition follows by applying the universal coefficient theorem to the short exact sequence in homology (the exactness of the sequence is preserved because $\Homology_2 (\ClassifyingSpace \Torus)$ is free).
    \end{proof}
    
\subsection{Equivariant Morse theory}
\label{sec:equivariant-morse-theory}

    Let us start by defining Morse data for the standard universal bundle for $\Circle$, namely $\InfiniteSphere \to \InfiniteComplexProjectiveSpace$.
    Elements of $\InfiniteSphere$ are vectors $(\eltUniversalSpace_0, \eltUniversalSpace_1, \ldots) \in \ComplexNumbers^\Infinity$ with unit norm (and only finitely-many nonzero entries).
    The function $\morseUniversalSpace_{\Circle} : \InfiniteSphere \to \RealNumbers$ given by $\morseUniversalSpace_{\Circle}(\eltUniversalSpace) = \sum_{r \ge 0} r |\eltUniversalSpace_r|^2$ descends to a Morse function on $\InfiniteComplexProjectiveSpace$.
    It has exactly one critical point in each even degree, in such a way that the degree-$2r$ critical point is the $r$-th standard basis vector.
    We use the round metric on $\InfiniteSphere$.
    It immediately follows that $\Cohomology^\ArbitraryIndex(\InfiniteComplexProjectiveSpace) = \Integers [\formalTorus]$, wherein $\formalTorus^r$ corresponds to the degree-$2r$ critical point.
    
    For the torus $\Torus$ with basis $\basisCocharacter$, we use the product decomposition to define Morse data for $(\InfiniteSphere)^\dimTorus \to (\InfiniteComplexProjectiveSpace)^\dimTorus$.
    Give $(\InfiniteSphere)^\dimTorus$ the product round metric and set the function $\morseUniversalSpace : (\InfiniteSphere)^\dimTorus \to \RealNumbers$ to be the sum of the functions $\morseUniversalSpace_{\Circle}$ on each copy of $\InfiniteSphere$.
    The function $\morseUniversalSpace$ descends to a Morse function on $(\InfiniteComplexProjectiveSpace)^\dimTorus$.
    
    The critical points in $(\InfiniteComplexProjectiveSpace)^\dimTorus$ are in bijection with $(\Integers_{\ge 0})^\dimTorus$, where the vector $\critClassifyingSpace = (\critClassifyingSpace_1, \ldots, \critClassifyingSpace_\dimTorus) \in (\Integers_{\ge 0})^\dimTorus$ corresponds to the critical point comprising of the $\critClassifyingSpace_i$-th unit vector in the $i$-th copy of $\InfiniteComplexProjectiveSpace$.
    We will use this notation to describe $\CriticalPointSet{\morseUniversalSpace, (\InfiniteComplexProjectiveSpace)^\dimTorus}$, and the notation $\formalTorus^\critClassifyingSpace$ for the corresponding class in $\Cohomology^\ArbitraryIndex ((\InfiniteComplexProjectiveSpace)^\dimTorus)$.
    For each critical point $\critClassifyingSpace \in \CriticalPointSet{\morseUniversalSpace, (\InfiniteComplexProjectiveSpace)^\dimTorus}$, there is a $\Torus$-orbit of points in $(\InfiniteSphere)^\dimTorus$, all of which are critical points of $\morseUniversalSpace : (\InfiniteSphere)^\dimTorus \to \RealNumbers$.
    Such critical points $\critUniversalSpace \in \CriticalPointSet{\morseUniversalSpace, (\InfiniteSphere)^\dimTorus}$ are identified by the notation $[\critUniversalSpace] = \critClassifyingSpace$.
    
    From now on, the notation $\UniversalSpace \Torus$ will mean $(\InfiniteSphere)^\dimTorus$, taken with respect to the basis $\basisCocharacter$, and $\ClassifyingSpace \Torus$ will mean $(\InfiniteComplexProjectiveSpace)^\dimTorus$.
    
    Fix a $\TorusAction$-invariant Riemannian metric on $\Manifold$.
    A \define{$\Torus$-equivariant Morse function} on $\Manifold$ is a function $\eqMorseFunction : \UniversalSpace \Torus \times \Manifold \to \RealNumbers$ which is $\Torus$-invariant, so $\eqMorseFunction(\eltTorus \Inverse \cdot \eltUniversalSpace, \TorusAction_{\eltTorus}(\eltManifold)) = \eqMorseFunction (\eltUniversalSpace, \eltManifold)$ holds for all $\eltTorus \in \Torus$; which is convex for all $\eltUniversalSpace \in \UniversalSpace \Torus$; and which satisfies the following condition about critical points.
    For every critical point $\critUniversalSpace \in \CriticalPointSet{\morseUniversalSpace, \UniversalSpace \Torus}$, the function $\eqMorseFunction(\critUniversalSpace, \Argument)$ is a convex Morse-Smale function on $\Manifold$, and moreover $\eqMorseFunction(\critUniversalSpace, \Argument) = \eqMorseFunction(\eltUniversalSpace, \Argument)$ for all $\eltUniversalSpace$ in a neighbourhood of $\critUniversalSpace$ in $\UnstableManifold(\critUniversalSpace) \Union \StableManifold(\critUniversalSpace)$.
    
    A \define{$\Torus$-equivariant critical point} of $\eqMorseFunction$ is a $\Torus$-equivalence class of points $(\critUniversalSpace, \critManifold) \in \UniversalSpace \Torus \times \Manifold$ where $\critUniversalSpace \in \CriticalPointSet{\morseUniversalSpace, \UniversalSpace \Torus}$ and $\critManifold \in \CriticalPointSet{\eqMorseFunction(\critUniversalSpace, \Argument), \Manifold}$ are critical points.
    Such $\Torus$-equivariant points are denoted $[\critUniversalSpace, \critManifold]$, and the set of $\Torus$-equivariant points is denoted $\eqCriticalPointSet{\eqMorseFunction, \Manifold}$.
    The \define{degree} $|\critUniversalSpace, \critManifold|$ of $[\critUniversalSpace, \critManifold]$ is simply the sum of the degrees of $\critUniversalSpace$ and $\critManifold$.
    
    A \define{$\Torus$-equivariant flowline} of $\eqMorseFunction$ is a $\Torus$-equivalence class of pairs $(\flowUniversalSpace, \flowManifold)$ such that $\flowUniversalSpace : \RealNumbers \to \UniversalSpace \Torus$ is a flowline of $\morseUniversalSpace$ and $\flowManifold : \RealNumbers \to \Manifold$ is a flowline of $\eqMorseFunction (\flowUniversalSpace(s), \Argument)$.
    Explicitly, $\flowManifold$ satisfies $\partial_s(\flowManifold(s)) = - (\Gradient \eqMorseFunction (\flowUniversalSpace(s), \Argument))_{\flowManifold(s)}$ and has finite energy.
    The flowline $\flowManifold$ depends on $s$ only on a compact interval of $\RealNumbers$ because $\eqMorseFunction(\flowUniversalSpace(\pm \infty), \Argument) = \eqMorseFunction(\flowUniversalSpace(s), \Argument)$ holds for all sufficiently large $\pm s \gg 0$.
    As such, the flowline $\flowManifold$ behaves analytically like a continuation map.
    The torus $\Torus$ acts on pairs $(\flowUniversalSpace, \flowManifold)$ by simultaneously acting on the codomains of the two maps, though it acts on $\UniversalSpace \Torus$ by the inverse of the usual action.
    This mimics the natural $\Torus$-action on maps $(\flowUniversalSpace, \flowManifold) : \RealNumbers \to \UniversalSpace \Torus \times \Manifold$ induced by the antidiagonal action on $\UniversalSpace \Torus \times \Manifold$.
    The space of $\Torus$-equivariant flowlines from $[\critUniversalSpace^-, \critManifold^-]$ to $[\critUniversalSpace^+, \critManifold^+]$ has dimension $|\critUniversalSpace^-, \critManifold^-| - |\critUniversalSpace^+, \critManifold^+|$.
    
    Exactly like the non-equivariant Morse cohomology definition in \autoref{sec:morse-cohomology}, the \define{$\Torus$-equivariant Morse cochain complex} is $\Integers \langle \eqCriticalPointSet{\eqMorseFunction, \Manifold} \rangle$ and the \define{$\Torus$-equivariant Morse differential} counts $\RealNumbers$-equivalence classes of $\Torus$-equivariant flowlines (see \autoref{fig:equivariant-morse-differential}).
    We denote the resulting \define{$\Torus$-equivariant Morse cohomology} by $\Cohomology^\ArbitraryIndex_{\Torus; \basisCocharacter} (\Manifold; \eqMorseFunction)$.
    
    The \define{(geometric\footnote{
        An alternative \define{algebraic module structure} is given by shifting coordinates in $\ClassifyingSpace \Torus$ (see \cite[Section~4.4.1]{liebenschutz-jones_intertwining_2020}).
    }) $\Cohomology^\ArbitraryIndex (\ClassifyingSpace \Torus)$-module structure} is given by counting `Y'-shaped graphs in $\UniversalSpace \Torus$ paired with flowlines in $\Manifold$ (see \autoref{fig:geometric-module-structure-on-equivariant-morse-cohomology}).
    Such a configuration is a $\Torus$-equivalence class of triples $(\flowUniversalSpace, \flowUniversalSpace_0, \flowManifold)$ where $[\flowUniversalSpace, \flowManifold]$ is a $\Torus$-equivariant flowline, and $\flowUniversalSpace_0$ is a perturbed half\textsuperscript{$+$} flowline of $\morseUniversalSpace$ satisfying $\flowUniversalSpace_0(0) = \flowUniversalSpace(0)$.
    For a critical point $\critClassifyingSpace \in \CriticalPointSet{\morseUniversalSpace, \ClassifyingSpace \Torus}$ and two $\Torus$-equivariant critical points $[\critUniversalSpace^\pm, \critManifold^\pm] \in \eqCriticalPointSet{\eqMorseFunction, \Manifold}$, let $\ModuliSpace([\critUniversalSpace^\pm, \critManifold^\pm]; \critClassifyingSpace)$ be the moduli space of such $[(\flowUniversalSpace, \flowUniversalSpace_0, \flowManifold)]$ which satisfy $[\flowUniversalSpace(\pm \Infinity), \flowManifold(\pm \Infinity)] = [\critUniversalSpace^\pm, \critManifold^\pm]$ and $[\flowUniversalSpace_0(+ \Infinity)] = \critClassifyingSpace$.
    The module structure is given by
        \begin{equation}
            \formalTorus^\critClassifyingSpace \cdot [\critUniversalSpace^+, \critManifold^+] = \sum_{\substack{[\critUniversalSpace^-, \critManifold^-] \\ \Dimension \ModuliSpace = 0}} \Count \ModuliSpace([\critUniversalSpace^\pm, \critManifold^\pm]; \critClassifyingSpace) \  [\critUniversalSpace^-, \critManifold^-].
        \end{equation}
        
    \begin{remark}
        [Dependence on the basis $\basisCocharacter$]
    \label{remark:morse-bott-construction-to-show-independence-of-torus-basis}
        Our construction of $\UniversalSpace \Torus$ uses the basis $\basisCocharacter$ of $\Torus$, but $\Torus$-equivariant cohomology does not depend on $\basisCocharacter$.
        While there are many ways to show independence of $\basisCocharacter$, not all approaches also allow us to show that the $\Torus$-equivariant quantum product \eqref{eqn:equivariant-quantum-product-definition} is independent of $\basisCocharacter$.
        Here we describe a Morse-theoretic construction which uses carefully-chosen Morse data to simplify the moduli spaces.
        In this remark, we are not concerned with independence of the Morse data on $\UniversalSpace \Torus$ because this can be demonstrated using standard homotopy arguments.
        
        Let $\basisCocharacter$ and $\basisCocharacter \Prime$ be two bases of $\Torus$.
        Let $E$ and $E\Prime$ denote the corresponding constructions of $\UniversalSpace \Torus$ (both spaces are $(\InfiniteSphere)^\dimTorus$, but they have different $\Torus$-actions).
        Let $\morseUniversalSpace$ and $\morseUniversalSpace\Prime$ denote the corresponding Morse functions.
        The product $E \times E\Prime$ with the product $\Torus$-action satisfies the criteria for $\UniversalSpace \Torus$: it is a contractible space with a free $\Torus$-action.
        It is sufficient to show $E$ and $E \times E\Prime$ yield isomorphic constructions of $\Torus$-equivariant cohomology.
        Set $\Projection_E : E \times E\Prime \to E$.
        The induced projection $(E \times E\Prime)/\Torus \to E/\Torus$ is a fibration with contractible fibre $E \Prime$.
        The function $\morseUniversalSpace + \morseUniversalSpace\Prime : (E \times E\Prime)/\Torus \to \RealNumbers$ is Morse-Bott with critical submanifolds isomorphic to $\Torus$.
        Pick a perfect Morse function on each critical submanifold.
        Given a critical point $\critUniversalSpace \in E$ of $\morseUniversalSpace$, denote by $\min(\critUniversalSpace) \in (\morseUniversalSpace \Prime)\Inverse (0) \subset E\Prime$ the unique point such that $[\critUniversalSpace, \min(\critUniversalSpace)]$ is the minimum on the critical submanifold $((\Torus \cdot \critUniversalSpace) \times (\morseUniversalSpace \Prime)\Inverse (0))/\Torus \subset (E \times E\Prime)/\Torus$.
        
        Fix a $\Torus$-equivariant Morse function $\eqMorseFunction_E : E \times \Manifold \to \RealNumbers$ from the construction of $\Torus$-equivariant cohomology which uses $E$.
        We can use the functions $\eqMorseFunction_E$ and $\morseUniversalSpace + \morseUniversalSpace\Prime$, together with the Morse functions on the critical submanifolds, to define a Morse-Bott construction of $\Torus$-equivariant cohomology which uses $E \times E\Prime$.
        Consider the map $\Projection\PullBack_E : \CriticalPointSet{E \times_{\Torus} \Manifold} \to \CriticalPointSet{(E \times E\Prime) \times_{\Torus} \Manifold}$ given by $\Projection\PullBack_E ([(\critUniversalSpace, \critManifold)]) = [((\critUniversalSpace, \min(\critUniversalSpace)), \critManifold)]$.
        The map $\Projection\PullBack_E$ induces a map on the Morse cochain complex, and we prove that $\Projection\PullBack_E$ is a chain map by computing the differential $d ([((\critUniversalSpace^+, \min(\critUniversalSpace^+)), \critManifold^+)])$.
        This differential counts isolated equivariant flowlines, and these are either flowlines that flow from $[((\critUniversalSpace^-, \min(\critUniversalSpace^-)), \critManifold^-)]$, which correspond to flowlines on $E \times_{\Torus} \Manifold$, or flowlines that are constant in  $\UniversalSpace \times \Manifold$, which correspond to the flowlines of the perfect Morse function on $\Torus \cdot \min(\critUniversalSpace^-)$.
        These latter flowlines cancel with each other because the Morse differential of a perfect Morse function on $\Torus$ is zero.
        Additionally, $\Projection\PullBack_E : \Cohomology^\ArbitraryIndex (E \times_{\Torus} \Manifold) \to \Cohomology^\ArbitraryIndex ((E \times E\Prime) \times_{\Torus} \Manifold)$ is an isomorphism.
        It is easiest to prove this by appealing to topological properties of pullback maps, using the functorial flowline construction below to assert $\Projection\PullBack_E$ is indeed a pullback map, however it is also possible to prove this directly by computing the differential on $(E \times E\Prime)/\Torus$.
        In any case, we have explicitly constructed a pullback map $\Projection\PullBack_E : \Cohomology^\ArbitraryIndex (E \times_{\Torus} \Manifold) \to \Cohomology^\ArbitraryIndex ((E \times E\Prime) \times_{\Torus} \Manifold)$ and found that it is an isomorphism, demonstrating the required independence.
        
        In general, pullback maps like $\Projection_E \PullBack$ count functorial flowlines\footnote{
            \label{footnote:functorial-flowline}%
            For the map $\phi : \Manifold^- \to \Manifold^+$, a \define{functorial flowline} is a pair of perturbed half\textsuperscript{$\pm$} flowlines $\flowManifold^\pm$ in $\Manifold^\pm$ which satisfy $\phi(\flowManifold^- (0)) = \flowManifold^+ (0)$.
            See \cite[Footnote~34]{liebenschutz-jones_intertwining_2020} or \cite[Section~1.3]{rot_functoriality_2014}.
        }, however in our setup the only isolated functorial flowlines for $\Projection_E \PullBack$ are constant in $E$ and $\Manifold$, and flow from the minimum in $E\Prime$, so $\Projection_E \PullBack$ is given by the above formula.
        Since the $\Torus$-equivariant Morse function does not depend on $E\Prime$, the two half flowlines on $\Manifold$ join to create a single flowline on $\Manifold$.
        It is this property that allows us to deduce that the $\Torus$-equivariant quantum product is independent of the basis $\basisCocharacter$.
    \end{remark}

\begin{figure}
\centering
\begin{subfigure}{0.3\textwidth}
    \centering
	\begin{center}
		\begin{tikzpicture}
			\node[inner sep=0] at (2,0) {\includegraphics[width=2.5 cm]{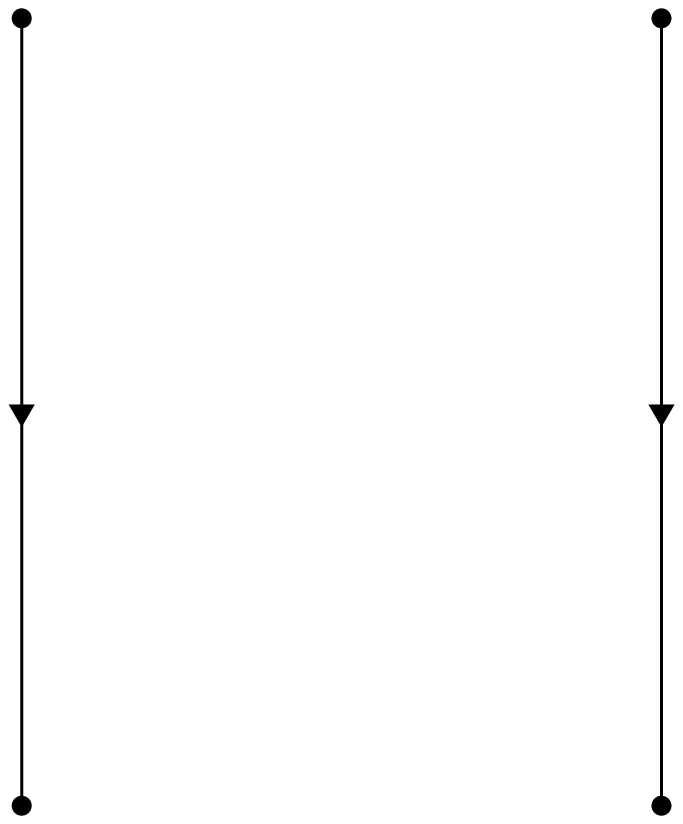}};
% 			\draw[step=1.0,black,thin] (0,-3) grid (4,3);
			
% 			\node at (0,0){0};
			\node at (0.7, 0.8){\tiny $\flowUniversalSpace$};
			\node at (0.9, 1.7){\tiny $\critUniversalSpace^-$};
			\node at (0.9, -1.6){\tiny $\critUniversalSpace^+$};
			\node at (3.2, 1.7){\tiny $\critManifold^-$};
			\node at (3.2, -1.6){\tiny $\critManifold^+$};
			\node at (3.0, 0.8){\tiny $\flowManifold$};
		\end{tikzpicture}
	\end{center}
    \caption{Equivariant differential} 
\label{fig:equivariant-morse-differential}
\end{subfigure}
\hspace{1em} 
\begin{subfigure}{0.3\textwidth}
  \centering
	\begin{center}
		\begin{tikzpicture}
			\node[inner sep=0] at (2,0) {\includegraphics[width=2.5 cm]{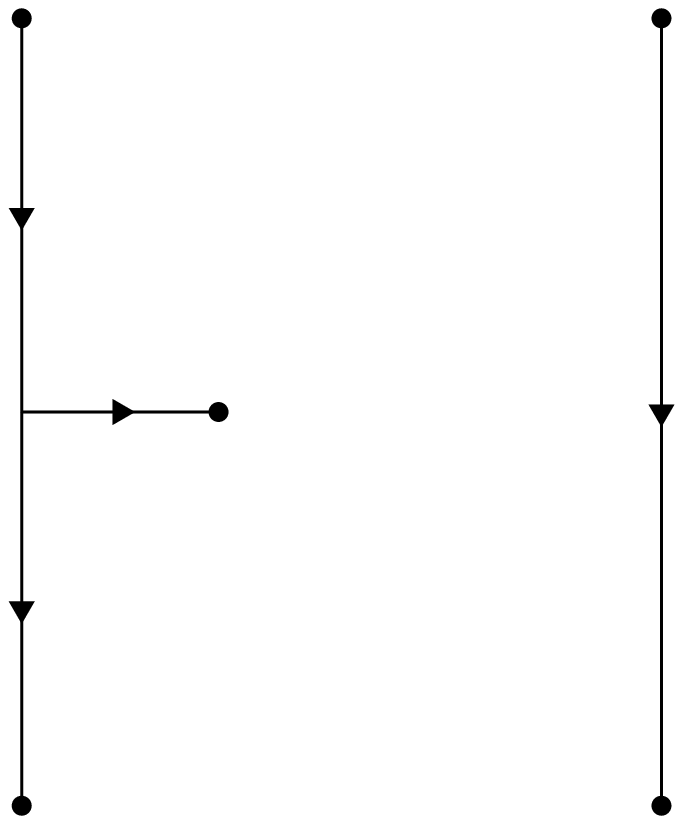}};
% 			\draw[step=1.0,black,thin] (0,-3) grid (4,3);
			
% 			\node at (0,0){0};
			\node at (0.6, 0.8){\tiny $\flowUniversalSpace$};
			\node at (1.3, 0.1){\tiny $\flowUniversalSpace_0$};
			\node at (0.9, 1.7){\tiny $\critUniversalSpace^-$};
			\node at (0.9, -1.6){\tiny $\critUniversalSpace^+$};
			\node at (3.2, 1.7){\tiny $\critManifold^-$};
			\node at (3.2, -1.6){\tiny $\critManifold^+$};
			\node at (3.0, 0.8){\tiny $\flowManifold$};
		\end{tikzpicture}
	\end{center}
    \caption{Geometric $\Cohomology^\ArbitraryIndex (\ClassifyingSpace \Torus)$-mod\-ule structure} 
    \label{fig:geometric-module-structure-on-equivariant-morse-cohomology}
\end{subfigure}
\hspace{1em}
\begin{subfigure}{0.3\textwidth}
    \centering
	\begin{center}
		\begin{tikzpicture}
			\node[inner sep=0] at (2,0) {\includegraphics[width=3 cm]{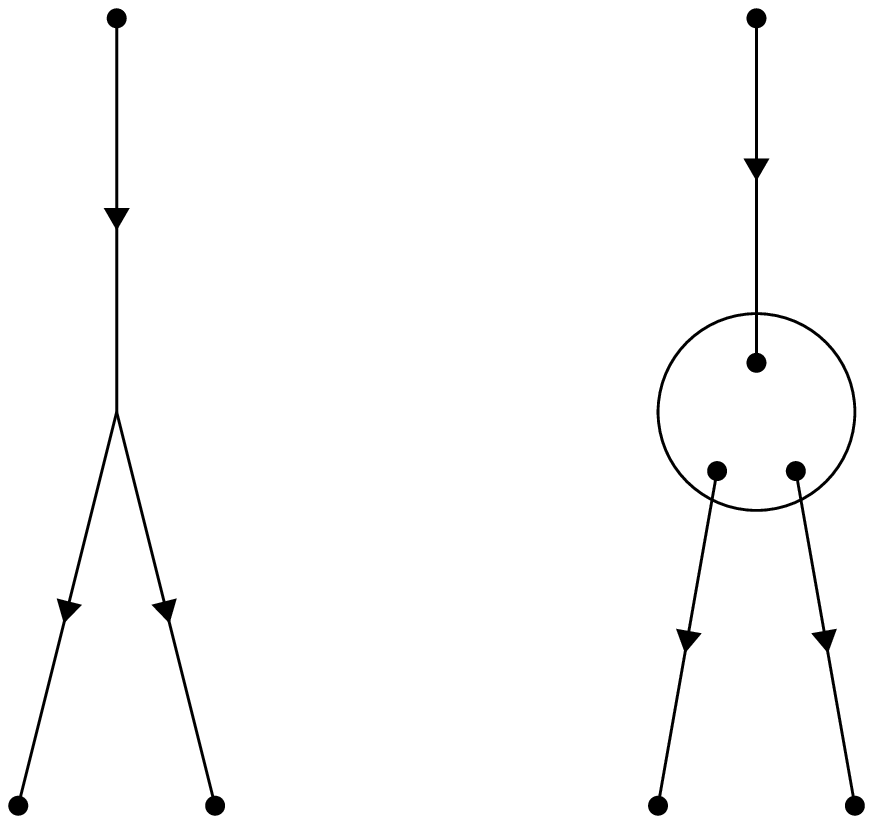}};
% 			\draw[step=1.0,black,thin] (0,-3) grid (4,3);
			
% 			\node at (0,0){0};
			\node at (0.7, 0.8){\tiny $\flowUniversalSpace^-_0$};
			\node at (0.5, -0.6){\tiny $\flowUniversalSpace^+_1$};
			\node at (1.3, -0.6){\tiny $\flowUniversalSpace^+_2$};
			\node at (0.9, 1.6){\tiny $\critUniversalSpace^-_0$};
			\node at (0.6, -1.5){\tiny $\critUniversalSpace^+_1$};
			\node at (1.3, -1.5){\tiny $\critUniversalSpace^+_2$};
			\node at (3.2, 1.6){\tiny $\critManifold^-_0$};
			\node at (2.8, -1.5){\tiny $\critManifold^+_1$};
			\node at (3.5, -1.5){\tiny $\critManifold^+_2$};
			\node at (2.9, 0.8){\tiny $\flowManifold^-_0$};
			\node at (2.7, -0.7){\tiny $\flowManifold^+_1$};
			\node at (3.6, -0.7){\tiny $\flowManifold^+_2$};
			\node at (2.7, 0.2){\tiny $u$};
		\end{tikzpicture}
	\end{center}
    \caption{Equivariant quantum product}
    \label{fig:equivariant-quantum-product}
\end{subfigure}
\caption{
    Our diagrams for equivariant configurations have two components.
    The left component describes the part that is mapped to $\UniversalSpace \Torus$ while the right component describes the part that is mapped to $\Manifold$.
    In general, our Morse flowlines flow downwards on the page, in the direction that the Morse function decreases.
    As such, the `Y'-shaped graphs are upside-down when compared to the letter `Y'.
} \label{fig:basic-equivariant-morse-theoretic-constructions}
\end{figure}
        
\subsection{Equivariant almost complex structures}
\label{sec:equivariant-almost-complex-structures}

    A \define{$\Torus$-equivariant almost complex structure} $\ACS^\eqnt$ is a choice of almost complex structure $\ACS^\eqnt_\eltUniversalSpace$ for each $\eltUniversalSpace \in \UniversalSpace \Torus$ such that the diagram
        \begin{equation}
        \label{eqn:equivariant-almost-complex-structure-commutative-diagram}
        \begin{tikzcd}[column sep=huge]
            \TangentSpace _\eltManifold \Manifold
            \arrow[r, "\ACS^\eqnt_{\eltUniversalSpace, \eltManifold}"]
            \arrow[d, "\Derivative \TorusAction_\eltTorus"]
            & \TangentSpace _\eltManifold \Manifold
            \arrow[d, "\Derivative \TorusAction_\eltTorus"] \\
            \TangentSpace _{\eltTorus \cdot \eltManifold} \Manifold
            \arrow[r, "\ACS^\eqnt_{\eltTorus\Inverse \cdot \eltUniversalSpace, \eltTorus \cdot \eltManifold}"]
            & \TangentSpace _{\eltTorus \cdot \eltManifold} \Manifold
        \end{tikzcd}
        \end{equation}
    commutes.
    This condition implies that the space of pairs $(\eltUniversalSpace, u)$ of elements $\eltUniversalSpace \in \UniversalSpace \Torus$ and $\ACS^\eqnt_\eltUniversalSpace$-holomorphic maps $u : \Projective^1 \to \Manifold$ inherits a natural $\Torus$-action given by $\eltTorus \cdot (\eltUniversalSpace, u) = (\eltTorus \Inverse \cdot \eltUniversalSpace, \TorusAction_\eltTorus \ComposedWith u)$.
    We will assume that our $\Torus$-equivariant almost complex structures are everywhere $\SymplecticForm$-compatible, regular and convex.
    Here, \define{convex} means there exists $\RadialCoord_0 \in \IntervalClosedOpen{1}{\Infinity}$ such that $- \ExteriorDerivative \RadialCoord \ComposedWith \ACS^\eqnt _{\eltUniversalSpace, \eltManifold} = \RadialCoord \ContactForm$ holds for all $\eltUniversalSpace \in \UniversalSpace \Torus$ and all $\eltManifold \in \Set{\RadialCoord \ge \RadialCoord_0}$.
    
    The regularity ensures that the space $\ModuliSpace(\DegreeTwoClass; \ACS^\eqnt, K)$ of pairs $(\eltUniversalSpace, u)$ with $\eltUniversalSpace \in K$ and $u \PushForward [\Projective^1] = \DegreeTwoClass$ is a manifold of dimension $2 \dimManifold + 2 \FirstChernClass(\TangentSpace \Manifold, \SymplecticForm)(\DegreeTwoClass) + \Dimension K$ for the $\Torus$-invariant compact subsets $K = (\HighDimensionalSphere{2r-1})^\dimTorus \subset \UniversalSpace \Torus$.

\subsection{Equivariant quantum cohomology}

    Analogously to non-equivariant quantum cohomology, the \define{$\Torus$-equivariant quantum cohomology} $\QuantumCohomology^\ArbitraryIndex_{\Torus; \basisCocharacter} (\Manifold, \TorusAction; \eqMorseFunction)$ of $\Manifold$ is the cohomology of $\NovikovRing \GradedCompletedTensorProduct \Integers \langle \eqCriticalPointSet{\eqMorseFunction, \Manifold} \rangle$ with the $\Torus$-equivariant Morse differential.
    Unlike the non-equivariant case, the set $\eqCriticalPointSet{\eqMorseFunction, \Manifold}$ is typically infinite, so we use the graded completed tensor product\footnote{
        The \define{graded completed tensor product} of two $\Integers$-graded $\Integers$-modules $A$ and $B$ is $A \GradedCompletedTensorProduct B = \DirectSum_{l \in \Integers} (A \GradedCompletedTensorProduct B)^l$, where the degree-$l$ summand is $(A \GradedCompletedTensorProduct B)^l = \DirectProduct_{l\Prime \in \Integers} A^{l\Prime} \Tensor B^{l - l\Prime}$.
        For us, any degree-$l$ element in the $\Integers$-module $\NovikovRing \GradedCompletedTensorProduct \Integers \langle \eqCriticalPointSet{\eqMorseFunction, \Manifold} \rangle$ has the property that, given $c \in \RealNumbers$ and $l\Prime, l\DoublePrime \in \Integers$, only finitely-many of the terms $\NovVariable^A \Tensor [\critUniversalSpace, \critManifold]$ with $|\NovVariable^A| = l\Prime$, $\SymplecticForm(A) \le c$ and $|[\critUniversalSpace]| \le l\DoublePrime$ are supported.
        This is an equivariant extension of the corresponding condition which defines the Novikov ring (see \autoref{sec:novikov-ring}).
    } $\GradedCompletedTensorProduct$ in place of the standard tensor product $\Tensor$.
    The $\NovikovRing$-module $\QuantumCohomology^\ArbitraryIndex_{\Torus; \basisCocharacter} (\Manifold, \TorusAction; \eqMorseFunction)$ is immediately a $\NovikovRing \GradedCompletedTensorProduct \Cohomology^\ArbitraryIndex (\ClassifyingSpace \Torus)$-module by combining the formal Novikov ring action with the geometric $\Cohomology^\ArbitraryIndex (\ClassifyingSpace \Torus)$-module structure.
    
    The $\Torus$-equivariant quantum product $\QuantumProduct$ gives $\QuantumCohomology^\ArbitraryIndex_{\Torus; \basisCocharacter} (\Manifold, \TorusAction; \eqMorseFunction)$ the structure of a $\NovikovRing \GradedCompletedTensorProduct \Cohomology^\ArbitraryIndex (\ClassifyingSpace \Torus)$-algebra.
    It is defined by counting `Y'-shaped graphs in $\UniversalSpace \Torus$ paired with deformed `Y'-shaped graphs in $\Manifold$ (see \autoref{fig:equivariant-quantum-product}).
    Explicitly, a \define{$\Torus$-equivariant deformed `Y'-shaped graph} is a $\Torus$-equivalence class of septuples $(\flowUniversalSpace^-_0, \flowUniversalSpace^+_1, \flowUniversalSpace^+_2, \flowManifold^-_0, \flowManifold^+_1, \flowManifold^+_2, u)$ where $[\flowUniversalSpace^\pm_i, \flowManifold^\pm_i]$ are perturbed $\Torus$-equivariant half\textsuperscript{$\pm$} flowlines and $u : \Projective^1 \to \Manifold$ is a $\eqACS_{\flowUniversalSpace^-_0(0)}$-holomorphic map subject to the conditions $\flowUniversalSpace^-_0(0) = \flowUniversalSpace^+_1(0) = \flowUniversalSpace^+_2(0)$ and $u(\pointSphere^\pm_i) = \flowManifold^\pm_i (0)$.
    Let $\ModuliSpace([\critUniversalSpace^\pm_i, \critManifold^\pm_i], \DegreeTwoClass)$ be the moduli space of such $\Torus$-equivariant deformed `Y'-shaped graphs with the obvious critical point conditions and $[\flowUniversalSpace^-_0(0), u] \in \eqModuliSpace(\DegreeTwoClass; \eqACS)$.
    The \define{$\Torus$-equivariant quantum product} is given by
        \begin{equation}
        \label{eqn:equivariant-quantum-product-definition}
            [\critUniversalSpace^+_1, \critManifold^+_1] \QuantumProduct [\critUniversalSpace^+_2, \critManifold^+_2] = \sum_{\substack{
                \DegreeTwoClass \in \Homology_2(\Manifold) \\
                [\critUniversalSpace^-_0, \critManifold^-_0] \in \eqCriticalPointSet{\eqMorseFunction, \Manifold} \\
                \Dimension \ModuliSpace = 0
            }} \Count \ModuliSpace([\critUniversalSpace^\pm_i, \critManifold^\pm_i], \DegreeTwoClass) \ \NovVariable^\DegreeTwoClass [\critUniversalSpace^-_0, \critManifold^-_0].
        \end{equation}
    Like the non-equivariant quantum product, the $\Torus$-equivariant quantum product is unital, graded-commutative and associative by standard homotopy arguments.
    Moreover, it is independent of the choice of $\Torus$-equivariant Morse function $\eqMorseFunction$ and the choice of $\Torus$-equivariant almost complex structure $\eqACS$.

\subsection{The extended torus}
\label{sec:extended-torus}

    Denote a copy of the circle $\Circle$ by $\AdditionalCircle$.
    Set $\ExtendedTorus = \AdditionalCircle \times \Torus$.
    The character $\ExtendedTorus \to \Circle$ which is a projection to $\AdditionalCircle$ is denoted $\Extended{\Character}_0$, and the corresponding variable in $\Cohomology^\ArbitraryIndex (\ClassifyingSpace \ExtendedTorus)$ is denoted $\formalAdditionalCircle$.
    We will use the notation $\Extended{\Argument}$ to denote actions, bases, (co)characters, etc. which correspond to $\ExtendedTorus$.
    
    We are interested in the $\ExtendedTorus$-equivariant quantum cohomology $\QuantumCohomology^\ArbitraryIndex _{\ExtendedTorus; \Extended{\basisCocharacter}} (\Manifold, \Extended{\TorusAction})$.
    Here, the $\ExtendedTorus$-action $\Extended{\TorusAction}$ on $\Manifold$ induced by $\TorusAction$ is uniquely determined by $\Extended{\TorusAction}\RestrictedTo{\Torus} = \TorusAction$ and $\Extended{\TorusAction}\RestrictedTo{\AdditionalCircle} = \Identity_\Manifold$.
    We use bases $\Extended{\basisCocharacter}$ of $\ExtendedTorus$ which are given by
        \begin{equation}
        \label{eqn:basis-rules-for-extended-torus}
            \begin{cases}
                \Extended{\basisCocharacter}_0 (s) = (s, \basisCocharacter_0(s)), \\
                \Extended{\basisCocharacter}_i (s) = (0, \basisCocharacter_i(s)),
            \end{cases}
        \end{equation}
    where $\basisCocharacter$ is a basis of $\Torus$ and $\basisCocharacter_0 : \Circle \to \Torus$ is a further cocharacter of $\Torus$.
    Such bases satisfy $(\Extended{\basisCocharacter}_i, \Extended{\Character}_0) = \KroneckerDelta_{i, 0}$, where $\KroneckerDelta$ is the Kronecker delta.
    Certain constructions use the basis $\Extended{\basisCocharacter}$, but the cohomological invariants are independent of the choice by \autoref{remark:morse-bott-construction-to-show-independence-of-torus-basis}.
    
\subsection{Lifting to equivariant quantum cohomology}
\label{sec:lifting-ordinary-to-equivariant-degree-two-coclasses}

    Let $\FixedPoint \in \Manifold$ be a fixed point.
    The inclusion map $\Set{\FixedPoint} \to \Manifold$ is $\ExtendedTorus$-equivariant, so there is an induced map $\FixedPoint \PullBack : \Cohomology^\ArbitraryIndex_{\ExtendedTorus} (\Manifold) \to \Cohomology^\ArbitraryIndex_{\ExtendedTorus}(\Set{\FixedPoint}) \cong \Cohomology^\ArbitraryIndex (\ClassifyingSpace \ExtendedTorus)$ as in \eqref{eqn:functorial-map-on-equivariant-cohomology-general}.
    The map $\FixedPoint \PullBack$ induces a splitting of the short exact sequence
        \begin{equation}
        \label{eqn:split-short-exact-sequence-borel-quotient}
            \begin{tikzcd}
                0
                \arrow[r]
                & \Cohomology^2 (\ClassifyingSpace \ExtendedTorus)
                \arrow[r]
                & \Cohomology^2_{\ExtendedTorus} (\Manifold, \Extended{\TorusAction})
                \arrow[l, bend right, "\FixedPoint \PullBack"']
                \arrow[r]
                & \Cohomology^2 (\Manifold)
                \arrow[l, bend right, dashed]
                \arrow[r]
                & 0.
            \end{tikzcd}
        \end{equation}
    The short exact sequence \eqref{eqn:split-short-exact-sequence-borel-quotient} is the degree-2 part of \eqref{eqn:sequence-on-cohomology-associated-to-borel-quotient} for $(\Manifold, \Extended{\TorusAction})$, and it is exact by \autoref{prop:sequence-relating-cohomology-and-equivariant-cohomology-is-short-exact-sequence}.
    The corresponding dashed map is 
        \begin{equation}
            \Cohomology^2 (\Manifold) \to \Cohomology^2_{\ExtendedTorus} (\Manifold, \Extended{\TorusAction}), \quad \DegreeTwoCoclass \mapsto \DegreeTwoCoclass^\FixedPoint \text{ for } \DegreeTwoCoclass \in \Cohomology^2 (\Manifold).
        \end{equation}
    % The corresponding dashed map $\Cohomology^2 (\Manifold) \to \Cohomology^2_{\ExtendedTorus} (\Manifold, \Extended{\TorusAction})$ is denoted $\DegreeTwoCoclass \mapsto \DegreeTwoCoclass^\FixedPoint$ for $\DegreeTwoCoclass \in \Cohomology^2 (\Manifold)$.
    
\subsection{Differential connection}

    There is a differential connection on $\QuantumCohomology^\ArbitraryIndex _{\ExtendedTorus} (\Manifold, \Extended{\TorusAction})$, in the sense of \autoref{def:partial-differential-connection}.
    To define this connection, we must describe the data $(k, A, \SpaceOfDerivations, P, \Connection)$, as in \autoref{sec:differential-connections-definition}.
    
    The integral domain $k$ is $\Cohomology^\ArbitraryIndex (\ClassifyingSpace \AdditionalCircle) \cong \Integers [\formalAdditionalCircle]$.
    The $k$-algebra $A$ is $\NovikovRing \GradedCompletedTensorProduct \Cohomology^\ArbitraryIndex (\ClassifyingSpace \ExtendedTorus)$.
    The $A$-module $P$ is $\QuantumCohomology^\ArbitraryIndex _{\ExtendedTorus} (\Manifold, \Extended{\TorusAction})$, which we will abbreviate by $\QuantumNotation$ for convenience.
    
    Associated to a cohomology class $\DegreeTwoCoclass \in \Cohomology^2 (\Manifold)$, we have a derivation $\dbydmult{\DegreeTwoCoclass}$ on $\NovikovRing \GradedCompletedTensorProduct \Cohomology^\ArbitraryIndex (\ClassifyingSpace \ExtendedTorus)$ given by
        \begin{equation}
        \label{eqn:basic-derivation-on-novikov-ring}
            \left( \dbydmult{\DegreeTwoCoclass} \right) (\NovVariable^\DegreeTwoClass \formalTorus^{\critClassifyingSpace}) = \DegreeTwoCoclass(\DegreeTwoClass) \NovVariable^\DegreeTwoClass \formalAdditionalCircle \formalTorus^{\critClassifyingSpace}.
        \end{equation}
    Note that the operation $\dbyd{\DegreeTwoCoclass}$ does not change the exponent of $\NovVariable$, so it actually behaves more like the derivation $t \dbyd{t}$ than pure differentiation $\dbyd{t}$ in a polynomial ring $k [t]$ (see \autoref{eg:differential-connection-on-polynomial-ring}).
    Nonetheless, the map \eqref{eqn:basic-derivation-on-novikov-ring} satisfies the Leibniz rule \eqref{eqn:Leibniz-rule-for-derivations}.
    The space of derivations is
        \begin{equation}
            \SpaceOfDerivations = \NovikovRing[\formalAdditionalCircle] \Tensor \Cohomology^2 (\Manifold),
        \end{equation}
    where $a \Tensor \DegreeTwoCoclass \in \SpaceOfDerivations$ corresponds to the derivation $a \cdot \dbydmult{\DegreeTwoCoclass}$ for $a \in \NovikovRing[\formalAdditionalCircle]$.

    Let $\FixedPoint \in \Manifold$ be a fixed point.
    Given $\DegreeTwoCoclass \in \Cohomology^2 (\Manifold)$, the map
        \begin{equation}
        \label{eqn:connection-definition-on-equivariant-quantum-cohomology}
            \Connection ^\FixedPoint _\DegreeTwoCoclass = \dbydmult{\DegreeTwoCoclass} + \DegreeTwoCoclass^\FixedPoint \QuantumProduct : \QuantumNotation \to \QuantumNotation
        \end{equation}
        given by
        \begin{equation}
        \label{eqn:connection-definition-on-equivariant-quantum-cohomology-explicit}
            \Connection^\FixedPoint_\DegreeTwoCoclass (\NovVariable^\DegreeTwoClass [\critUniversalSpace, \critManifold]) = \DegreeTwoCoclass(\DegreeTwoClass)   \formalAdditionalCircle \cdot \NovVariable^\DegreeTwoClass [\critUniversalSpace, \critManifold] + \DegreeTwoCoclass^\FixedPoint \QuantumProduct (\NovVariable^\DegreeTwoClass [\critUniversalSpace, \critManifold])
        \end{equation}
    satisfies the Leibniz rule \eqref{eqn:Leibniz-rule-for-abstract-connections}.
    It naturally extends to a connection
        \begin{equation}
        \label{eqn:definition-quantum-connection-final-form}
            \Connection^\FixedPoint : \SpaceOfDerivations \to \Homomorphisms_{\Integers [\formalAdditionalCircle]} (\QuantumNotation, \QuantumNotation).
        \end{equation}

    \begin{theorem}
    \label{thm:differential-connection-flatness-quantum}
        The connection $\Connection^\FixedPoint$ is flat.
    \end{theorem}
    
    \begin{proof}
        By \autoref{lem:flatness-sufficient-on-generating-set-differential-connection}, it is sufficient to check that the curvature \eqref{eqn:curvature-definition} vanishes for any two classes $\DegreeTwoCoclass, \secondDegreeTwoCoclass \in \Cohomology^2 (\Manifold)$.
        This further reduces to verifying $[\Connection^\FixedPoint_\DegreeTwoCoclass, \Connection^\FixedPoint_\secondDegreeTwoCoclass] = 0$ because the derivations $\dbydmult{\DegreeTwoCoclass}$ and $\dbydmult{\secondDegreeTwoCoclass}$ commute.
        
        Let us write out the terms of $\Connection^\FixedPoint_\DegreeTwoCoclass \Connection^\FixedPoint_\secondDegreeTwoCoclass (\NovVariable^\DegreeTwoClass [\critUniversalSpace, \critManifold])$:
            \begin{align}
                \Connection^\FixedPoint_\DegreeTwoCoclass \Connection^\FixedPoint_\secondDegreeTwoCoclass (\NovVariable^\DegreeTwoClass [\critUniversalSpace, \critManifold])
                &= \Connection^\FixedPoint_\DegreeTwoCoclass \left( \secondDegreeTwoCoclass (\DegreeTwoClass) \  \NovVariable^\DegreeTwoClass \formalAdditionalCircle [\critUniversalSpace, \critManifold] +  \secondDegreeTwoCoclass^\FixedPoint \QuantumProduct (\NovVariable^\DegreeTwoClass [\critUniversalSpace, \critManifold])\right) \nonumber
                \\ \label{eqnpt:double-diff}
                &= \DegreeTwoCoclass(\DegreeTwoClass) \secondDegreeTwoCoclass (\DegreeTwoClass) \  \NovVariable^\DegreeTwoClass \formalAdditionalCircle^2 [\critUniversalSpace, \critManifold]
                \\ \label{eqnpt:diff-pre-mult} & \qquad
                + \secondDegreeTwoCoclass (\DegreeTwoClass) \  \DegreeTwoCoclass^\FixedPoint \QuantumProduct \NovVariable^\DegreeTwoClass \formalAdditionalCircle [\critUniversalSpace, \critManifold]
                \\ \label{eqnpt:diff-post-mult} & \qquad
                + \left( \dbydmult{\DegreeTwoCoclass} \right) \big(
                \secondDegreeTwoCoclass^\FixedPoint \QuantumProduct \NovVariable^\DegreeTwoClass [\critUniversalSpace, \critManifold] \big)
                \\ \label{eqnpt:double-mult} & \qquad
                + \DegreeTwoCoclass^\FixedPoint \QuantumProduct \big( \secondDegreeTwoCoclass^\FixedPoint \QuantumProduct \NovVariable^\DegreeTwoClass [\critUniversalSpace, \critManifold] \big).
            \end{align}
        The first term \eqref{eqnpt:double-diff} is clearly symmetric in $\DegreeTwoCoclass$ and $\secondDegreeTwoCoclass$.
        The last term \eqref{eqnpt:double-mult} is also symmetric in $\DegreeTwoCoclass$ and $\secondDegreeTwoCoclass$ because the $\ExtendedTorus$-equivariant quantum product is graded-commutative and each of  $\DegreeTwoCoclass^\FixedPoint$ and $\secondDegreeTwoCoclass^\FixedPoint$ has even degree.
        
        We will show that the remaining two terms \eqref{eqnpt:diff-pre-mult} and \eqref{eqnpt:diff-post-mult} cancel with their counterparts in $- \Connection^\FixedPoint_\secondDegreeTwoCoclass \Connection^\FixedPoint_\DegreeTwoCoclass$.
        It is these four terms together that cancel (they do not pair off).

        Analogously to how the $\ExtendedTorus$-equivariant quantum product $\QuantumProduct : \QuantumNotation^{\Tensor 2} \to \QuantumNotation$ counts $\ExtendedTorus$-equivariant deformed `Y'-shaped graphs, the map $\Psi : \QuantumNotation^{\Tensor 3} \to \QuantumNotation$ counts \define{$\ExtendedTorus$-equivariant deformed `$\Psi$\!'-shaped graphs} (see \autoref{fig:equivariant-psi-map}). % \! to cope with italics font
        These graphs are $\ExtendedTorus$-equivalence classes of 10-tuples $(\flowUniversalSpace^-_0, \flowUniversalSpace^+_1, \flowUniversalSpace^+_2, \flowUniversalSpace^+_3, \flowManifold^-_0, \flowManifold^+_1, \flowManifold^+_2, \flowManifold^+_3, u, \pointSphere^+_3)$ where $[\flowUniversalSpace^\pm_i, \flowManifold^\pm_i]$ are $\ExtendedTorus$-equivariant perturbed half\textsuperscript{$\pm$} flowlines, $u : \Projective^1 \to \Manifold$ is a $\eqACS_{\flowUniversalSpace^-_0 (0)}$-holomorphic map and $\pointSphere^+_3$ is a point in $\Projective^1 \setminus \Set{\pointSphere^-_0, \pointSphere^+_1, \pointSphere^+_2}$, subject to the conditions $\flowUniversalSpace^-_0 (0) = \flowUniversalSpace^+_1 (0) = \flowUniversalSpace^+_2 (0) = \flowUniversalSpace^+_3 (0)$ and $u (\pointSphere^\pm_i) = \flowManifold^\pm_i (0)$.
        
        $\Psi$ is graded-commutative.
        We show this by permuting the $\ExtendedTorus$-equivariant half\textsuperscript{+} flowlines.
        Suppose the 10-tuple $(\flowUniversalSpace^-_0, \flowUniversalSpace^+_1, \flowUniversalSpace^+_2, \flowUniversalSpace^+_3, \flowManifold^-_0, \flowManifold^+_1, \flowManifold^+_2, \flowManifold^+_3, u, \pointSphere^+_3)$ is counted by $\Psi (a, b, c)$.
        Let $\phi : \Projective^1 \to \Projective^1$ be the unique automorphism which fixes $\pointSphere^-_0$ and $\pointSphere^+_1$ and satisfies $\phi (\pointSphere^+_2) = \pointSphere^+_3$.
        The 10-tuple $(\flowUniversalSpace^-_0, \flowUniversalSpace^+_1, \flowUniversalSpace^+_3, \flowUniversalSpace^+_2, \flowManifold^-_0, \flowManifold^+_1, \flowManifold^+_3, \flowManifold^+_2, u \ComposedWith \phi, \phi \Inverse (\pointSphere^+_2))$ is counted by the map $\Psi(a, c, b)$.
        Here, we have switched $[\flowUniversalSpace^+_2, \flowManifold^+_2]$ and $[\flowUniversalSpace^+_3, \flowManifold^+_3]$ and modified $\Projective^1$ to accommodate the switch.
        Implicitly, we are appealing to the fact that $\Psi$ is independent of the perturbations of the $\ExtendedTorus$-equivariant Morse functions.
        Altogether, this yields $\Psi (a, b, c) = (-1)^{|b||c|} \Psi(a, c, b)$.
        The same argument demonstrates $\Psi (a, b, c) = (-1)^{|a||b|} \Psi(b, a, c)$.
        
        In the counts for $\Psi(\NovVariable^\DegreeTwoClass [\critUniversalSpace, \critManifold], \DegreeTwoCoclass^\FixedPoint, \secondDegreeTwoCoclass^\FixedPoint)$, the 2-dimensional freedom of the point $\pointSphere_3^+$ is matched by the 2-dimensional intersection condition $[\flowUniversalSpace^+_3, \flowManifold^+_3] (\Infinity) = \secondDegreeTwoCoclass^\FixedPoint$.
        As such, these are the same counts as for $(\NovVariable^\DegreeTwoClass [\critUniversalSpace, \critManifold]) \QuantumProduct \DegreeTwoCoclass^\FixedPoint$, except where the $\ExtendedTorus$-equivariant deformed `Y'-shaped graphs have weights $\secondDegreeTwoCoclass(u \PushForward[\Projective^1])$ (see \autoref{fig:eqnt-weighted-quantum-product}).
        This gives the relation
            \begin{equation}
            \label{eqn:remainder-of-differentiation-commuting-with-quantum-product}
                \Psi (\NovVariable^\DegreeTwoClass [\critUniversalSpace, \critManifold], \DegreeTwoCoclass^\FixedPoint, \secondDegreeTwoCoclass^\FixedPoint)
                = \left( \dbydmult{\secondDegreeTwoCoclass} \right) \big(
                \DegreeTwoCoclass^\FixedPoint \QuantumProduct \NovVariable^\DegreeTwoClass [\critUniversalSpace, \critManifold] \big)
                - \DegreeTwoCoclass^\FixedPoint \QuantumProduct \left( \left( \dbydmult{\secondDegreeTwoCoclass} \right) \left( \NovVariable^\DegreeTwoClass [\critUniversalSpace, \critManifold] \right) \right).
            \end{equation}
        This relation, together with the same relation with $\DegreeTwoCoclass$ and $\secondDegreeTwoCoclass$ switched and the graded-commutativity statement $\Psi (\NovVariable^\DegreeTwoClass [\critUniversalSpace, \critManifold], \DegreeTwoCoclass^\FixedPoint, \secondDegreeTwoCoclass^\FixedPoint) = \Psi (\NovVariable^\DegreeTwoClass [\critUniversalSpace, \critManifold], \secondDegreeTwoCoclass^\FixedPoint, \DegreeTwoCoclass^\FixedPoint)$, completes the proof of flatness.
    \end{proof}

\begin{figure}
\centering
\begin{subfigure}{0.4\textwidth}
    \centering
	\begin{center}
		\begin{tikzpicture}
			\node[inner sep=0] at (2.5,0) {\includegraphics[width=4 cm]{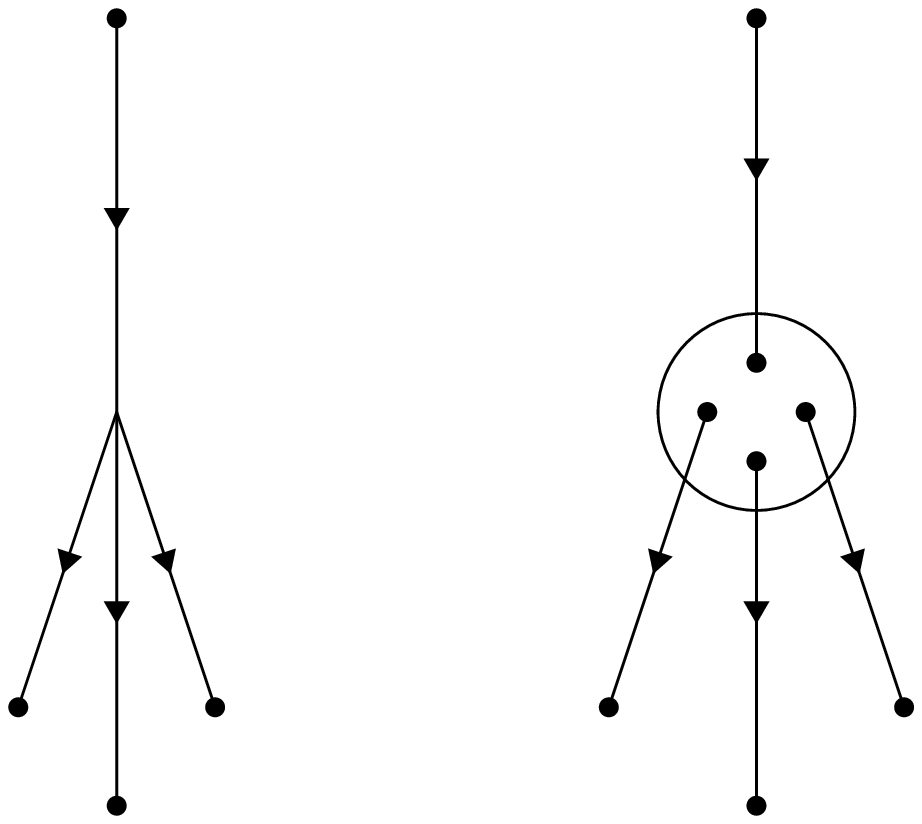}};
% 			\draw[step=1.0,black,thin] (0,-3) grid (5,3);
			
% 			\node at (0,0){0};
			\node at (0.7, 0.9){\tiny $\flowUniversalSpace^-_0$};
			\node at (0.6, -0.5){\tiny $\flowUniversalSpace^+_1$};
			\node at (1.2, -1.2){\tiny $\flowUniversalSpace^+_2$};
			\node at (1.4, -0.5){\tiny $\flowUniversalSpace^+_3$};
			\node at (1.0, 2.0){\tiny $\critUniversalSpace^-_0$};
			\node at (0.6, -1.5){\tiny $\critUniversalSpace^+_1$};
% 			\node at (1.0, -1.9){\tiny $\critUniversalSpace^+_2$};
% 			\node at (1.5, -1.5){\tiny $\critUniversalSpace^+_3$};
			\node at (3.8, 2.0){\tiny $\critManifold^-_0$};
			\node at (3.2, -1.5){\tiny $\critManifold^+_1$};
			\node at (3.8, -1.9){\tiny $\DegreeTwoCoclass^\FixedPoint$};
			\node at (4.5, -1.5){\tiny $\secondDegreeTwoCoclass^\FixedPoint$};
			\node at (3.6, 1.1){\tiny $\flowManifold^-_0$};
			\node at (3.1, -0.7){\tiny $\flowManifold^+_1$};
			\node at (4.0, -0.9){\tiny $\flowManifold^+_2$};
			\node at (4.5, -0.7){\tiny $\flowManifold^+_3$};
			\node at (3.4, 0.4){\tiny $u$};
		\end{tikzpicture}
	\end{center}
    \caption{$\ExtendedTorus$-equivariant deformed `$\Psi$'-shaped graph} 
\label{fig:equivariant-psi-map}
\end{subfigure}
\hspace{1em} 
\begin{subfigure}{0.4\textwidth}
  \centering
	\begin{center}
		\begin{tikzpicture}
			\node[inner sep=0] at (2.5,0) {\includegraphics[width=4 cm]{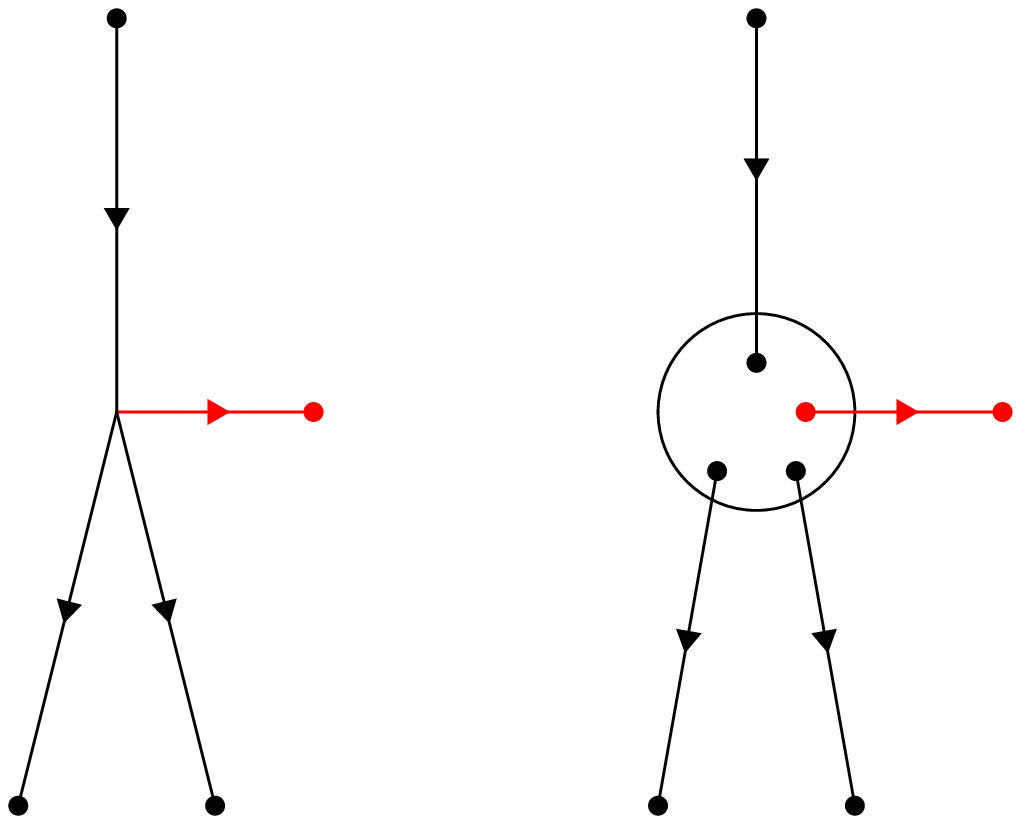}};
% 			\draw[step=1.0,black,thin] (0,-3) grid (5,3);
			
% 			\node at (0,0){0};
			\node at (0.7, 0.9){\tiny $\flowUniversalSpace^-_0$};
			\node at (0.6, -0.6){\tiny $\flowUniversalSpace^+_1$};
			\node at (1.3, -0.6){\tiny $\flowUniversalSpace^+_2$};
			\node at (1.4, 0.2){\tiny $\flowUniversalSpace^+_3$};
			\node at (1.0, 1.8){\tiny $\critUniversalSpace^-_0$};
			\node at (0.6, -1.7){\tiny $\critUniversalSpace^+_1$};
% 			\node at (1.4, -1.7){\tiny $\critUniversalSpace^+_2$};
% 			\node at (2.0, 0){\tiny $\critUniversalSpace^+_3$};
			\node at (3.5, 1.8){\tiny $\critManifold^-_0$};
			\node at (3.2, -1.7){\tiny $\critManifold^+_1$};
			\node at (3.9, -1.7){\tiny $\DegreeTwoCoclass^\FixedPoint$};
			\node at (4.7, 0){\tiny $\secondDegreeTwoCoclass^\FixedPoint$};
			\node at (3.3, 1.1){\tiny $\flowManifold^-_0$};
			\node at (3.1, -0.7){\tiny $\flowManifold^+_1$};
			\node at (3.9, -0.7){\tiny $\flowManifold^+_2$};
			\node at (4.2, 0.2){\tiny $\flowManifold^+_3$};
			\node at (3.1, 0.4){\tiny $u$};
		\end{tikzpicture}
	\end{center}
    \caption{$\secondDegreeTwoCoclass^\FixedPoint$-weighted equivariant product by $\DegreeTwoCoclass^\FixedPoint$}
    \label{fig:eqnt-weighted-quantum-product}
\end{subfigure}
\caption{
    The $\ExtendedTorus$-equivariant deformed `$\Psi$'-shaped graphs that $\Psi$ counts are the same configurations that the \textcolor{red}{$\secondDegreeTwoCoclass^\FixedPoint$-weighted} equivariant product by $\DegreeTwoCoclass^\FixedPoint$ counts.
    Only the interpretation of the flowline $[\flowUniversalSpace_3^+, \flowManifold^+_3]$ changes between the two maps.
} \label{fig:psi-map-diagrams}
\end{figure}

\subsection{Clutching bundle}
\label{sec:clutching-bundle-construction}

    Fix a cocharacter $\Cocharacter \in \LatticeCocharacters{\Torus}$ for which the action $\TorusAction \ComposedWith \Cocharacter$ is linear of nonnegative slope.
    We denote the set of such \define{$\TorusAction$-nonnegative cocharacters} by $\NonnegativeLatticeCocharacters{\TorusAction} {\Torus}$.
    The \define{clutching bundle} $\ClutchingBundle{\Cocharacter}$ is a symplectic $\Manifold$-bundle over the sphere $\Sphere$ which is associated to $\Cocharacter$.
    We outline the construction here, but further details for closed $\Manifold$ are found in \cite{seidel_$_1997} and for convex $\Manifold$ in \cite{ritter_floer_2014}.
    
    Write the sphere $\Sphere$ as a union of its two hemispheres $\Hemisphere^-$ and $\Hemisphere^+$.
    Each hemisphere is a closed disc, and they are glued along the equator.
    Use coordinates for $\Hemisphere^\pm$ from the closed unit disc in $\ComplexNumbers$ and for the equator $\Circle = \RealNumbers / \Integers$, so the gluing identifications are
        \begin{equation}
            \eltCircle \in \Circle \leftrightarrow \ExponentialNumber ^ {2 \PiNumber \ImaginaryNumber \eltCircle} \in \Boundary \Hemisphere^- \leftrightarrow \ExponentialNumber ^ {- 2 \PiNumber \ImaginaryNumber \eltCircle} \in \Boundary \Hemisphere^+. 
        \end{equation}
    The bundle $\ClutchingBundle{\Cocharacter}$ is the union of the trivial bundles $\Hemisphere^\pm \times \Manifold$ over the hemispheres, glued over the equator via
        \begin{equation}
            \Boundary \Hemisphere^- \times \Manifold \ni (\ExponentialNumber ^ {2 \PiNumber \ImaginaryNumber \eltCircle}, \eltManifold) \leftrightarrow (\ExponentialNumber ^ {- 2 \PiNumber \ImaginaryNumber \eltCircle}, \TorusAction_{\Cocharacter(\eltCircle)}(\eltManifold)) \in \Boundary \Hemisphere^+ \times \Manifold.
        \end{equation}
    Denote the projection map by $\ClutchingProjection : \ClutchingBundle{\Cocharacter} \to \Sphere$.
    The vertical tangent bundle $\VerticalTangentSpace \ClutchingBundle{\Cocharacter} = \ker \Derivative \ClutchingProjection$ is naturally equipped with the symplectic bilinear form $\ClutchingBilinearForm_{\eltSphere, \eltManifold} = \SymplecticForm_\eltManifold$ for all $\eltSphere \in \Hemisphere^\pm$.
    The action $\TorusAction \ComposedWith \Cocharacter$ is symplectic, so $\ClutchingBilinearForm$ is well-defined.
    
    The circle $\AdditionalCircle$ acts on $\Sphere$ by rotation.
    The fixed points of the $\AdditionalCircle$-action are the poles $\Pole^\pm = 0 \in \Hemisphere^\pm$.
    The complement of the poles $\Sphere \setminus \Set{\Pole^\pm}$ is isomorphic to an open cylinder via
        \begin{equation}
        \label{eqn:cylinder-coordinates-for-sphere}
            \RealNumbers \times \Circle \ni (s, \eltCircle) \mapsto \left\{ \begin{array}{lll}
            \ExponentialNumber ^ {2 \PiNumber (s + \ImaginaryNumber \eltCircle)} &\in \Hemisphere^- & \text{if $s \le 0$} \\
            \ExponentialNumber ^ {- 2 \PiNumber (s + \ImaginaryNumber \eltCircle)} &\in \Hemisphere^+ & \text{if $s \ge 0$,}
            \end{array} \right.
        \end{equation}
    and the circle $\AdditionalCircle$ acts via $\eltAdditionalCircle \cdot (s, \eltCircle) = (s, \eltCircle + \eltAdditionalCircle)$.
    This rotation action on the sphere $\Sphere$ lifts to a $\ExtendedTorus$-action on $\ClutchingBundle{\Cocharacter}$.
    The element $(\eltAdditionalCircle, \eltTorus) \in \ExtendedTorus = \AdditionalCircle \times \Torus$ acts via
        \begin{equation}
        \label{eqn:extended-torus-action-on-clutching-bundle}
        \begin{array}{lllllll}
            \Hemisphere^- \times \Manifold & \ni & (\eltSphere, \eltManifold) & \mapsto & (\ExponentialNumber ^ {2 \PiNumber \ImaginaryNumber \eltAdditionalCircle} \eltSphere, \TorusAction_{\eltTorus - \Cocharacter(\eltAdditionalCircle)} (\eltManifold)) & \in & \Hemisphere^- \times \Manifold \\
            \Hemisphere^+ \times \Manifold & \ni & (\eltSphere, \eltManifold) & \mapsto & (\ExponentialNumber ^ {-2 \PiNumber \ImaginaryNumber \eltAdditionalCircle} \eltSphere, \TorusAction_\eltTorus (\eltManifold)) & \in & \Hemisphere^+ \times \Manifold
        \end{array}
        \end{equation}
    The fibres above the poles $\Pole^\pm$ are $\ExtendedTorus$-invariant under this action.
    Identifying each fibre with $\Manifold$ using the local trivialisation $\Hemisphere^\pm \times \Manifold$, the $\ExtendedTorus$-action over the south pole $\Pole^+$ is $\Extended{\TorusAction}$ and the $\ExtendedTorus$-action over the north pole $\Pole^-$ is $\Cocharacter \cdot \Extended{\TorusAction} = \Extended{\TorusAction} \ComposedWith \Extended{\Cocharacter} \Inverse$, where $\Extended{\Cocharacter}$ is the automorphism
        \begin{equation}
            \Extended{\Cocharacter} : \ExtendedTorus \to \ExtendedTorus, \quad (\eltAdditionalCircle, \eltTorus) \mapsto (\eltAdditionalCircle, \eltTorus + \Cocharacter(\eltAdditionalCircle)).
        \end{equation}
    
\subsection{Sections of the clutching bundle}
\label{sec:sections-of-clutching-bundle}

    We are concerned with pseudoholomorphic sections of $\ClutchingProjection : \ClutchingBundle{\Cocharacter} \to \Sphere$.
    With this in mind, let $\SphereACS$ be the complex structure on $\Sphere$ induced by \eqref{eqn:cylinder-coordinates-for-sphere}.
    Let $\ClutchingACS$ be an almost complex structure on $\ClutchingBundle{\Cocharacter}$ for which $\ClutchingProjection$ is a $(\ClutchingACS, \SphereACS)$-holomorphic map, which restricts to a convex $\ClutchingBilinearForm$-compatible almost complex structure on $\VerticalTangentSpace \ClutchingBundle{\Cocharacter}$ and which satisfies a Floer-type convexity condition outside of a compact subset (see \cite[Definition~26]{ritter_floer_2014}).
    Such an almost complex structure is \define{admissible}.
    For a regular admissible $\ClutchingACS$, the space of $(\SphereACS, \ClutchingACS)$-holomorphic sections $\ClutchingSection : \Sphere \to \ClutchingBundle{\Cocharacter}$ with $\ClutchingSection \PushForward [\Sphere] = \DegreeTwoClass \in \Homology_2 (\ClutchingBundle{\Cocharacter})$ is a manifold of dimension $2 \dimManifold + 2 \FirstChernClass(\VerticalTangentSpace \ClutchingBundle{\Cocharacter}, \ClutchingBilinearForm)(\DegreeTwoClass)$.
    
    Let $\FixedPoint \in \Manifold$ be a fixed point.
    There is a \define{fixed section} $\ClutchingSection^\FixedPoint : \Sphere \to \ClutchingBundle{\Cocharacter}$ associated to $\FixedPoint$ given by $\ClutchingSection^\FixedPoint (\eltSphere) = (\eltSphere, \FixedPoint)$ for $\eltSphere \in \Sphere$.
    Set $|\Cocharacter, \FixedPoint| = - 2 \FirstChernClass(\VerticalTangentSpace \ClutchingBundle{\Cocharacter}, \ClutchingBilinearForm)(\ClutchingSection^\FixedPoint \PushForward [\Sphere])$.
    This quantity $|\Cocharacter, \FixedPoint|$ is equal to twice the sum of the weights of $\TorusAction \ComposedWith \Cocharacter$ around $\FixedPoint$ \cite[Lemma~2.2]{mcduff_topological_2006}.
    The section $\ClutchingSection^\FixedPoint$ splits the short exact sequence
        \begin{equation}
        \label{eqn:short-exact-sequence-on-homology-of-clutching-bundle}
            \begin{tikzcd}
                0 \arrow[r] & \arrow[r, "\Pole^- \PushForward"] \Homology_2 (\Manifold) & \arrow[r] \Homology_2 (\ClutchingBundle{\Cocharacter}) & \arrow[r] \Homology_2 (\Sphere) \cong \Integers \arrow[l, bend right, "\ClutchingSection^\FixedPoint \PushForward"'] & 0.
            \end{tikzcd}
        \end{equation}
    from \autoref{lem:short-exact-sequence-on-degree-2-homology-of-clutching-bundle}.
    In \eqref{eqn:short-exact-sequence-on-homology-of-clutching-bundle}, the map $\Pole^- : \Manifold \to \ClutchingBundle{\Cocharacter}$ is the inclusion map of the fibre above the north pole $\Pole^-$.
    Given a class $\DegreeTwoClass \in \Homology_2 (\Manifold)$, set $\DegreeTwoClass^\FixedPoint = \Pole^- \PushForward (\DegreeTwoClass) + \ClutchingSection^\FixedPoint \PushForward [\Sphere]$.
    Note $2 \FirstChernClass(\VerticalTangentSpace \ClutchingBundle{\Cocharacter}, \ClutchingBilinearForm)(\DegreeTwoClass^\FixedPoint) = 2 \FirstChernClass(\TangentSpace \Manifold, \SymplecticForm)(\DegreeTwoClass) - |\Cocharacter, \FixedPoint|$.
    The assignment $\DegreeTwoClass \mapsto \DegreeTwoClass^\FixedPoint$ is a bijection between classes $\DegreeTwoClass \in \Homology_2 (\Manifold)$, which are recorded by the Novikov ring, and those classes in $\Homology_2 (\ClutchingBundle{\Cocharacter})$ which can represent sections of $\ClutchingBundle{\Cocharacter}$.
    
    A \define{spiked section} is a triple $(\flowManifold^-, \ClutchingSection, \flowManifold^+)$ of perturbed half\textsuperscript{$\pm$} flowlines $\flowManifold^\pm$ in $\Manifold$ together with a $(\SphereACS, \ClutchingACS)$-holomorphic section $\ClutchingSection$ which satisfies $\ClutchingSection (\Pole^\pm) = \flowManifold^\pm (0)$.
    Let $\ModuliSpace (\critManifold^-, \critManifold^+, \DegreeTwoClass^\FixedPoint)$ denote the space of spiked sections $(\flowManifold^-, \ClutchingSection, \flowManifold^+)$ which satisfy $\flowManifold^\pm (\pm \Infinity) = \critManifold^\pm$ and $\ClutchingSection \PushForward [\Sphere] = \DegreeTwoClass^\FixedPoint$.
    The \define{quantum Seidel map} 
        \begin{equation}
        \label{eqn:quantum-seidel-map-definition}
            \QuantumSeidel(\Cocharacter, \FixedPoint) : \QuantumCohomology^\ArbitraryIndex (\Manifold; \MorseFunction) \to \QuantumCohomology^{\ArbitraryIndex + |\Cocharacter, \FixedPoint|} (\Manifold; \MorseFunction)
        \end{equation}
    counts spiked sections; it is given by
        \begin{equation}
            \QuantumSeidel (\Cocharacter, \FixedPoint) \ (\critManifold^+) = \sum _{\substack{
                \critManifold^-, \DegreeTwoClass \\
                \Dimension \ModuliSpace = 0
            }} \Count \ModuliSpace (\critManifold^-, \critManifold^+, \DegreeTwoClass^\FixedPoint) \ \NovVariable^\DegreeTwoClass \critManifold^-.
        \end{equation}
    The quantum Seidel map commutes with the quantum product.
    
\subsection{The equivariant quantum Seidel map}
    
    Analogously to \autoref{sec:equivariant-almost-complex-structures}, a \define{$\ExtendedTorus$-equiv\-ari\-ant admissible almost complex structure} $\eqClutchingACS$ is a choice of admissible almost complex structure on $\ClutchingBundle{\Cocharacter}$ for each $\eltUniversalSpace \in \UniversalSpace \ExtendedTorus$ which makes an analogous diagram to \eqref{eqn:equivariant-almost-complex-structure-commutative-diagram} commute.
    The convexity of $\eqClutchingACS\RestrictedTo{\VerticalTangentSpace \ClutchingBundle{\Cocharacter}}$ and the Floer-type convexity condition must hold in a common region at infinity for all $\eltUniversalSpace \in \UniversalSpace \ExtendedTorus$.
    
    The \define{$\ExtendedTorus$-equivariant quantum Seidel map}
        \begin{equation}
        \label{eqn:equivariant-quantum-seidel-map-definition}
            \QuantumSeidel_{\ExtendedTorus}(\Cocharacter, \FixedPoint) : \QuantumCohomology^\ArbitraryIndex _{\ExtendedTorus} (\Manifold, \Extended{\TorusAction}; \eqMorseFunction_+) \to \QuantumCohomology^{\ArbitraryIndex + |\Cocharacter, \FixedPoint|}  _{\ExtendedTorus} (\Manifold, \Cocharacter \cdot \Extended{\TorusAction}; \eqMorseFunction_-)
        \end{equation}
    counts $\ExtendedTorus$-equivariant spiked sections, analogously to \eqref{eqn:quantum-seidel-map-definition}.
    These \define{$\ExtendedTorus$-equiv\-ari\-ant spiked sections} are $\ExtendedTorus$-equivalence classes of quadruples $(\flowUniversalSpace, \flowManifold^-, \ClutchingSection, \flowManifold^+)$ where $\flowUniversalSpace$ is a flowline in $\UniversalSpace \ExtendedTorus$, $\flowManifold^\pm$ are perturbed half\textsuperscript{$\pm$} flowlines of $\eqMorseFunction_{\pm}(\flowUniversalSpace(s), \Argument)$ and $\ClutchingSection$ is a $(\SphereACS, \eqClutchingACS_{\flowUniversalSpace(0)})$-holo\-mor\-phic section which satisfies $\ClutchingSection (\Pole^\pm) = \flowManifold^\pm (0)$.
    See \autoref{fig:equivariant-quantum-seidel-map}.
    The two $\ExtendedTorus$-equivariant Morse functions $\eqMorseFunction_{\pm} : \UniversalSpace \ExtendedTorus \times \Manifold \to \RealNumbers$ satisfy different relations because the domain and codomain have different $\ExtendedTorus$-actions.
    
    A homotopy argument shows that $\QuantumSeidel_{\ExtendedTorus}(\Cocharacter, \FixedPoint)$ commutes with the geometric $\Cohomology^\ArbitraryIndex (\ClassifyingSpace \ExtendedTorus)$-action.
    
\subsection{Intertwining relation}

    Let $\DegreeTwoCoclass \in \Cohomology^2_{\ExtendedTorus} (\ClutchingBundle{\Cocharacter})$ be a $\ExtendedTorus$-equivariant degree-2 coclass for the $\ExtendedTorus$-action \eqref{eqn:extended-torus-action-on-clutching-bundle} on the total space $\ClutchingBundle{\Cocharacter}$.
    Denote by $\DegreeTwoCoclass^\pm$ the restrictions of $\DegreeTwoCoclass$ to the fibres above the poles $\Pole^\pm$.
    This gives two elements $\DegreeTwoCoclass^- \in \Cohomology^2_{\ExtendedTorus} (\Manifold, \Cocharacter \cdot \Extended{\TorusAction})$ and $\DegreeTwoCoclass^+ \in \Cohomology^2_{\ExtendedTorus} (\Manifold, \Extended{\TorusAction})$.
    
    The \define{$\DegreeTwoCoclass$-weighted $\ExtendedTorus$-equiv\-ari\-ant quantum Seidel map} 
        \begin{equation}
        \label{eqn:weighted-equivariant-quantum-seidel-map-definition}
            \WeightedQuantumSeidel_{\ExtendedTorus} (\Cocharacter, \FixedPoint, \DegreeTwoCoclass) : \QuantumCohomology^\ArbitraryIndex _{\ExtendedTorus} (\Manifold, \Extended{\TorusAction}; \eqMorseFunction_+) \to \QuantumCohomology^{\ArbitraryIndex + |\Cocharacter, \FixedPoint|}  _{\ExtendedTorus} (\Manifold, \Cocharacter \cdot \Extended{\TorusAction}; \eqMorseFunction_-)
        \end{equation}
    counts $\ExtendedTorus$-equiv\-ari\-ant spiked sections $[\flowUniversalSpace, \flowManifold^-, \ClutchingSection, \flowManifold^+]$, but weighted by $\DegreeTwoCoclass([\flowUniversalSpace(0), \ClutchingSection([\Sphere])])$.
    We construct $\WeightedQuantumSeidel_{\ExtendedTorus} (\Cocharacter, \FixedPoint, \DegreeTwoCoclass)$ in the proof of \autoref{thm:intertwining-relation-quantum-extended-torus} using a $\ExtendedTorus$-equivariant half\textsuperscript{$+$} flowline to $\DegreeTwoCoclass$ to record the weight.
    
    \begin{theorem}
        [Intertwining relation]
        \label{thm:intertwining-relation-quantum-extended-torus}
        The relation
            \begin{equation}
            \label{eqn:intertwining-relation-extended-quantum}
                 \QuantumSeidel_{\ExtendedTorus} (\Cocharacter, \FixedPoint) (x \QuantumProduct \DegreeTwoCoclass^+) - \QuantumSeidel_{\ExtendedTorus} (\Cocharacter, \FixedPoint) (x) \QuantumProduct \DegreeTwoCoclass^- = \formalAdditionalCircle \  \WeightedQuantumSeidel_{\ExtendedTorus} (\Cocharacter, \FixedPoint, \DegreeTwoCoclass) (x)
            \end{equation}
        holds for all $x \in \QuantumCohomology^\ArbitraryIndex _{\ExtendedTorus} (\Manifold, \Extended{\TorusAction}; \eqMorseFunction_+)$.
    \end{theorem}
    
    \autoref{thm:intertwining-relation-quantum-extended-torus} is the $\ExtendedTorus$-equivariant version of the $\Circle$-equivariant intertwining relation that we showed in \cite{liebenschutz-jones_intertwining_2020}, and the proof is almost identical.
    The new proofs in \autoref{sec:equivariant-floer-cohomology-all-content} use similar ideas.
    
    \begin{proof}
        [Sketch proof]
        Fix a $\ExtendedTorus$-equivariant identification  $\minUniversalSpace \to \ExtendedTorus$, where $\minUniversalSpace$ is the minimal locus of $\morseUniversalSpace : \UniversalSpace \ExtendedTorus \to \RealNumbers$.
        Composing with the character $\Extended{\Character}_0 : \ExtendedTorus \to \AdditionalCircle$, this gives a $\AdditionalCircle$-equivariant map $\arg : \minUniversalSpace \to \AdditionalCircle$.
        
        Likewise, define the $\AdditionalCircle$-equivariant map $\arg : \Sphere \setminus \Set{\Pole^\pm} \to \AdditionalCircle$ via $(s, \eltCircle) \mapsto \eltCircle$, using the cylindrical coordinates from \eqref{eqn:cylinder-coordinates-for-sphere}.
        
        The homotopy which gives \eqref{eqn:intertwining-relation-extended-quantum} on cohomology is given by counting $\ExtendedTorus$-equivalence classes of octuples $(\flowUniversalSpace, \flowManifold^-, \ClutchingSection, \flowManifold^+, \flowUniversalSpace^{\min}, \flowUniversalSpace_0, \flowManifold_0, \Pole_0)$ where $[\flowUniversalSpace, \flowManifold^-, \ClutchingSection, \flowManifold^+]$ is a $\ExtendedTorus$-equivariant spiked section, $\flowUniversalSpace^{\min}$ is a perturbed half\textsuperscript{$+$} flowline in $\UniversalSpace \ExtendedTorus$, $[\flowUniversalSpace_0, \flowManifold_0]$ is a $\ExtendedTorus$-equivariant half\textsuperscript{$+$} flowline that converges to $\DegreeTwoCoclass$ and $\Pole_0 \in \Sphere \setminus \Set{\Pole^\pm}$ is a point which together satisfy the conditions $\flowUniversalSpace(0) = \flowUniversalSpace^{\min} (0) = \flowUniversalSpace_0 (0)$, $\flowManifold_0 (0) = \ClutchingSection (\Pole_0)$ and $\arg(\Pole_0) + \arg( \flowUniversalSpace^{\min}(+\Infinity)) = 0$.
        See \autoref{fig:equivariant-quantum-seidel-map-homotopy}.
        The last condition leaves a 1-dimensional freedom in the point $\Pole_0$ along the line of longitude between the poles and with fixed argument.
        
        The map $\WeightedQuantumSeidel_{\ExtendedTorus} (\Cocharacter, \FixedPoint, \DegreeTwoCoclass)$ is given by the same count, but with $\flowUniversalSpace^{\min}$ omitted.
        Each section $\ClutchingSection$ is counted as many times as there are intersections of $\ClutchingSection (\Sphere)$ and $\StableManifold(\DegreeTwoCoclass)$.
        The 2-dimensional freedom of $\Pole_0 \in \Sphere \setminus \Set{\Pole^\pm}$ is cancelled by the 2-dimensional intersection condition with $\DegreeTwoCoclass$.
        
        The two terms on the left-hand side of \eqref{eqn:intertwining-relation-extended-quantum} arise as the point $\Pole_0$ converges to one of the poles $\Pole^\pm$.
        The term on the right-hand side of \eqref{eqn:intertwining-relation-extended-quantum} arises as $\flowUniversalSpace^{\min}$ breaks.
        The half\textsuperscript{$+$} flowline $\flowUniversalSpace^{\min}$ breaks into a pair of flowlines, the first of which is a half\textsuperscript{$+$} flowline from $\flowUniversalSpace(0)$ to a critical point $\critUniversalSpace$, where the critical point $\critClassifyingSpace = [\critUniversalSpace]$ corresponds to $\formalAdditionalCircle$ (the critical point $\critClassifyingSpace$ exists and is unique because of our choice of basis satisfying \eqref{eqn:basis-rules-for-extended-torus}).
        The second flowline is from $\critUniversalSpace$ to a point $\critUniversalSpace^{\min} \in \min (\morseUniversalSpace)$ which satisfies $\arg (\critUniversalSpace^{\min}) = \arg (\eltSphere_0)$.
        It turns out that the point $\critUniversalSpace^{\min}$ and the entire second flowline are uniquely determined by $\critUniversalSpace$ and $\eltSphere_0$, so we discard this second flowline from the moduli space without losing any information.
        Using a homotopy, we decouple the flowline to $\critClassifyingSpace$ from the rest of the configuration, and the resulting moduli spaces yield the maps $\formalAdditionalCircle$ and $\WeightedQuantumSeidel_{\ExtendedTorus} (\Cocharacter, \FixedPoint, \DegreeTwoCoclass)$ respectively
        (more details in \cite{liebenschutz-jones_intertwining_2020}).

        % The two terms on the left-hand side of \eqref{eqn:intertwining-relation-extended-quantum} arise as the point $\Pole_0$ converges to one of the poles $\Pole^\pm$.
        % The term on the right-hand side of \eqref{eqn:intertwining-relation-extended-quantum} arises as the half\textsuperscript{$+$} flowline $\flowUniversalSpace^{\min}$ breaks.
        % The broken flowline $\flowUniversalSpace^{\min}$ flows to the critical point $\critClassifyingSpace$ corresponding to $\formalAdditionalCircle$ (this critical point exists and is unique because of our choice of basis satisfying \eqref{eqn:basis-rules-for-extended-torus}).
        % We can decouple the contributions of the flowline to $\critClassifyingSpace$ and the rest of the configuration, which is how we arrive at the composition of the two maps $\formalAdditionalCircle$ and $\WeightedQuantumSeidel_{\ExtendedTorus} (\Cocharacter, \FixedPoint, \DegreeTwoCoclass)$.
    \end{proof}

\begin{figure}
\centering
\begin{subfigure}{0.4\textwidth}
    \centering
	\begin{center}
		\begin{tikzpicture}
			\node[inner sep=0] at (2.5,0) {\includegraphics[width=4 cm]{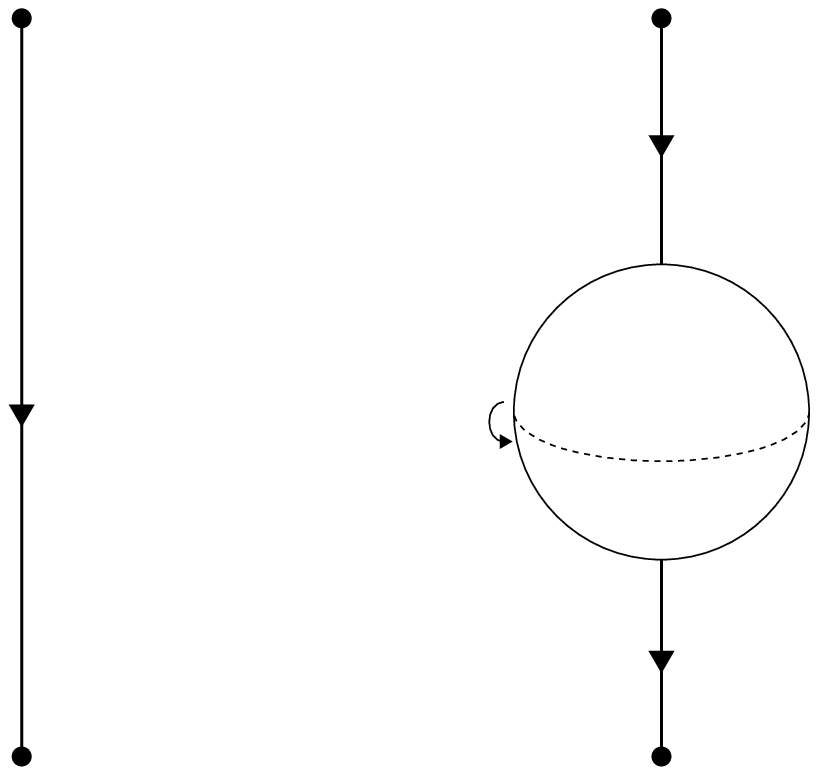}};
% 			\draw[step=1.0,black,thin] (0,-3) grid (5,3);
			
% 			\node at (0,0){0};
			\node at (0.4, 0.8){\tiny $\flowUniversalSpace$};
			\node at (3.1, 0.5){\tiny $\ClutchingSection$};
			\node at (0.7, 2.1){\tiny $\critUniversalSpace^-$};
			\node at (0.7, -2.0){\tiny $\critUniversalSpace^+$};
			\node at (3.9, 2.1){\tiny $\critManifold^-$};
			\node at (3.9, -2.0){\tiny $\critManifold^+$};
			\node at (2.8, -0.1){\tiny $\Cocharacter$};
			\node at (4.0, 0){\tiny $\DegreeTwoClass^\FixedPoint$};
			\node at (4.0, 1.3){\tiny $\flowManifold^-$};
			\node at (4.0, -1.3){\tiny $\flowManifold^+$};
		\end{tikzpicture}
	\end{center}
    \caption{$\ExtendedTorus$-equivariant quantum Seidel map} 
\label{fig:equivariant-quantum-seidel-map}
\end{subfigure}
\hspace{2em} 
\begin{subfigure}{0.4\textwidth}
  \centering
	\begin{center}
		\begin{tikzpicture}
			\node[inner sep=0] at (2.5,0) {\includegraphics[width=5.0 cm]{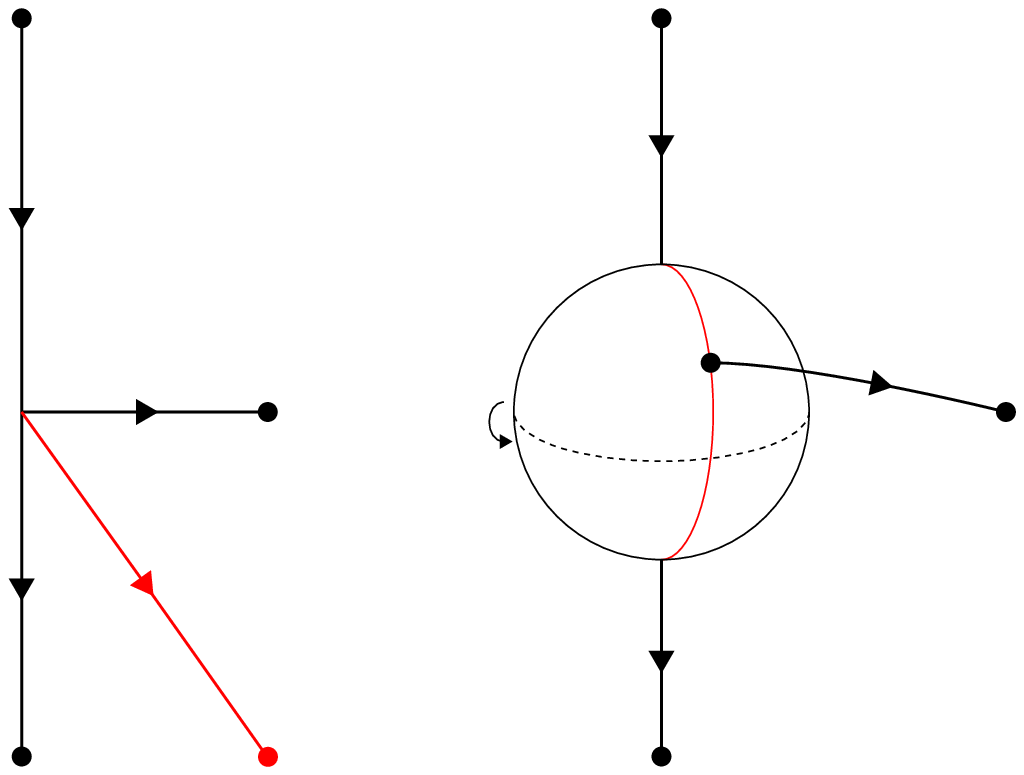}};
% 			\draw[step=1.0,black,thin] (0,-3) grid (5,3);
			
% 			\node at (0,0){0};
			\node at (-0.1, 0.8){\tiny $\flowUniversalSpace$};
			\node at (2.6, 0.5){\tiny $\ClutchingSection$};
			\node at (0.1, 2.0){\tiny $\critUniversalSpace^-$};
			\node at (0.1, -2.0){\tiny $\critUniversalSpace^+$};
			\node at (3.3, 2.0){\tiny $\critManifold^-$};
			\node at (3.3, -2.0){\tiny $\critManifold^+$};
			\node at (2.3, -0.1){\tiny $\Cocharacter$};
			\node at (3.5, 1.3){\tiny $\flowManifold^-$};
			\node at (3.5, -1.3){\tiny $\flowManifold^+$};
			\node at (0.7, 0){\tiny $\flowUniversalSpace_0$};
			\node at (4.3, 0.2){\tiny $\flowManifold_0$};
			\node at (1, -0.9){\tiny $\flowUniversalSpace^{\min}$};
			\node at (3.3, 0.1){\tiny $\eltSphere_0$};
			\node at (5.1, -0.1){\tiny $\DegreeTwoCoclass$};
		\end{tikzpicture}
	\end{center}
    \caption{Homotopy from proof of intertwining relation}
    \label{fig:equivariant-quantum-seidel-map-homotopy}
\end{subfigure}
\caption{
    An additional flowline $\flowUniversalSpace^{\min}$ in $\UniversalSpace \ExtendedTorus$ is used to describe a \textcolor{red}{path} between the poles.
    It is from this path that the half\textsuperscript{+} flowline $\flowManifold_0$ flows.
} \label{fig:equivariant-quantum-seidel-map-with-homotopy}
\end{figure}
    
    The fact there is no global $\AdditionalCircle$-equivariant function $\UniversalSpace \ExtendedTorus \to \AdditionalCircle$ is why the right-hand side of \eqref{eqn:intertwining-relation-extended-quantum} is nonzero.
    The best we can do is define such a function on $\StableManifold (\minUniversalSpace)$, which is a dense subset.
    In the proof, the function $\StableManifold (\minUniversalSpace) \to \AdditionalCircle$ arises as the composition of flowing along the flowline $\flowUniversalSpace^{\min}$ with the map $\arg : \minUniversalSpace \to \AdditionalCircle$.
    
\subsection{Pullback group isomorphisms}
\label{sec:pullback-group-isomorphisms-on-cohomology}

    The group isomorphism $\Extended{\Cocharacter} : \ExtendedTorus \to \ExtendedTorus$ associated to a cocharacter $\Cocharacter : \Circle \to \Torus$ is given by $(\eltAdditionalCircle, \eltTorus) \mapsto (\eltAdditionalCircle, \eltTorus + \Cocharacter (\eltAdditionalCircle))$.
    The induced map $(\ClassifyingSpace \Extended{\Cocharacter}) \PullBack : \Cohomology^\ArbitraryIndex (\ClassifyingSpace \ExtendedTorus) \to \Cohomology^\ArbitraryIndex (\ClassifyingSpace \ExtendedTorus)$ is given by $[\Extended{\Character}] \mapsto [\Extended{\Character} \ComposedWith \Extended{\Cocharacter}]$ for characters $\Extended{\Character} : \ExtendedTorus \to \Circle$.
    Note that we have $\Extended{\Character}_0 =  \Extended{\Character}_0 \ComposedWith \Extended{\Cocharacter}$, giving $(\ClassifyingSpace \Extended{\Cocharacter}) \PullBack (\formalAdditionalCircle) = \formalAdditionalCircle$.
    
    In this paper, we are using an explicit classifying bundle $(\InfiniteSphere)^{\dimTorus + 1} \to (\InfiniteComplexProjectiveSpace)^{\dimTorus + 1}$ for $\UniversalSpace \ExtendedTorus \to \ClassifyingSpace \ExtendedTorus$ which is associated to a basis $\Extended{\basisCocharacter}$.
    Denote by $\CoordwiseProductAction$ the coordinate-wise action of $(\Circle)^{\dimTorus + 1}$ on $(\InfiniteSphere)^{\dimTorus + 1}$.
    The action of $\ExtendedTorus$ on $(\InfiniteSphere)^{\dimTorus + 1}$ is $\CoordwiseProductAction \ComposedWith \Extended{\basisCocharacter}\Inverse$, where $\Extended{\basisCocharacter} : (\Circle)^{\dimTorus + 1} \to \ExtendedTorus$ is the isomorphism associated to the basis $\Extended{\basisCocharacter}$.
    The identity map $\Identity : (\InfiniteSphere)^{\dimTorus + 1} \to (\InfiniteSphere)^{\dimTorus + 1}$ satisfies
        \begin{equation}
        \label{eqn:functorial-relation-satisfied-for-identity-map-when-changing-basis}
            \Identity \big( \CoordwiseProductAction ( \ \Extended{\basisCocharacter}\Inverse (\eltExtendedTorus), \ \Extended{\eltUniversalSpace} \ ) \big) = \CoordwiseProductAction \big( \  (\Extended{\Cocharacter} \ComposedWith \Extended{\basisCocharacter}) \Inverse ( \Extended{\Cocharacter} (\eltExtendedTorus)), \ \Identity (\Extended{\eltUniversalSpace}) \ \big)
        \end{equation}
    for all $\Extended{\eltUniversalSpace} \in (\InfiniteSphere)^{\dimTorus + 1}$ and all $\eltExtendedTorus \in \ExtendedTorus$.
    This relation \eqref{eqn:functorial-relation-satisfied-for-identity-map-when-changing-basis} is the relation \eqref{eqn:equivariant-functoriality-map-relation} applied to our setting.
    Therefore the identity map $\Identity : (\InfiniteSphere)^{\dimTorus + 1} \to (\InfiniteSphere)^{\dimTorus + 1}$ is $\UniversalSpace \Extended{\Cocharacter}$, where the domain has associated basis $\Extended{\basisCocharacter}$ and the codomain has the associated basis $\Extended{\Cocharacter} \ComposedWith \Extended{\basisCocharacter}$.
    
    This explicit construction of $\UniversalSpace \Extended{\Cocharacter}$ as the identity map using different bases extends to the Borel homotopy quotients.
    The identity map on the Borel homotopy quotients induces the isomorphism
        \begin{equation}
        \label{eqn:pull-back-group-isom-on-equivariant-cohomology}
            (\ClassifyingSpace \Extended{\Cocharacter}) \PullBack : \Cohomology^\ArbitraryIndex _{\ExtendedTorus, \Extended{\basisCocharacter}} (\Manifold, \Extended{\TorusAction} \ComposedWith \Extended{\Cocharacter} \Inverse; \eqMorseFunction) \to  \Cohomology^\ArbitraryIndex _{\ExtendedTorus, \Extended{\Cocharacter} \ComposedWith \Extended{\basisCocharacter}} (\Manifold, \Extended{\TorusAction}; \eqMorseFunction).
        \end{equation}
    Denote the basis $\Extended{\Cocharacter} \ComposedWith \Extended{\basisCocharacter}$ by $\Cocharacter \cdot  \Extended{\basisCocharacter}$.
    
    The map \eqref{eqn:pull-back-group-isom-on-equivariant-cohomology} is compatible with the $\ExtendedTorus$-equivariant quantum product $\QuantumProduct$ because it is essentially the identity map.
    Thus $(\ClassifyingSpace \Extended{\Cocharacter}) \PullBack (x \QuantumProduct y) = (\ClassifyingSpace \Extended{\Cocharacter}) \PullBack (x) \QuantumProduct (\ClassifyingSpace \Extended{\Cocharacter}) \PullBack (y)$ holds.
        
\subsection{Shift operator}

    The \define{shift operator} on quantum cohomology is the composition
        \begin{equation}
            \ShiftOperator_\Cocharacter^\FixedPoint = (\ClassifyingSpace \Extended{\Cocharacter}) \PullBack \ComposedWith \QuantumSeidel _{\ExtendedTorus} (\Cocharacter, \FixedPoint) : \QuantumCohomology^\ArbitraryIndex _{\ExtendedTorus; \Extended{\basisCocharacter}} (\Manifold, \Extended{\TorusAction}) \to \QuantumCohomology^{\ArbitraryIndex + |\Cocharacter, \FixedPoint|} _{\ExtendedTorus; \Cocharacter \cdot \Extended{\basisCocharacter}} (\Manifold, \Extended{\TorusAction}).
        \end{equation}
    The shift operator is defined on the commutative monoid $\NonnegativeLatticeCocharacters{\TorusAction}{\Torus}$ of $\TorusAction$-nonnegative cocharacters $\Cocharacter : \Circle \to \Torus$.
    
    \begin{theorem}
    \label{thm:shift-operator-flat-on-quantum-individually}
        The shift operator $\ShiftOperator^\FixedPoint$ is flat (and hence well-defined on $\NonnegativeLatticeCocharacters{\TorusAction}{\Torus}$).
    \end{theorem}
    
    \begin{proof}
        We must show $\ShiftOperator^\FixedPoint _\Cocharacter \ShiftOperator^\FixedPoint _{\Cocharacter \Prime} = \ShiftOperator^\FixedPoint _{\Cocharacter + \Cocharacter \Prime}$ for any two $\Cocharacter, \Cocharacter \Prime \in \NonnegativeLatticeCocharacters{\TorusAction}{\Torus}$.
        This statement can be proved directly by gluing sections exactly as in the non-equivariant case \cite[Theorem~11.4.3]{mcduff_j-holomorphic_2004}.
        On the other hand, it is much easier to show $\ShiftOperator^\FixedPoint _\Cocharacter \ShiftOperator^\FixedPoint _{\Cocharacter \Prime} = \ShiftOperator^\FixedPoint _{\Cocharacter + \Cocharacter \Prime}$ directly for the shift operators $\ShiftOperator^\FixedPoint$ on $\ExtendedTorus$-equivariant Floer cohomology.
        We can deduce the quantum statement from the Floer statement,  \autoref{thm:flatness-of-difference-differential-connection-on-floer-cohomology}, using $\ExtendedTorus$-equivariant PSS isomorphisms (see \autoref{prop:equivariant-gluing-for-seidel-map}).
        This second approach is analogous to that taken in \cite{seidel_$_1997}.
    \end{proof}
    
    \begin{theorem}
    \label{thm:overall-flatness-of-difference-differential-connection-in-equivariant-quantum-cohomology}
        The difference-differential connection $(\ShiftOperator^\FixedPoint, \Connection^\FixedPoint)$ is flat.
    \end{theorem}
    
    \begin{proof}
        The differential connection $\Connection^\FixedPoint$ is flat by \autoref{thm:differential-connection-flatness-quantum}.
        The shift operator $\ShiftOperator^\FixedPoint$ is flat by \autoref{thm:shift-operator-flat-on-quantum-individually}.
        
        It remains to show $[\Connection^\FixedPoint _\DegreeTwoCoclass, \ShiftOperator^\FixedPoint _\Cocharacter] = 0$ for any $\DegreeTwoCoclass \in \Cohomology^2 (\Manifold)$ and any $\Cocharacter \in \NonnegativeLatticeCocharacters{\TorusAction}{\Torus}$.
        By \autoref{lem:class-existence-for-quantum-flatness-from-intertwining}, there is a class $\beta \in \Cohomology^2 _{\ExtendedTorus} (\ClutchingBundle{\Cocharacter})$ which satisfies $\beta^+ = \DegreeTwoCoclass^\FixedPoint$, $(\ClassifyingSpace \Extended{\Cocharacter}) \PullBack (\beta^-) = \DegreeTwoCoclass^\FixedPoint$ and $\beta (\DegreeTwoClass^\FixedPoint) = \DegreeTwoCoclass (\DegreeTwoClass)$ for $\DegreeTwoClass \in \Homology_2 (\Manifold)$.
        The first two conditions imply that the intertwining relation \autoref{thm:intertwining-relation-quantum-extended-torus} applied to $\beta$, post-composed with $(\ClassifyingSpace \Extended{\Cocharacter}) \PullBack$, gives
            \begin{equation}
            \label{eqn:intertwining-preparing-for-flatness}
                 \ShiftOperator _\Cocharacter ^\FixedPoint (x \QuantumProduct \DegreeTwoCoclass^\FixedPoint) - \ShiftOperator _\Cocharacter ^\FixedPoint (x) \QuantumProduct \DegreeTwoCoclass^\FixedPoint = (\ClassifyingSpace \Extended{\Cocharacter}) \PullBack \ \formalAdditionalCircle \  \WeightedQuantumSeidel_{\ExtendedTorus} (\Cocharacter, \FixedPoint, \beta) (x)
            \end{equation}
        for $x \in \QuantumCohomology^\ArbitraryIndex _{\ExtendedTorus; \Extended{\basisCocharacter}} (\Manifold, \Extended{\TorusAction})$.
        The third condition $\beta (\DegreeTwoClass^\FixedPoint) = \DegreeTwoCoclass (\DegreeTwoClass)$ gives
            \begin{equation}
            \label{eqn:weighted-seidel-map-recorded-by-differentiation}
                 \left( \dbydmult{\DegreeTwoCoclass} \right) \ShiftOperator_\Cocharacter^\FixedPoint (x) -  \ShiftOperator_\Cocharacter^\FixedPoint \left( \left( \dbydmult{\DegreeTwoCoclass} \right) x \right)  = (\ClassifyingSpace \Extended{\Cocharacter}) \PullBack \  \formalAdditionalCircle \  \WeightedQuantumSeidel_{\ExtendedTorus} (\Cocharacter, \FixedPoint, \beta) (x),
            \end{equation}
        analogously to \eqref{eqn:remainder-of-differentiation-commuting-with-quantum-product}.
        The difference $\eqref{eqn:weighted-seidel-map-recorded-by-differentiation} - \eqref{eqn:intertwining-preparing-for-flatness}$ is $[\Connection^\FixedPoint _\DegreeTwoCoclass, \ShiftOperator^\FixedPoint _\Cocharacter] (x) = 0$.
    \end{proof}
    
\section{Equivariant Floer cohomology}
\label{sec:equivariant-floer-cohomology-all-content}

    The symplectic manifold $(\Manifold, \SymplecticForm)$ has the same assumptions as \autoref{sec:assumptions-on-symplectic-manifold}.
    
\subsection{Hamiltonian dynamics}
\label{sec:hamiltonian-dynamics}

    The time-dependent Hamiltonian function $\Hamiltonian : \Circle \times \Manifold \to \RealNumbers$ is \define{linear of slope} $\Slope \in \RealNumbers \setminus \ReebPeriods$ if the equation $\Hamiltonian_t (\ConvexCoordMap (\RadialCoord, \eltContactManifold)) = \Slope \RadialCoord + \text{constant}$ holds outside a compact set.
    The (1-periodic Hamiltonian) orbits $\HamiltonianOrbit : \Circle \to \Manifold$ of such a Hamiltonian $\Hamiltonian$ are the solutions of $\partial_t  \HamiltonianOrbit(t) = \HamiltonianVectorField{\Hamiltonian, t}$, where the Hamiltonian vector field $\HamiltonianVectorField{\Hamiltonian, t}$ is determined by $\SymplecticForm( \Argument, \HamiltonianVectorField{\Hamiltonian, t}) = \ExteriorDerivative \Hamiltonian_t (\Argument)$.
    At infinity, the Hamiltonian vector field $\HamiltonianVectorField{\Hamiltonian, t} = \Slope \ReebVectorField$ is a multiple of the Reeb vector field, but since $\Slope \in \RealNumbers \setminus \ReebPeriods$ is not a Reeb period, there are no 1-periodic Hamiltonian orbits in this region.
    It follows that a regular linear Hamiltonian of slope $\Slope$ has only finitely-many orbits.
    Denote the set of orbits of $\Hamiltonian$ by $\HamiltonianOrbitSet(\Hamiltonian)$.
    
    A \define{filling} of an orbit $\HamiltonianOrbit : \Circle \to \Manifold$ is a continuous map $f : \Disc \to \Manifold$ which satisfies $f( \ExponentialNumber^{2 \PiNumber \ImaginaryNumber t} ) = \HamiltonianOrbit(t)$, where $\Disc = \Set{ \Modulus{z} \le 1 } \subset \ComplexNumbers$ is a 2-disc.
    Two fillings $f, f\Prime$ of the same orbit $\HamiltonianOrbit$ combine to give a map $(f \DisjointUnion_x f\Prime) : \Sphere = \Disc \DisjointUnion_{\Circle} \Disc \to \Manifold$ (orient $\Disc \DisjointUnion_{\Circle} \Disc$ with the standard orientation on the first copy of $\Disc$ and the opposite orientation on the second copy of $\Disc$).
    These fillings are \define{equivalent} if $(f \DisjointUnion_x f\Prime) \PushForward [\Sphere] = 0 \in \Homology_2 (\Manifold)$ holds.
    If $(f \DisjointUnion_x f\Prime) \PushForward [\Sphere] = \DegreeTwoClass \in \Homology_2 (\Manifold)$ holds, then we write $[f] = \DegreeTwoClass \ConnectedSum [f\Prime]$ for the corresponding equivalence classes.
    This coincides with the connected sum of $\DegreeTwoClass$ with $f\Prime$ (see \cite{hofer_floer_1995}).
    We denote an orbit $\HamiltonianOrbit$ together with an equivalence class of fillings by $\WithFilling{\HamiltonianOrbit}$.
    Denote the set of such pairs by $\WithFilling{\HamiltonianOrbitSet} (\Hamiltonian)$.
    
    The Conley-Zehnder index $|\WithFilling{\HamiltonianOrbit}|$ of $\WithFilling{\HamiltonianOrbit}$ is well-defined, for a filling of $\HamiltonianOrbit$ induces a trivialisation of $\HamiltonianOrbit \PullBack \TangentSpace \Manifold$.
    We use the convention for the Conley-Zehnder index for which the Conley-Zehnder index of a critical point of a small Hamiltonian recovers the Morse index, as in \cite{ritter_floer_2014}.
    For this convention, we have $|\DegreeTwoClass \ConnectedSum \WithFilling{\HamiltonianOrbit}| = |\WithFilling{\HamiltonianOrbit}| - 2 \FirstChernClass (\TangentSpace \Manifold, \SymplecticForm) (\DegreeTwoClass)$.
    
    For each orbit $\HamiltonianOrbit_i \in \HamiltonianOrbitSet(\Hamiltonian)$, fix a filling $\WithFilling{\HamiltonianOrbit}_i \in \WithFilling{\HamiltonianOrbitSet}(\Hamiltonian)$ for the orbit.
    The \define{Floer cochain complex} $\FloerCochainComplex^\ArbitraryIndex (\Manifold, \Slope; \Hamiltonian)$ is the graded $\NovikovRing$-module $\NovikovRing \langle \WithFilling{\HamiltonianOrbit}_i \rangle _i$.
    It is independent of the choices of fillings because the alternative filling $\WithFilling{\HamiltonianOrbit}\Prime _i = (-\DegreeTwoClass) \ConnectedSum \WithFilling{\HamiltonianOrbit}_i$ corresponds to the element $\NovVariable^\DegreeTwoClass \WithFilling{\HamiltonianOrbit}_i$.
    Equivalently, we have $\FloerCochainComplex^\ArbitraryIndex (\Manifold, \Slope; \Hamiltonian) = \NovikovRing \Tensor \Integers \langle \WithFilling{\HamiltonianOrbitSet} (\Hamiltonian) \rangle / {\sim}$, where the relation $\sim$ is generated by 
        \begin{equation}
        \label{eqn:cohomological-convention-for-fillings}
            \NovVariable^\DegreeTwoClass \WithFilling{\HamiltonianOrbit} \sim (-\DegreeTwoClass) \ConnectedSum \WithFilling{\HamiltonianOrbit}.
        \end{equation}
    
    \begin{remark}
        [Cohomological conventions for fillings]
        Poincar\'{e} duality can be established in Morse theory by reversing the flowlines \cite[Section~4.3]{audin_morse_2014}.
        The \define{homology Morse differential} for the Morse function $\MorseFunction$ is $d_\ArbitraryIndex (\critManifold^-; \MorseFunction) = \critManifold^+ + \cdots$, and it counts ($\RealNumbers$-equivalence classes of) flowlines $\flowManifold (s)$ from the critical point $\critManifold^-$ to the critical point $\critManifold^+$.
        The \define{cohomology Morse differential} for $\MorseFunction$, in contrast, is $d^\ArbitraryIndex(\critManifold^+; \MorseFunction) = \critManifold^- + \cdots$.
        If we reverse the flowline $\flowManifold(s)$, we get the flowline $\flowManifold (-s)$ of the Morse function $- \MorseFunction$ which goes from $\critManifold^+$ to $\critManifold^-$.
        The cohomology Morse differential for $-\MorseFunction$ is $d^\ArbitraryIndex(\critManifold^-; -\MorseFunction) = \critManifold^+ + \cdots$, so it coincides with $d_\ArbitraryIndex (\critManifold^-; \MorseFunction)$.
        Therefore the Morse homology of $\MorseFunction$ is isomorphic to the Morse cohomology of $-\MorseFunction$, as desired.
        
        Similarly, in Floer theory, we reverse the Floer solutions to recover Poincar\'{e} duality.
        If $\FloerSolution(s, t)$ is a Floer solution for $\Hamiltonian_t$, then $\FloerSolution(-s, -t)$ is a Floer solution for $- \Hamiltonian_{-t}$.
        The homological convention $\NovVariable^\DegreeTwoClass \WithFilling{\HamiltonianOrbit} \sim \DegreeTwoClass \ConnectedSum \WithFilling{\HamiltonianOrbit}$ for fillings was described in \cite[Section~5]{hofer_floer_1995}.
        Under Poincar\'{e} duality, the filling $f : \Disc \to \Manifold$ for $\HamiltonianOrbit(t)$ is mapped to the filling $\Conjugate{f} : \Disc \to \Manifold$ for $\HamiltonianOrbit(-t)$, where $\Conjugate{f}$ denotes the precomposition of $f$ with the complex conjugate map.
        Therefore, the filling $\DegreeTwoClass \ConnectedSum f$ is mapped to $(-\DegreeTwoClass) \ConnectedSum \Conjugate{f}$ under Poincar\'{e} duality.
        With our Novikov ring $\NovikovRing$ from \autoref{sec:novikov-ring}, the relation \eqref{eqn:cohomological-convention-for-fillings} is the correct convention for the $\NovikovRing$-ring action on Floer cohomology.
        This remark corrects the convention in \cite[Section~2.7]{ritter_floer_2014} and an earlier draft of \cite{liebenschutz-jones_intertwining_2020}.
    \end{remark}

\subsection{Floer solutions}

    A \define{Floer datum} is a pair $(\Hamiltonian, \ACS)$ consisting of a time-dependent linear Hamiltonian $\Hamiltonian$ and a time-dependent convex $\SymplecticForm$-compatible almost complex structure $\ACS$.
    A \define{Floer solution} is a smooth map $\FloerSolution : \RealNumbers \times \Circle \to \Manifold$ which satisfies
        \begin{equation}
            \partial_s \FloerSolution - \ACS_t \left( \partial_t \FloerSolution - \HamiltonianVectorField{\Hamiltonian, t} \right) = 0
        \end{equation}
    for all $(s, t) \in \RealNumbers \times \Circle$ and whose energy
        \begin{equation}
            \Energy(\FloerSolution) = \int_{\RealNumbers \times \Circle} \Norm{\partial_s \FloerSolution}^2 _{\SymplecticForm(\Argument, \ACS_t \Argument)} \wrt{s} \wrt{t}
        \end{equation}
    is finite.
    Here, the norm is taken with respect to the time-dependent metric $\SymplecticForm(\Argument, \ACS_t \Argument)$.
    
    For a regular Floer datum $(\Hamiltonian, \ACS)$, we consider the moduli space $\ModuliSpace (\WithFilling{\HamiltonianOrbit}^-, \WithFilling{\HamiltonianOrbit}^+)$ of Floer solutions $\FloerSolution$ which satisfy the uniform limits $\FloerSolution (s, \Argument) \to \HamiltonianOrbit^\pm (\Argument)$ and $\partial_s \FloerSolution (s, \Argument) \to 0$ as $s \to \pm \Infinity$ and the condition $\WithFilling{\HamiltonianOrbit}^+ = \FloerSolution \ConnectedSum \WithFilling{\HamiltonianOrbit}^-$ on the fillings\footnote{
        Here, the notation $\FloerSolution \ConnectedSum \WithFilling{\HamiltonianOrbit}^-$ denotes the cylinder $u : \RealNumbers \times \Circle \to \Manifold$ and the closed disc $\WithFilling{\HamiltonianOrbit}^- : \Disc \to \Manifold$ glued along their common boundary $\HamiltonianOrbit^- : \Circle \to \Manifold$, which is a filling of the orbit $\HamiltonianOrbit^+$.
    }.
    This moduli space is a smooth manifold of dimension $|\WithFilling{\HamiltonianOrbit}^-| - |\WithFilling{\HamiltonianOrbit}^+|$ and has a free $\RealNumbers$-action given by $s$-translation (except for constant $\FloerSolution$).
    Moreover, the Floer solutions are restricted to a compact subset of $\Manifold$ by a maximum principle.
    The moduli space may be compactified by broken Floer solutions with bubbling via standard compactification and gluing theorems, and bubbling for $\dim \ModuliSpace (\WithFilling{\HamiltonianOrbit}^-, \WithFilling{\HamiltonianOrbit}^+) \le 2$ is avoided by regularity (see \cite{salamon_lectures_1997}).
    
    The \define{Floer differential} counts Floer solutions mod $s$-translation.
    It is given by
        \begin{equation}
            d \WithFilling{\HamiltonianOrbit}^+ = \sum_{\substack{
                \WithFilling{\HamiltonianOrbit}^- \in \WithFilling{\HamiltonianOrbitSet}(\Hamiltonian) \\
                |\WithFilling{\HamiltonianOrbit}^-| - |\WithFilling{\HamiltonianOrbit}^+| = 1
            }}
            \Count \left(            \frac{\ModuliSpace(\WithFilling{\HamiltonianOrbit}^-, \WithFilling{\HamiltonianOrbit}^+)}{\RealNumbers} \right) \ \WithFilling{\HamiltonianOrbit}^-.
        \end{equation}
    The orientation of $\ModuliSpace(\WithFilling{\HamiltonianOrbit}^-, \WithFilling{\HamiltonianOrbit}^+)$ comes from a choice of \emph{coherent orientation}, which shall be implicit throughout (see \cite[Appendix~B]{ritter_topological_2013} for details).
    
    The \define{Floer cohomology} $\FloerCohomology^\ArbitraryIndex (\Manifold, \Slope; \Hamiltonian, \ACS)$ is the cohomology of the cochain complex $(\FloerCochainComplex^\ArbitraryIndex (\Manifold, \Slope; \Hamiltonian), d)$.
    
\subsection{Continuation maps}

    Given two Floer data $(\Hamiltonian^-, \ACS^-)$ and $(\Hamiltonian^+, \ACS^+)$, a \define{homotopy} between them is an $\RealNumbers$-dependent family of Floer data $(\Hamiltonian_s, \ACS_s)$ which depends on $s \in \RealNumbers$ only on a bounded interval and equals $(\Hamiltonian^\pm, \ACS^\pm)$ as $s \to \pm \infty$.
    The homotopy is \define{monotone} if the slope of $\Hamiltonian_s$ is nonincreasing as $s$ increases.
    Notice this means the slopes $\Slope^\pm$ of $\Hamiltonian^\pm$ satisfy $\Slope^- \ge \Slope^+$.
    Denote by $\delta$ the difference $\Slope^- - \Slope^+ \ge 0$ of the slopes $\Slope^\pm$ of $\Hamiltonian^\pm$.
    
    A \define{continuation solution} is a smooth map $\FloerSolution : \RealNumbers \times \Circle \to \Manifold$ which satisfies
        \begin{equation}
            \partial_s \FloerSolution - \ACS_{s, t} \left( \partial_t \FloerSolution - \HamiltonianVectorField{\Hamiltonian, s, t} \right) = 0
        \end{equation}
    and whose energy
        \begin{equation}
            \Energy(\FloerSolution) = \int_{\RealNumbers \times \Circle} \Norm{\partial_s \FloerSolution}^2 _{\SymplecticForm(\Argument, \ACS_{s, t} \Argument)} \wrt{s} \wrt{t}
        \end{equation}
    is finite.
    That is, a continuation solution is a Floer solution, but with $s$-dependent Floer data.
    The moduli space $\ModuliSpace_s (\WithFilling{\HamiltonianOrbit}^-, \WithFilling{\HamiltonianOrbit}^+)$ of continuation solutions between $\WithFilling{\HamiltonianOrbit}^\pm \in \WithFilling{\HamiltonianOrbitSet}(\Hamiltonian^\pm)$ is defined as for Floer solutions.
    For regular Floer data and a regular monotone homotopy, the moduli spaces $\ModuliSpace_s (\WithFilling{\HamiltonianOrbit}^-, \WithFilling{\HamiltonianOrbit}^+)$ are smooth manifolds of dimension $|\WithFilling{\HamiltonianOrbit}^-| - |\WithFilling{\HamiltonianOrbit}^+|$.
    The map which counts continuation solutions is a chain map.
    The \define{continuation map} between the Floer data is the induced map on Floer cohomology
        \begin{equation}
            \ContinuationMap^\delta : \FloerCohomology^\ArbitraryIndex (\Manifold, \Slope^+; \Hamiltonian^+, \ACS^+) \to \FloerCohomology^\ArbitraryIndex (\Manifold, \Slope^-; \Hamiltonian^-, \ACS^-)
        \end{equation}
    and it is independent of the choice of monotone homotopy.
    The composition of two continuation maps is itself a continuation map.
    If $\Slope^-$ and $\Slope^+$ belong to the same subinterval of $\RealNumbers \setminus \ReebPeriods$, then the continuation map between Floer data of those slopes is an isomorphism \cite[Lemma~11]{ritter_floer_2014}.
    In particular, Floer cohomology only depends on the slope of the Hamiltonian, and is otherwise independent of the choice of Floer datum.
    
    The \define{symplectic cohomology} $\SymplecticCohomology^\ArbitraryIndex(\Manifold)$ of $\Manifold$ is the direct limit of the direct system $\FloerCohomology^\ArbitraryIndex (\Manifold, \Slope)$ for $\Slope \in \RealNumbers \setminus \ReebPeriods$.
    The maps of the direct system are the continuation maps, and the underlying poset is $(\RealNumbers \setminus \ReebPeriods, \le)$.
    
\subsection{Equivariant Floer cohomology}

    Let $\Torus$ and $\TorusAction : \Torus \times \Manifold \to \Manifold$ be as in \autoref{sec:torus-and-torus-action-on-manifold}.
    The extended torus $\ExtendedTorus = \AdditionalCircle \times \Torus$ acts on the (contractible) loop space $\ContractibleLoopSpace{\Manifold} = \Set{\text{contractible $\HamiltonianOrbit : \Circle \to \Manifold$}}$ via
        \begin{equation}
        \label{eqn:action-on-loop-space}
            ((\eltAdditionalCircle, \eltTorus) \cdot \HamiltonianOrbit) (t) = \Extended{\TorusAction}_{(\eltAdditionalCircle, \eltTorus)} (\HamiltonianOrbit(t - \eltAdditionalCircle)) =  \TorusAction_{\eltTorus} (\HamiltonianOrbit(t - \eltAdditionalCircle)).
        \end{equation}
    Recall that $\AdditionalCircle$ acts trivially under $\Extended{\TorusAction}$ as in \autoref{sec:extended-torus}.
    Our model of $\ExtendedTorus$-equivariant Floer cohomology is a Borel-style model which captures the action \eqref{eqn:action-on-loop-space} on the Hamiltonian orbits.
    
    Fix a basis $\Extended{\basisCocharacter}$ for $\ExtendedTorus$.
    Let $\UniversalSpace \ExtendedTorus$ be the space $(\InfiniteSphere)^{\dimTorus + 1}$ defined using the basis and let $\morseUniversalSpace : \UniversalSpace \ExtendedTorus \to \RealNumbers$ be the corresponding Morse function, as in \autoref{sec:equivariant-morse-theory}.
    A \define{$\ExtendedTorus$-equivariant Hamiltonian function} is a function $\eqHamiltonian : \UniversalSpace \ExtendedTorus \times \Circle \times \Manifold \to \RealNumbers$ which is $\ExtendedTorus$-invariant, so it satisfies $\eqHamiltonian_{\eltUniversalSpace, t} (\eltManifold) = \eqHamiltonian_{(\eltAdditionalCircle, \eltTorus) \Inverse \cdot \eltUniversalSpace, \eltAdditionalCircle + t} (\Extended{\TorusAction}_{(\eltAdditionalCircle, \eltTorus)} (\eltManifold))$, and which satisfies the following condition on the critical points in $\UniversalSpace \ExtendedTorus$.
    For every critical point $\critUniversalSpace \in \CriticalPointSet{\morseUniversalSpace, \UniversalSpace \ExtendedTorus}$, the function $\eqHamiltonian_{\critUniversalSpace, \Argument} (\Argument)$ is a regular Hamiltonian on $\Circle \times \Manifold$, and moreover $\eqHamiltonian_{\critUniversalSpace, \Argument} (\Argument) = \eqHamiltonian_{\eltUniversalSpace, \Argument} (\Argument)$ for all $\eltUniversalSpace$ in a neighbourhood of $\critUniversalSpace$ in $\UnstableManifold(\critUniversalSpace) \Union \StableManifold(\critUniversalSpace)$.
    Moreover, we ask that $\eqHamiltonian$ is \define{linear} of some slope $\Slope \in \RealNumbers \setminus \ReebPeriods$, which means that there is $\RadialCoord_0 > 1$ such that $\eqHamiltonian_{\Argument, \Argument} (\ConvexCoordMap (\RadialCoord, \eltContactManifold)) = \Slope \RadialCoord + \text{constant}$ for all $\RadialCoord \ge \RadialCoord_0$.
    
    A \define{$\ExtendedTorus$-equivariant Hamiltonian orbit} is a $\ExtendedTorus$-equivalence class $[\critUniversalSpace, \HamiltonianOrbit] \in \UniversalSpace \ExtendedTorus \times_{\ExtendedTorus} \ContractibleLoopSpace{\Manifold}$ with $\critUniversalSpace \in \CriticalPointSet{\morseUniversalSpace, \UniversalSpace \ExtendedTorus}$ and $\HamiltonianOrbit \in \HamiltonianOrbitSet (\eqHamiltonian_{\critUniversalSpace, \Argument})$.
    The set of such equivariant orbits is denoted $\HamiltonianOrbitSet^{\eqnt} (\eqHamiltonian)$.
    We can attach equivalence classes of fillings as per \autoref{sec:hamiltonian-dynamics}.
    The corresponding set is denoted $\WithFilling{\HamiltonianOrbitSet}^{\eqnt} (\eqHamiltonian)$ and the index of $[\critUniversalSpace, \WithFilling{\HamiltonianOrbit}] \in \WithFilling{\HamiltonianOrbitSet}^{\eqnt} (\eqHamiltonian)$ is $|\critUniversalSpace, \WithFilling{\HamiltonianOrbit}| = |\critUniversalSpace| + |\WithFilling{\HamiltonianOrbit}|$.
    
    For each equivariant orbit $[\critUniversalSpace, \HamiltonianOrbit] \in \HamiltonianOrbitSet^{\eqnt} (\eqHamiltonian)$, fix a filling $\WithFilling{\HamiltonianOrbit}_{\FillingBasis}$ of $\HamiltonianOrbit$.
    Let $\FillingBasis \subseteq \WithFilling{\HamiltonianOrbitSet}^{\eqnt} (\eqHamiltonian)$ be the set of such equivariant orbits with fillings.
    Such a set is a \define{basis of fillings}.
    The \define{$\ExtendedTorus$-equivariant Floer cochain complex} $\FloerCochainComplex^\ArbitraryIndex_{\ExtendedTorus} (\Manifold, \Extended{\TorusAction}, \Slope; \eqHamiltonian)$ is the $\Integers$-graded $\NovikovRing$-module with degree-$l$ summand
        \begin{equation}
        \label{eqn:equivariant-floer-cochain-complex-definition}
            \FloerCochainComplex^l_{\ExtendedTorus} (\Manifold, \Extended{\TorusAction}, \Slope; \eqHamiltonian) = 
            \DirectProduct_{[\critUniversalSpace, \WithFilling{\HamiltonianOrbit}_{\FillingBasis}] \in \FillingBasis}
            \left( \NovikovRing \langle [\critUniversalSpace, \WithFilling{\HamiltonianOrbit}_{\FillingBasis}] \rangle \right)^l.
        \end{equation}
    Let us unpack this definition.
    First, this is independent of the choices in $\FillingBasis$ just as in the non-equivariant definition: the equivariant orbit $[\critUniversalSpace, (-\DegreeTwoClass) \ConnectedSum \WithFilling{\HamiltonianOrbit}_{\FillingBasis}] \in \WithFilling{\HamiltonianOrbitSet}^{\eqnt} (\eqHamiltonian)$ corresponds to the element $\NovVariable^\DegreeTwoClass [\critUniversalSpace, \WithFilling{\HamiltonianOrbit}_{\FillingBasis}]$.
    Second, for each critical point $\critClassifyingSpace \in \CriticalPointSet{\morseUniversalSpace, \ClassifyingSpace \ExtendedTorus}$, there are finitely-many equivariant orbits $[\critUniversalSpace, \HamiltonianOrbit] \in \HamiltonianOrbitSet^{\eqnt} (\eqHamiltonian)$ with $[\critUniversalSpace] = \critClassifyingSpace$.
    Thus the part of the complex \eqref{eqn:equivariant-floer-cochain-complex-definition} made up of such orbits is isomorphic to $\FloerCochainComplex^{l - |\critClassifyingSpace|} (\Manifold, \Slope; \eqHamiltonian_{\critUniversalSpace, \Argument})$ for any $\critUniversalSpace$ with $[\critUniversalSpace] = \critClassifyingSpace$.
    
    A \define{$\ExtendedTorus$-equivariant time-dependent almost complex structure} $\eqACS$ is a choice of almost complex structure $\eqACS_{\eltUniversalSpace, t}$ for all $\eltExtendedTorus \in \UniversalSpace \ExtendedTorus$ and all $t \in \Circle$ such that the diagram
        \begin{equation}
        \label{eqn:equivariant-time-dependent-almost-complex-structure-commutative-diagram}
        \begin{tikzcd}[column sep=huge, row sep=large]
            \TangentSpace _\eltManifold \Manifold
            \arrow[r, "\ACS^\eqnt_{\eltUniversalSpace, t}"]
            \arrow[d, "\Derivative \Extended{\TorusAction}_{(\eltAdditionalCircle, \eltTorus)}"]
            & \TangentSpace _\eltManifold \Manifold
            \arrow[d, "\Derivative \Extended{\TorusAction}_{(\eltAdditionalCircle, \eltTorus)}"] \\
            \TangentSpace _{\Extended{\TorusAction}_{(\eltAdditionalCircle, \eltTorus)} (\eltManifold)} \Manifold
            \arrow[r, "\ACS^\eqnt_{(\eltAdditionalCircle, \eltTorus) \Inverse \cdot \eltUniversalSpace, t + \eltAdditionalCircle}"]
            & \TangentSpace _{\Extended{\TorusAction}_{(\eltAdditionalCircle, \eltTorus)} (\eltManifold)} \Manifold
        \end{tikzcd}
        \end{equation}
    commutes, and which satisfies the following condition on the critical points in $\UniversalSpace \ExtendedTorus$.
    For every critical point $\critUniversalSpace \in \CriticalPointSet{\morseUniversalSpace, \UniversalSpace \ExtendedTorus}$, the time-dependent almost complex structure $\eqACS_{\critUniversalSpace, \Argument}$ is regular, and moreover $\eqACS_{\critUniversalSpace, \Argument} = \eqACS_{\eltUniversalSpace, \Argument}$ for all $\eltUniversalSpace$ in a neighbourhood of $\critUniversalSpace$ in $\UnstableManifold(\critUniversalSpace) \Union \StableManifold(\critUniversalSpace)$.
    A \define{$\ExtendedTorus$-equivariant Floer datum} is a pair $(\eqHamiltonian, \eqACS)$ comprising a linear  $\ExtendedTorus$-equivariant time-dependent Hamiltonian function of slope $\Slope$ and a convex $\SymplecticForm$-compatible $\ExtendedTorus$-equivariant time-dependent almost complex structure $\eqACS$.

\begin{figure}
\centering
\begin{subfigure}{0.4\textwidth}
    \centering
	\begin{center}
		\begin{tikzpicture}
			\node[inner sep=0] at (2.5,0) {\includegraphics[width=4 cm]{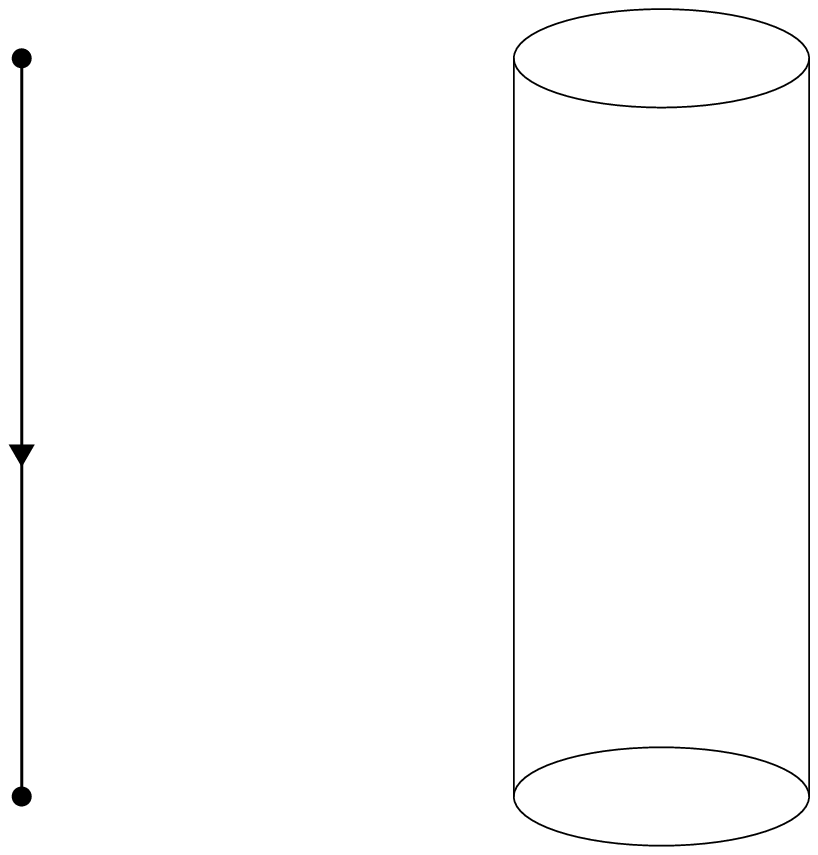}};
% 			\draw[step=1.0,black,thin] (0,-3) grid (5,3);
			
% 			\node at (0,0){0};
			\node at (0.4, 0.8){\tiny $\flowUniversalSpace$};
			\node at (2.8, 0.8){\tiny $\FloerSolution$};
			\node at (0.7, 2.1){\tiny $\critUniversalSpace^-$};
			\node at (0.7, -2.1){\tiny $\critUniversalSpace^+$};
			\node at (4.6, 2.1){\tiny $\WithFilling{\HamiltonianOrbit}^-$};
			\node at (4.6, -2.1){\tiny $\WithFilling{\HamiltonianOrbit}^+$};
		\end{tikzpicture}
	\end{center}
    \caption{Equivariant differential} 
\label{fig:equivariant-floer-differential}
\end{subfigure}
\hspace{1em} 
\begin{subfigure}{0.4\textwidth}
  \centering
	\begin{center}
		\begin{tikzpicture}
			\node[inner sep=0] at (2.5,0) {\includegraphics[width=4 cm]{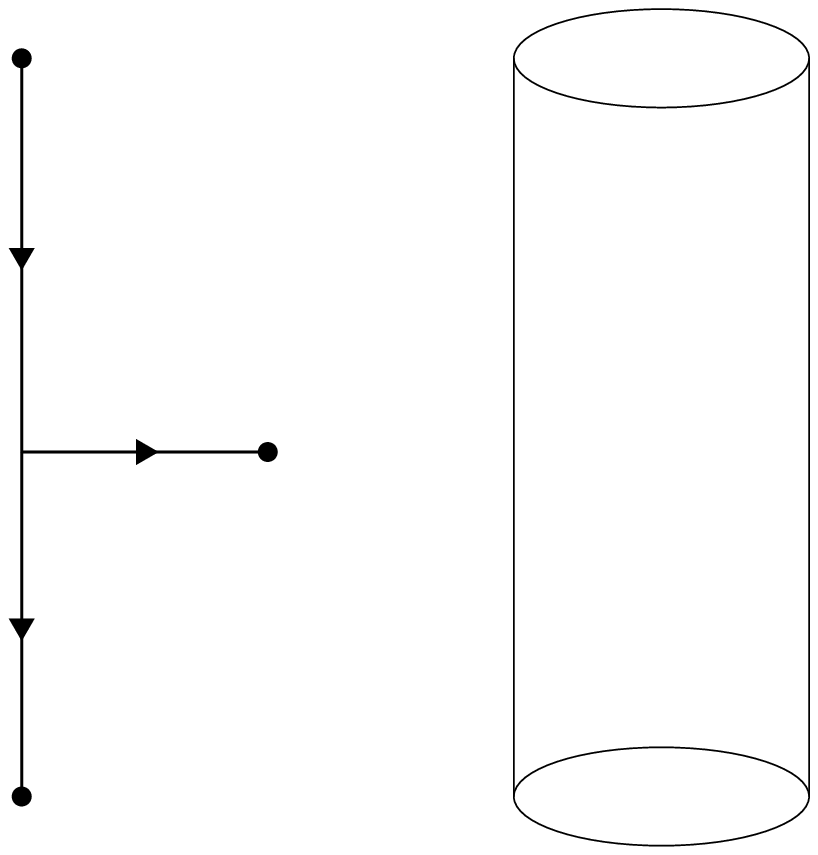}};
% 			\draw[step=1.0,black,thin] (0,-3) grid (5,3);
			
% 			\node at (0,0){0};
			\node at (0.4, 0.8){\tiny $\flowUniversalSpace$};
			\node at (2.8, 0.8){\tiny $\FloerSolution$};
			\node at (0.7, 2.1){\tiny $\critUniversalSpace^-$};
			\node at (0.7, -2.1){\tiny $\critUniversalSpace^+$};
			\node at (4.6, 2.1){\tiny $\WithFilling{\HamiltonianOrbit}^-$};
			\node at (4.6, -2.1){\tiny $\WithFilling{\HamiltonianOrbit}^+$};
			\node at (1.2, 0.1){\tiny $\flowUniversalSpace_0$};
		\end{tikzpicture}
	\end{center}
    \caption{Geometric $\Cohomology^\ArbitraryIndex (\ClassifyingSpace \Torus)$-mod\-ule structure} 
    \label{fig:geometric-module-structure-on-equivariant-floer-cohomology}
\end{subfigure}
\caption{
    The equivariant Floer differential and geometric module action acutely resemble those of equivariant Morse cohomology (\autoref{fig:basic-equivariant-morse-theoretic-constructions}).
    The cylinders on the right-hand side of the diagrams denote a Floer cylinder $\RealNumbers \times \Circle \to \Manifold$.
} \label{fig:basic-equivariant-floer-theoretic-constructions}
\end{figure}
    
    A \define{$\ExtendedTorus$-equivariant Floer solution} is a $\ExtendedTorus$-equivalence class of pairs $(\flowUniversalSpace, \FloerSolution)$ where $\flowUniversalSpace : \RealNumbers \to \UniversalSpace \ExtendedTorus$ is a flowline of $\morseUniversalSpace$ and $\FloerSolution : \RealNumbers \times \Circle \to \Manifold$ is a continuation solution for the $s$-dependent Floer data $(\eqHamiltonian_{\flowUniversalSpace (s), \Argument}, \eqACS_{\flowUniversalSpace (s), \Argument})$.
    See \autoref{fig:equivariant-floer-differential}.
    Notice how the condition on the $\ExtendedTorus$-equivariant Floer datum near the critical points in $\UniversalSpace \ExtendedTorus$ means this $s$-dependent Floer data depends on $s$ only on a bounded interval.
    The $\ExtendedTorus$-action on such pairs is given by $(\eltAdditionalCircle, \eltTorus) \cdot (\flowUniversalSpace, \FloerSolution) = ((\eltAdditionalCircle, \eltTorus) \Inverse \cdot \flowUniversalSpace, \FloerSolution\Prime)$ with $\FloerSolution\Prime (s, t) = \Extended{\TorusAction}_{(\eltAdditionalCircle, \eltTorus)} (\FloerSolution (s, t - \eltAdditionalCircle))$.
    
    The $\ExtendedTorus$-equivariant Floer solution $[\flowUniversalSpace, \FloerSolution]$ \define{converges} to the $\ExtendedTorus$-equivariant orbits $[\critUniversalSpace^\pm, \WithFilling{\HamiltonianOrbit}^\pm]$ as $s \to \pm \infty$ if $\flowUniversalSpace \in \ModuliSpace(\critUniversalSpace^-, \critUniversalSpace^+)$ and $\FloerSolution \in \ModuliSpace_s (\WithFilling{\HamiltonianOrbit}^-, \WithFilling{\HamiltonianOrbit}^+)$ hold for some choices of representatives.
    For a regular $\ExtendedTorus$-equivariant Floer datum, the moduli space $\ModuliSpace([\critUniversalSpace^-, \WithFilling{\HamiltonianOrbit}^-], [\critUniversalSpace^+, \WithFilling{\HamiltonianOrbit}^+])$ of such $\ExtendedTorus$-equivariant Floer solutions is a smooth manifold of dimension $|\critUniversalSpace^-, \WithFilling{\HamiltonianOrbit}^-| - |\critUniversalSpace^+, \WithFilling{\HamiltonianOrbit}^+|$.
    The \define{$\ExtendedTorus$-equivariant Floer differential} counts equivariant Floer solutions mod $s$-translation.
    It is given by
        \begin{equation}
            d [\critUniversalSpace^+, \WithFilling{\HamiltonianOrbit}^+] = 
            \sum_{\substack{
                [\critUniversalSpace^-, \WithFilling{\HamiltonianOrbit}^-] \in \WithFilling{\HamiltonianOrbitSet}^\eqnt(\eqHamiltonian) \\
                |\critUniversalSpace^-, \WithFilling{\HamiltonianOrbit}^-| - |\critUniversalSpace^+, \WithFilling{\HamiltonianOrbit}^+| = 1
            }}
            \Count \left(            \frac{\ModuliSpace([\critUniversalSpace^-, \WithFilling{\HamiltonianOrbit}^-], [\critUniversalSpace^+, \WithFilling{\HamiltonianOrbit}^+])}{\RealNumbers} \right) \ [\critUniversalSpace^-, \WithFilling{\HamiltonianOrbit}^-].
        \end{equation}
    
    The \define{$\ExtendedTorus$-equivariant Floer cohomology} $\FloerCohomology^\ArbitraryIndex_{\ExtendedTorus, \Extended{\basisCocharacter}} (\Manifold, \Extended{\TorusAction}, \Slope; \eqHamiltonian, \eqACS)$ is the cohomology of the $\ExtendedTorus$-equivariant Floer cochain complex.
    It is independent of the $\ExtendedTorus$-equivariant Floer datum by standard homotopy arguments and it is independent of $\Extended{\basisCocharacter}$ by \autoref{remark:morse-bott-construction-to-show-independence-of-torus-basis}.
    
    As with the equivariant Morse cohomology in \autoref{sec:equivariant-morse-theory}, $\ExtendedTorus$-equivariant Floer cohomology has a (geometric) $\Cohomology^\ArbitraryIndex (\ClassifyingSpace \ExtendedTorus)$-module structure.
    It is given by counting $\ExtendedTorus$-equivariant Floer solutions together with a perturbed  half\textsuperscript{$+$} flowline in $\UniversalSpace \ExtendedTorus$ (see \autoref{fig:geometric-module-structure-on-equivariant-floer-cohomology}).
    Combined with the formal $\NovikovRing$-multiplication action on the Floer cochain complex, this yields a $\NovikovRing \GradedCompletedTensorProduct \Cohomology^\ArbitraryIndex (\ClassifyingSpace \ExtendedTorus)$-module structure on $\ExtendedTorus$-equivariant Floer cohomology.
    
    As with (non-equivariant) Floer cohomology, we can construct $\ExtendedTorus$-equivariant continuation maps corresponding to $\ExtendedTorus$-equivariant monotone homotopies between different $\ExtendedTorus$-equivariant Floer data.
    These $\ExtendedTorus$-equivariant continuation maps have the same properties: they compose and they are isomorphisms if the slopes are in the same subinterval of $\RealNumbers \setminus \ReebPeriods$.
    Therefore $\ExtendedTorus$-equivariant Floer cohomology depends only on the slope of the $\ExtendedTorus$-equivariant Hamiltonian.
    
    \define{$\ExtendedTorus$-equivariant symplectic cohomology} $\SymplecticCohomology^\ArbitraryIndex _{\ExtendedTorus} (\Manifold, \Extended{\TorusAction})$ is the direct limit of the direct system $\FloerCohomology^\ArbitraryIndex _{\ExtendedTorus} (\Manifold, \Extended{\TorusAction}, \Slope)$ for $\Slope \in \RealNumbers \setminus \ReebPeriods$.
    
\subsection{Differentiation}

    Fix a regular $\ExtendedTorus$-equivariant Floer datum $(\eqHamiltonian, \eqACS)$ and a basis of fillings $\FillingBasis$.
    For $\DegreeTwoCoclass \in \Cohomology^2 (\Manifold)$, the differentiation map 
        \begin{equation}
             \dbyd{\DegreeTwoCoclass} : \FloerCochainComplex^\ArbitraryIndex_{\ExtendedTorus} (\Manifold, \Extended{\TorusAction}, \Slope; \eqHamiltonian, \eqACS) \to \FloerCochainComplex^\ArbitraryIndex_{\ExtendedTorus} (\Manifold, \Extended{\TorusAction}, \Slope; \eqHamiltonian, \eqACS)
        \end{equation}
    is given by $\NovVariable^\DegreeTwoClass [\critUniversalSpace, \WithFilling{\HamiltonianOrbit}] \mapsto \DegreeTwoCoclass (\DegreeTwoClass) \NovVariable^\DegreeTwoClass [\critUniversalSpace, \WithFilling{\HamiltonianOrbit}]$ for all $[\critUniversalSpace, \WithFilling{\HamiltonianOrbit}] \in \FillingBasis$.
    This map depends on the choice of basis $\FillingBasis$.
    The differenatiation map $\dbyd{\DegreeTwoCoclass}$ is \emph{not} a chain map.
    Instead, we have
        \begin{equation}
        \label{eqn:chain-map-failure-for-formal-differentiation}
            \dbyd{\DegreeTwoCoclass} \ComposedWith d - d \ComposedWith \dbyd{\DegreeTwoCoclass} = d_\DegreeTwoCoclass^\FillingBasis,
        \end{equation}
    where $d_\DegreeTwoCoclass^\FillingBasis$ is the $\DegreeTwoCoclass$-weighted differential \eqref{eqn:weighted-equivariant-floer-differential} for the basis $\FillingBasis$.
    
    An alternative expression for the Floer differential using the filling basis $\FillingBasis$ is
        \begin{equation}
        \label{eqn:equivariant-floer-differential-using-basis}
            d [\critUniversalSpace^+, \WithFilling{\HamiltonianOrbit}^+_{\FillingBasis}] = \hspace{-3em}
            \sum_{\substack{
                \DegreeTwoClass \in \Homology_2 (\Manifold) \\
                [\critUniversalSpace^-, \WithFilling{\HamiltonianOrbit}^-_{\FillingBasis}] \in \FillingBasis \\
                |\critUniversalSpace^-, (-\DegreeTwoClass) \ConnectedSum \WithFilling{\HamiltonianOrbit}^-_{\FillingBasis}| - |\critUniversalSpace^+, \WithFilling{\HamiltonianOrbit}^+_{\FillingBasis}| = 1
            }} \hspace{-1em}
            \sum_{
                \RealNumbers \cdot [\flowUniversalSpace, \FloerSolution] \in \frac{\ModuliSpace([\critUniversalSpace^-, (-\DegreeTwoClass) \ConnectedSum \WithFilling{\HamiltonianOrbit}^-_{\FillingBasis}], [\critUniversalSpace^+, \WithFilling{\HamiltonianOrbit}^+_{\FillingBasis}])}{\RealNumbers}
            } \hspace{-2em}
            \Count \left( \RealNumbers \cdot [\flowUniversalSpace, \FloerSolution] \right)
            \NovVariable^\DegreeTwoClass
            [\critUniversalSpace^-, \WithFilling{\HamiltonianOrbit}^-_{\FillingBasis}].
        \end{equation}
    In this expression, the count operation $\Count \left( \RealNumbers \cdot [\flowUniversalSpace, \FloerSolution] \right)$ simply captures the sign associated to this point of the quotient moduli space.
    Note that $\FloerSolution \ConnectedSum ( (-\DegreeTwoClass) \ConnectedSum \WithFilling{\HamiltonianOrbit}^-_{\FillingBasis}) = \WithFilling{\HamiltonianOrbit}^+_{\FillingBasis}$ holds, so $\DegreeTwoClass$ is determined by the Floer solution $\FloerSolution$ and the fillings $\WithFilling{\HamiltonianOrbit}^\pm_{\FillingBasis}$.
    The \define{$\DegreeTwoCoclass$-weighted differential} $d_\DegreeTwoCoclass^\FillingBasis$ for $\FillingBasis$ is given by
        \begin{equation}
        \label{eqn:weighted-equivariant-floer-differential}
            d_\DegreeTwoCoclass^\FillingBasis [\critUniversalSpace^+, \WithFilling{\HamiltonianOrbit}^+_{\FillingBasis}] = \hspace{-3em}
            \sum_{\substack{
                \DegreeTwoClass \in \Homology_2 (\Manifold) \\
                [\critUniversalSpace^-, \WithFilling{\HamiltonianOrbit}^-_{\FillingBasis}] \in \FillingBasis \\
                |\critUniversalSpace^-, (-\DegreeTwoClass) \ConnectedSum \WithFilling{\HamiltonianOrbit}^-_{\FillingBasis}| - |\critUniversalSpace^+, \WithFilling{\HamiltonianOrbit}^+_{\FillingBasis}| = 1
            }} \hspace{-1em}
            \sum_{
                \RealNumbers \cdot [\flowUniversalSpace, \FloerSolution] \in \frac{\ModuliSpace([\critUniversalSpace^-, (-\DegreeTwoClass) \ConnectedSum \WithFilling{\HamiltonianOrbit}^-_{\FillingBasis}], [\critUniversalSpace^+, \WithFilling{\HamiltonianOrbit}^+_{\FillingBasis}])}{\RealNumbers}
            } \hspace{-2em}
            \Count \left( \RealNumbers \cdot [\flowUniversalSpace, \FloerSolution] \right)
            \DegreeTwoCoclass (\DegreeTwoClass)
            \NovVariable^\DegreeTwoClass
            [\critUniversalSpace^-, \WithFilling{\HamiltonianOrbit}^-_{\FillingBasis}].
        \end{equation}
    The $\DegreeTwoCoclass$-weighted differential $d_\DegreeTwoCoclass^\FillingBasis$ has an additional \define{weight} $\DegreeTwoCoclass(\DegreeTwoClass)$, but otherwise resembles \eqref{eqn:equivariant-floer-differential-using-basis}.
    
\subsection{Equivariant quantum action}

    The quantum cohomology action on Floer cohomology is the map $\QuantumAction : \QuantumCohomology^\ArbitraryIndex (\Manifold) \Tensor \FloerCohomology^\ArbitraryIndex (\Manifold, \Slope) \to \FloerCohomology^\ArbitraryIndex (\Manifold, \Slope)$ which counts Floer solutions $\FloerSolution$ together with a perturbed half\textsuperscript{$+$} flowline $\flowManifold_0$ from $\FloerSolution(0, 0) = \flowManifold_0 (0)$ (see \autoref{fig:quantum-action-on-floer-cohomology-definition}).
    The map satisfies $\DegreeTwoCoclass \QuantumAction (\secondDegreeTwoCoclass \QuantumAction \WithFilling{\HamiltonianOrbit}) = (\DegreeTwoCoclass \QuantumProduct \secondDegreeTwoCoclass) \QuantumAction \WithFilling{\HamiltonianOrbit}$, where the \emph{quantum} product $(\DegreeTwoCoclass \QuantumProduct \secondDegreeTwoCoclass)$ arises as the Floer solution may bubble when the two perturbed half\textsuperscript{$+$} flowlines come together (see \autoref{fig:quantum-action-on-floer-cohomology-bubbling}).

\begin{figure}
\centering
\begin{subfigure}{0.3\textwidth}
    \centering
	\begin{center}
		\begin{tikzpicture}
			\node[inner sep=0] at (2,0) {\includegraphics[width=3 cm]{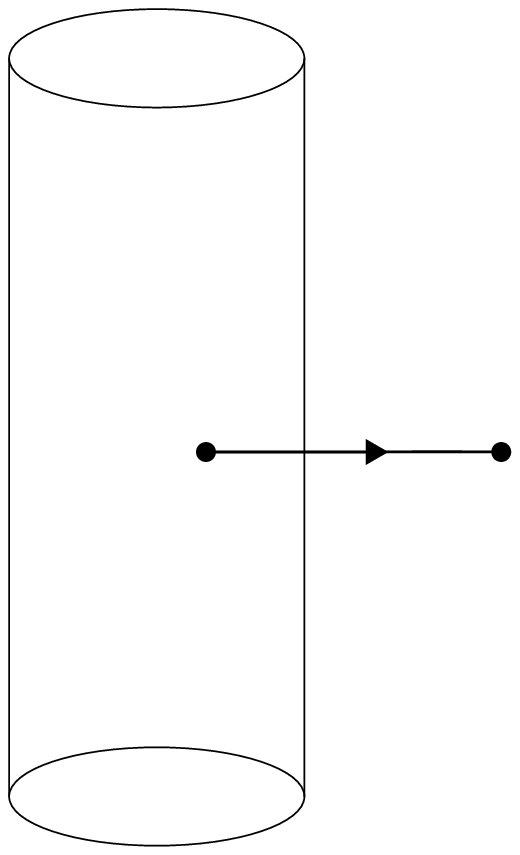}};
% 			\draw[step=1.0,black,thin] (0,-3) grid (4,3);
			
% 			\node at (0,0){0};
			\node at (1.3,-0.15){\tiny $(0, 0)$};
			\node at (0.4,1.3){\tiny $\FloerSolution$};
			\node at (3.6,-0.15){\tiny $\DegreeTwoCoclass$};
		\end{tikzpicture}
	\end{center}
    \caption{Quantum action on Floer cohomology} 
\label{fig:quantum-action-on-floer-cohomology-definition}
\end{subfigure}
\hspace{1em} 
\begin{subfigure}{0.3\textwidth}
  \centering
	\begin{center}
		\begin{tikzpicture}
			\node[inner sep=0] at (2,0) {\includegraphics[width=3 cm]{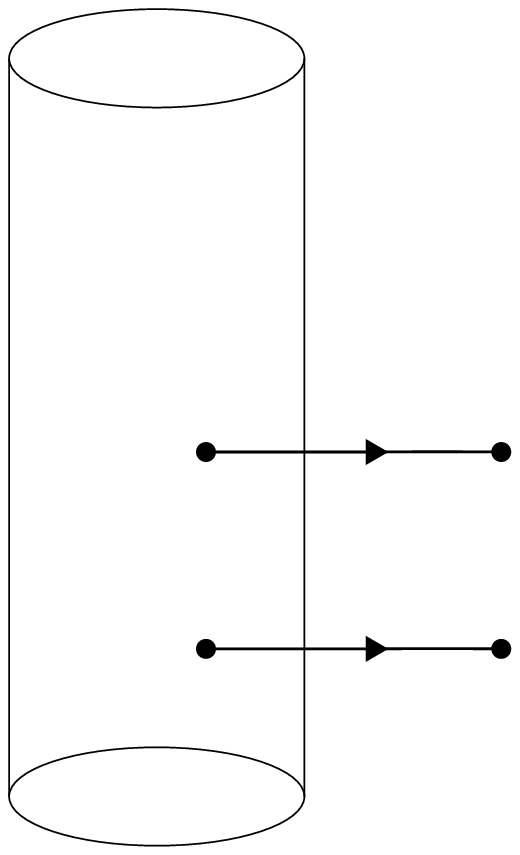}};
% 			\draw[step=1.0,black,thin] (0,-3) grid (4,3);
			
% 			\node at (0,0){0};
			\node at (1.3,-0.15){\tiny $(0, 0)$};
			\node at (1.3,-1.3){\tiny $(s, 0)$};
			\node at (0.4,1.3){\tiny $\FloerSolution$};
			\node at (3.6,-0.15){\tiny $\DegreeTwoCoclass$};
			\node at (3.6,-1.3){\tiny $\secondDegreeTwoCoclass$};
		\end{tikzpicture}
	\end{center}
    \caption{Homotopy with two outgoing half\textsuperscript{+} flowlines} 
    \label{fig:quantum-action-on-floer-cohomology-homotopy}
\end{subfigure}
\hspace{1em}
\begin{subfigure}{0.3\textwidth}
    \centering
	\begin{center}
		\begin{tikzpicture}
			\node[inner sep=0] at (2,0) {\includegraphics[width=3 cm]{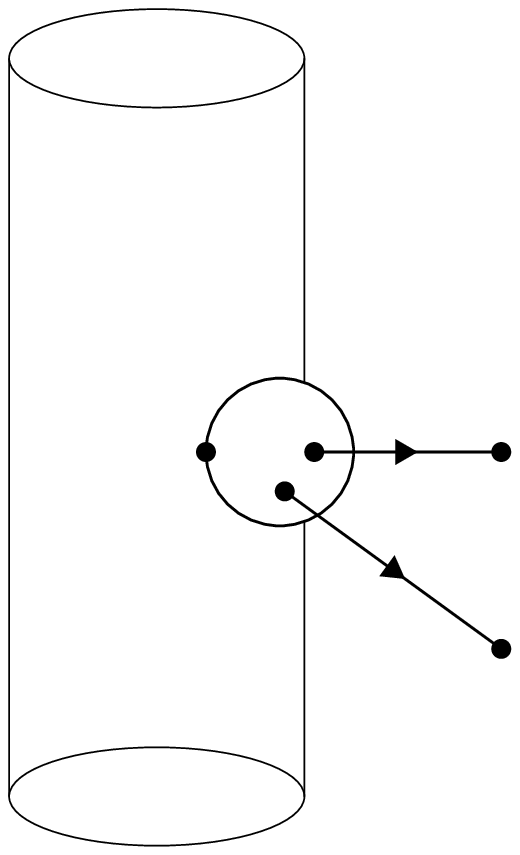}};
% 			\draw[step=1.0,black,thin] (0,-3) grid (4,3);
			
% 			\node at (0,0){0};
			\node at (1.3,-0.15){\tiny $(0, 0)$};
			\node at (0.4,1.3){\tiny $\FloerSolution$};
			\node at (3.6,-0.15){\tiny $\DegreeTwoCoclass$};
			\node at (3.6,-1.3){\tiny $\secondDegreeTwoCoclass$};
		\end{tikzpicture}
	\end{center}
    \caption{Bubbled configuration as the half\textsuperscript{+} flowlines come together}
    \label{fig:quantum-action-on-floer-cohomology-bubbling}
\end{subfigure}
\caption{
    To show the quantum action on Floer cohomology (\textsc{a}) satisfies $\DegreeTwoCoclass \QuantumAction (\secondDegreeTwoCoclass \QuantumAction \WithFilling{\HamiltonianOrbit}) = (\DegreeTwoCoclass \QuantumProduct \secondDegreeTwoCoclass) \QuantumAction \WithFilling{\HamiltonianOrbit}$, construct a homotopy with two half\textsuperscript{+} flowlines as in (\textsc{b}).
    The limit $s \to \infty$ gives the left-hand side $\DegreeTwoCoclass \QuantumAction (\secondDegreeTwoCoclass \QuantumAction \WithFilling{\HamiltonianOrbit})$ as the Floer solution breaks, while the limit $s \to 0^+$ gives the bubbled configuration (\textsc{c}) which is homotopic to the right-hand side $(\DegreeTwoCoclass \QuantumProduct \secondDegreeTwoCoclass) \QuantumAction \WithFilling{\HamiltonianOrbit}$.
} \label{fig:quantum-action-on-floer-cohomology}
\end{figure}
    
    To extend this action to the $\ExtendedTorus$-equivariant setting, we must change the $\FloerSolution(0, 0) = \flowManifold_0 (0)$ condition so that it respects the rotation action \eqref{eqn:action-on-loop-space}.
    Our strategy is just like the homotopy we constructed to prove the intertwining relation, so we use an additional perturbed half\textsuperscript{$+$} flowline in $\UniversalSpace \ExtendedTorus$ together with the argument function $\arg : \minUniversalSpace \to \AdditionalCircle$ to control the input $t_0$ in the new condition $\FloerSolution(0, t_0) = \flowManifold_0 (0)$.
    
    Explicitly, the new map 
        \begin{equation}
        \label{eqn:equivariant-quantum-action-on-equivariant-floer-cohomology}
            \QuantumAction : \Cohomology^2 _{\ExtendedTorus} (\Manifold, \Extended{\TorusAction}) \Tensor \FloerCochainComplex^\ArbitraryIndex _{\ExtendedTorus} (\Manifold, \Extended{\TorusAction}, \Slope; \eqHamiltonian, \eqACS) \to \FloerCochainComplex^\ArbitraryIndex _{\ExtendedTorus} (\Manifold, \Extended{\TorusAction}, \Slope; \eqHamiltonian, \eqACS)
        \end{equation}
    counts $\ExtendedTorus$-equivalence classes of sextuples $(\flowUniversalSpace, \FloerSolution, \flowUniversalSpace_0, \flowManifold_0, \flowUniversalSpace^{\min}, t_0)$ where $[\flowUniversalSpace, \FloerSolution]$ is a $\ExtendedTorus$-equivariant Floer solution, $[\flowUniversalSpace_0, \flowManifold_0]$ is a $\ExtendedTorus$-equivariant perturbed half\textsuperscript{$+$} flowline, $\flowUniversalSpace^{\min}$ is a perturbed half\textsuperscript{$+$} flowline in $\UniversalSpace \ExtendedTorus$ and $t_0 \in \Circle$ is an element of the circle which together satisfy $\flowUniversalSpace(0) = \flowUniversalSpace_0 (0) = \flowUniversalSpace^{\min} (0)$, $\FloerSolution (0, t_0) = \flowManifold_0 (0)$ and $\arg (\flowUniversalSpace^{\min} (+\infty)) + t_0 = 0$.
    See \autoref{fig:quantum-action-on-equivariant-floer-cohomology-definition}.
    We are interested in applying the action $\QuantumAction$ to the classes $\DegreeTwoCoclass^\FixedPoint \in \Cohomology^2 _{\ExtendedTorus} (\Manifold, \Extended{\TorusAction})$.
    
    Using the same Morse data as for the flowline $[\flowUniversalSpace_0, \flowManifold_0]$ in the construction of \eqref{eqn:equivariant-quantum-action-on-equivariant-floer-cohomology}, let $f : \Disc \to \Manifold$ be a filling for $[\critUniversalSpace, \WithFilling{\HamiltonianOrbit}]$ for which $(\critUniversalSpace, f (\Disc)) \Intersection \PerturbedStableManifold ( \DegreeTwoCoclass^\FixedPoint)$ is a transverse intersection.
    Here, the \define{(perturbed) stable manifold} $\PerturbedStableManifold (\DegreeTwoCoclass^\FixedPoint)$ of $\ExtendedTorus$-equivariant critical point $\DegreeTwoCoclass^\FixedPoint$ is the set of points $(\eltUniversalSpace, \eltManifold) \in \UniversalSpace \ExtendedTorus \times \Manifold$ for which there is a perturbed half\textsuperscript{$+$} flowline $[\flowUniversalSpace_0, \flowManifold_0]$ to $\DegreeTwoCoclass^\FixedPoint$ with $(\flowUniversalSpace_0, \flowManifold_0) (0) = (\eltUniversalSpace, \eltManifold)$.
    By regularity, no intersections occur on the Hamiltonian orbit $[\critUniversalSpace, \HamiltonianOrbit]$ itself.
    Set 
        \begin{equation}
            \FillingIntersectionMap_\DegreeTwoCoclass ([\critUniversalSpace, \WithFilling{\HamiltonianOrbit}]) = \Count ((\critUniversalSpace, f (\Disc)) \Intersection \PerturbedStableManifold ( \DegreeTwoCoclass^\FixedPoint)) \  [\critUniversalSpace, \WithFilling{\HamiltonianOrbit}],
        \end{equation}
    a degree-0 map on the $\ExtendedTorus$-equivariant Floer cochain complex.
    Similarly, let $\FillingIntersectionMap_\DegreeTwoCoclass^\FillingBasis$ be the $\NovikovRing$-linear map given by
        \begin{equation}
            \FillingIntersectionMap_\DegreeTwoCoclass^\FillingBasis (\NovVariable^\DegreeTwoClass [\critUniversalSpace, \WithFilling{\HamiltonianOrbit}_{\FillingBasis}]) = \Count ((\critUniversalSpace, f_\FillingBasis (\Disc)) \Intersection \PerturbedStableManifold ( \DegreeTwoCoclass^\FixedPoint)) \  \NovVariable^\DegreeTwoClass [\critUniversalSpace, \WithFilling{\HamiltonianOrbit}_{\FillingBasis}],
        \end{equation}
    where $f_\FillingBasis$ is a filling for $[\critUniversalSpace, \WithFilling{\HamiltonianOrbit}_{\FillingBasis}]$.
    The convention \eqref{eqn:cohomological-convention-for-fillings} yields
        \begin{equation}
        \label{eqn:basis-independence-for-differentiation-with-filling-map}
            \dbyd{\DegreeTwoCoclass} - \FillingIntersectionMap^\FillingBasis_\DegreeTwoCoclass = - \FillingIntersectionMap_\DegreeTwoCoclass.
        \end{equation}
    Both $\FillingIntersectionMap_\DegreeTwoCoclass$ and $\FillingIntersectionMap^\FillingBasis_\DegreeTwoCoclass$ are independent of the actual fillings $f$ and $f_\FillingBasis$ used, but they do depend on the choice of Morse data for $\PerturbedStableManifold ( \DegreeTwoCoclass^\FixedPoint))$.
        
    Let the (basis-free) \define{$\DegreeTwoCoclass^\FixedPoint$-weighted differential $d_\DegreeTwoCoclass$} be given by
        \begin{equation}
        \label{eqn:weighted-floer-differential-intersections-in-floer-solution-only}
            d_\DegreeTwoCoclass [\critUniversalSpace^+, \WithFilling{\HamiltonianOrbit}^+] = 
            \sum_{\substack{
                [\critUniversalSpace^-, \WithFilling{\HamiltonianOrbit}^-] \in \WithFilling{\HamiltonianOrbitSet}^\eqnt(\eqHamiltonian) \\
                |\critUniversalSpace^-, \WithFilling{\HamiltonianOrbit}^-| - |\critUniversalSpace^+, \WithFilling{\HamiltonianOrbit}^+| = 1 \\
                \RealNumbers \cdot [(\flowUniversalSpace, \FloerSolution)] \in \frac{\ModuliSpace([\critUniversalSpace^-, \WithFilling{\HamiltonianOrbit}^-], [\critUniversalSpace^+, \WithFilling{\HamiltonianOrbit}^+])}{\RealNumbers}
            }}
            \ \Count ((\flowUniversalSpace, \FloerSolution) (\RealNumbers \times \Circle) \Intersection \PerturbedStableManifold ( \DegreeTwoCoclass^\FixedPoint)) \  [\critUniversalSpace^-, \WithFilling{\HamiltonianOrbit}^-].
        \end{equation}
    This map counts Floer solutions weighted by their intersections with $\PerturbedStableManifold ( \DegreeTwoCoclass^\FixedPoint)$, and, unlike $d_\DegreeTwoCoclass^\FillingBasis$, it does not incorporate any intersections with the fillings.
    Combining the definitions yields
        \begin{equation}
            d_\DegreeTwoCoclass = d_\DegreeTwoCoclass^\FillingBasis - \FillingIntersectionMap^\FillingBasis_\DegreeTwoCoclass d + d \FillingIntersectionMap^\FillingBasis_\DegreeTwoCoclass.
        \end{equation}
    
    \begin{proposition}
    \label{prop:quantum-multiplication-commutes-with-differential-with-error}
        For the class $\DegreeTwoCoclass^\FixedPoint \in \Cohomology^2 _{\ExtendedTorus} (\Manifold, \Extended{\TorusAction})$, we have
            \begin{equation}
            \label{eqn:chain-map-failure-for-equivariant-quantum-action}
                 d (\DegreeTwoCoclass^\FixedPoint \QuantumAction \Argument) - \DegreeTwoCoclass^\FixedPoint \QuantumAction d (\Argument) = \formalAdditionalCircle d_\DegreeTwoCoclass.
            \end{equation}
    \end{proposition}
    
    \begin{proof}
        The proof is analogous to the proof of the intertwining relation \autoref{thm:intertwining-relation-quantum-extended-torus}.
        The 1-dimensional moduli space of $\ExtendedTorus$-equivalence classes $[\flowUniversalSpace, \FloerSolution, \flowUniversalSpace_0, \flowManifold_0, \flowUniversalSpace^{\min}, t_0]$ has a boundary given by breaking either the $\ExtendedTorus$-equivariant Floer solution $[\flowUniversalSpace, \FloerSolution]$, breaking the $\ExtendedTorus$-equivariant perturbed half\textsuperscript{$+$} flowline to $\DegreeTwoCoclass^\FixedPoint$ or breaking the perturbed half\textsuperscript{$+$} flowline $\flowUniversalSpace^{\min}$.
        
        The terms on the left-hand side of \eqref{eqn:chain-map-failure-for-equivariant-quantum-action} come from the $\ExtendedTorus$-equivariant Floer solution breaking, and the term on the right-hand side comes from the perturbed half\textsuperscript{$+$} flowline $\flowUniversalSpace^{\min}$ breaking.
        When the perturbed half\textsuperscript{$+$} flowline $\flowUniversalSpace^{\min}$ breaks, it breaks into a perturbed half\textsuperscript{$+$} flowline to $\critClassifyingSpace \in \CriticalPointSet{\morseUniversalSpace, \ClassifyingSpace \ExtendedTorus}$ and a flowline (considered modulo $s$-translation) between $\critClassifyingSpace$ and $\minUniversalSpace$.
        Only the case $\formalTorus^c = \formalAdditionalCircle$ has a 0-dimensional moduli space by regularity.
        The flowline between $\critClassifyingSpace$ and $\minUniversalSpace$ is uniquely determined by the rest of the configuration, for its endpoint in $\minUniversalSpace$ must satisfy $\arg ( \Argument) + t_0 = 0$.
        Therefore the moduli space is unchanged if we omit the flowline between $\critClassifyingSpace$ and $\minUniversalSpace$ (we omit the argument relation too, so now $t_0 \in \Circle$ is arbitrary).
        The perturbed half\textsuperscript{$+$} flowline to $\critClassifyingSpace$ contributes the multiplication by $\formalAdditionalCircle$.
        The $\ExtendedTorus$-equivariant half\textsuperscript{$+$} flowline to $\DegreeTwoCoclass^\FixedPoint$ means ($\RealNumbers$-equivalence classes of) $\ExtendedTorus$-equivariant Floer solutions $[\flowUniversalSpace, \FloerSolution]$ are counted as many times as they intersect the (perturbed) stable manifold of $\DegreeTwoCoclass^\FixedPoint$.
        % If the class $\DegreeTwoClass$ satisfies $\FloerSolution \ConnectedSum ( (-\DegreeTwoClass) \ConnectedSum \WithFilling{\HamiltonianOrbit}^-_{\FillingBasis}) = \WithFilling{\HamiltonianOrbit}^+_{\FillingBasis}$, as in \eqref{eqn:equivariant-floer-differential-using-basis} and \eqref{eqn:weighted-equivariant-floer-differential}, then $\FloerSolution$ with the fillings $\WithFilling{\HamiltonianOrbit}^\pm_{\FillingBasis}$ attached\footnote{
        %     The filling for $\WithFilling{\HamiltonianOrbit}^+$ is given the opposite orientation.
        % } is a map from the 2-sphere to $\Manifold$ of class $\DegreeTwoClass$.
        % Therefore the number of intersections is $\DegreeTwoCoclass^\FixedPoint (\eltUniversalSpace \PushForward \DegreeTwoClass)$, and all these intersections occur in $\FloerSolution$ because the sum of the intersections in the fillings is zero by assumption.
        % Here, $\eltUniversalSpace \PushForward \DegreeTwoClass$ denotes the pushforward of the class $\DegreeTwoClass$ under the fibre inclusion map $\Manifold \Inclusion \UniversalSpace \ExtendedTorus \times_{\ExtendedTorus} \Manifold$ at $\eltUniversalSpace \in \UniversalSpace \ExtendedTorus$, and it is independent of $\eltUniversalSpace$.
        % Finally, we have $\DegreeTwoCoclass^\FixedPoint (\eltUniversalSpace \PushForward \DegreeTwoClass) = \DegreeTwoCoclass (\DegreeTwoClass)$ by simple diagram chasing with \eqref{eqn:split-short-exact-sequence-borel-quotient}.
    \end{proof}

\begin{figure}
\centering
\begin{subfigure}{0.4\textwidth}
    \centering
	\begin{center}
		\begin{tikzpicture}
			\node[inner sep=0] at (2.5,0) {\includegraphics[width=4 cm]{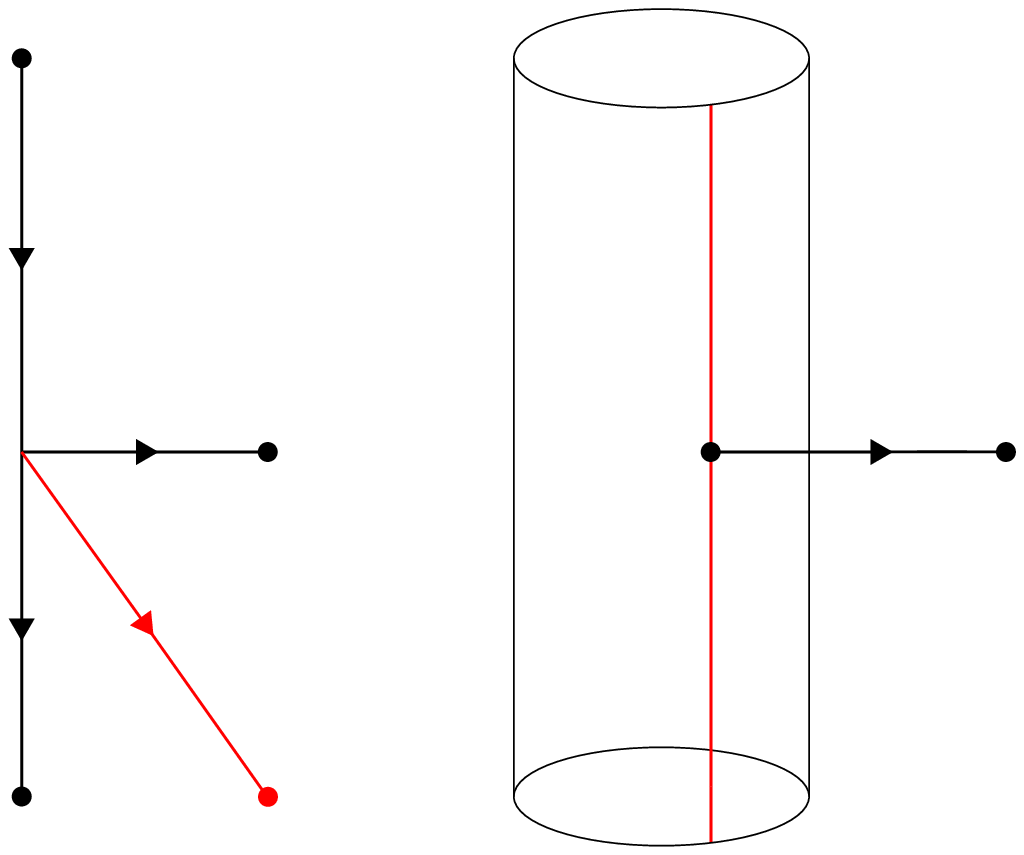}};
% 			\draw[step=1.0,black,thin] (0,-3) grid (5,3);
			
% 			\node at (0,0){0};
			\node at (2.9,-0.1){\tiny $(0, \textcolor{red}{t_0})$};
			\node at (4.0,0.1){\tiny $\flowManifold_0$};
			\node at (0.4, 0.8){\tiny $\flowUniversalSpace$};
			\node at (1.0, 0.1){\tiny $\flowUniversalSpace_0$};
			\node at (1.3, -0.6){\tiny $\flowUniversalSpace^{\min}$};
			\node at (0.4, -0.1){\tiny $0$};
			\node at (2.4, 0.8){\tiny $\FloerSolution$};
			\node at (4.7, -0.1){\tiny $\DegreeTwoCoclass^\FixedPoint$};
		\end{tikzpicture}
	\end{center}
    \caption{Equivariant action $\DegreeTwoCoclass^\FixedPoint \QuantumAction$ on equivariant Floer cohomology} 
\label{fig:quantum-action-on-equivariant-floer-cohomology-definition}
\end{subfigure}
\hspace{1em} 
\begin{subfigure}{0.4\textwidth}
  \centering
	\begin{center}
		\begin{tikzpicture}
			\node[inner sep=0] at (2.5,0) {\includegraphics[width=4.5 cm]{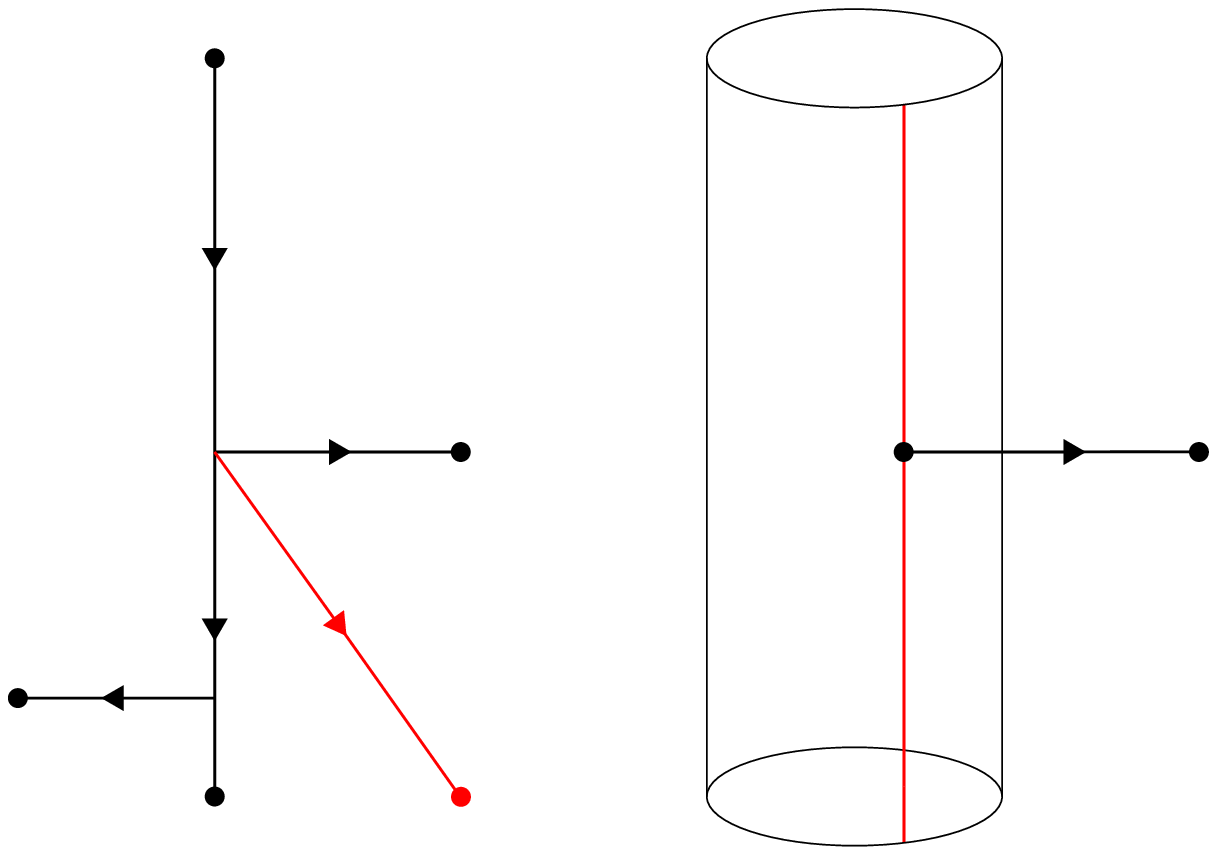}};
% 			\draw[step=1.0,black,thin] (0,-3) grid (5,3);
			
% 			\node at (0,0){0};
			\node at (3.2,-0.1){\tiny $(0, \textcolor{red}{t_0})$};
			\node at (4.2,0.1){\tiny $\flowManifold_0$};
			\node at (0.9, 0.8){\tiny $\flowUniversalSpace$};
			\node at (1.5, 0.1){\tiny $\flowUniversalSpace_0$};
			\node at (1.8, -0.6){\tiny $\flowUniversalSpace^{\min}$};
			\node at (0.9, -0.1){\tiny $0$};
			\node at (1.2, -1){\tiny $s_1$};
			\node at (0.7, -0.9){\tiny $\flowUniversalSpace_1$};
			\node at (0.2, -1.0){\tiny $\formalTorus^\critClassifyingSpace$};
			\node at (2.7, 0.8){\tiny $\FloerSolution$};
			\node at (4.7, -0.3){\tiny $\DegreeTwoCoclass^\FixedPoint$};
		\end{tikzpicture}
	\end{center}
    \caption{Homotopy for geometric module structure and connection} 
    \label{fig:geometric-action-and-connection-homotopy}
\end{subfigure}

\vspace{1em}

\begin{subfigure}{0.4\textwidth}
    \centering
	\begin{center}
		\begin{tikzpicture}
			\node[inner sep=0] at (2.5,0) {\includegraphics[width=4 cm]{figures/equivariant-action-on-floer.eps}};
% 			\draw[step=1.0,black,thin] (0,-3) grid (5,3);
			
% 			\node at (0,0){0};
			\node at (2.8,-0.1){\tiny $(s_0, \textcolor{red}{t_0})$};
			\node at (4.0,0.1){\tiny $\flowManifold_0$};
			\node at (0.4, 0.8){\tiny $\flowUniversalSpace$};
			\node at (1.0, 0.1){\tiny $\flowUniversalSpace_0$};
			\node at (1.3, -0.6){\tiny $\flowUniversalSpace^{\min}$};
			\node at (0.4, -0.1){\tiny $s_0$};
			\node at (2.4, 0.8){\tiny $\FloerSolution$};
			\node at (4.7, -0.1){\tiny $\DegreeTwoCoclass^\FixedPoint$};
		\end{tikzpicture}
	\end{center}
    \caption{Homotopy for continuation maps and connection}
    \label{fig:continuation-maps-and-connection-homotopy}
\end{subfigure}
\hspace{1em}
\begin{subfigure}{0.4\textwidth}
    \centering
	\begin{center}
		\begin{tikzpicture}
			\node[inner sep=0] at (2.5,0) {\includegraphics[width=4.5 cm]{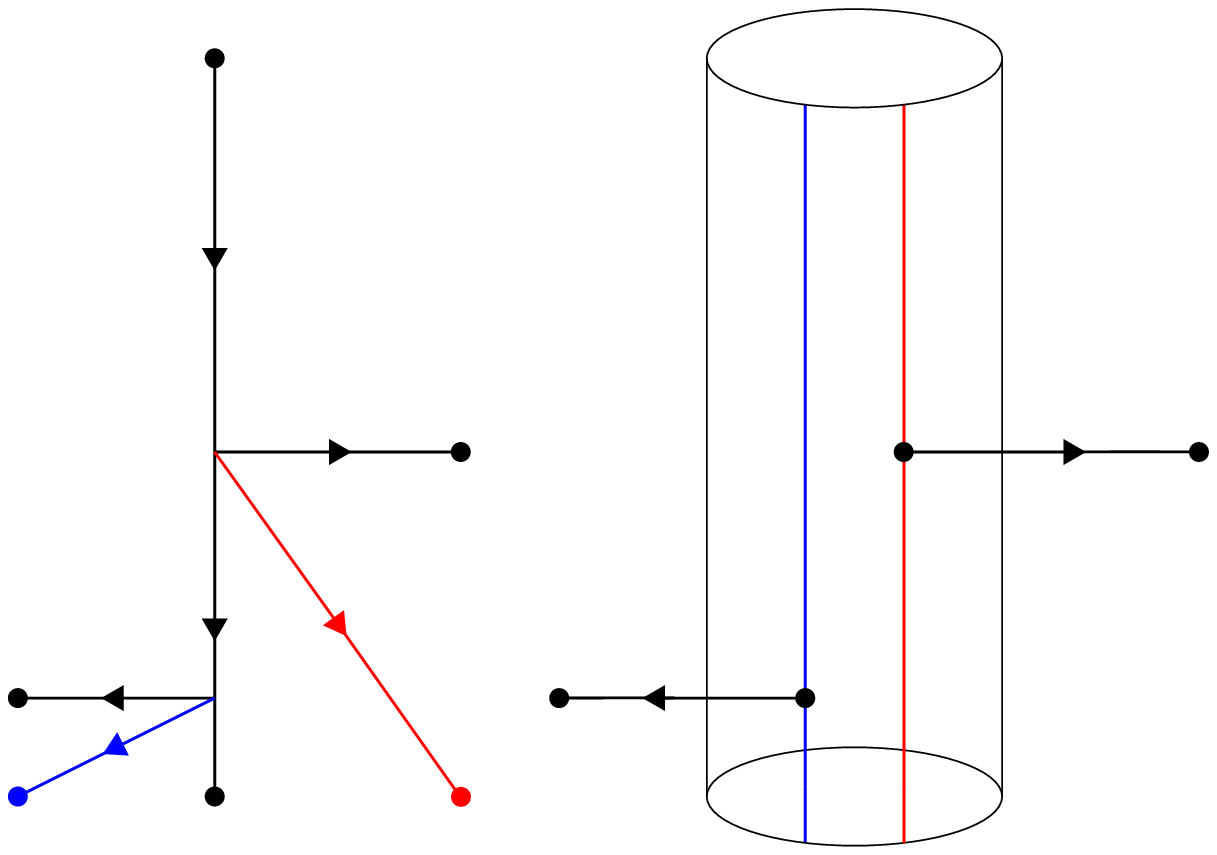}};
% 			\draw[step=1.0,black,thin] (0,-3) grid (5,3);
			
% 			\node at (0,0){0};
			\node at (3.2,-0.1){\tiny $(0, \textcolor{red}{t_0})$};
			\node at (3.7,-1.0){\tiny $(s_1, \textcolor{blue}{t_1})$};
			\node at (4.2,0.1){\tiny $\flowManifold_0$};
			\node at (2.6,-0.9){\tiny $\flowManifold_1$};
			\node at (0.9, 0.8){\tiny $\flowUniversalSpace$};
			\node at (1.5, 0.1){\tiny $\flowUniversalSpace_0$};
			\node at (1.8, -0.6){\tiny $\flowUniversalSpace^{\min}_0$};
			\node at (0.5, -1.5){\tiny $\flowUniversalSpace^{\min}_1$};
			\node at (0.9, -0.1){\tiny $0$};
			\node at (1.2, -1){\tiny $s_1$};
			\node at (0.7, -0.9){\tiny $\flowUniversalSpace_1$};
			\node at (2.7, 0.8){\tiny $\FloerSolution$};
			\node at (4.7, -0.3){\tiny $\DegreeTwoCoclass^\FixedPoint$};
			\node at (2.3, -1.2){\tiny $\secondDegreeTwoCoclass^\FixedPoint$};
		\end{tikzpicture}
	\end{center}
    \caption{Homotopy for flatness}
    \label{fig:floer-connection-is-flat-homotopy}
\end{subfigure}
\caption{
    The connection comprises a formal differentiation operation and the equivariant action $\DegreeTwoCoclass^\FixedPoint \QuantumAction$ in (\textsc{a}).
    The action is modified into three different homotopies which are used to show the connection is compatible with the geometric $\Cohomology^\ArbitraryIndex (\ClassifyingSpace \ExtendedTorus)$-module structure (\textsc{b}), compatible with continuation maps  (\textsc{c}) and flat (\textsc{d}).
}
\end{figure}
        
\subsection{Floer connection}

    The \define{Floer connection} $\Connection^{\FixedPoint, \FillingBasis}$ on $\FloerCohomology^\ArbitraryIndex_{\ExtendedTorus} (\Manifold, \Extended{\TorusAction}, \Slope; \eqHamiltonian, \eqACS)$ is the collection of maps
        \begin{equation}
            \begin{gathered}
                \Connection^{\FixedPoint, \FillingBasis} : \FloerCohomology^\ArbitraryIndex_{\ExtendedTorus} (\Manifold, \Extended{\TorusAction}, \Slope; \eqHamiltonian, \eqACS) \to \FloerCohomology^{\ArbitraryIndex + 2} _{\ExtendedTorus} (\Manifold, \Extended{\TorusAction}, \Slope; \eqHamiltonian, \eqACS) \\
                \Connection^{\FixedPoint, \FillingBasis} _{\DegreeTwoCoclass} (\Argument) = \formalAdditionalCircle \dbyd{\DegreeTwoCoclass} (\Argument) + \DegreeTwoCoclass^\FixedPoint \QuantumAction (\Argument) - \formalAdditionalCircle \FillingIntersectionMap^\FillingBasis_\DegreeTwoCoclass (\Argument)
            \end{gathered}
        \end{equation}
    for each $\DegreeTwoCoclass \in \Cohomology^2 (\Manifold)$.
    The map $\Connection^{\FixedPoint, \FillingBasis} _{\DegreeTwoCoclass}$ is a chain map because the error terms in \eqref{eqn:chain-map-failure-for-formal-differentiation} and \eqref{eqn:chain-map-failure-for-equivariant-quantum-action} cancel with each other.
    The Leibniz rule is satisfied for multiplication with $\NovikovRing$ because the differentiation operation $\formalAdditionalCircle \dbyd{\DegreeTwoCoclass}$ satisfies the Leibniz rule on the cochain complex while the other two maps are $\NovikovRing$-linear on the cochain complex.
    The difference $\dbyd{\DegreeTwoCoclass} - \FillingIntersectionMap^\FillingBasis_\DegreeTwoCoclass$ is independent of $\FillingBasis$ by \eqref{eqn:basis-independence-for-differentiation-with-filling-map}, and hence $\Connection^{\FixedPoint, \FillingBasis}$ is independent of $\FillingBasis$.
    This independence can also be shown using a homotopy construction based on continuation maps (\autoref{prop:connection-compatible-with-continuation-maps}).
    
    \begin{proposition}
    \label{prop:floer-connection-compatible-with-equivariant-action}
        $\Connection^{\FixedPoint, \FillingBasis}_{\DegreeTwoCoclass}$ is a $\Cohomology^\ArbitraryIndex (\ClassifyingSpace \ExtendedTorus)$-linear map.
    \end{proposition}
    
    \begin{proof}
        We construct another homotopy in the intertwining relation style.
        Fix $\formalTorus^\critClassifyingSpace \in \Cohomology^\ArbitraryIndex (\ClassifyingSpace \ExtendedTorus)$ and $\DegreeTwoCoclass \in \Cohomology^2 (\Manifold)$.
        We have 
            \begin{equation}
            \label{eqn:differentiation-relation-for-equivariant-action-on-equivariant-floer-cohomology}
                (-\FillingIntersectionMap_\DegreeTwoCoclass) (\formalTorus^\critClassifyingSpace \ \Argument) - \formalTorus^\critClassifyingSpace ( -\FillingIntersectionMap_\DegreeTwoCoclass) (\Argument) = \formalTorus^\critClassifyingSpace_\DegreeTwoCoclass (\Argument)
            \end{equation}
        where $\formalTorus^\critClassifyingSpace_\DegreeTwoCoclass$ denotes the $\DegreeTwoCoclass$-weighted geometric action of $\formalTorus^\critClassifyingSpace$, defined analogously to \eqref{eqn:weighted-floer-differential-intersections-in-floer-solution-only}.
        The homotopy counts $\ExtendedTorus$-equivalence classes of septuples $(\flowUniversalSpace, \FloerSolution, \flowUniversalSpace_0, \flowManifold_0, \flowUniversalSpace^{\min}, t_0, s_1, \flowUniversalSpace_1)$ where $(\flowUniversalSpace, \FloerSolution, \flowUniversalSpace_0, \flowManifold_0, \flowUniversalSpace^{\min}, t_0)$ is as for \eqref{eqn:equivariant-quantum-action-on-equivariant-floer-cohomology} and $\flowUniversalSpace_1$ is a perturbed half\textsuperscript{$+$} flowline in $\UniversalSpace \ExtendedTorus$ which satisfies $\flowUniversalSpace (s_1) = \flowUniversalSpace_1 (0)$ and $s_1 \in \RealNumbers \setminus \Set{0}$.
        See \autoref{fig:geometric-action-and-connection-homotopy}.
        The components of the 1-dimensional moduli space boundary which survive to cohomology are the limits $s_1 \to \pm \infty$ and the breaking of the perturbed half\textsuperscript{$+$} flowline $\flowUniversalSpace^{\min}$.
        The left and right limits $s_1 \to 0^\pm$ cancel with each other.
        This yields
            \begin{equation}
            \label{eqn:intertwining-relation-for-equivariant-action-on-equivariant-floer-cohomology}
                 \formalTorus^\critClassifyingSpace (\DegreeTwoCoclass^\FixedPoint \QuantumAction (\Argument)) - \DegreeTwoCoclass^\FixedPoint \QuantumAction (\formalTorus^\critClassifyingSpace \ \Argument) \sim \formalAdditionalCircle \formalTorus^\critClassifyingSpace_{\DegreeTwoCoclass} (\Argument).
            \end{equation}
        % Here, $\sim$ indicates this is a chain homotopy, for 
        For simplicity, we have omitted the terms which arise as the $\ExtendedTorus$-equivariant Floer solution breaks.
        These are the terms $d \ComposedWith h$ and $h \ComposedWith d$, where $h$ denotes the homotopy which counts the $\ExtendedTorus$-equivalence classes of sextuples described above.
        The symbol $\sim$ denotes that each side is homotopic to the other.
        This is a slight abuse of notation because the two sides are not chain maps, however when combined with \eqref{eqn:differentiation-relation-for-equivariant-action-on-equivariant-floer-cohomology} in \eqref{eqn:connection-commutes-with-geometric-action}, both sides are indeed chain maps.
        
        While the derivation of \eqref{eqn:intertwining-relation-for-equivariant-action-on-equivariant-floer-cohomology} is mostly like the proof of \autoref{prop:quantum-multiplication-commutes-with-differential-with-error}, there is an additional step of performing an $s$-translation when interpreting the breaking of $\flowUniversalSpace^{\min}$ as $\formalAdditionalCircle \formalTorus^\critClassifyingSpace_{\DegreeTwoCoclass} (\Argument)$.
        More precisely, precomposing $[\flowUniversalSpace, \FloerSolution]$ with the translation $s \mapsto s + s_1$ changes the intersection condition $\flowUniversalSpace (s_1) = \flowUniversalSpace_1 (0)$ to $\flowUniversalSpace (0) = \flowUniversalSpace_1 (0)$.
        After the translation, the intersection between $[\flowUniversalSpace, \FloerSolution]$ and $[\flowUniversalSpace_0, \flowManifold_0]$ occurs at an unconstrained point $(-s_1, t_0) \in \RealNumbers \times \Circle$.
        Therefore, we recover an $\DegreeTwoCoclass$-weighted count of the configurations which determine the `geometric action of $\formalTorus^\critClassifyingSpace$' map.
        
        Together, equations \eqref{eqn:differentiation-relation-for-equivariant-action-on-equivariant-floer-cohomology} and \eqref{eqn:intertwining-relation-for-equivariant-action-on-equivariant-floer-cohomology} yield
            \begin{equation}
            \label{eqn:connection-commutes-with-geometric-action}
                [\Connection^{\FixedPoint, \FillingBasis}, \formalTorus^\critClassifyingSpace] (\Argument) = \formalAdditionalCircle \left( (-\FillingIntersectionMap_\DegreeTwoCoclass) (\formalTorus^\critClassifyingSpace \ \Argument) - \formalTorus^\critClassifyingSpace (-\FillingIntersectionMap_\DegreeTwoCoclass) (\Argument) \right) - \left( \formalTorus^\critClassifyingSpace (\DegreeTwoCoclass^\FixedPoint \QuantumAction (\Argument)) - \DegreeTwoCoclass^\FixedPoint \QuantumAction (\formalTorus^\critClassifyingSpace \ \Argument) \right) \sim 0
            \end{equation}
        as desired.
    \end{proof}
    
\subsection{Connection and continuation maps}

    Let $(\Hamiltonian^{\eqnt,\pm}, \ACS^{\eqnt,\pm})$ be two $\ExtendedTorus$-equivariant Floer data, and let $\FillingBasis^\pm$ be filling bases for the data.
    Suppose the slopes $\Slope^\pm$ of $\Hamiltonian^{\eqnt,\pm}$ satisfy $\delta = \Slope^- - \Slope^+ \ge 0$, and let $(\eqHamiltonian_s, \eqACS_s)$ be a regular $\ExtendedTorus$-equivariant monotone homotopy between the data.
    The corresponding $\ExtendedTorus$-equivariant continuation map $\ContinuationMap^\delta$ is given for $[\critUniversalSpace^+, \WithFilling{\HamiltonianOrbit}^+_{\FillingBasis^+}] \in \FillingBasis^+$ by
        \begin{equation}
        \label{eqn:equivariant-floer-continuation-map}
            \ContinuationMap^\delta [\critUniversalSpace^+, \WithFilling{\HamiltonianOrbit}^+_{\FillingBasis^+}] = \hspace{-3em}
            \sum_{\substack{
                \DegreeTwoClass \in \Homology_2 (\Manifold) \\
                [\critUniversalSpace^-, \WithFilling{\HamiltonianOrbit}^-_{\FillingBasis^-}] \in \FillingBasis^- \\
                |\critUniversalSpace^-, (-\DegreeTwoClass) \ConnectedSum \WithFilling{\HamiltonianOrbit}^-_{\FillingBasis^-}| - |\critUniversalSpace^+, \WithFilling{\HamiltonianOrbit}^+_{\FillingBasis^+}| = 0
            }} \hspace{-2em}
            \sum_{
                [\flowUniversalSpace, \FloerSolution] \in \ModuliSpace_s([\critUniversalSpace^-, (-\DegreeTwoClass) \ConnectedSum \WithFilling{\HamiltonianOrbit}^-_{\FillingBasis^-}], [\critUniversalSpace^+, \WithFilling{\HamiltonianOrbit}^+_{\FillingBasis^+}])
            } \hspace{-1em}
            \Count \left( [\flowUniversalSpace, \FloerSolution] \right)
            \NovVariable^\DegreeTwoClass
            [\critUniversalSpace^-, \WithFilling{\HamiltonianOrbit}^-_{\FillingBasis^-}].
        \end{equation}
    
    \begin{proposition}
    \label{prop:connection-compatible-with-continuation-maps}
        The connections $\Connection^{\FixedPoint, \FillingBasis^\pm}$ satisfy
            \begin{equation}
                \Connection^{\FixedPoint, \FillingBasis^-} \ComposedWith \ContinuationMap^\delta = \ContinuationMap^\delta \ComposedWith \Connection^{\FixedPoint, \FillingBasis^+}.
            \end{equation}
    \end{proposition}
    
    \begin{proof}
        This is another variation of the intertwining relation proof.
        Fix data for the $\ExtendedTorus$-equivariant quantum actions $\QuantumAction_{\FillingBasis^\pm}$.
        Choose an $s$-dependent homotopy between this Morse data.
        We have 
            \begin{equation}
            \label{eqn:differentiation-intertwining-with-continuation-map}
                 (-\FillingIntersectionMap_\DegreeTwoCoclass) (\ContinuationMap^\delta) - \ContinuationMap^\delta  (-\FillingIntersectionMap_\DegreeTwoCoclass) = \ContinuationMap^\delta_{\DegreeTwoCoclass},
            \end{equation}
        where $\ContinuationMap^\delta_{\DegreeTwoCoclass}$ is the $\DegreeTwoCoclass$-weighted continuation map which counts $\ExtendedTorus$-equivariant continuation solutions weighted by their intersections with $s$-dependent half\textsuperscript{$+$} flowlines for the $s$-dependent homotopy.
        Thus $\ContinuationMap^\delta_{\DegreeTwoCoclass}$ is the $s$-dependent version of \eqref{eqn:weighted-floer-differential-intersections-in-floer-solution-only}.
        
        Define a map which counts $\ExtendedTorus$-equivalence classes of septuples $(\flowUniversalSpace, \FloerSolution, \flowUniversalSpace_0, \flowManifold_0, \flowUniversalSpace^{\min}, s_0, t_0)$ where $(\flowUniversalSpace, \FloerSolution, \flowUniversalSpace_0, \flowManifold_0, \flowUniversalSpace^{\min}, t_0)$ differs from \eqref{eqn:equivariant-quantum-action-on-equivariant-floer-cohomology} only in its intersection conditions: instead we have $\flowUniversalSpace(s_0) = \flowUniversalSpace_0(0) = \flowUniversalSpace^{\min}(0)$ and $\FloerSolution (s_0, t_0) = \flowManifold_0 (0)$.
        See \autoref{fig:continuation-maps-and-connection-homotopy}.
        Thus for $ \pm s_0 \gg 0$, the half\textsuperscript{$+$} flowlines use the data corresponding to  $\QuantumAction_{\FillingBasis^\pm}$.
        The limits $s_0 \to \pm \infty$ and the breaking of the half\textsuperscript{$+$} flowline $\flowUniversalSpace^{\min}$ form the boundary of the 1-dimensional moduli space, and yield
            \begin{equation}
                 \ContinuationMap^\delta (\DegreeTwoCoclass^\FixedPoint \QuantumAction_{\FillingBasis^+} (\Argument)) - \DegreeTwoCoclass^\FixedPoint \QuantumAction_{\FillingBasis^-} \ContinuationMap^\delta (\Argument) \sim \formalAdditionalCircle \ContinuationMap^\delta_{\DegreeTwoCoclass} (\Argument),
            \end{equation}
        which together with \eqref{eqn:differentiation-intertwining-with-continuation-map} gives the desired compatibility.
    \end{proof}
    
    In particular, the connection $\Connection^\FixedPoint$ on $\FloerCohomology^\ArbitraryIndex_{\ExtendedTorus} (\Manifold, \Extended{\TorusAction}, \Slope)$ is well-defined and independent of the choice of $\ExtendedTorus$-equivariant Floer datum and filling basis.
    Moreover, the connection is well-defined in the direct limit as $\Slope \to \infty$, implying the following corollary.
    
    \begin{corollary}
        There is an induced connection $\Connection^\FixedPoint$ on $\ExtendedTorus$-equivariant symplectic cohomology $\SymplecticCohomology^\ArbitraryIndex_{\ExtendedTorus} (\Manifold, \Extended{\TorusAction})$.
    \end{corollary}
    
    The proof of \autoref{prop:connection-compatible-with-continuation-maps} may be adapted to prove that the connection $\Connection^\FixedPoint$ on $\FloerCohomology^\ArbitraryIndex_{\ExtendedTorus} (\Manifold, \Extended{\TorusAction}, 0)$ is isomorphic to the connection \eqref{eqn:definition-quantum-connection-final-form}, and the isomorphisms are the $\ExtendedTorus$-equivariant PSS maps.
    Here, a Hamiltonian of slope 0 and the corresponding PSS maps are defined as in \cite[Theorem~37]{ritter_floer_2014}.
    
\subsection{Flatness of Floer connection}
        
    The Floer connection is flat.
    The proof is a Floer-theoretic version of the proof of \autoref{thm:differential-connection-flatness-quantum}, and it is more involved because $\DegreeTwoCoclass^\FixedPoint \QuantumAction (\secondDegreeTwoCoclass^\FixedPoint \QuantumAction \Argument)$ and $\secondDegreeTwoCoclass^\FixedPoint \QuantumAction (\DegreeTwoCoclass^\FixedPoint \QuantumAction \Argument)$ are not chain maps, let alone equal.
    
    \begin{theorem}
        \label{thm:flatness-of-differential-floer-connection}
        The connection $\Connection^\FixedPoint$ is flat.
    \end{theorem}
    
    \begin{proof}
        As per the proof of \autoref{thm:differential-connection-flatness-quantum}, we need only check the commutativity of $\Connection^\FixedPoint_\DegreeTwoCoclass$ and $\Connection^\FixedPoint _\secondDegreeTwoCoclass$ for classes $\DegreeTwoCoclass, \secondDegreeTwoCoclass \in \Cohomology^2 (\Manifold)$.
        The maps $(-\FillingIntersectionMap_\DegreeTwoCoclass)$ and $(-\FillingIntersectionMap_\secondDegreeTwoCoclass)$ clearly commute.
        We have
            \begin{equation}
            \label{eqn:differentiation-intertwining-with-quantum-action}
                (-\FillingIntersectionMap_\secondDegreeTwoCoclass) (\DegreeTwoCoclass^\FixedPoint \QuantumAction \Argument) - \DegreeTwoCoclass^\FixedPoint \QuantumAction (-\FillingIntersectionMap_\secondDegreeTwoCoclass) (\Argument) = (\DegreeTwoCoclass^\FixedPoint \QuantumAction {})_{\secondDegreeTwoCoclass} (\Argument),
            \end{equation}
        where $(\DegreeTwoCoclass^\FixedPoint \QuantumAction {})_{\secondDegreeTwoCoclass}$ denotes the $\secondDegreeTwoCoclass$-weighted $\ExtendedTorus$-equivariant quantum action of $\DegreeTwoCoclass^\FixedPoint$, analogously to \eqref{eqn:weighted-floer-differential-intersections-in-floer-solution-only}.
        An analogous equation holds for $\secondDegreeTwoCoclass^\FixedPoint$ and $\DegreeTwoCoclass$.
        Below, we apply \autoref{prop:floer-connection-compatible-with-equivariant-action} to get from \eqref{eqn:commuting-floer-connection-expand} to \eqref{eqn:commuting-floer-connection-pass-through-formal-additional-circle}, and we apply \eqref{eqn:differentiation-intertwining-with-quantum-action} to get from \eqref{eqn:commuting-floer-connection-fully-expand} to \eqref{eqn:commuting-floer-connection-simplified}.
        This gives us
            \begin{align}
                [\Connection^\FixedPoint _\DegreeTwoCoclass, \Connection^\FixedPoint _\secondDegreeTwoCoclass] (\Argument) &=
                \Connection^\FixedPoint _\DegreeTwoCoclass \left(\formalAdditionalCircle (-\FillingIntersectionMap_\secondDegreeTwoCoclass) (\Argument) + \secondDegreeTwoCoclass^\FixedPoint \QuantumAction (\Argument) \right)
                - \Connection^\FixedPoint _\secondDegreeTwoCoclass \left(\formalAdditionalCircle (-\FillingIntersectionMap_\DegreeTwoCoclass) (\Argument) + \DegreeTwoCoclass^\FixedPoint \QuantumAction (\Argument) \right) 
                \label{eqn:commuting-floer-connection-expand}
                \\
                &= \formalAdditionalCircle \Connection^\FixedPoint _\DegreeTwoCoclass \left( (-\FillingIntersectionMap_\secondDegreeTwoCoclass) (\Argument) \right) + \Connection^\FixedPoint _\DegreeTwoCoclass \left( \secondDegreeTwoCoclass^\FixedPoint \QuantumAction (\Argument) \right) \nonumber
                \\ & \qquad
                - \formalAdditionalCircle \Connection^\FixedPoint _\secondDegreeTwoCoclass \left( (-\FillingIntersectionMap_\DegreeTwoCoclass) (\Argument) \right) - \Connection^\FixedPoint _\secondDegreeTwoCoclass \left(\DegreeTwoCoclass^\FixedPoint \QuantumAction (\Argument) \right) 
                \label{eqn:commuting-floer-connection-pass-through-formal-additional-circle}
                \\
                &= \formalAdditionalCircle^2 (-\FillingIntersectionMap_\DegreeTwoCoclass) \left( (-\FillingIntersectionMap_\secondDegreeTwoCoclass) (\Argument) \right) + \formalAdditionalCircle \DegreeTwoCoclass^\FixedPoint \QuantumAction \left( (-\FillingIntersectionMap_\secondDegreeTwoCoclass) (\Argument) \right) + \formalAdditionalCircle (-\FillingIntersectionMap_\DegreeTwoCoclass) \left( \secondDegreeTwoCoclass^\FixedPoint \QuantumAction (\Argument) \right) \nonumber
                \\ & \hspace{20em} + \DegreeTwoCoclass^\FixedPoint \QuantumAction \left( \secondDegreeTwoCoclass^\FixedPoint \QuantumAction (\Argument) \right) \nonumber
                \\ & \qquad
                - \formalAdditionalCircle^2 (-\FillingIntersectionMap_\secondDegreeTwoCoclass) \left( (-\FillingIntersectionMap_\DegreeTwoCoclass) (\Argument) \right) - \formalAdditionalCircle \secondDegreeTwoCoclass^\FixedPoint \QuantumAction \left( (-\FillingIntersectionMap_\DegreeTwoCoclass) (\Argument) \right) - \formalAdditionalCircle (-\FillingIntersectionMap_\secondDegreeTwoCoclass) \left(\DegreeTwoCoclass^\FixedPoint \QuantumAction (\Argument) \right) \nonumber
                \\ & \hspace{20em} - \secondDegreeTwoCoclass^\FixedPoint \QuantumAction \left(\DegreeTwoCoclass^\FixedPoint \QuantumAction (\Argument) \right) 
                \label{eqn:commuting-floer-connection-fully-expand}
                \\
                &= - \formalAdditionalCircle (\DegreeTwoCoclass^\FixedPoint \QuantumAction {})_\secondDegreeTwoCoclass (\Argument) + \formalAdditionalCircle (\secondDegreeTwoCoclass^\FixedPoint \QuantumAction {})_\DegreeTwoCoclass (\Argument) + \DegreeTwoCoclass^\FixedPoint \QuantumAction \left( \secondDegreeTwoCoclass^\FixedPoint \QuantumAction (\Argument) \right) - \secondDegreeTwoCoclass^\FixedPoint \QuantumAction \left(\DegreeTwoCoclass^\FixedPoint \QuantumAction (\Argument) \right).
                \label{eqn:commuting-floer-connection-simplified}
            \end{align}
        
        We will find a homotopy which shows that \eqref{eqn:commuting-floer-connection-simplified} is chain homotopic to zero.
        The homotopy is drawn in \autoref{fig:floer-connection-is-flat-homotopy}.
        It counts $\ExtendedTorus$-equivalence classes of 11-tuples $(\flowUniversalSpace, \FloerSolution, \flowUniversalSpace_0, \flowManifold_0, \flowUniversalSpace^{\min}_0, t_0, \flowUniversalSpace_1, \flowManifold_1, \flowUniversalSpace^{\min}_1, s_1, t_1)$, where $[\flowUniversalSpace, \FloerSolution]$ is a $\ExtendedTorus$-equivariant Floer solution, $[\flowUniversalSpace_i, \flowManifold_i]$ are $\ExtendedTorus$-equivariant perturbed half\textsuperscript{$+$} flowlines, $\flowUniversalSpace^{\min}_i$ are perturbed half\textsuperscript{$+$} flowlines in $\UniversalSpace \ExtendedTorus$, $t_i \in \Circle$ are elements of the circle and $s_1 \in \RealNumbers \setminus \Set{0}$ is a real number.
        The 11-tuples satisfy the conditions
            \begin{equation}
                \begin{gathered}
                    \flowUniversalSpace (0) = \flowUniversalSpace_0 (0) = \flowUniversalSpace^{\min}_0 (0), \\
                    \flowUniversalSpace (s_1) = \flowUniversalSpace_1 (0) = \flowUniversalSpace^{\min}_1 (0), \\
                    \FloerSolution (0, t_0) = \flowManifold_0 (0), \\
                    \FloerSolution (s_1, t_1) = \flowManifold_1 (0)
                \end{gathered}
                \qquad \qquad
                \begin{gathered}
                    \arg (\flowUniversalSpace^{\min}_i (+\infty)) + t_i = 0, \\
                    [\flowUniversalSpace_0, \flowManifold_0] (+\infty) = \DegreeTwoCoclass^\FixedPoint, \\
                    [\flowUniversalSpace_1, \flowManifold_1] (+\infty) = \secondDegreeTwoCoclass^\FixedPoint.
                \end{gathered}
            \end{equation}
        
        The boundary components of the 1-dimensional moduli space that contribute to \eqref{eqn:commuting-floer-connection-simplified} are the limits $s_1 \to \pm \infty$ and the breaking of each of the flowlines $\flowUniversalSpace^{\min}_0$ and $\flowUniversalSpace^{\min}_1$.
        The limit $s_1 \to +\infty$ contributes the term $\DegreeTwoCoclass^\FixedPoint \QuantumAction \left( \secondDegreeTwoCoclass^\FixedPoint \QuantumAction (\Argument) \right)$ and the limit $s_1 \to -\infty$ contributes $- \secondDegreeTwoCoclass^\FixedPoint \QuantumAction \left(\DegreeTwoCoclass^\FixedPoint \QuantumAction (\Argument) \right)$.
        When $\flowUniversalSpace^{\min}_1$ breaks, we get $\formalAdditionalCircle (\secondDegreeTwoCoclass^\FixedPoint \QuantumAction {})_\DegreeTwoCoclass (\Argument)$ as in the proof of \autoref{prop:quantum-multiplication-commutes-with-differential-with-error}.
        Using $s$-translation as in the proof of \autoref{prop:floer-connection-compatible-with-equivariant-action}, we recover the term $- \formalAdditionalCircle (\DegreeTwoCoclass^\FixedPoint \QuantumAction {})_\secondDegreeTwoCoclass (\Argument)$ when $\flowUniversalSpace^{\min}_0$ breaks.
        
        The boundary components for the left and right limits $s_1 \to 0^\pm$ cancel with each other.
        The breaking of $[\flowUniversalSpace, \FloerSolution]$ contributes the $hd + dh$, where $h$ is the chain homotopy map which counts isolated $\ExtendedTorus$-equivalence classes of such 11-tuples.
        Therefore \eqref{eqn:commuting-floer-connection-simplified} is chain homotopic to 0, and this completes the proof.
    \end{proof}
    
\subsection{Floer Seidel map}
\label{sec:floer-seidel-map}

    Let $\Cocharacter \in \NonnegativeLatticeCocharacters{\TorusAction}{\Torus}$ be a $\TorusAction$-nonnegative cocharacter, that is a cocharacter for which $\TorusAction \ComposedWith \Cocharacter$ is a linear action of nonnegative slope on $\Manifold$.
    The \define{pullback Hamiltonian} $\Cocharacter \PullBack \Hamiltonian$ of the Hamiltonian $\Hamiltonian$ is given by 
        \begin{equation}
        \label{eqn:pullback-hamiltonian-definition}
            (\Cocharacter \PullBack \Hamiltonian)_t (\eltManifold) = \Hamiltonian_t ( \TorusAction_{\Cocharacter(t)} (\eltManifold)) - \ActionHamiltonian ^{\Cocharacter} ( \TorusAction_{\Cocharacter(t)} (\eltManifold)),
        \end{equation}
    where $\ActionHamiltonian^\Cocharacter$ is the Hamiltonian of the action $\TorusAction \ComposedWith \Cocharacter$ (without loss of generality, impose $\min \ActionHamiltonian^\Cocharacter = 0$ to remove the freedom of choosing a constant).
    Let $\CocharacterSlope{\Cocharacter}$ be the slope of $\ActionHamiltonian^\Cocharacter$.
    The assignment $\HamiltonianOrbit \mapsto \Cocharacter \PullBack \HamiltonianOrbit$, where $\Cocharacter \PullBack \HamiltonianOrbit$ is given by $t \mapsto \TorusAction_{\Cocharacter (- t)} \HamiltonianOrbit(t)$, is a bijection $\HamiltonianOrbitSet (\Hamiltonian) \to \HamiltonianOrbitSet (\Cocharacter \PullBack \Hamiltonian)$.
    
    The \define{pullback almost complex structure} $\Cocharacter \PullBack \ACS$ is given by $(\Cocharacter \PullBack \ACS)_t = (\Derivative \TorusAction_{\Cocharacter(t)}) \Inverse \ComposedWith \ACS_t \ComposedWith \Derivative \TorusAction_{\Cocharacter(t)}$.
    The assignment $\FloerSolution \mapsto \Cocharacter \PullBack \FloerSolution$, where $\Cocharacter \PullBack \HamiltonianOrbit$ is given by $(s, t) \mapsto \TorusAction_{\Cocharacter (- t)} \FloerSolution(s, t)$, is a bijection between the Floer solutions of Floer datum $(\Hamiltonian, \ACS)$ and those of $(\Cocharacter \PullBack \Hamiltonian, \Cocharacter \PullBack \ACS)$ \cite[Lemma~4.3]{seidel_$_1997}.
    
    Let $\FixedPoint \in \Manifold$ be a fixed point.
    Every orbit $\WithFilling{\HamiltonianOrbit}$ has a choice of filling $f$ for which $f(0) = \FixedPoint$.
    For such a filling $f$, the \define{pullback filling} $\Cocharacter \PullBack f$ is well-defined.
    Explicitly, this is given by $(\Cocharacter \PullBack f) (\ExponentialNumber^{2 \PiNumber (s + \ImaginaryNumber t)}) = \TorusAction_{\Cocharacter(-t)} f (\ExponentialNumber^{2 \PiNumber (s + \ImaginaryNumber t)})$.
    Define $(\Cocharacter, \FixedPoint) \PullBack \WithFilling{ \HamiltonianOrbit }$ to be the Hamiltonian orbit $\Cocharacter \PullBack \HamiltonianOrbit$ with the equivalence class of fillings $[\Cocharacter \PullBack f]$, with $f$ as above.
    
    The \define{Floer Seidel map} is the map
        \begin{equation}
            \FloerSeidel (\Cocharacter, \FixedPoint) : \FloerCohomology^\ArbitraryIndex (\Manifold, \Slope; \Hamiltonian, \ACS) \to \FloerCohomology^{\ArbitraryIndex + |\Cocharacter, \FixedPoint|} (\Manifold, \Slope - \CocharacterSlope{\Cocharacter}; \Cocharacter \PullBack \Hamiltonian, \Cocharacter \PullBack \ACS)
        \end{equation}
    given by $\WithFilling{\HamiltonianOrbit} \mapsto (\Cocharacter, \FixedPoint) \PullBack \WithFilling{\HamiltonianOrbit}$, and it is an isomorphism of cochain complexes.
    The Floer Seidel map commutes with continuation maps \cite[Corollary~4.8]{seidel_$_1997}, and we can show this using $s$-dependent pullback constructions analogous to the above.
    As such, $\FloerSeidel (\Cocharacter, \FixedPoint)$ is independent of the Floer datum.
    
\subsection{Equivariant Floer Seidel map}

    In \cite{liebenschutz-jones_intertwining_2020}, we defined an $\Circle$-equivariant Floer Seidel map.
    The only modification required to upgrade the non-equivariant construction in \autoref{sec:floer-seidel-map} to the equivariant setup was to incorporate a pullback of the action on the contractible loop space $\ContractibleLoopSpace{\Manifold}$.
    The same is true for the $\ExtendedTorus$-equivariant Floer Seidel map.
    The action on $\ContractibleLoopSpace{\Manifold}$ changes from
        \begin{equation*}
            ((\eltAdditionalCircle, \eltTorus) \cdot \HamiltonianOrbit) (t) = \Extended{\TorusAction}_{(\eltAdditionalCircle, \eltTorus)} (\HamiltonianOrbit(t - \eltAdditionalCircle)) =  \TorusAction_{\eltTorus} (\HamiltonianOrbit(t - \eltAdditionalCircle)),
        \end{equation*}
    as in \eqref{eqn:action-on-loop-space},
    to
        \begin{equation}
        \label{eqn:pullback-action-on-loop-space}
            ((\eltAdditionalCircle, \eltTorus) \cdot \HamiltonianOrbit) (t) =  (\Cocharacter \cdot \Extended{\TorusAction})_{(\eltAdditionalCircle, \eltTorus)} (\HamiltonianOrbit(t - \eltAdditionalCircle)) = \TorusAction_{\eltTorus - \Cocharacter(\eltAdditionalCircle)} (\HamiltonianOrbit(t - \eltAdditionalCircle)).
        \end{equation}
    This change in action from $\Extended{\TorusAction}$ to $\Cocharacter \cdot \Extended{\TorusAction}$ is compatible with the pullback data, so the $\ExtendedTorus$-equivariant pullback Hamiltonian $\Cocharacter \PullBack \eqHamiltonian$, defined as in \eqref{eqn:pullback-hamiltonian-definition}, is $\ExtendedTorus$-equivariant with respect to the action $\Cocharacter \cdot \Extended{\TorusAction}$; it satisfies 
        \begin{equation}
            (\Cocharacter \PullBack \eqHamiltonian)_{\eltUniversalSpace, t} (\eltManifold) = (\Cocharacter \PullBack \eqHamiltonian)_{(\eltAdditionalCircle, \eltTorus) \Inverse \cdot \eltUniversalSpace, \eltAdditionalCircle + t} ((\Cocharacter \cdot \Extended{\TorusAction})_{(\eltAdditionalCircle, \eltTorus)} (\eltManifold)).
        \end{equation}
        
    The \define{$\ExtendedTorus$-equivariant Floer Seidel map} is the map
        \begin{equation}
        \label{eqn:equivariant-floer-seidel-map-definition}
            \FloerSeidel_{\ExtendedTorus} (\Cocharacter, \FixedPoint) : \FloerCohomology^\ArbitraryIndex_{\ExtendedTorus} (\Manifold, \Extended{\TorusAction}, \Slope; \eqHamiltonian, \eqACS) \to \FloerCohomology^{\ArbitraryIndex + |\Cocharacter, \FixedPoint|}_{\ExtendedTorus} (\Manifold, \Cocharacter \cdot \Extended{\TorusAction}, \Slope - \CocharacterSlope{\Cocharacter}; \Cocharacter \PullBack \eqHamiltonian, \Cocharacter \PullBack \eqACS)
        \end{equation}
    given by $[\critUniversalSpace, \WithFilling{\HamiltonianOrbit}] \mapsto [\critUniversalSpace, (\Cocharacter, \FixedPoint) \PullBack \WithFilling{\HamiltonianOrbit}]$.
    Like its non-equivariant counterpart, the $\ExtendedTorus$-equivariant Floer Seidel map is an isomorphism of the cochain complexes, and it commutes with $\ExtendedTorus$-equivariant continuation maps.
    It is compatible with the $\NovikovRing \GradedCompletedTensorProduct \Cohomology^\ArbitraryIndex (\ClassifyingSpace \ExtendedTorus)$-module structure.
    
    \begin{proposition}
    \label{prop:equivariant-gluing-for-seidel-map}
        There is a commutative diagram
            \begin{equation}
                \label{eqn:diagram-relating-quantum-and-floer-seidel-maps-equivariant-case}
                \begin{tikzcd}[row sep=3.5em, column sep=-2em]
                    \QuantumCohomology^\ArbitraryIndex_{\ExtendedTorus} (\Manifold, \Extended{\TorusAction})
                        \arrow[rr, "{\QuantumSeidel_{\ExtendedTorus} (\Cocharacter, \FixedPoint)}"]
                        \arrow[d, "\substack{\text{$\ExtendedTorus$-equivariant}\\\text{PSS map}}"', "\cong"] 
                        &
                        & \QuantumCohomology^{\ArbitraryIndex + |\Cocharacter, \FixedPoint|}_{\ExtendedTorus}
                        (\Manifold, \Cocharacter \cdot \Extended{\TorusAction})
                    \\
                    \FloerCohomology^\ArbitraryIndex_{\ExtendedTorus} (\Manifold, \Extended{\TorusAction}, 0)
                        \arrow[rd, "{\FloerSeidel_{\ExtendedTorus} (\Cocharacter, \FixedPoint)}"', "\cong"]
                        &
                        & \FloerCohomology ^{\ArbitraryIndex + |\Cocharacter, \FixedPoint|} _{\ExtendedTorus} (\Manifold, \Cocharacter \cdot \Extended{\TorusAction}, 0)
                        \arrow[u, "\substack{\text{$\ExtendedTorus$-equivariant}\\\text{PSS map}}"', "\cong"]
                    \\
                    & \FloerCohomology ^{\ArbitraryIndex + |\Cocharacter, \FixedPoint|} _{\ExtendedTorus} (\Manifold, \Cocharacter \cdot \Extended{\TorusAction}, - \CocharacterSlope{\Cocharacter})
                        \arrow[ru, "\substack{\text{$\ExtendedTorus$-equivariant}\\\text{continuation map}}"', pos=0.6]
                        &                  
                \end{tikzcd}
            \end{equation}
        where slope 0 Hamiltonians and PSS maps are defined as in \cite[Theorem~37]{ritter_floer_2014}.
    \end{proposition}
    
    \begin{proof}
        This is an equivariant version of \cite[Section~8]{seidel_$_1997}, or more precisely its extension from closed manifolds to convex manifolds in \cite[Section~5.7]{ritter_floer_2014}.
        The proof otherwise extends to our setup (see \autoref{fig:equivariant-gluing-argument}).
    \end{proof}

\begin{figure}
\centering
\begin{subfigure}{0.4\textwidth}
    \centering
	\begin{center}
		\begin{tikzpicture}
			\node[inner sep=0] at (2.5,0) {\includegraphics[width=4 cm]{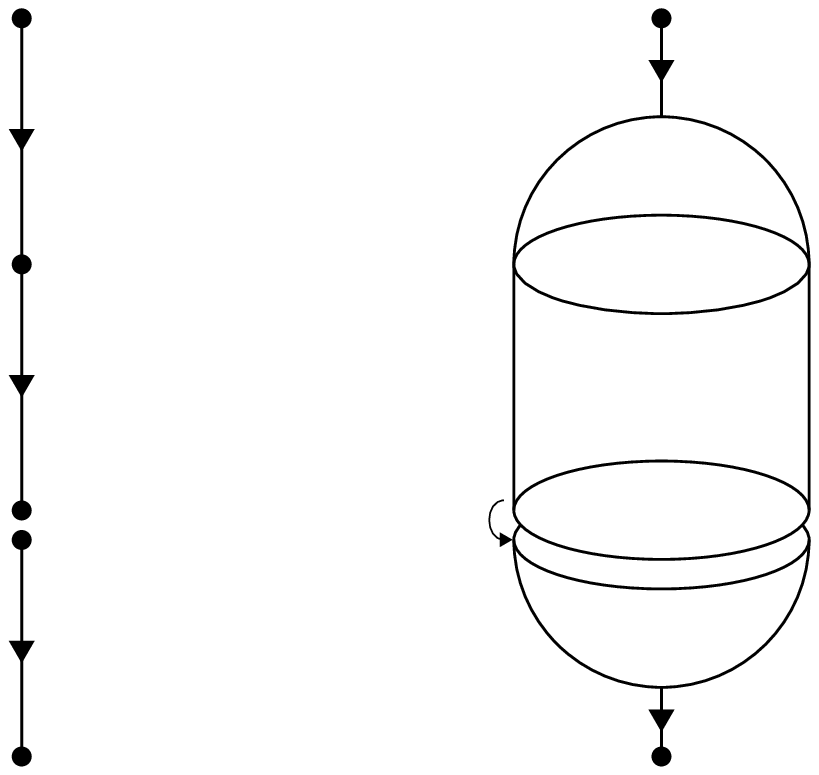}};
% 			\draw[step=1.0,black,thin] (0,-3) grid (5,3);
			
% 			\node at (0,0){0};
			\node at (2.9, 0.2){\tiny $\FloerSolution$};
			\node at (0.7, 2.1){\tiny $\critUniversalSpace^-$};
			\node at (0.7, -2.0){\tiny $\critUniversalSpace^+$};
			\node at (3.9, 2.1){\tiny $\critManifold^-$};
			\node at (3.9, -2.0){\tiny $\critManifold^+$};
			\node at (2.8, -0.6){\tiny $\Cocharacter$};
			\node at (4.0, 1.7){\tiny $\flowManifold^-$};
			\node at (4.0, -1.6){\tiny $\flowManifold^+$};
			\node at (0.4, -0.7){\tiny $=$};
		\end{tikzpicture}
	\end{center}
    \caption{Equivariant PSS maps together with Floer Seidel map and continuation map} 
\label{fig:quantum-action-on-eadsfquivariant-floer-cohomology-definition}
\end{subfigure}
\hspace{1em} 
\begin{subfigure}{0.4\textwidth}
  \centering
	\begin{center}
		\begin{tikzpicture}
			\node[inner sep=0] at (2.5,0) {\includegraphics[width=4 cm]{figures/equivariant-seidel-map.eps}};
% 			\draw[step=1.0,black,thin] (0,-3) grid (5,3);
			
% 			\node at (0,0){0};
			\node at (0.4, 0.8){\tiny $\flowUniversalSpace$};
			\node at (3.1, 0.5){\tiny $\ClutchingSection$};
			\node at (0.7, 2.1){\tiny $\critUniversalSpace^-$};
			\node at (0.7, -2.0){\tiny $\critUniversalSpace^+$};
			\node at (3.9, 2.1){\tiny $\critManifold^-$};
			\node at (3.9, -2.0){\tiny $\critManifold^+$};
			\node at (2.8, -0.1){\tiny $\Cocharacter$};
			\node at (4.0, 1.3){\tiny $\flowManifold^-$};
			\node at (4.0, -1.3){\tiny $\flowManifold^+$};
		\end{tikzpicture}
	\end{center}
    \caption{Equivariant quantum Seidel map from \autoref{fig:equivariant-quantum-seidel-map}} 
    \label{fig:geometric-action-anasdfd-connection-homotopy}
\end{subfigure}
\caption{
    Gluing together PSS maps, the Floer Seidel map and a continuation map (\textsc{a}) yields the quantum Seidel map (\textsc{b}).
}
\label{fig:equivariant-gluing-argument}
\end{figure}
    
\subsection{Shift operator}

    As in \autoref{sec:pullback-group-isomorphisms-on-cohomology}, there is a pullback isomorphism
        \begin{equation}
        \label{eqn:pullback-group-isom-on-equivariant-floer-cohomology}
            (\ClassifyingSpace \Extended{\Cocharacter}) \PullBack : \FloerCohomology^\ArbitraryIndex _{\ExtendedTorus, \Extended{\basisCocharacter}} (\Manifold, \Cocharacter \cdot \Extended{\TorusAction}, \Slope) \to  \FloerCohomology^\ArbitraryIndex _{\ExtendedTorus, \Cocharacter \cdot \Extended{\basisCocharacter}} (\Manifold, \Extended{\TorusAction}, \Slope).
        \end{equation}
    This map translates between different notation which ultimately describe the same moduli spaces.
    Recall $(\ClassifyingSpace \Extended{\Cocharacter}) \PullBack$ satisfies $(\ClassifyingSpace \Extended{\Cocharacter}) \PullBack ([\Extended{\Character}] \ \Argument) = [\Extended{\Character} \ComposedWith \Extended{\Cocharacter}] \  (\ClassifyingSpace \Extended{\Cocharacter}) \PullBack (\Argument)$ for characters $\Extended{\Character} : \ExtendedTorus \to \Circle$.
    
    The \define{shift operator} on $\ExtendedTorus$-equivariant Floer cohomology is the map 
        \begin{equation}
            \ShiftOperator^\FixedPoint _\Cocharacter : \FloerCohomology^\ArbitraryIndex _{\ExtendedTorus, \Extended{\basisCocharacter}} (\Manifold, \Extended{\TorusAction}, \Slope) \to  \FloerCohomology^{\ArbitraryIndex + |\Cocharacter, \FixedPoint|} _{\ExtendedTorus, \Cocharacter \cdot \Extended{\basisCocharacter}} (\Manifold, \Extended{\TorusAction}, \Slope)
        \end{equation}
    given by the composition
        \begin{equation}
        \label{eqn:shift-operator-floer-cohomology-definition}
            \ShiftOperator^\FixedPoint _\Cocharacter = (\ClassifyingSpace \Extended{\Cocharacter}) \PullBack \ComposedWith \ContinuationMap^{\CocharacterSlope{\Cocharacter}} \ComposedWith \FloerSeidel_{\ExtendedTorus} (\Cocharacter, \FixedPoint),
        \end{equation}
    where $\ContinuationMap^{\CocharacterSlope{\Cocharacter}}$ is the continuation map which increases the slope by $\CocharacterSlope{\Cocharacter}$.
    Note that the three maps in \eqref{eqn:shift-operator-floer-cohomology-definition} which make up $\ShiftOperator^\FixedPoint _\Cocharacter$ commute, so the order of these maps is unimportant.
    
    \begin{theorem}
    \label{thm:flatness-of-difference-differential-connection-on-floer-cohomology}
        The difference-differential connection $(\ShiftOperator^\FixedPoint, \Connection^\FixedPoint)$ on $\ExtendedTorus$-equivariant Floer cohomology is flat.
    \end{theorem}
    
    We prove \autoref{thm:flatness-of-difference-differential-connection-on-floer-cohomology} in \autoref{sec:proof-of-flatness-of-difference-connection-on-floer-cohomology}.
    
    This difference-differential connection on the $\ExtendedTorus$-equivariant Floer cohomology of slope 0 is isomorphic to the difference-differential connection on $\ExtendedTorus$-equivariant quantum cohomology.
    That the PSS isomorphism preserves the connection $\Connection^\FixedPoint$ follows from an adaptation of \autoref{prop:connection-compatible-with-continuation-maps}, and that it preserves the shift operator $\ShiftOperator^\FixedPoint$ follows from \autoref{prop:equivariant-gluing-for-seidel-map}.
    
    The difference-differential connection also commutes with $\ExtendedTorus$-equivariant continuation maps.
    Therefore there is an induced difference-differential connection on $\ExtendedTorus$-equivariant symplectic cohomology.
    This difference-differential connection will be flat, since \autoref{thm:flatness-of-difference-differential-connection-on-floer-cohomology} holds for all slopes in the direct limit.
        
    \begin{remark}
        [Negative cocharacters]
        Unlike the quantum Seidel map, the \emph{Floer} Seidel map itself is well-defined for cocharacters $\Cocharacter$ which induce a $\Circle$-action $\TorusAction \ComposedWith \Cocharacter$ of negative slope \cite[Remark,~page~1046]{ritter_floer_2014}.
        Similarly, the $\ExtendedTorus$-equivariant Floer Seidel map $\FloerSeidel_{\ExtendedTorus} (\Cocharacter, \FixedPoint)$ is well-defined for such cocharacters and is an isomorphism of the cochain complexes.
        In \eqref{eqn:equivariant-floer-seidel-map-definition}, the slope $\Slope$ in the domain changes to $\Slope - \CocharacterSlope{\Cocharacter}$.
        For $\TorusAction$-negative cocharacters, these slopes satisfy $\Slope - \CocharacterSlope{\Cocharacter} > \Slope$, and hence the change in slope cannot be undone with a continuation map.
        Therefore we cannot construct a shift operator on $\FloerCohomology^\ArbitraryIndex_{\ExtendedTorus} (\Manifold, \Extended{\TorusAction}, \Slope)$ corresponding to $\Cocharacter$.
        
        On the other hand, the composition $(\ClassifyingSpace \Extended{\Cocharacter}) \PullBack \ComposedWith \FloerSeidel_{\ExtendedTorus} (\Cocharacter, \FixedPoint)$ commutes with continuation maps, and hence induces a shift operator $\ShiftOperator^\FixedPoint_\Cocharacter$ on $\ExtendedTorus$-equivariant \emph{symplectic} cohomology $\SymplecticCohomology^\ArbitraryIndex_{\ExtendedTorus} (\Manifold, \Extended{\TorusAction})$.
        This shift operator commutes with the connection $\Connection^\FixedPoint$ by \autoref{prop:connection-commutes-with-equivariant-floer-seidel-map} and \autoref{prop:connection-commutes-with-pullback-operator}.
        Therefore $\SymplecticCohomology^\ArbitraryIndex_{\ExtendedTorus} (\Manifold, \Extended{\TorusAction})$ has a flat difference-differential connection defined on the group of all cocharacters, not just on the monoid $\NonnegativeLatticeCocharacters{\TorusAction}{\Torus}$.
    \end{remark}
    
\subsection{Proof of flatness of connection}
\label{sec:proof-of-flatness-of-difference-connection-on-floer-cohomology}

    To show $\Connection^\FixedPoint_\DegreeTwoCoclass$ commutes with $\ShiftOperator^\FixedPoint _\Cocharacter$, we will show that $\Connection^\FixedPoint_\DegreeTwoCoclass$ commutes each of the three maps in the composition \eqref{eqn:shift-operator-floer-cohomology-definition} separately.
    
    \begin{proposition}
    \label{prop:connection-commutes-with-equivariant-floer-seidel-map}
        The map $\Connection^\FixedPoint_\DegreeTwoCoclass$ commutes with the $\ExtendedTorus$-equivariant Floer Seidel map $\FloerSeidel_{\ExtendedTorus} (\Cocharacter, \FixedPoint)$.
    \end{proposition}
    
    \begin{proof}
        We use the $\ExtendedTorus$-equiv\-ari\-ant Floer data $(\eqHamiltonian, \eqACS)$ and the filling basis $\FillingBasis$ on the domain $\FloerCohomology^\ArbitraryIndex_{\ExtendedTorus} (\Manifold, \Extended{\TorusAction}, \Slope)$.
        We use the pullback data $(\Cocharacter \PullBack \eqHamiltonian, \Cocharacter \PullBack \eqACS)$ and the pullback filling basis $(\Cocharacter, \FixedPoint) \PullBack \FillingBasis$ on the codomain.
        With these choices, $\FloerSeidel_{\ExtendedTorus} (\Cocharacter, \FixedPoint)$ is an isomorphism between the filling bases, and therefore preserves the differentiation operation $\dbyd{\DegreeTwoCoclass}$.
        The map $\FloerSeidel_{\ExtendedTorus} (\Cocharacter, \FixedPoint)$ does not modify flowlines in $\UniversalSpace \ExtendedTorus$, so it induces a natural isomorphism on the moduli spaces used to define the map $\formalAdditionalCircle \cdot$.
        Thus, with the data we have chosen, $\FloerSeidel_{\ExtendedTorus} (\Cocharacter, \FixedPoint)$ commutes with $\formalAdditionalCircle \dbyd{\DegreeTwoCoclass}$ on the cochain complexes, giving
            \begin{equation}
            \label{eqn:equivariant-floer-seidel-map-commutes-with-differentiation-connection-on-chains}
                \FloerSeidel_{\ExtendedTorus} (\Cocharacter, \FixedPoint) \ComposedWith \left( \formalAdditionalCircle \dbyd{\DegreeTwoCoclass} \right) = \left( \formalAdditionalCircle \dbyd{\DegreeTwoCoclass} \right) \ComposedWith \FloerSeidel_{\ExtendedTorus} (\Cocharacter, \FixedPoint).
            \end{equation}
            
        The two other terms in $\Connection^\FixedPoint_\DegreeTwoCoclass$, namely $\DegreeTwoCoclass^\FixedPoint \QuantumAction$ and $\FillingIntersectionMap^\FillingBasis_\DegreeTwoCoclass$, both depend on the choice of $\ExtendedTorus$-equivariant Morse data on $\Manifold$.
        The $\ExtendedTorus$-action on $\Manifold$ in the domain of $\FloerSeidel_{\ExtendedTorus} (\Cocharacter, \FixedPoint)$ is $\Extended{\TorusAction}$, while on the codomain it is $\Cocharacter \cdot \Extended{\TorusAction}$.
        The fact that these two actions are different means that we cannot use the same $\ExtendedTorus$-equivariant Morse data in defining these terms of $\Connection^\FixedPoint_\DegreeTwoCoclass$.
        To emphasise the different underlying $\ExtendedTorus$-action, we will incorporate a subscript $\Extended{\TorusAction}$ or $\Cocharacter \cdot \Extended{\TorusAction}$ in our notation.
        Note that each of the classes $\DegreeTwoCoclass^\FixedPoint _{\Extended{\TorusAction}} \in \Cohomology^2_{\ExtendedTorus} (\Manifold, \Extended{\TorusAction})$ and $\DegreeTwoCoclass^\FixedPoint _{\Cocharacter \cdot \Extended{\TorusAction}} \in \Cohomology^2_{\ExtendedTorus} (\Manifold, \Cocharacter \cdot \Extended{\TorusAction})$ is independently defined using the split short exact sequence \eqref{eqn:split-short-exact-sequence-borel-quotient}, however the relation $(\ClassifyingSpace \Extended{\Cocharacter}) \PullBack \DegreeTwoCoclass^\FixedPoint _{\Cocharacter \cdot \Extended{\TorusAction}} = \DegreeTwoCoclass^\FixedPoint _{\Extended{\TorusAction}}$ readily follows by a diagram chasing argument.
        
        We will show
            \begin{equation}
            \label{eqn:chain-homotopy-equivalence-for-floer-seidel-commuting-with-quantum-action}
                \FloerSeidel_{\ExtendedTorus} (\Cocharacter, \FixedPoint) \ComposedWith \left( \DegreeTwoCoclass^\FixedPoint _{\Extended{\TorusAction}} \QuantumAction {} - \formalAdditionalCircle \FillingIntersectionMap^\FillingBasis_{\DegreeTwoCoclass, \Extended{\TorusAction}} \right) \sim \left( \DegreeTwoCoclass^\FixedPoint _{\Cocharacter \cdot \Extended{\TorusAction}} \QuantumAction {} - \formalAdditionalCircle \FillingIntersectionMap^{(\Cocharacter, \FixedPoint) \PullBack \FillingBasis}_{\DegreeTwoCoclass, \Cocharacter \cdot \Extended{\TorusAction}} \right) \ComposedWith \FloerSeidel_{\ExtendedTorus} (\Cocharacter, \FixedPoint),
            \end{equation}
        where $\sim$ denotes a homotopy equivalence.
        This means there is a map $K$ such that $dK + Kd$ is equal to the difference of the two sides of \eqref{eqn:chain-homotopy-equivalence-for-floer-seidel-commuting-with-quantum-action}, and this makes sense even though neither side is a chain map.
        Combining \eqref{eqn:equivariant-floer-seidel-map-commutes-with-differentiation-connection-on-chains} and \eqref{eqn:chain-homotopy-equivalence-for-floer-seidel-commuting-with-quantum-action} yields a chain homotopy between the chain maps $\FloerSeidel_{\ExtendedTorus} (\Cocharacter, \FixedPoint) \ComposedWith \Connection^\FixedPoint _\DegreeTwoCoclass$ and $\Connection^\FixedPoint _\DegreeTwoCoclass \ComposedWith \FloerSeidel_{\ExtendedTorus} (\Cocharacter, \FixedPoint)$, completing the proof.
        
        To construct a homotopy between the two sides of \eqref{eqn:chain-homotopy-equivalence-for-floer-seidel-commuting-with-quantum-action}, we will need to find a homotopy between $\DegreeTwoCoclass^\FixedPoint _{\Extended{\TorusAction}} \QuantumAction$ and $\DegreeTwoCoclass^\FixedPoint _{\Cocharacter \cdot \Extended{\TorusAction}} \QuantumAction$.
        These are $\ExtendedTorus$-equivariant cohomology classes for different actions, so we will use the clutching bundle $\ClutchingBundle{\Cocharacter}$ from \autoref{sec:clutching-bundle-construction} to create the homotopy.
        By \autoref{lem:class-existence-for-quantum-flatness-from-intertwining}, there is a class $\beta \in \Cohomology^2 _{\ExtendedTorus} (\ClutchingBundle{\Cocharacter})$ which satisfies
            \begin{equation}
                (\ClutchingFibreInclusion^+_{\Pole^+}) \PullBack \beta = \DegreeTwoCoclass^\FixedPoint _{\Extended{\TorusAction}}, \quad (\ClutchingFibreInclusion^-_{\Pole^-}) \PullBack \beta = \DegreeTwoCoclass^\FixedPoint _{\Cocharacter \cdot \Extended{\TorusAction}}, \quad \beta (\DegreeTwoClass^\FixedPoint) = \DegreeTwoCoclass (\DegreeTwoClass) \text{ for }\DegreeTwoClass \in \Homology_2 (\Manifold).
            \end{equation}
        Here, we have denoted the fibre inclusions over the poles $\Pole^\pm \in \Sphere$ by $\ClutchingFibreInclusion^\pm_{\Pole^\pm} : \Manifold \to \ClutchingBundle{\Cocharacter}$.
        Thus, informally, the two classes $\DegreeTwoCoclass^\FixedPoint _{\Extended{\TorusAction}}$ and $\DegreeTwoCoclass^\FixedPoint _{\Cocharacter \cdot \Extended{\TorusAction}}$ coincide in the clutching bundle, and we have reduced the problem to finding a homotopy which moves the intersection with the class $\beta$ from the fibre over one pole to the fibre over the other pole.
        We define this homotopy below, however note that we do require further homotopies in our proof of \eqref{eqn:chain-homotopy-equivalence-for-floer-seidel-commuting-with-quantum-action}.
        
        \begin{definition}
            [The homotopy $K$]
            \label{def:homotopy-k-in-flatness-proof}
            The map $K^+$, which we have drawn in \autoref{fig:homotopy-for-flatness-of-shift-operator-floer-positive-end}, counts $\ExtendedTorus$-equivalence classes of 7-tuples $(\flowUniversalSpace, \FloerSolution, \flowUniversalSpace^{\min}, \flowUniversalSpace_0, \flowManifold_0, s_0, t_0)$, where $[\flowUniversalSpace, \FloerSolution]$ is a $\ExtendedTorus$-equivariant Floer solution, $\flowUniversalSpace^{\min}$ is a perturbed half\textsuperscript{+} flowline in $\UniversalSpace \ExtendedTorus$, $[\flowUniversalSpace_0, \flowManifold_0]$ is a perturbed $\ExtendedTorus$-equivariant half\textsuperscript{+} flowline in $\ClutchingBundle{\Cocharacter}$ which converges to $\beta$ at $+ \infty$, $s_0 \in \IntervalOpen{0}{\infty}$ is a positive real number and $t_0 \in \Circle$ is an element of the circle.
            Together, they must satisfy $\flowUniversalSpace(0) = \flowUniversalSpace^{\min} (0) = \flowUniversalSpace_0 (0)$, $\arg (\flowUniversalSpace^{\min} (+ \infty)) + t_0 = 0$ and $\ClutchingFibreInclusion ^+ _{(s_0, t_0)} (\FloerSolution (0, t_0)) = \flowManifold_0 (0)$.
            Here, $\ClutchingFibreInclusion^+ _{(s_0, t_0)} : \Manifold \to \ClutchingBundle{\Cocharacter}$ is the fibre inclusion map over the point $(s_0, t_0) \in \Sphere$, using coordinates from \eqref{eqn:cylinder-coordinates-for-sphere}, which uses the trivialisation over the pole $\Pole^+$ to identify the fibre with $\Manifold$.
            The map $K^+$ is an endomorphism of the $\ExtendedTorus$-equivariant Floer cochain complex for the data $(\eqHamiltonian, \eqACS)$.

\begin{figure}
\centering
	\begin{center}
		\begin{tikzpicture}
			\node[inner sep=0] at (0,0) {\includegraphics[width=8 cm]{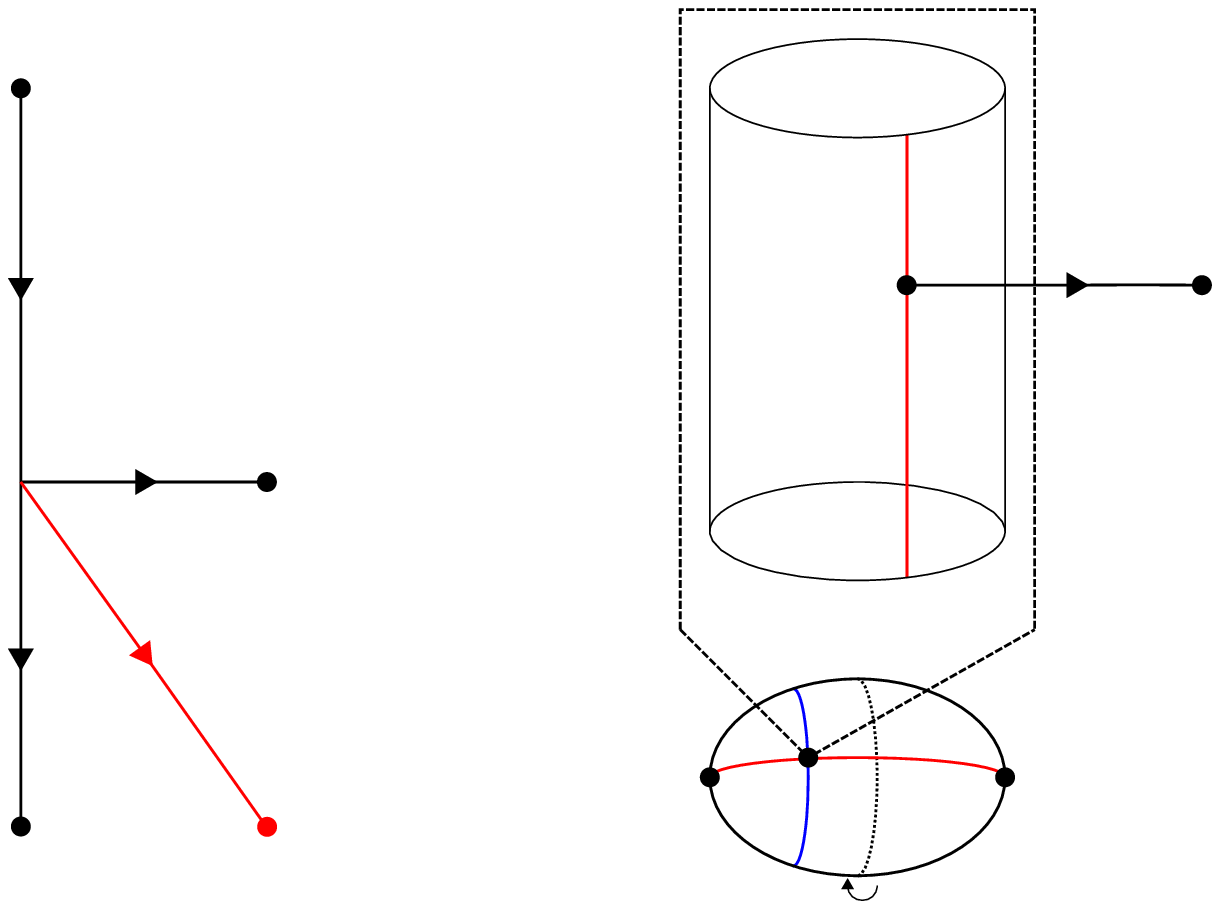}};
% 			\draw[step=1.0,gray,thin] (-5,-6) grid (5,6);
			
% 			\node at (0,0){0};
			\node at (1.5, 1.1){\tiny $(0, \textcolor{red}{t_0})$};
			\node at (3.3, 1.3){\tiny $\flowManifold_0$};
			\node at (-4.1, 1.0){\tiny $\flowUniversalSpace$};
			\node at (-3.0, 0.0){\tiny $\flowUniversalSpace_0$};
			\node at (-2.7, -1.3){\tiny $\flowUniversalSpace^{\min}$};
			\node at (-4.0, -0.2){\tiny $0$};
			\node at (0.8, 1.8){\tiny $\FloerSolution$};
			\node at (4.1, 1.1){\tiny $\beta$};
			\node at (1.3, -2.2){\tiny $(\textcolor{blue}{s_0}, \textcolor{red}{t_0})$};
			\node at (0.4, -2.1){\tiny $\Pole^+$};
			\node at (2.9, -2.1){\tiny $\Pole^-$};
			\node at (1.7, -3.1){\tiny $\Cocharacter$};
			\node at (1.7, 3.2){\tiny fibre $\ClutchingFibreInclusion^+ _{(\textcolor{blue}{s_0}, \textcolor{red}{t_0})}(\Manifold)$};
			\node at (2.9, -2.6){\tiny base $\Sphere$};
		\end{tikzpicture}
	\end{center}
\caption{
    For the map $K^+$, we embed the Floer cylinder in the bundle $\ClutchingBundle{\Cocharacter}$, and the half\textsuperscript{+} flowline $\flowManifold_0$ to $\beta$ lives in the bundle.
}
\label{fig:homotopy-for-flatness-of-shift-operator-floer-positive-end}
\end{figure}

            Similarly, we define a map $K^-$ which counts similar $\ExtendedTorus$-equivalence classes of 7-tuples, but these instead satisfy $\ClutchingFibreInclusion ^- _{(s_0, t_0)} (\FloerSolution (0, t_0)) = \flowManifold_0 (0)$ and $s_0 \in \IntervalOpen{-\infty}{0}$.
            The map $\ClutchingFibreInclusion ^- _{(s_0, t_0)} : \Manifold \to \ClutchingBundle{\Cocharacter}$ is the fibre inclusion map using the trivialisation over the pole $\Pole^-$ to identify the fibres with $\Manifold$.
            The map $K^-$ is an endomorphism of the $\ExtendedTorus$-equivariant Floer cochain complex for the pullback data $(\Cocharacter \PullBack \eqHamiltonian, \Cocharacter \PullBack \eqACS)$.
        
            Finally, define 
            \begin{equation}
                K = (\FloerSeidel_{\ExtendedTorus} (\Cocharacter, \FixedPoint) \ComposedWith K^+) + (K^- \ComposedWith \FloerSeidel_{\ExtendedTorus} (\Cocharacter, \FixedPoint)),
            \end{equation}
            which is a map $K : \FloerCohomology^\ArbitraryIndex_{\ExtendedTorus} (\Manifold, \Extended{\TorusAction}, \Slope) \to \FloerCohomology^{\ArbitraryIndex + |\Cocharacter, \FixedPoint| + 1} _{\ExtendedTorus} (\Manifold, \Cocharacter \cdot \Extended{\TorusAction}, \Slope - \CocharacterSlope{\Cocharacter})$.
        \end{definition}
        
        Two of the boundary components of the 1-dimensional moduli spaces for $K$ will be the limits $s_0 \to \pm \infty$.
        To characterise these boundary components, let $K^\pm _{\pm \infty}$ be the maps defined as per $K^\pm$ in \autoref{def:homotopy-k-in-flatness-proof}, except with the intersection condition over the fibre at $(s_0, t_0)$ replaced by the condition $\ClutchingFibreInclusion ^\pm _{\Pole^\pm} (\FloerSolution (0, t_0)) = \flowManifold_0 (0)$.
        The moduli spaces for $K^\pm _{\pm \infty}$ do not record the parameter $s_0$ since this intersection condition does not depend on $s_0$.
        
        Denote by $\FillingIntersectionMap^{\FillingBasis, +}_\beta$ the endomorphism of $\FloerCohomology^\ArbitraryIndex_{\ExtendedTorus} (\Manifold, \Extended{\TorusAction}, \Slope)$ which counts the intersections of $(\critUniversalSpace, \ClutchingFibreInclusion^+_{\Pole^+} (f(\Disc)))$ with $\PerturbedStableManifold(\beta)$ for the $\ExtendedTorus$-equivariant Hamiltonian orbit $[\critUniversalSpace, \WithFilling{\HamiltonianOrbit}]$  with filling $f$.
        This is the natural analogue of $\FillingIntersectionMap^{\FillingBasis}_{\DegreeTwoCoclass, \Extended{\TorusAction}}$ which uses $\beta$ instead of $\DegreeTwoCoclass^\FixedPoint$.
        Similarly define the endomorphism $\FillingIntersectionMap^{(\Cocharacter, \FixedPoint) \PullBack \FillingBasis, -}_\beta$ of $\FloerCohomology^{\ArbitraryIndex} _{\ExtendedTorus} (\Manifold, \Cocharacter \cdot \Extended{\TorusAction}, \Slope - \CocharacterSlope{\Cocharacter})$, which is the analogue of $\FillingIntersectionMap^{(\Cocharacter, \FixedPoint) \PullBack \FillingBasis}_{\DegreeTwoCoclass, \Cocharacter \cdot \Extended{\TorusAction}}$ that uses $\beta$.
        
        We now have the ingredients to prove \eqref{eqn:chain-homotopy-equivalence-for-floer-seidel-commuting-with-quantum-action}, and therefore complete the proof.
        \autoref{lemma:homotopy-between-maps-at-poles-and-maps-on-manifold} states that we can indeed change to maps using $\beta$ over the poles, implying that \eqref{eqn:chain-homotopy-equivalence-for-floer-seidel-commuting-with-quantum-action} is equivalent to
            \begin{equation}
            \label{eqn:chain-homotopy-equivalence-for-floer-seidel-commuting-with-homotopy-extremes}
                \FloerSeidel_{\ExtendedTorus} (\Cocharacter, \FixedPoint) \ComposedWith \left( K^+ _{+ \infty} - \formalAdditionalCircle \FillingIntersectionMap^{\FillingBasis, +}_\beta \right) \sim \left( K^- _{- \infty} - \formalAdditionalCircle \FillingIntersectionMap^{(\Cocharacter, \FixedPoint) \PullBack \FillingBasis, -}_\beta \right) \ComposedWith \FloerSeidel_{\ExtendedTorus} (\Cocharacter, \FixedPoint).
            \end{equation}
        We show \eqref{eqn:chain-homotopy-equivalence-for-floer-seidel-commuting-with-homotopy-extremes} holds in \autoref{lemma:homotopy-equivalence-for-floer-flatness-holds-on-clutching-bundle} using the homotopy $K$.
        Therefore \eqref{eqn:chain-homotopy-equivalence-for-floer-seidel-commuting-with-quantum-action} holds as required.
    \end{proof}
        
    \begin{lemma}
    \label{lemma:homotopy-between-maps-at-poles-and-maps-on-manifold}
        We have homotopy equivalences
            \begin{align}
                K^+ _{+ \infty} - \formalAdditionalCircle \FillingIntersectionMap^{\FillingBasis, +}_\beta &\sim \DegreeTwoCoclass^\FixedPoint _{\Extended{\TorusAction}} \QuantumAction {} - \formalAdditionalCircle \FillingIntersectionMap^\FillingBasis_{\DegreeTwoCoclass, \Extended{\TorusAction}} \label{eqn:homotopy-equivalence-for-flatness-transfer-to-clutching-bundle-+} \\
                K^- _{- \infty} - \formalAdditionalCircle \FillingIntersectionMap^{(\Cocharacter, \FixedPoint) \PullBack \FillingBasis, -}_\beta &\sim \DegreeTwoCoclass^\FixedPoint _{\Cocharacter \cdot \Extended{\TorusAction}} \QuantumAction {} - \formalAdditionalCircle \FillingIntersectionMap^{(\Cocharacter, \FixedPoint) \PullBack \FillingBasis}_{\DegreeTwoCoclass, \Cocharacter \cdot \Extended{\TorusAction}}. \label{eqn:homotopy-equivalence-for-flatness-transfer-to-clutching-bundle--}
            \end{align}
    \end{lemma}
    
    \begin{proof}
        We will show \eqref{eqn:homotopy-equivalence-for-flatness-transfer-to-clutching-bundle-+} holds; the same argument proves \eqref{eqn:homotopy-equivalence-for-flatness-transfer-to-clutching-bundle--} as well.
        Denote by $\MorseFunction^{\ClutchingBundle{\Cocharacter}} : \UniversalSpace \ExtendedTorus \times \ClutchingBundle{\Cocharacter} \to \RealNumbers$ the equivariant Morse function on $\ClutchingBundle{\Cocharacter}$ and by $\MorseFunction^{\Manifold} : \UniversalSpace \ExtendedTorus \times \Manifold \to \RealNumbers$ the equivariant Morse function on $\Manifold$ for the action $\Extended{\TorusAction}$.
        We have $\beta \in \Cohomology^2_{\ExtendedTorus} (\ClutchingBundle{\Cocharacter}; \MorseFunction^{\ClutchingBundle{\Cocharacter}})$ and $\DegreeTwoCoclass^\FixedPoint_{\Extended{\TorusAction}} \in \Cohomology^2_{\ExtendedTorus} (\Manifold, \Extended{\TorusAction}; \MorseFunction^\Manifold)$.
        Let $\MorseFunction^{\Disc^+} : \Disc^+ \to \RealNumbers$ be the function $\MorseFunction^{\Disc^+} (\eltSphere) = |\eltSphere|^2$.
        Let $\MorseFunction^+ : \UniversalSpace \ExtendedTorus \times \ClutchingBundle{\Cocharacter} \to \RealNumbers$ be a Morse function which extends the sum $\MorseFunction^\Manifold + \MorseFunction^{\Disc^+} : \UniversalSpace \ExtendedTorus \times \ClutchingProjection\Inverse (\Disc^+) \to \RealNumbers$.
        The Morse function $\MorseFunction^+$ has no flowlines which flow away from the fibre $\ClutchingFibreInclusion^+_{\Pole^+} (\Manifold)$.
        
        We will construct a homotopy which modifies the flowline $[\flowUniversalSpace_0, \flowManifold_0]$ in the definition of $K^+_{+\infty}$ so that it uses $\MorseFunction^+$ instead of $\MorseFunction^{\ClutchingBundle{\Cocharacter}}$.
        Since $[\flowUniversalSpace_0, \flowManifold_0]$ is, in fact, an \emph{$s$-dependent} flowline, we need $s$-dependent perturbations of the above Morse functions.
        The perturbation of $\MorseFunction^{\ClutchingBundle{\Cocharacter}}$ is unrestricted, however for $\MorseFunction^+$ we must use a perturbation for which no flowline flows away from the fibre $\ClutchingFibreInclusion^+_{\Pole^+} (\Manifold)$.
        For this, take a perturbation $\MorseFunction^\Manifold_s$ of $\MorseFunction^\Manifold$, and impose $\MorseFunction^+_s = \MorseFunction^\Manifold_s + \MorseFunction^{\Disc^+}$ on $\RealNumbers \times \UniversalSpace \ExtendedTorus \times \ClutchingProjection\Inverse (\Disc^+)$.
        
        Let $\MorseFunction^{+, \mathrm{hty}}_{\lambda, s} : \IntervalOpen{0}{\infty} \times \RealNumbers \times \UniversalSpace \ExtendedTorus \times \ClutchingBundle{\Cocharacter} \to \RealNumbers$ be a function which also depends on a real parameter $\lambda \in \IntervalOpen{0}{\infty}$ and which satifies $\MorseFunction^{+, \mathrm{hty}}_{\lambda, s} = \MorseFunction^+_s$ on $s \le \lambda - 1$ and $\MorseFunction^{+, \mathrm{hty}}_{\lambda, s} = \MorseFunction^{\ClutchingBundle{\Cocharacter}}_s$ on $s \ge \lambda$.
        Let $h$ be the map which counts $\ExtendedTorus$-equivalence classes of tuples $(\flowUniversalSpace, \FloerSolution, \flowUniversalSpace^{\min}, \flowUniversalSpace_0, \flowManifold_0, s_0, t_0, \lambda)$, where $\lambda \in \IntervalOpen{0}{\infty}$ is a new parameter and $[\flowUniversalSpace, \FloerSolution, \flowUniversalSpace^{\min}, \flowUniversalSpace_0, \flowManifold_0, s_0, t_0]$ is as per the definition of $K^+_{+\infty}$ in \autoref{def:homotopy-k-in-flatness-proof}, except that $[\flowUniversalSpace_0, \flowManifold_0]$ is an $s$-dependent flowline for the Morse function $\MorseFunction^{+, \mathrm{hty}}_{\lambda, s}$.
        We have drawn this configuration in \autoref{fig:homotopy-for-flatness-of-shift-operator-floer-positive-pole-fibre}.

\begin{figure}
\centering
	\begin{center}
		\begin{tikzpicture}
			\node[inner sep=0] at (0,0) {\includegraphics[width=8 cm]{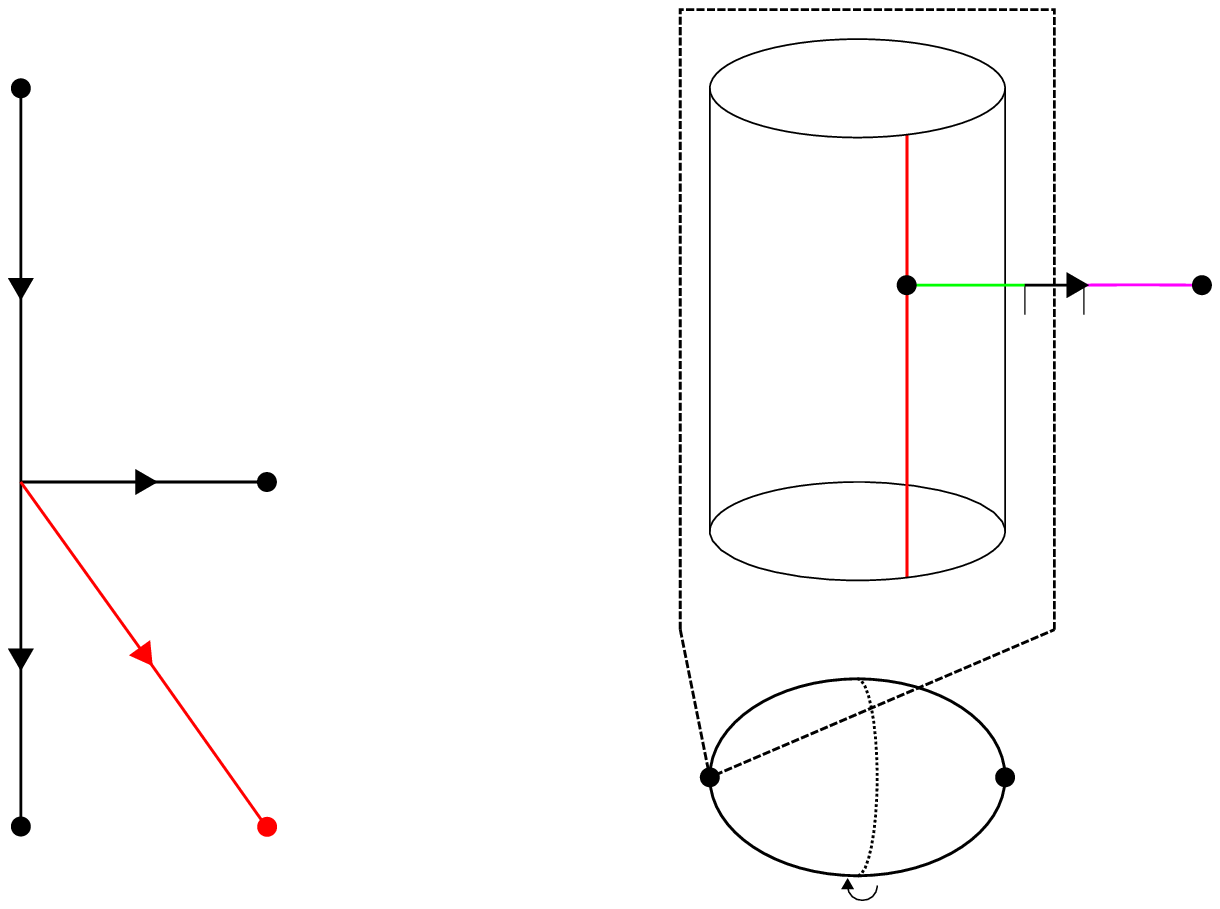}};
% 			\draw[step=1.0,gray,thin] (-5,-6) grid (5,6);
			
% 			\node at (0,0){0};
			\node at (1.5, 1.1){\tiny $(0, \textcolor{red}{t_0})$};
			\node at (2.4, 1.25){\tiny $\flowManifold_0$};
			\node at (2.65, 0.83){\tiny ${\lambda {-} 1}$};
			\node at (3.14, 0.85){\tiny $\lambda$};
			\node at (-4.1, 1.0){\tiny $\flowUniversalSpace$};
			\node at (-3.0, 0.0){\tiny $\flowUniversalSpace_0$};
			\node at (-2.7, -1.3){\tiny $\flowUniversalSpace^{\min}$};
			\node at (-4.0, -0.2){\tiny $0$};
			\node at (0.8, 1.8){\tiny $\FloerSolution$};
			\node at (4.1, 1.1){\tiny $\beta$};
			\node at (0.4, -2.1){\tiny $\Pole^+$};
			\node at (2.9, -2.1){\tiny $\Pole^-$};
			\node at (1.7, -3.1){\tiny $\Cocharacter$};
			\node at (1.7, 3.2){\tiny fibre $\ClutchingFibreInclusion^+ _{\Pole^+}(\Manifold)$};
			\node at (2.9, -2.6){\tiny base $\Sphere$};
		\end{tikzpicture}
	\end{center}
\caption{
    For the map $h$, we embed the Floer cylinder in the bundle $\ClutchingBundle{\Cocharacter}$ in the fibre above $\Pole^+$, and the half\textsuperscript{+} flowline $\flowManifold_0$ to $\beta$ lives in the bundle.
    The flowline $\flowManifold_0$ uses the data $\textcolor{green}{\MorseFunction^+_s}$ on $\IntervalClosed{0}{\lambda - 1}$ and it uses the data $\textcolor{magenta}{\MorseFunction^{\ClutchingBundle{\Cocharacter}}_s}$ on $\IntervalClosedOpen{\lambda}{\infty}$.
}
\label{fig:homotopy-for-flatness-of-shift-operator-floer-positive-pole-fibre}
\end{figure}
        
        We will show the boundary of the 1-dimensional moduli spaces yields \eqref{eqn:homotopy-equivalence-for-flatness-transfer-to-clutching-bundle-+}.
        The contributions from the equivariant Floer solution breaking and from the flowline $[\flowUniversalSpace_0, \flowManifold_0]$ breaking produce the terms $d h + h d$.
        Under the limit $\lambda \to 0$, $[\flowUniversalSpace_0, \flowManifold_0]$ is a flowline of $\MorseFunction^{\ClutchingBundle{\Cocharacter}}_s$, so this boundary component yields the term $K^+_{+\infty}$.
        We will show that the limit $\lambda \to \infty$ yields the term $\DegreeTwoCoclass^\FixedPoint_{\Extended{\TorusAction}} \QuantumAction$ and that the breaking of $\flowUniversalSpace^{\min}$ yields the remaining two terms in \eqref{eqn:homotopy-equivalence-for-flatness-transfer-to-clutching-bundle-+}, up to a further homotopy equivalence.
        
        Under the limit $\lambda \to \infty$, the flowline $[\flowUniversalSpace_0, \flowManifold_0]$ breaks into a pair of flowlines: a half\textsuperscript{$+$} flowline $[\flowUniversalSpace_1, \flowManifold_1]$ from $[\flowUniversalSpace(0), \ClutchingFibreInclusion^+_{\Pole^+} (\FloerSolution(0, t_0))]$ to a $\ExtendedTorus$-equivariant critical point $[\critUniversalSpace_1, \critManifold_1]$ and a further flowline $[\flowUniversalSpace_2, \flowManifold_2]$ from $[\critUniversalSpace_1, \critManifold_1]$ to $\beta$.
        The flowline $[\flowUniversalSpace_1, \flowManifold_1]$ uses the Morse function $\MorseFunction^+_s$, so it lies completely within the fibre $\ClutchingFibreInclusion^+_{\Pole^+} (\Manifold)$.
        Identifying this fibre with $\Manifold$, we see that $[\flowUniversalSpace_1, \flowManifold_1]$ is a half\textsuperscript{$+$} flowline of $\MorseFunction^\Manifold_s$.
        The other flowline $[\flowUniversalSpace_2, \flowManifold_2]$ uses $\MorseFunction^+$ on $s \le - 1$ and uses $\MorseFunction^{\ClutchingBundle{\Cocharacter}}$ on $s \ge 0$.
        To see this, consider a translation of the original flowline $[\flowUniversalSpace_0, \flowManifold_0]$ by $- \lambda$ so that it is a flowline with domain $\IntervalClosedOpen{- \lambda}{\infty}$ which uses $\MorseFunction^+_{s - \lambda}$ for $s \in\IntervalClosed{-\lambda}{-1}$ and $\MorseFunction^{\ClutchingBundle{\Cocharacter}}_{s - \lambda}$ for $s \in\IntervalClosedOpen{0}{\infty}$.
        Under the limit $\lambda \to \infty$, the flowline $[\flowUniversalSpace_2, \flowManifold_2]$ is the component of the broken flowline according to these coordinates on $[\flowUniversalSpace_0, \flowManifold_0]$.
        In particular, the standard Morse continuation map between the $\ExtendedTorus$-equivariant Morse functions $\MorseFunction^{\ClutchingBundle{\Cocharacter}}$ and $\MorseFunction^+$ counts the flowlines $[\flowUniversalSpace_2, \flowManifold_2]$ by definition.
        We restrict to those flowlines $[\flowUniversalSpace_2, \flowManifold_2]$ which flow from the fibre $\ClutchingFibreInclusion^+_{\Pole^+}(\Manifold)$ to get
            \begin{equation}
            \label{eqn:homotopy-between-morse-functions-on-clutching-bundle-recovers-degree-two-coclass}
                \DegreeTwoCoclass^\FixedPoint_{\Extended{\TorusAction}} = (\ClutchingFibreInclusion^+_{\Pole^+}) \PullBack \beta = \sum_{\substack{
                    [\critUniversalSpace_1, \critManifold_1] \in \eqCriticalPointSet{\MorseFunction^+}
                    \\
                    \text{with $\critManifold_1 \in \ClutchingFibreInclusion^+_{\Pole^+}(\Manifold)$}
                }} \ 
                \sum_{[\flowUniversalSpace_2, \flowManifold_2] \in \ModuliSpace ([\critUniversalSpace_1, \critManifold_1], \beta)} \Count([\flowUniversalSpace_2, \flowManifold_2])
                \ [\critUniversalSpace_1, \critManifold_1].
            \end{equation}
        We can also derive \eqref{eqn:homotopy-between-morse-functions-on-clutching-bundle-recovers-degree-two-coclass} using functorial flowlines to model the pullback map $(\ClutchingFibreInclusion^+_{\Pole^+}) \PullBack$ (see \autoref{footnote:functorial-flowline}).
        In summary, $[\flowUniversalSpace_1, \flowManifold_1]$ is really a half\textsuperscript{$+$} flowline in $\Manifold$ which uses the Morse function $\MorseFunction^+$ and which flows to $\DegreeTwoCoclass^\FixedPoint_{\Extended{\TorusAction}}$.
        Therefore, the $\lambda \to \infty$ limit exactly equals $\DegreeTwoCoclass^\FixedPoint_{\Extended{\TorusAction}} \QuantumAction$.
        
        As per our previous proofs, the breaking of $\flowUniversalSpace^{\min}$ is isolated exactly when it breaks with the intermediate critical point $\formalAdditionalCircle$, and moreover we can decouple the contribution of the flowline to $\formalAdditionalCircle$ from the rest of the configuration.
        Therefore the contribution from $\flowUniversalSpace^{\min}$ breaking is homotopy equivalent to $\formalAdditionalCircle \ComposedWith C_\beta$.
        Here, $C_\beta$ counts $\ExtendedTorus$-equivalence classes of tuples $(\flowUniversalSpace, \FloerSolution, \flowUniversalSpace_0, \flowManifold_0, t_0, \lambda)$ where $[\flowUniversalSpace, \FloerSolution]$ is a $\ExtendedTorus$-equivariant Floer solution, $[\flowUniversalSpace_0, \flowManifold_0]$ is a half\textsuperscript{$+$} flowline to $\beta$, $t_0 \in \Circle$ is an unconstrained point on the circle and $\lambda \in \IntervalOpen{0}{\infty}$ is a positive number.
        The tuple satisfies the conditions $\flowUniversalSpace(0) = \flowUniversalSpace_0 (0)$ and $\ClutchingFibreInclusion^+_{+\infty} (\FloerSolution(0, t_0)) = \flowManifold_0(0)$.
        The moduli space has dimension $|\critUniversalSpace^-, \WithFilling{\HamiltonianOrbit}^-| - |\critUniversalSpace^+, \WithFilling{\HamiltonianOrbit}^+| - |\beta| + 2$, where the $+2$ comes from the freedom in the parameters $t_0$ and $\lambda$.
        Since $|\beta| = 2$, these moduli spaces are isolated only when $[\flowUniversalSpace, \FloerSolution]$ is constant in $s$.
        Therefore $C_\beta$ counts the isolated flowlines $[\flowUniversalSpace_0, \flowManifold_0]$ of $\MorseFunction^{+, \text{hty}}_{\lambda, s}$ which satisfy $\flowUniversalSpace_0 (0) = \critUniversalSpace$ and $\flowManifold_0 (0) = \HamiltonianOrbit(t_0)$ for the given $\ExtendedTorus$-equivariant Hamiltonian orbit $[\critUniversalSpace, \WithFilling{\HamiltonianOrbit}]$, where $t_0$ and $\lambda$ can vary.
        
        We will show $C_\beta = \FillingIntersectionMap^\FillingBasis_{\DegreeTwoCoclass, \Extended{\TorusAction}} - \FillingIntersectionMap^{\FillingBasis, +}_\beta$.
        The maps $\FillingIntersectionMap^\FillingBasis_{\DegreeTwoCoclass, \Extended{\TorusAction}}$ and $\FillingIntersectionMap^{\FillingBasis, +}_\beta$ count the intersections of the prescribed filling of $[\critUniversalSpace, \WithFilling{\HamiltonianOrbit}]$ with the perturbed stable manifolds of $\DegreeTwoCoclass^\FixedPoint_{\Extended{\TorusAction}}$ and $\beta$ respectively.
        These perturbed stable manifolds use the Morse functions $\MorseFunction^\Manifold_s$ and $\MorseFunction^{\ClutchingBundle{\Cocharacter}}_s$ respectively.
        Consider now the intersections of the filling with the perturbed stable manifold of $\beta$ for the function $\MorseFunction^{+, \mathrm{hty}}_{\lambda, s}$, for fixed $\lambda \in \IntervalOpen{0}{\infty}$.
        The limit as $\lambda \to \infty$ recovers $\FillingIntersectionMap^\FillingBasis_{\DegreeTwoCoclass, \Extended{\TorusAction}}$ (using \eqref{eqn:homotopy-between-morse-functions-on-clutching-bundle-recovers-degree-two-coclass} as above) and the limit $\lambda \to 0$ recovers $\FillingIntersectionMap^{\FillingBasis, +}_\beta$.
        As $\lambda$ varies, the intersections will vary smoothly within the filling.
        At isolated values of $\lambda$, however, the intersections may enter or leave the filling.
        When this happens, we will have an intersection which lies on the boundary of the filling, and this boundary is the Hamiltonian orbit $\HamiltonianOrbit$.
        Thus, $C_\beta$ exactly records the new and lost intersections, giving $C_\beta = \FillingIntersectionMap^\FillingBasis_{\DegreeTwoCoclass, \Extended{\TorusAction}} - \FillingIntersectionMap^{\FillingBasis, +}_\beta$ as desired.
    \end{proof}
    
    \begin{lemma}
    \label{lemma:homotopy-equivalence-for-floer-flatness-holds-on-clutching-bundle}
        The homotopy equivalence \eqref{eqn:chain-homotopy-equivalence-for-floer-seidel-commuting-with-homotopy-extremes} holds.
    \end{lemma}
    
    \begin{proof}
        We will show that the map $K = (\FloerSeidel_{\ExtendedTorus} (\Cocharacter, \FixedPoint) \ComposedWith K^+) + (K^- \ComposedWith \FloerSeidel_{\ExtendedTorus} (\Cocharacter, \FixedPoint))$ from \autoref{def:homotopy-k-in-flatness-proof} is a homotopy\footnote{
            Technically, we add a further homotopy to $K$ to achieve this, but $K$ is the important part of the homotopy.
        } between the two sides of \eqref{eqn:chain-homotopy-equivalence-for-floer-seidel-commuting-with-homotopy-extremes}.
        To show this, we must consider each boundary component of the 1-dimensional moduli spaces of $K$.
        The limits as $s_0 \to \pm \infty$ recover the $K^\pm_{\pm \infty}$ terms.
        We will show that the boundary components corresponding to the limits $s_0 \to 0^\pm$ cancel with each other and that the breaking of the flowline $\flowUniversalSpace^{\min}$ contributes the remaining terms $\FloerSeidel_{\ExtendedTorus} (\Cocharacter, \FixedPoint) \ComposedWith - \formalAdditionalCircle \FillingIntersectionMap^{\FillingBasis, +} _\beta$ and $- \formalAdditionalCircle \FillingIntersectionMap^{(\Cocharacter, \FixedPoint) \PullBack \FillingBasis, -}_\beta \ComposedWith \FloerSeidel_{\ExtendedTorus} (\Cocharacter, \FixedPoint)$.
        
        First, we will show that the limits $s_0 \to 0^\pm$ cancel with each other.
        Let $K^\pm _0$ denote the versions of the maps $K^\pm$ with the parameter $s_0 = 0$ fixed, so that $K^\pm _0$ is the $s_0 \to 0^\pm$ boundary component of $K^\pm$.
        If $[\flowUniversalSpace, \FloerSolution, \flowUniversalSpace^{\min}, \flowUniversalSpace_0, \flowManifold_0, 0, t_0]$ is a $\ExtendedTorus$-equivalence class counted by $K^+ _0$, then $[\flowUniversalSpace, \Cocharacter \PullBack \FloerSolution, \flowUniversalSpace^{\min}, \flowUniversalSpace_0, \flowManifold_0, 0, t_0]$ is counted by $K^- _0$.
        Indeed, the clutching bundle is essentially constructed so that this holds.
        This identification extends to an isomorphism between the moduli spaces for $K^+_0$ and $K^-_0$.
        The isomorphism is sign-reversing because the limits $s_0 \to 0^\pm$ approach in opposite directions.
        Therefore the contributions of $\FloerSeidel_{\ExtendedTorus} (\Cocharacter, \FixedPoint) \ComposedWith K^+ _{0}$ and $K^- _{0} \ComposedWith \FloerSeidel_{\ExtendedTorus} (\Cocharacter, \FixedPoint)$ cancel with each other, as desired.
        
        Next, we consider the contribution of the boundary component in which the half\textsuperscript{+} flowline $\flowUniversalSpace^{\min}$ breaks.
        As is standard in our intertwining-style proofs, only the breaking to $\formalAdditionalCircle$ is isolated and we can use a homotopy to decouple the flowline to $\formalAdditionalCircle$ from the rest of the configuration.
        Therefore, the contribution of this boundary component is chain homotopic to $\FloerSeidel_{\ExtendedTorus} (\Cocharacter, \FixedPoint) \ComposedWith \formalAdditionalCircle \ComposedWith K^+ _{\formalAdditionalCircle} + \formalAdditionalCircle \ComposedWith K^- _{\formalAdditionalCircle} \ComposedWith \FloerSeidel_{\ExtendedTorus} (\Cocharacter, \FixedPoint)$, where the maps $K^\pm _{\formalAdditionalCircle}$ count $\ExtendedTorus$-equivalence classes of 6-tuples $(\flowUniversalSpace, \FloerSolution, \flowUniversalSpace_0, \flowManifold_0, s_0, t_0)$ just as for the maps $K^\pm$, but without the flowline $\flowUniversalSpace^{\min}$.
        The dimension of these moduli spaces is $|\critUniversalSpace^-, \WithFilling{\HamiltonianOrbit}^-| - |\critUniversalSpace^+, \WithFilling{\HamiltonianOrbit}^+| - |\beta| + 2$.
        The $+2$ comes from the freedom in the parameters $s_0$ and $t_0$.
        Recall $|\beta| = 2$.
        Thus the moduli spaces are isolated only when $|\critUniversalSpace^-, \WithFilling{\HamiltonianOrbit}^-| = |\critUniversalSpace^+, \WithFilling{\HamiltonianOrbit}^+|$ holds, and this equation holds exactly for the Floer solutions which are constant in $s$.

        Let $[\flowUniversalSpace, \FloerSolution] \equiv [\critUniversalSpace, \HamiltonianOrbit]$ be a $\ExtendedTorus$-equivariant Floer solution which is constant in $s \in \RealNumbers$.
        Let $C_{[\flowUniversalSpace, \FloerSolution]}$ be the total count of the moduli spaces of $K^+ _{\formalAdditionalCircle}$ which use the Floer solution $[\flowUniversalSpace, \FloerSolution]$ and the moduli spaces of $K^- _{\formalAdditionalCircle}$ which use $[\flowUniversalSpace, \Cocharacter \PullBack \FloerSolution]$.
        The map $\FloerSeidel_{\ExtendedTorus} (\Cocharacter, \FixedPoint)  \ComposedWith K^+ _{\formalAdditionalCircle} + K^- _{\formalAdditionalCircle} \ComposedWith \FloerSeidel_{\ExtendedTorus} (\Cocharacter, \FixedPoint)$ satisfies 
            \begin{equation}
            \label{eqn:writing-homotopy-as-counting-intersections-only}
                [\critUniversalSpace, \WithFilling{\HamiltonianOrbit}_{\FillingBasis}] \mapsto C_{[\flowUniversalSpace, \FloerSolution]} \cdot [\critUniversalSpace, (\Cocharacter, \FixedPoint) \PullBack \WithFilling{\HamiltonianOrbit}_{\FillingBasis}].
            \end{equation}
        Since $\ClutchingFibreInclusion^+ _{(s_0, t_0)} (\HamiltonianOrbit(t_0)) = \ClutchingFibreInclusion^- _{(s_0, t_0)} (\Cocharacter \PullBack \HamiltonianOrbit(t_0))$ holds, the count $C_{[\flowUniversalSpace, \FloerSolution]}$ is the same as the count of the points $(s_0, t_0) \in \RealNumbers \times \Circle$ for which $(\critUniversalSpace, \ClutchingFibreInclusion^+ _{(s_0, t_0)} (\HamiltonianOrbit(t_0))) \in \PerturbedStableManifold(\beta)$ is satisfied.
        The map $(s, t) \mapsto \ClutchingFibreInclusion^+ _{(s, t)} (\HamiltonianOrbit(t))$ describes a cylinder in $\ClutchingBundle{\Cocharacter}$ which lifts the cylinder $\Sphere \setminus \Set{\Pole^\pm}$, and $C_{[\flowUniversalSpace, \FloerSolution]}$ is the number of intersections of this cylinder with $\PerturbedStableManifold(\beta)$.
        
        We use the prescribed fillings of $\WithFilling{\HamiltonianOrbit}_{\FillingBasis}$ and $(\Cocharacter, \FixedPoint) \PullBack \WithFilling{\HamiltonianOrbit}_{\FillingBasis}$ to cap either end of this cylinder.
        Recall that the expressions $\FillingIntersectionMap^{\FillingBasis, +}_\beta ([\critUniversalSpace, \WithFilling{\HamiltonianOrbit}_\FillingBasis])$ and $\FillingIntersectionMap^{(\Cocharacter, \FixedPoint) \PullBack \FillingBasis, -}_\beta ([\critUniversalSpace, (\Cocharacter, \FixedPoint) \PullBack \WithFilling{\HamiltonianOrbit}_{\FillingBasis}])$ capture the intersections of these fillings with $\PerturbedStableManifold(\beta)$.
        The capped cylinder is contractible by the definition of the pullback filling basis, and hence the signed sum of its intersections with $\PerturbedStableManifold(\beta)$ is 0.
        This yields
            \begin{equation}
            \label{eqn:vmin-breaking-recovers-difference-of-intersections-with-caps}
                \FloerSeidel_{\ExtendedTorus} (\Cocharacter, \FixedPoint)  \ComposedWith K^+ _{\formalAdditionalCircle} + K^- _{\formalAdditionalCircle} \ComposedWith \FloerSeidel_{\ExtendedTorus} (\Cocharacter, \FixedPoint) = \FloerSeidel_{\ExtendedTorus} (\Cocharacter, \FixedPoint)  \ComposedWith \FillingIntersectionMap^{\FillingBasis, +}_\beta - \FillingIntersectionMap^{(\Cocharacter, \FixedPoint) \PullBack \FillingBasis, -}_\beta \ComposedWith \FloerSeidel_{\ExtendedTorus} (\Cocharacter, \FixedPoint),
            \end{equation}
        since the left-hand side is \eqref{eqn:writing-homotopy-as-counting-intersections-only}.
        
        In summary, the contribution of the breaking of the half\textsuperscript{+} flowline $\flowUniversalSpace^{\min}$ is homotopic to $\FloerSeidel_{\ExtendedTorus} (\Cocharacter, \FixedPoint) \ComposedWith \formalAdditionalCircle \ComposedWith K^+ _{\formalAdditionalCircle} + \formalAdditionalCircle \ComposedWith K^- _{\formalAdditionalCircle} \ComposedWith \FloerSeidel_{\ExtendedTorus} (\Cocharacter, \FixedPoint)$.
        This map is equal to $\formalAdditionalCircle (\FloerSeidel_{\ExtendedTorus} (\Cocharacter, \FixedPoint)  \ComposedWith K^+ _{\formalAdditionalCircle} + K^- _{\formalAdditionalCircle} \ComposedWith \FloerSeidel_{\ExtendedTorus} (\Cocharacter, \FixedPoint) )$ because $\FloerSeidel_{\ExtendedTorus} (\Cocharacter, \FixedPoint)$ commutes with $\formalAdditionalCircle$.
        This final map recovers the required terms by \eqref{eqn:vmin-breaking-recovers-difference-of-intersections-with-caps}.
    \end{proof}
    
    \begin{proposition}
    \label{prop:connection-commutes-with-pullback-operator}
        The map $\Connection^\FixedPoint_\DegreeTwoCoclass$ commutes with the map $(\ClassifyingSpace \Extended{\Cocharacter}) \PullBack$.
    \end{proposition}
    
    \begin{proof}
        The pullback isomorphism $(\ClassifyingSpace \Extended{\Cocharacter}) \PullBack$ does not modify the filling basis, so it preserves the differentiation operation $\dbyd{\DegreeTwoCoclass}$.
        Moreover, it exactly preserves the moduli spaces which are used by the maps $\formalAdditionalCircle \cdot$, $\FillingIntersectionMap^\FillingBasis_\DegreeTwoCoclass$ and $\DegreeTwoCoclass^\FixedPoint \QuantumAction$.
        Therefore the pullback isomorphism $(\ClassifyingSpace \Extended{\Cocharacter}) \PullBack$ commutes with the connection $\Connection^\FixedPoint _\DegreeTwoCoclass = \formalAdditionalCircle \dbyd{\DegreeTwoCoclass} + {\DegreeTwoCoclass^\FixedPoint \QuantumAction} - \formalAdditionalCircle \FillingIntersectionMap^\FillingBasis_\DegreeTwoCoclass$ on the cochain complexes.
        Thus, we have
            \begin{equation}
                (\ClassifyingSpace \Extended{\Cocharacter}) \PullBack \ComposedWith \left( \formalAdditionalCircle \dbyd{\DegreeTwoCoclass} + {\left( \DegreeTwoCoclass^\FixedPoint _{\Cocharacter \cdot \Extended{\TorusAction}} \QuantumAction \right)} - \formalAdditionalCircle \FillingIntersectionMap^\FillingBasis_\DegreeTwoCoclass \right) = \left( \formalAdditionalCircle \dbyd{\DegreeTwoCoclass} + \left( \DegreeTwoCoclass^\FixedPoint _{\Extended{\TorusAction}} \QuantumAction \right) - \formalAdditionalCircle \FillingIntersectionMap^\FillingBasis_\DegreeTwoCoclass \right) \ComposedWith (\ClassifyingSpace \Extended{\Cocharacter}) \PullBack
            \end{equation}
        as desired.
    \end{proof}
    
    \begin{proposition}
    \label{prop:connection-commutes-with-floer-shift-operator}
        The map $\Connection^\FixedPoint_\DegreeTwoCoclass$ commutes with the shift operator $\ShiftOperator^\FixedPoint _\Cocharacter$.
    \end{proposition}
    
    \begin{proof}
        The shift operator is the composition $(\ClassifyingSpace \Extended{\Cocharacter}) \PullBack \ComposedWith \ContinuationMap^{\CocharacterSlope{\Cocharacter}} \ComposedWith \FloerSeidel_{\ExtendedTorus} (\Cocharacter, \FixedPoint)$.
        We showed that the connection $\Connection^\FixedPoint _\DegreeTwoCoclass$ commutes with the map $(\ClassifyingSpace \Extended{\Cocharacter}) \PullBack$ in \autoref{prop:connection-commutes-with-pullback-operator}, with the continuation map $\ContinuationMap^{\CocharacterSlope{\Cocharacter}}$ in \autoref{prop:connection-compatible-with-continuation-maps}, and with the Floer Seidel map $\FloerSeidel_{\ExtendedTorus} (\Cocharacter, \FixedPoint)$ in \autoref{prop:connection-commutes-with-equivariant-floer-seidel-map}.
        Therefore the connection also commutes with the composition as desired.
    \end{proof}
    
    \begin{proof}
        [Proof of \autoref{thm:flatness-of-difference-differential-connection-on-floer-cohomology}]
        The flatness theorem \autoref{thm:flatness-of-difference-differential-connection-on-floer-cohomology} comprises three separate claims.
        The first claim is that $\Connection^\FixedPoint_\DegreeTwoCoclass$ commutes with $\Connection^\FixedPoint _\secondDegreeTwoCoclass$ for two classes $\DegreeTwoCoclass, \secondDegreeTwoCoclass \in \Cohomology^2 (\Manifold)$, which was shown in \autoref{thm:flatness-of-differential-floer-connection}.
        The second claim is that $\Connection^\FixedPoint_\DegreeTwoCoclass$ commutes with $\ShiftOperator^\FixedPoint _\Cocharacter$, which was shown in \autoref{prop:connection-commutes-with-floer-shift-operator}.
        The third and final claim is $\ShiftOperator^\FixedPoint _\Cocharacter \ShiftOperator^\FixedPoint _\secondCocharacter = \ShiftOperator^\FixedPoint _{\Cocharacter + \Cocharacter\Prime}$ for two $\TorusAction$-nonnegative cocharacters $\Cocharacter, \secondCocharacter \in \NonnegativeLatticeCocharacters{\TorusAction}{\Torus}$.
        This final claim follows readily from the fact that the maps which make up \eqref{eqn:shift-operator-floer-cohomology-definition} commute with each other, and individually compose: we have $(\ClassifyingSpace \Extended{\Cocharacter}) \PullBack (\ClassifyingSpace \Extended{\Cocharacter \Prime}) \PullBack = (\ClassifyingSpace \Extended{(\Cocharacter {+} \Cocharacter\Prime)}) \PullBack$, $ \ContinuationMap^{\CocharacterSlope{\Cocharacter}} \ContinuationMap^{\CocharacterSlope{\Cocharacter\Prime}} = \ContinuationMap^{\CocharacterSlope{\Cocharacter} + \CocharacterSlope{\Cocharacter\Prime}}$ and $\FloerSeidel_{\ExtendedTorus} (\Cocharacter, \FixedPoint) \FloerSeidel_{\ExtendedTorus} (\Cocharacter \Prime, \FixedPoint) = \FloerSeidel_{\ExtendedTorus} (\Cocharacter + \Cocharacter \Prime, \FixedPoint)$.
    \end{proof}
    
\section{Toric manifolds}
\label{sec:toric-manifolds-general}

While we discuss toric manifolds in Sections~\ref{sec:closed-toric-manifolds-non-equivariant-statements}--\ref{sec:connection-computation-on-closed-toric-manifold}, the reader may wish to consult \autoref{sec:projective-plane-computations} for a running example (the projective plane $\Projective^2$).
    
\subsection{Closed toric manifold}
\label{sec:closed-toric-manifolds-non-equivariant-statements}
    
    Let $\Manifold$ be a closed monotone $2 \dimManifold$-dimensional symplectic manifold with an effective Hamiltonian action $\TorusAction : \Torus \times \Manifold \to \Manifold$ of the $\dimManifold$-dimensional torus $\Torus = (\Circle)^\dimManifold$.
    Such a pair $(\Manifold, \TorusAction)$ is a monotone \define{closed toric manifold}.
    It satisfies all of the hypotheses of Sections~\ref{sec:assumptions-on-symplectic-manifold} and \ref{sec:torus-and-torus-action-on-manifold} because $\Manifold$ is simply connected \cite[Theorem~12.1.10]{cox_toric_2011}.
    
    With the decomposition $\Torus = (\Circle)^\dimManifold$, we identify the Lie algebra $\TorusLieAlgebra = \Lie(\Torus)$ with $\RealNumbers^\dimManifold$.
    The lattice of cocharacters $\LatticeCocharacters{\Torus}$ is naturally isomorphic to the lattice $\Integers^\dimManifold \subset \RealNumbers^\dimManifold = \TorusLieAlgebra$ under the map $\Cocharacter \mapsto \partial_{\eltCircle} \Cocharacter \EvaluateAt{\eltCircle = 0} \in \TorusLieAlgebra$.
    Similarly, we identify $\DualTorusLieAlgebra$ with $\RealNumbers^\dimManifold$, such that the lattice of characters $\LatticeCharacters{\Torus}$ is identified with $\Integers^\dimManifold \subset \RealNumbers^\dimManifold = \DualTorusLieAlgebra$.
    
    The moment map $\MomentMap{\TorusAction} : \Manifold \to \DualTorusLieAlgebra$ of $\TorusAction$ is uniquely determined up to an additive constant.
    The \define{moment polytope $\MomentPolytope$} is the image of $\MomentMap{\TorusAction}$; it is a convex polytope in $\DualTorusLieAlgebra = \RealNumbers^\dimManifold$.
    The \define{facets} $\Facet_i$ of $\MomentPolytope$ are its codimension-1 faces.
    There exist uniquely-determined real numbers $\FacetBoundaryCondition_i \in \RealNumbers$ such that the moment polytope is given by
        \begin{equation}
        \label{eqn:moment-polytope-equation}
            \MomentPolytope = \SetCondition{\pointDualTorusLieAlgebra \in \RealNumbers^\dimManifold}{\text{$\InnerProduct{\InwardNormalVector_i}{\pointDualTorusLieAlgebra} \ge \FacetBoundaryCondition_i$ for $i = 1, \ldots, \numberFacets$}},
        \end{equation}
    where $\InwardNormalVector_i \in \Integers^\dimManifold$ is the \define{primitive inward-pointing normal vector} to the facet $\Facet_i$ for each $i = 1, \ldots, \numberFacets$.
    The moment polytope determines the pair $(\Manifold, \TorusAction)$ up to equivariant symplectomorphism by Delzant's theorem.
    
    The faces of $\MomentPolytope$ are the nonempty subsets of the form $\Face_I = \Set{\text{$\InnerProduct{\pointDualTorusLieAlgebra}{\InwardNormalVector_i} = \FacetBoundaryCondition_i$ for $i \in I$}} \subseteq \MomentPolytope$ for a subset of indices $I = \Set{i_1, \ldots, i_p} \subseteq \Set{1, \ldots, \numberFacets}$.
    Associate to each face $\Face_I$ the cone $\Cone_I \subset \RealNumbers^\dimManifold = \TorusLieAlgebra$ given by the $\RealNumbers^{\ge 0}$-span of the vectors $\SetCondition{\InwardNormalVector_i}{i \in I}$.
    Thus we assign the ray $\RealNumbers^{\ge 0} \cdot \InwardNormalVector_i$ to the facet $\Facet_i$.
    The \define{(inward) normal fan $\InwardNormalFan$} is the fan comprising these cones.
    The fan is \define{complete}: the union of the cones in $\InwardNormalFan$ is all of $\RealNumbers^\dimManifold$; and \define{smooth}: the generators $\SetCondition{\InwardNormalVector_i}{i \in I}$ of each cone extend to a $\Integers$-basis of $\Integers^\dimManifold$.
    
    Let $\Integers [\ToricGenerator_1, \ldots, \ToricGenerator_\numberFacets]$ be the polynomial ring with formal variables $\ToricGenerator_i$ in degree 2.
    The \define{linear relations} are the polynomials $\sum_i \InnerProduct{\InwardNormalVector_i}{\pointDualTorusLieAlgebra} \ToricGenerator_i$, for all $\pointDualTorusLieAlgebra \in \Integers^\dimManifold$ (or for all $\pointDualTorusLieAlgebra$ in a basis of $\Integers^\dimManifold$).
    Let $\LinearRelationsIdeal$ be the ideal generated by the linear relations.
    The \define{Stanley-Reisner ideal} is the ideal
        \begin{equation}
            \StanleyReisnerIdeal = \langle \ToricGenerator_{i_1} \cdots \ToricGenerator_{i_p}  \mid \text{$i_1, \ldots, i_p$ are distinct and do not generate a cone of $\InwardNormalFan$}\rangle.
        \end{equation}
    Set $\Divisor_i = \MomentMap{\TorusAction} \Inverse (\Facet_i)$, and denote by $\Divisor_i \Dual \in \Cohomology^2 (\Manifold)$ the Poincar\'{e} dual of $\Divisor_i$.
    
    \begin{theorem}
        [{Cohomology presentation \cite[Theorem~12.4.4]{cox_toric_2011}}]
    \label{thm:cohomology-presentation-toric-compact}
        The map $\ToricGenerator_i \mapsto \Divisor_i \Dual$ induces a ring isomorphism
            \begin{equation}
            \label{eqn:cohomology-presentation-compact-toric-manifold}
                {\Integers [\ToricGenerator_1, \ldots, \ToricGenerator_\numberFacets]} / {(\LinearRelationsIdeal + \StanleyReisnerIdeal)} \overset{\cong}{\to} \Cohomology^\ArbitraryIndex (\Manifold).
            \end{equation}
        The map satisfies $\sum_i \ToricGenerator_i \mapsto \FirstChernClass (\TangentSpace \Manifold, \SymplecticForm)$ and $- \sum \FacetBoundaryCondition_i \ToricGenerator_i \mapsto [\SymplecticForm]$.
    \end{theorem}
    
    Poincar\'{e} duality yields $\Homology_2 (\Manifold) \cong \Cohomology^{2 \dimManifold - 2} (\Manifold)$ because the cohomology ring has no torsion \cite[Proposition~1.5]{franz_describing_2010}.
    We can use this description of $\Homology_2 (\Manifold)$ together with the expressions for $\FirstChernClass (\TangentSpace \Manifold, \SymplecticForm)$ and $[\SymplecticForm]$ from \autoref{thm:cohomology-presentation-toric-compact} to construct our Novikov ring $\NovikovRing$.
    
    Let $\Cocharacter_i \in \LatticeCocharacters{\Torus}$ be the cocharacter corresponding to the vector $\InwardNormalVector_i \in \Integers^\dimManifold \subset \TorusLieAlgebra$.
    The set $\Divisor_i$ is the (minimal) fixed locus of $\TorusAction \ComposedWith \Cocharacter_i$.
    Given a vertex $\vertexMomentPolytope$ of $\MomentPolytope$, let $\FixedPoint_v = \MomentMap{\TorusAction} \Inverse (\vertexMomentPolytope)$ be the corresponding fixed point of $\TorusAction$ in $\Manifold$.
    For any vertex $\vertexMomentPolytope$ of the facet $\Facet_i$, we have $\QuantumSeidel (\Cocharacter_i, \FixedPoint_\vertexMomentPolytope) (1) = \Divisor_i \Dual$ \cite[Example~5.3]{mcduff_topological_2006}.
    
    \begin{lemma}
    \label{lem:change-fixed-point-over-edge-of-polytope}
        Let $\edgeMomentPolytope$ be an edge of $\MomentPolytope$ between two vertices $\vertexMomentPolytope^\pm$, and let $\edgeVectorMomentPolytope \in \Integers^\dimManifold$ be the vector from $\vertexMomentPolytope^-$ to $\vertexMomentPolytope^+$.
        The Seidel map satisfies
            \begin{equation}
                \QuantumSeidel (\Cocharacter_i, \FixedPoint_{\vertexMomentPolytope^-}) (1) = \NovVariable^{ \InnerProduct{\InwardNormalVector_i}{\edgeVectorMomentPolytope} [\edgeMomentPolytope]} \QuantumSeidel (\Cocharacter_i, \FixedPoint_{\vertexMomentPolytope^+}) (1),
            \end{equation}
        where $[\edgeMomentPolytope] \in \Homology_2 (\Manifold)$ is the class $[\MomentMap{\TorusAction} \Inverse \edgeMomentPolytope]$.
    \end{lemma}
    
    \begin{proof}
        It is sufficient to show $\ClutchingSection^{\FixedPoint_{\vertexMomentPolytope_+}} \PushForward [\Sphere] = \ClutchingSection^{\FixedPoint_{\vertexMomentPolytope_-}} \PushForward [\Sphere] + \Pole^- \PushForward (\InnerProduct{\InwardNormalVector_i}{\edgeVectorMomentPolytope} [\edgeMomentPolytope])$ (this notation is from \autoref{sec:sections-of-clutching-bundle}), and moreover it is sufficient to show this on $\MomentMap{\TorusAction} \Inverse (\edgeMomentPolytope) \cong \Projective^1$.
        This can be either shown with a direct homotopy between the two fixed sections or deduced from previous calculations on $\Projective^1$ (for example \cite[Section~8.2]{liebenschutz-jones_intertwining_2020}).
    \end{proof}
    
    Any edge $\edgeMomentPolytope$ is the intersection of $n-1$ facets $\edgeMomentPolytope = \Facet_{i_1} \Intersection \cdots \Intersection \Facet_{i_{n-1}}$, so it is easy to describe the class $[\edgeMomentPolytope] = (\ToricGenerator_{i_1} \cdots \ToricGenerator_{i_{n-1}}) \Dual$ using our presentation of $\Homology_2 (\Manifold)$.
    Repeated application of \autoref{lem:change-fixed-point-over-edge-of-polytope} along a path of edges from a given vertex $\vertexMomentPolytope$ to a vertex lying on $\Facet_i$ yields
        \begin{equation}
        \label{eqn:quantum-seidel-maps-with-computed-novikov-exponents-at-1}
            \QuantumSeidel (\Cocharacter_i, \FixedPoint_\vertexMomentPolytope) (1) = \NovVariable^{\NovikovExponentForDivisor_i} \Divisor_i \Dual
        \end{equation} 
    for a class $\NovikovExponentForDivisor_i \in \Homology_2 (\Manifold)$.
    
    The set of distinct indices $I = \Set{i_1, \ldots, i_p}$ is \define{primitive} if the cone $\Cone_I$ is not in $\InwardNormalFan$, but the cone $\Cone_{I \Prime}$ is in $\InwardNormalFan$ for any proper subset $I \Prime \subsetneq I$.
    Let $I$ be primitive.
    The vector $v = \InwardNormalVector_{i_1} + \cdots + \InwardNormalVector_{i_p}$ lies in a cone $\Cone_J$ in $\InwardNormalFan$, for some subset of indices $J = \Set{j_1, \ldots, j_s}$.
    Therefore there exist positive integers $c_1, \ldots, c_s \in \Integers^{>0}$ which give
        \begin{equation}
            \label{eqn:vector-relating-primitive-index-set-with-counterparts-defining-cone}
            \InwardNormalVector_{i_1} + \cdots + \InwardNormalVector_{i_p} = c_1\InwardNormalVector_{j_1} + \cdots + c_s \InwardNormalVector_{j_s}.
        \end{equation}
    The indices $j_1, \ldots, j_s$ and integers $c_1, \ldots, c_s$ are unique up to permutation \cite[Definition~5.2]{batyrev_quantum_1993}, and moreover the sets $I$ and $J$ are disjoint.
    
    Fix a vertex $\vertexMomentPolytope$ of $\MomentPolytope$.
    % We computed $\QuantumSeidel (\Cocharacter_i, \FixedPoint_\vertexMomentPolytope) (1)$ above, and found that it was of the form $\NovVariable^{\NovikovExponentForDivisor_i} \Divisor_i \Dual$ for a class $\NovikovExponentForDivisor_i \in \Homology_2 (\Manifold)$.
    Set $\ToricGeneratorWithNovikovWeighting_i = \NovVariable^{\NovikovExponentForDivisor_i} \ToricGenerator_i \in \NovikovRing [\ToricGenerator_1, \ldots, \ToricGenerator_\numberFacets]$, with $\NovikovExponentForDivisor_i$ as in \eqref{eqn:quantum-seidel-maps-with-computed-novikov-exponents-at-1}.
    The \define{quantum Stanley-Reisner ideal} is
        \begin{equation}
            \QuantumStanleyReisnerIdeal = \langle \ToricGeneratorWithNovikovWeighting_{i_1} \cdots \ToricGeneratorWithNovikovWeighting_{i_p} - \ToricGeneratorWithNovikovWeighting_{j_1}^{c_1} \cdots \ToricGeneratorWithNovikovWeighting_{j_s}^{c_s} \mid \text{$I = \Set{i_1, \ldots, i_p}$ is primitive}\rangle.
        \end{equation}
    
    \begin{theorem}
        [{Quantum cohomology presentation \cite[Proposition~5.2]{mcduff_topological_2006}}]
    \label{thm:quantum-cohomology-presentation-closed-toric-manifolds}
        The map $\ToricGenerator_i \mapsto \Divisor_i \Dual$ induces a ring isomorphism
            \begin{equation}
            \label{eqn:presentation-for-quantum-cohomology-of-closed-toric-manifold}
                {\NovikovRing [\ToricGenerator_1, \ldots, \ToricGenerator_\numberFacets]} / {(\LinearRelationsIdeal + \QuantumStanleyReisnerIdeal)} \overset{\cong}{\to} \QuantumCohomology^\ArbitraryIndex (\Manifold).
            \end{equation}
    \end{theorem}
    
    \begin{proof}
        [Sketch proof]
        The map $\ToricGenerator_i \mapsto \Divisor_i \Dual$ satisfies $\ToricGeneratorWithNovikovWeighting_i \mapsto \QuantumSeidel (\Cocharacter_i, \FixedPoint_\vertexMomentPolytope) (1)$.
        The equality of vectors \eqref{eqn:vector-relating-primitive-index-set-with-counterparts-defining-cone} means the corresponding cocharacters are equal, giving $\QuantumSeidel (\Cocharacter_{i_1} \cdots \Cocharacter_{i_p}, \FixedPoint_\vertexMomentPolytope) (1) = \QuantumSeidel (\Cocharacter_{j_1}^{c_1} \cdots \Cocharacter_{j_s}^{c_s}, \FixedPoint_\vertexMomentPolytope) (1)$.
        The left-hand side may be expanded as a composition of Seidel maps $\QuantumSeidel (\Cocharacter_{i_1}, \FixedPoint_\vertexMomentPolytope) \cdots \QuantumSeidel(\Cocharacter_{i_p}, \FixedPoint_\vertexMomentPolytope) (1)$ by \eqref{eqn:non-equivariant-quantum-seidel-map-composition-of-actions}, and then further rearranged to a product $\QuantumSeidel (\Cocharacter_{i_1}, \FixedPoint_\vertexMomentPolytope) (1) \cdots \QuantumSeidel(\Cocharacter_{i_p}, \FixedPoint_\vertexMomentPolytope) (1)$ because the quantum Seidel map is a module map with respect to quantum multiplication.
        The same procedure applies to the right-hand side.
        Therefore the generators of $\QuantumStanleyReisnerIdeal$ vanish under $\ToricGeneratorWithNovikovWeighting_i \mapsto \QuantumSeidel (\Cocharacter_i, \FixedPoint_\vertexMomentPolytope) (1)$, so the map \eqref{eqn:presentation-for-quantum-cohomology-of-closed-toric-manifold} is well-defined.
        
        The relation $\ToricGeneratorWithNovikovWeighting_{i_1} \cdots \ToricGeneratorWithNovikovWeighting_{i_p} - \ToricGeneratorWithNovikovWeighting_{j_1}^{c_1} \cdots \ToricGeneratorWithNovikovWeighting_{j_s}^{c_s}$ for the primitive set $I$ can be rearranged to 
            \begin{equation}
            \label{eqn:quantum-stanley-reisner-higher-order-description}
                \ToricGenerator_{i_1} \cdots \ToricGenerator_{i_p} - \NovVariable^{\NovikovExponentForDivisor_I} \ToricGenerator_{j_1}^{c_1} \cdots \ToricGenerator_{j_s}^{c_s}
            \end{equation} 
        with $\SymplecticForm (\NovikovExponentForDivisor_I) > 0$ \cite{batyrev_quantum_1993}.
        Therefore the relations in $\QuantumStanleyReisnerIdeal$ are exactly the relations in $\StanleyReisnerIdeal$, except they have higher-order corrections.
        Here, a monomial is \define{higher-order} if its $\NovikovRing$-coefficient $\NovVariable^\DegreeTwoClass$ satisfies $\SymplecticForm(\DegreeTwoClass) > 0$.
        We deduce the map \eqref{eqn:presentation-for-quantum-cohomology-of-closed-toric-manifold} is both surjective and injective from \autoref{thm:cohomology-presentation-toric-compact} using induction, recursively incorporating the higher-order corrections \cite[Lemma ~5.1]{mcduff_topological_2006}.
    \end{proof}
    
\subsection{Presentation of equivariant quantum cohomology}

    Recall the canonical isomorphism $\SymmetricAlgebra (\LatticeCharacters{\Torus}) \cong \Cohomology^\ArbitraryIndex (\ClassifyingSpace \Torus)$ from \autoref{sec:classifying-space-for-torus}.
    Therefore $\Cohomology^2 (\ClassifyingSpace \Torus)$ is isomorphic to the lattice $\Integers^\dimManifold \subset \DualTorusLieAlgebra$.
    Via this isomorphism, we define a map $\Cohomology^2 (\ClassifyingSpace \Torus) = \Integers^\dimManifold \to \Integers [\ToricGenerator_1, \ldots, \ToricGenerator_\numberFacets]$ by 
        \begin{equation}
        \label{eqn:character-maps-on-equivariant-cohomology-presentation}
            \formalTorus \mapsto - \sum_{i=1}^\numberFacets \InnerProduct{\InwardNormalVector_i}{\formalTorus} \ToricGenerator_i.
        \end{equation}
    
    Associated to the $\Torus$-invariant subset $\Divisor_i \subset \Manifold$, there is an \define{equivariant cohomology class} $[\Divisor_i]_\Torus \in \Cohomology^2 _\Torus (\Manifold, \TorusAction)$ \cite[Proposition~12.4.13]{cox_toric_2011}.
    These are the $\Torus$-equivariant analogues of the Poincar\'{e} duals $\Divisor_i \Dual$.
    
    \begin{theorem}
        [{Equivariant cohomology presentation \cite[Theorem~12.4.14]{cox_toric_2011}}]
    \label{thm:equivariant-cohomology-presentation-closed-toric-manifold}
        The map $\ToricGenerator_i \mapsto [\Divisor_i]_\Torus$ induces a ring isomorphism
            \begin{equation}
            \label{eqn:equivariant-cohomology-presentation-closed-toric-manifold}
                {\Integers [\ToricGenerator_1, \ldots, \ToricGenerator_\numberFacets]} / {\StanleyReisnerIdeal} \overset{\cong}{\to} \Cohomology^\ArbitraryIndex _\Torus (\Manifold, \TorusAction).
            \end{equation}
        This extends to a $\Cohomology^\ArbitraryIndex (\ClassifyingSpace \Torus)$-algebra isomorphism under \eqref{eqn:character-maps-on-equivariant-cohomology-presentation}.
    \end{theorem}
    
    \begin{remark}
        [Surprising minus sign]
        The minus sign in \eqref{eqn:character-maps-on-equivariant-cohomology-presentation} is perhaps surprising.
        Consider the case $\Torus = \Circle$.
        We use the universal bundle $\InfiniteSphere \to \InfiniteComplexProjectiveSpace$ for $\Circle$ and we have $\Cohomology^\ArbitraryIndex (\InfiniteComplexProjectiveSpace) \cong \Integers [\formalTorus]$ (see \autoref{sec:classifying-space-for-torus}).
        Our $\Circle$-bundle $\InfiniteSphere \to \InfiniteComplexProjectiveSpace$ is the unit sphere bundle of the complex line bundle $\LineBundle_{\InfiniteComplexProjectiveSpace}(+1) \to \InfiniteComplexProjectiveSpace$.
        The generator $\formalTorus$ is equal to the first Chern class $\FirstChernClass (\LineBundle_{\InfiniteComplexProjectiveSpace}(+1))$.
        
        Consider $\ComplexNumbers$ with the standard $\Circle$-action.
        The Borel homotopy quotient $\InfiniteSphere \times_{\Circle} \ComplexNumbers \to \InfiniteComplexProjectiveSpace$ is isomorphic to $\LineBundle_{\InfiniteComplexProjectiveSpace}(-1) \to \InfiniteComplexProjectiveSpace$, where the minus sign arises since the action on $\InfiniteSphere$ is reversed.
        The class $\formalTorus = \FirstChernClass (\LineBundle_{\InfiniteComplexProjectiveSpace}(+1))$ is equal to $- \FirstChernClass(\LineBundle_{\InfiniteComplexProjectiveSpace}(-1))$, where $\FirstChernClass(\LineBundle_{\InfiniteComplexProjectiveSpace}(-1)) = \FirstChernClass(\InfiniteSphere \times_{\Circle} \ComplexNumbers)$ plays the same role as the equivariant cohomology classes $[\Divisor_i]_\Torus$.
    \end{remark}
    
    \begin{theorem}
        [Equivariant quantum cohomology presentation]
    \label{thm:equivariant-quantum-cohomology-presentation-closed-toric-manifold}
        The map $\ToricGenerator_i \mapsto [\Divisor_i]_\Torus$ induces a ring isomorphism
            \begin{equation}
            \label{eqn:equivariant-quantum-cohomology-presentation-map-closed-toric-manifolds}
                {\NovikovRing \GradedCompletedTensorProduct \Integers [\ToricGenerator_1, \ldots, \ToricGenerator_\numberFacets]} / {\QuantumStanleyReisnerIdeal} \overset{\cong}{\to} \QuantumCohomology^\ArbitraryIndex _\Torus (\Manifold, \TorusAction).
            \end{equation}
        This extends to a $\Cohomology^\ArbitraryIndex (\ClassifyingSpace \Torus)$-algebra isomorphism under \eqref{eqn:character-maps-on-equivariant-cohomology-presentation}.
    \end{theorem}
    
    \begin{proof}
    
        This proof is the $\Torus$-equivariant analogue of the proof of \autoref{thm:quantum-cohomology-presentation-closed-toric-manifolds}.
        
        Exactly like the non-equivariant case, $\QuantumSeidel_{\Torus} (\Cocharacter_i, \FixedPoint_\vertexMomentPolytope) (1)$ only counts the fixed sections of the clutching bundle $\ClutchingBundle{\Cocharacter_i}$ at points in $\Divisor_i$.
        Therefore the map \eqref{eqn:equivariant-quantum-cohomology-presentation-map-closed-toric-manifolds} satisfies $\ToricGeneratorWithNovikovWeighting_i \mapsto \NovVariable^{\NovikovExponentForDivisor_i} [\Divisor_i]_{\Torus} = \QuantumSeidel_{\Torus} (\Cocharacter_i, \FixedPoint_\vertexMomentPolytope) (1)$ just like \eqref{eqn:quantum-seidel-maps-with-computed-novikov-exponents-at-1}.
        The $\Torus$-equivariant quantum Seidel map intertwines the $\Torus$-equivariant quantum product (it is the $\ExtendedTorus$-equivariant quantum Seidel map which does not), hence the map \eqref{eqn:equivariant-quantum-cohomology-presentation-map-closed-toric-manifolds} vanishes on $\QuantumStanleyReisnerIdeal$.
        Thus \eqref{eqn:equivariant-quantum-cohomology-presentation-map-closed-toric-manifolds} is well-defined.
        
        Let $B_k$ be the ideal in $\Integers [\ToricGenerator_1, \ldots, \ToricGenerator_\numberFacets]$ generated by the degree-$k$ products of the polynomials in \eqref{eqn:character-maps-on-equivariant-cohomology-presentation}.
        Equivalently, $B_k$ is generated by the image of $\SymmetricAlgebra^k (\LatticeCharacters{\Torus})$.
        Notice $\Cohomology^\ArbitraryIndex _\Torus (\Manifold, \TorusAction) = \varprojlim {\Integers [\ToricGenerator_1, \ldots, \ToricGenerator_\numberFacets]} / {(\StanleyReisnerIdeal + B_k)}$ as ${k \to \infty}$ under the isomorphism \eqref{eqn:equivariant-cohomology-presentation-closed-toric-manifold}.
        Quotienting by $B_k$ acts like restricting $\InfiniteSphere \times_{\Circle} \Manifold$ to $\HighDimensionalSphere{2k - 1} \times_{\Circle} \Manifold$ in the $\Torus = \Circle$ case, and similarly in general.
        
        Elements of $\NovikovRing \GradedCompletedTensorProduct \Integers [\ToricGenerator_1, \ldots, \ToricGenerator_\numberFacets]$ are formal sums of monomials $\NovVariable^\DegreeTwoClass \Tensor p$ with a finiteness property: there are only finitely many such monomials with $p \notin B_k$ and the energy $\SymplecticForm (\DegreeTwoClass)$ bounded above.
        For any given $k$, we can use the methods of higher-order corrections in $\NovVariable$ from \autoref{thm:quantum-cohomology-presentation-closed-toric-manifolds} to show injectivity and surjectivity modulo $B_k$.
        Letting $k \to \infty$, we deduce the bijectivity of \eqref{eqn:equivariant-quantum-cohomology-presentation-map-closed-toric-manifolds}.
    \end{proof}
    
\subsection{Connection}
\label{sec:connection-computation-on-closed-toric-manifold}

    First, we modify the presentation \eqref{eqn:equivariant-quantum-cohomology-presentation-map-closed-toric-manifolds} to get a presentation for $\QuantumCohomology^\ArbitraryIndex _{\ExtendedTorus} (\Manifold, \Extended{\TorusAction})$: the map $\ToricGenerator_i \mapsto [\Divisor_i]_{\ExtendedTorus}$ and $\formalAdditionalCircle \mapsto \formalAdditionalCircle$ induces a $\Cohomology^\ArbitraryIndex (\ClassifyingSpace \ExtendedTorus)$-module isomorphism
        \begin{equation}
        \label{eqn:equivariant-quantum-cohomology-presentation-map-closed-toric-manifolds-extended}
            {\NovikovRing \GradedCompletedTensorProduct \Integers [\ToricGenerator_1, \ldots, \ToricGenerator_\numberFacets, \formalAdditionalCircle]} / {\QuantumStanleyReisnerIdeal} \overset{\cong}{\to} \QuantumCohomology^\ArbitraryIndex _{\ExtendedTorus} (\Manifold, \Extended{\TorusAction}).
        \end{equation}
    Since the copy of $\AdditionalCircle$ in $\ExtendedTorus = \AdditionalCircle \times \Torus$ acts trivially on $\Manifold$, we merely incorporate a new formal variable $\formalAdditionalCircle$ which generates $\Cohomology^\ArbitraryIndex (\ClassifyingSpace \AdditionalCircle)$.
    We write $\Cohomology^\ArbitraryIndex (\ClassifyingSpace \ExtendedTorus) = \SymmetricAlgebra (\LatticeCharacters{\Torus}) \Tensor \Integers [\formalAdditionalCircle]$.
    
    Let $\vertexMomentPolytope$ be a vertex of the polytope $\MomentPolytope$, and let $\NeighbouringFacets (\vertexMomentPolytope) = \SetCondition{i}{\vertexMomentPolytope \in \Facet_i}$ be the set of indices for the facets neighbouring $\vertexMomentPolytope$.
    The \define{star} of $\vertexMomentPolytope$ is the union of the interiors of all faces containing $\vertexMomentPolytope$, denoted $\Star(\vertexMomentPolytope)$.
    The facets $\Facet_i$ with $i \in \NeighbouringFacets(\vertexMomentPolytope)$ are exactly the facets which have nonempty intersection with $\Star(\vertexMomentPolytope)$.
    Therefore the restriction map $\Cohomology^2 _{\ExtendedTorus} (\Manifold, \Extended{\TorusAction}) \to \Cohomology^2 _{\ExtendedTorus} (\MomentMap{\TorusAction} \Inverse \Star(\vertexMomentPolytope), \Extended{\TorusAction})$ satisfies $[\Divisor_i]_{\ExtendedTorus} \mapsto 0$ if and only if $i \notin \NeighbouringFacets (\vertexMomentPolytope)$.
    The inclusion $\Set{\FixedPoint_\vertexMomentPolytope} \to \MomentMap{\TorusAction} \Inverse \Star(\vertexMomentPolytope)$ is an equivariant retraction, so $\Cohomology^2 _{\ExtendedTorus} (\MomentMap{\TorusAction} \Inverse \Star(\vertexMomentPolytope), \Extended{\TorusAction})$ is naturally isomorphic to $\Cohomology^2 _{\ExtendedTorus} (\Set{\FixedPoint_\vertexMomentPolytope})$.
    This determines the map $\FixedPoint_\vertexMomentPolytope \PullBack$ in \eqref{eqn:split-short-exact-sequence-borel-quotient-closed-toric-manifold}.
    
    Recall the set $\SetCondition{\InwardNormalVector_i}{i \in \NeighbouringFacets(\vertexMomentPolytope)}$ is a basis for $\Integers^\dimManifold \subset \TorusLieAlgebra$ since $\InwardNormalFan$ is smooth.
    Let $\SetCondition{\formalTorus^\vertexMomentPolytope _i}{i \in \NeighbouringFacets(\vertexMomentPolytope)}$ be the $\InnerProduct{\Argument}{\Argument}$-dual basis to $\SetCondition{\InwardNormalVector_i}{i \in \NeighbouringFacets(\vertexMomentPolytope)}$.
    The linear relation corresponding to $\formalTorus^\vertexMomentPolytope_i$ is of the form 
        \begin{equation}
        \label{eqn:linear-relation-for-facet-neighbouring-vertex}
            \ToricGenerator_i = \sum_{j \notin \NeighbouringFacets(\vertexMomentPolytope)} a_{j i} \ToricGenerator_i
        \end{equation}
    for each $i \in \NeighbouringFacets(\vertexMomentPolytope)$.
    Therefore the map
        \begin{equation}
        \label{eqn:generators-not-neighbouring-vertex-generate-degree-2}
            \Integers \langle \ToricGenerator_i : i \notin \NeighbouringFacets (\vertexMomentPolytope) \rangle \to {\Integers \langle \ToricGenerator_1, \ldots, \ToricGenerator_\numberFacets \rangle} / {(\text{linear relations})}
        \end{equation}
    is an isomorphism.
    The right-hand side of \eqref{eqn:generators-not-neighbouring-vertex-generate-degree-2} is the degree 2 summand of the presentation \eqref{eqn:cohomology-presentation-compact-toric-manifold} because the elements of $\StanleyReisnerIdeal$ are at least degree 4.
    Putting this together, the short exact sequence \eqref{eqn:split-short-exact-sequence-borel-quotient} is
        \begin{equation}
        \label{eqn:split-short-exact-sequence-borel-quotient-closed-toric-manifold}
            \begin{tikzcd}
                0
                \arrow[r]
                & \Cohomology^2 (\ClassifyingSpace \ExtendedTorus)
                \arrow[r]
                \arrow[d, equal]
                & \Cohomology^2_{\ExtendedTorus} (\Manifold, \Extended{\TorusAction})
                \arrow[l, bend right, "\FixedPoint_\vertexMomentPolytope \PullBack"']
                \arrow[r]
                & \Cohomology^2 (\Manifold)
                \arrow[l, bend right, dashed]
                \arrow[r]
                & 0
                \\
                0
                \arrow[r]
                & \Cohomology^2 (\ClassifyingSpace \ExtendedTorus)
                \arrow[r]
                & \Integers \langle \ToricGenerator_i, \formalAdditionalCircle \rangle
                \arrow[l, bend left, "\FixedPoint_\vertexMomentPolytope \PullBack (\ToricGenerator_i) = 0 \iff i \notin \NeighbouringFacets (\vertexMomentPolytope)"]
                \arrow[r]
                \arrow[u, "\cong"]
                & \Integers \langle \ToricGenerator_i : i \notin \NeighbouringFacets (\vertexMomentPolytope) \rangle
                \arrow[l, bend left, dashed, "\ToricGenerator_i \mapsto \ToricGenerator_i"]
                \arrow[r]
                \arrow[u, "\cong"]
                & 0.
            \end{tikzcd}
        \end{equation}
    It is easiest to define the differential connection for $\ToricGenerator_i$ with $i \notin \NeighbouringFacets(\vertexMomentPolytope)$, for then the lifting map $\DegreeTwoCoclass \mapsto \DegreeTwoCoclass^{\FixedPoint_\vertexMomentPolytope}$ is simply given by $\ToricGenerator_i \mapsto \ToricGenerator_i$.
    The differential connection is
        \begin{equation}
        \label{eqn:connection-for-toric-manifolds}
            \Connection^\vertexMomentPolytope_i = \formalAdditionalCircle \dbyd{\ToricGenerator_i} + \ToricGenerator_i \text{ for } i \notin \NeighbouringFacets(\vertexMomentPolytope).
        \end{equation}
    
    \begin{remark}
        [Differentiation]
    \label{remark:evaluating-differentiation-using-toric-presentation}
        The equivariant quantum cohomology $\QuantumCohomology^\ArbitraryIndex _{\ExtendedTorus} (\Manifold, \Extended{\TorusAction})$ is isomorphic as a $\NovikovRing \GradedCompletedTensorProduct \Cohomology^\ArbitraryIndex (\ClassifyingSpace \ExtendedTorus)$-module to $\NovikovRing \GradedCompletedTensorProduct  (\Cohomology^\ArbitraryIndex (\ClassifyingSpace \ExtendedTorus) \Tensor \Cohomology^\ArbitraryIndex (\Manifold))$.
        Therefore the elements of $\QuantumCohomology^\ArbitraryIndex _{\ExtendedTorus} (\Manifold, \Extended{\TorusAction})$ may be written as formal sums of monomials $\NovVariable^\DegreeTwoClass p(\formalTorus, \formalAdditionalCircle) r(\ToricGenerator)$ where $p(\formalTorus, \formalAdditionalCircle) \in \Cohomology^\ArbitraryIndex(\ClassifyingSpace \ExtendedTorus)$ is any monomial and $r(\ToricGenerator) \in \Integers [\ToricGenerator_1, \ldots, \ToricGenerator_\numberFacets]$ does not vanish under \eqref{eqn:cohomology-presentation-compact-toric-manifold}.
        The differentiation operation $\dbyd{\ToricGenerator_i}$ is defined term-wise when elements are \emph{written in this form}.
        We have $\dbyd{\ToricGenerator_i} (\NovVariable^\DegreeTwoClass p(\formalTorus, \formalAdditionalCircle) r(\ToricGenerator)) = \ToricGenerator_i (\DegreeTwoClass) \  \NovVariable^\DegreeTwoClass p(\formalTorus, \formalAdditionalCircle) r(\ToricGenerator)$ for such monomials.
    \end{remark}
    
    The shift operator $\ShiftOperator^\vertexMomentPolytope _{\Cocharacter_i} = \ShiftOperator^\vertexMomentPolytope _i$ satisfies 
        \begin{equation}
        \label{eqn:shift-operator-at-1-for-toric-manifolds}
            \ShiftOperator^\vertexMomentPolytope _i (1) = \ToricGeneratorWithNovikovWeighting_i = \NovVariable^{\NovikovExponentForDivisor_i} \ToricGenerator_i,
        \end{equation}
    where $\NovikovExponentForDivisor_i \in \Homology_2 (\Manifold)$ is defined as in \eqref{eqn:quantum-seidel-maps-with-computed-novikov-exponents-at-1}.
    We have 
        \begin{equation}
        \label{eqn:pullback-map-for-toric-manifolds}
            (\ClassifyingSpace \Extended{\Cocharacter_i}) \PullBack (\formalTorus) = {\formalTorus \ComposedWith \Extended{\Cocharacter_i}} = \formalTorus + \InnerProduct{\InwardNormalVector_i}{\formalTorus} \formalAdditionalCircle,
        \end{equation}
    where in the second equality we write out the character $\formalTorus \ComposedWith \Extended{\Cocharacter_i}$ in terms of the original character $\formalTorus$ and $\formalAdditionalCircle$.
    
    \begin{theorem}
        We can compute $\ShiftOperator^\vertexMomentPolytope_i$ on the whole ring $\QuantumCohomology^\ArbitraryIndex _{\ExtendedTorus} (\Manifold, \Extended{\TorusAction})$ with respect to the presentation \eqref{eqn:equivariant-quantum-cohomology-presentation-map-closed-toric-manifolds-extended} by combining \eqref{eqn:shift-operator-at-1-for-toric-manifolds} with the flatness equations $\ShiftOperator^\vertexMomentPolytope _i \Connection^\vertexMomentPolytope_j = \Connection^\vertexMomentPolytope_j \ShiftOperator^\vertexMomentPolytope _i$ for $j \notin \NeighbouringFacets(\vertexMomentPolytope)$, where $\Connection^\vertexMomentPolytope_j$ is given by \eqref{eqn:connection-for-toric-manifolds}, and the relations $\ShiftOperator^\vertexMomentPolytope _i (\formalTorus \ToricGenerator) = (\ClassifyingSpace \Extended{\Cocharacter_i}) \PullBack (\formalTorus) \  \ShiftOperator^\vertexMomentPolytope _i (\ToricGenerator)$ for $\formalTorus \in \Integers^\dimManifold \subset \DualTorusLieAlgebra$, where $(\ClassifyingSpace \Extended{\Cocharacter_i}) \PullBack (\formalTorus)$ is given by \eqref{eqn:pullback-map-for-toric-manifolds}.
        
        Moreover since the set of normal vectors $\Set{\InwardNormalVector_i}_{i = 1}^{\numberFacets}$ spans $\Integers^\dimManifold \cong \LatticeCocharacters{\Torus}$, the shift operators $\ShiftOperator^\vertexMomentPolytope_i$ may be combined to compute $\ShiftOperator^\vertexMomentPolytope _{\Cocharacter}$ for any cocharacter $\Cocharacter$.
    \end{theorem}
    
    \begin{proof}
        We prove that $\ShiftOperator^\vertexMomentPolytope_i$ is determined by these ingredients on all monomials.
        Since $\ShiftOperator^\vertexMomentPolytope_i$ satisfies $\ShiftOperator^\vertexMomentPolytope_i (\formalAdditionalCircle \ToricGenerator) = \formalAdditionalCircle \ShiftOperator^\vertexMomentPolytope_i (\ToricGenerator)$, we need only consider monomials in the variables $\ToricGenerator_j$.
        We use induction on the degree $k$ of the monomial.
        The base case $k=0$ is \eqref{eqn:shift-operator-at-1-for-toric-manifolds}.
        For the induction step, take a monomial $\ToricGenerator_j p$, where $p$ is a monomial of degree $k$.
        
        If $j \notin \NeighbouringFacets(\vertexMomentPolytope)$ holds, we rearrange the flatness equation for $\Connection^\vertexMomentPolytope_j$ to get
            \begin{equation}
            \label{eqn:shift-operator-equation-for-toric-generator-not-neighbouring}
                \ShiftOperator^\vertexMomentPolytope_i (\ToricGenerator_j p) = \Connection^\vertexMomentPolytope_j \ShiftOperator^\vertexMomentPolytope_i (p) - \formalAdditionalCircle \ShiftOperator^\vertexMomentPolytope_i \left( \dbyd{\ToricGenerator_j} p \right).
            \end{equation}
        By the induction hypothesis, the right-hand side of \eqref{eqn:shift-operator-equation-for-toric-generator-not-neighbouring} is known.
        Here, we are using the fact that $\dbyd{\ToricGenerator_j} p$ is a sum of monomials of degree $\le k$ with coefficients in the Novikov ring.
        This follows because, to put $p$ into the form required by \autoref{remark:evaluating-differentiation-using-toric-presentation}, we apply the relations \eqref{eqn:character-maps-on-equivariant-cohomology-presentation} and \eqref{eqn:quantum-stanley-reisner-higher-order-description}, both of which only decrease the degree of the monomials.
        
        Conversely, if $j \in \NeighbouringFacets(\vertexMomentPolytope)$ holds, then the linear relation \eqref{eqn:linear-relation-for-facet-neighbouring-vertex} combined with \eqref{eqn:character-maps-on-equivariant-cohomology-presentation} yields
            \begin{equation}
            \label{eqn:shift-operator-equation-for-toric-generator-neighbouring}
                \ShiftOperator^\vertexMomentPolytope_i (\ToricGenerator_j p) = \sum_{l \notin \NeighbouringFacets(\vertexMomentPolytope)} a_{l j} \  \ShiftOperator^\vertexMomentPolytope_i (\ToricGenerator_l p) - \ShiftOperator^\vertexMomentPolytope_i (\formalTorus^\vertexMomentPolytope _j p),
            \end{equation}
        where $\formalTorus^\vertexMomentPolytope _j$ is the dual to $\InwardNormalVector_j$
        The shift operator was computed above for the monomials $\ToricGenerator_l p$ with $l \notin \NeighbouringFacets(\vertexMomentPolytope)$, and we have
            \begin{equation}
                \ShiftOperator^\vertexMomentPolytope_i (\formalTorus^\vertexMomentPolytope_j p) = (\formalTorus^\vertexMomentPolytope_j + \InnerProduct{\InwardNormalVector_i}{\formalTorus^\vertexMomentPolytope_j} \formalAdditionalCircle) \  \ShiftOperator^\vertexMomentPolytope_i (p)
            \end{equation}
        using \eqref{eqn:pullback-map-for-toric-manifolds}.
        Therefore the right-hand side of \eqref{eqn:shift-operator-equation-for-toric-generator-neighbouring} is determined.
    \end{proof}
    
\subsection{Projective plane}
\label{sec:projective-plane-computations}

    An instructive example of a closed monotone toric manifold is the complex projective plane $\Projective^2$.
    The 2-dimensional torus $\Torus = (\Circle)^2$ acts on $\Projective^2$ via $(\eltTorus_1, \eltTorus_2) \cdot [\eltComplexProjectiveSpace_0, \eltComplexProjectiveSpace_1, \eltComplexProjectiveSpace_2] = [\eltComplexProjectiveSpace_0, \ExponentialNumber^{2 \PiNumber \ImaginaryNumber \eltTorus_1} \eltComplexProjectiveSpace_1, \ExponentialNumber^{2 \PiNumber \ImaginaryNumber \eltTorus_2} \eltComplexProjectiveSpace_2]$.
    The moment map $\MomentMap{} : \Projective^2 \to \RealNumbers^2$ given by
        \begin{equation}
            \MomentMap{} ([\eltComplexProjectiveSpace_0, \eltComplexProjectiveSpace_1, \eltComplexProjectiveSpace_2]) = \frac{1}{\sum_i | \eltComplexProjectiveSpace_i |^2} (| \eltComplexProjectiveSpace_1 |^2, | \eltComplexProjectiveSpace_2 |^2)
        \end{equation}
    yields the moment polytope and fan as in \autoref{fig:projective-plane-moment-polytope-and-fan}.
    The vertices $\vertexMomentPolytope_0 = (0, 0)$, $\vertexMomentPolytope_1 = (1, 0)$ and $\vertexMomentPolytope_2 = (0, 1)$ correspond to the fixed points $[1, 0, 0]$, $[0, 1, 0]$ and $[0, 0, 1]$ respectively.
    The inward-pointing normal vectors are $\InwardNormalVector_0 = (-1, -1)$, $\InwardNormalVector_1 = (1, 0)$ and $\InwardNormalVector_2 = (0, 1)$, and the corresponding real numbers for \eqref{eqn:moment-polytope-equation} are $\FacetBoundaryCondition_0 = -1$, $\FacetBoundaryCondition_1 = 0$ and $\FacetBoundaryCondition_2 = 0$ respectively.
    The facets $\Facet_i$ correspond to the invariant subsets $\Divisor_i = \Set{\eltComplexProjectiveSpace_i = 0}$.
    
\begin{figure}[ht]
	\begin{center}
		\begin{tikzpicture}
			\node[inner sep=0] at (0,0) {\includegraphics[width=10 cm]{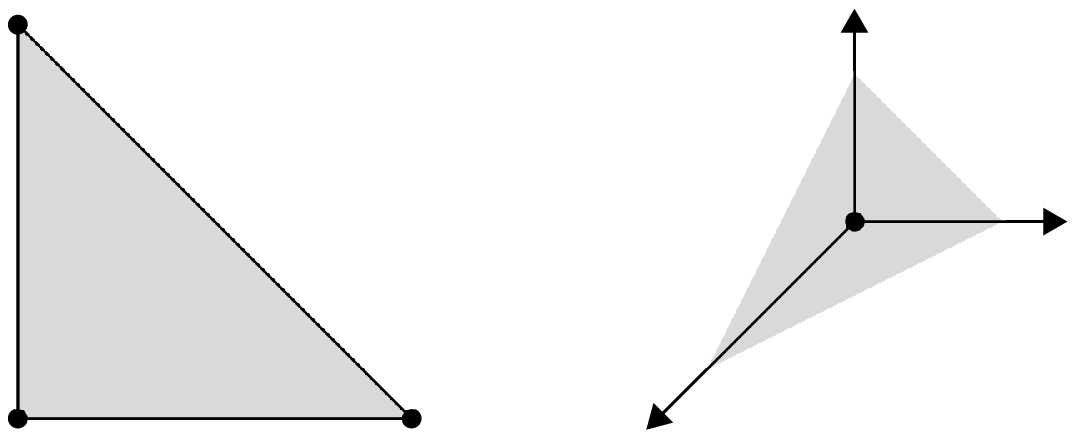}};
% 			\draw[step=1.0,black,thin] (-6,-5) grid (6,5);
			\node at (-5.1,0){$\Facet_1$};
			\node at (-3.1,-2.1){$\Facet_2$};
			\node at (-2.8,0.2){$\Facet_0$};
			\node at (5.2,0){$\InwardNormalVector_1$};
			\node at (3,2.1){$\InwardNormalVector_2$};
			\node at (0.9,-2.1){$\InwardNormalVector_0$};
			\node at (-0.9,-2.1){$\vertexMomentPolytope_1$};
			\node at (-5.1,2.1){$\vertexMomentPolytope_2$};
			\node at (-5.1,-2.1){$\vertexMomentPolytope_0$};
		\end{tikzpicture}
	\end{center}
	\caption{The moment polytope and the inner normal fan for $\Projective^2$.}
	\label{fig:projective-plane-moment-polytope-and-fan}
\end{figure}
    
    The linear relations are $\ToricGenerator_1 - \ToricGenerator_0$ and $\ToricGenerator_2 - \ToricGenerator_0$, and the Stanley-Reisner ideal is $\StanleyReisnerIdeal = \langle \ToricGenerator_0 \ToricGenerator_1 \ToricGenerator_2 \rangle$.
    The presentation of the cohomology ring $\Cohomology^\ArbitraryIndex (\Projective^2)$ from \autoref{thm:cohomology-presentation-toric-compact} is $\Integers [\ToricGenerator_0, \ToricGenerator_1, \ToricGenerator_2] / \langle\ToricGenerator_1 - \ToricGenerator_0, \ToricGenerator_2 - \ToricGenerator_0, \ToricGenerator_0 \ToricGenerator_1 \ToricGenerator_2 \rangle$.
    The map $\ToricGenerator_i \mapsto \ToricGenerator$ induces an isomorphism to $\Integers [\ToricGenerator] / \ToricGenerator^3$, which is the standard presentation of $\Cohomology^\ArbitraryIndex (\Projective^2)$.
    Using this presentation, we have $\FirstChernClass (\TangentSpace \Projective^2, \SymplecticForm) = 3 \ToricGenerator$ and $[\SymplecticForm] = \ToricGenerator$.
    Therefore, the Novikov ring is $\NovikovRing = \Integers [\NovVariable]$, with $\NovVariable$ in degree 6, and $\dbyd{\ToricGenerator} \NovVariable^k = k \NovVariable^k$ acts like $\NovVariable \dbyd{\NovVariable}$.
    
    Set $\vertexMomentPolytope = \vertexMomentPolytope_0$.
    This yields $\ToricGeneratorWithNovikovWeighting_1 = \ToricGenerator_1$ and $\ToricGeneratorWithNovikovWeighting_2 = \ToricGenerator_2$.
    To find $\ToricGeneratorWithNovikovWeighting_0$, we use the edge $\edgeMomentPolytope = \Facet_2$ from $\vertexMomentPolytope_0$ to $\vertexMomentPolytope_1$ which has vector $\edgeVectorMomentPolytope = (1, 0)$.
    \autoref{lem:change-fixed-point-over-edge-of-polytope} yields
        \begin{equation}
            \QuantumSeidel (\Cocharacter_0, \FixedPoint_{\vertexMomentPolytope_0}) (1) = \NovVariable^{ \InnerProduct{\InwardNormalVector_0}{\edgeVectorMomentPolytope} [\edgeMomentPolytope]} \QuantumSeidel (\Cocharacter_0, \FixedPoint_{\vertexMomentPolytope_1}) (1) = \NovVariable^{-1} \QuantumSeidel (\Cocharacter_0, \FixedPoint_{\vertexMomentPolytope_1}) (1),
        \end{equation}
    giving $\ToricGeneratorWithNovikovWeighting_0 = \NovVariable^{-1} \ToricGenerator_0$.
    The only primitive set is $\Set{0, 1, 2}$, for which \eqref{eqn:vector-relating-primitive-index-set-with-counterparts-defining-cone} reads $\InwardNormalVector_0 + \InwardNormalVector_1 + \InwardNormalVector_2 = 0$.
    Thus the quantum Stanley-Reisner ideal is $\QuantumStanleyReisnerIdeal = \langle \ToricGeneratorWithNovikovWeighting_0 \ToricGeneratorWithNovikovWeighting_1 \ToricGeneratorWithNovikovWeighting_2 - 1 \rangle = \langle \ToricGenerator_0 \ToricGenerator_1 \ToricGenerator_2 - \NovVariable \rangle$.
    \autoref{thm:quantum-cohomology-presentation-closed-toric-manifolds} yields the presentation
        \begin{equation}
            \QuantumCohomology^\ArbitraryIndex (\Projective^2) \cong \frac{\NovikovRing [\ToricGenerator_0, \ToricGenerator_1, \ToricGenerator_2]}{\langle \ToricGenerator_1 - \ToricGenerator_0, \ToricGenerator_2 - \ToricGenerator_0, \ToricGenerator_0 \ToricGenerator_1 \ToricGenerator_2 - \NovVariable \rangle}.
        \end{equation}
    The map $\ToricGenerator_i \mapsto \ToricGenerator$ recovers the familiar presentation $\NovikovRing [\ToricGenerator] / (\ToricGenerator^3 - \NovVariable)$.
    
    The lattice $\Integers^2 \subset \DualTorusLieAlgebra$ is generated by the vectors $\formalTorus_1 = (1, 0)$ and $\formalTorus_2 = (0, 1)$, so we write $\Cohomology^\ArbitraryIndex (\ClassifyingSpace \Torus) = \Integers [\formalTorus_1, \formalTorus_2]$.
    \autoref{thm:equivariant-cohomology-presentation-closed-toric-manifold} gives the presentation
        \begin{equation}
            \Cohomology^\ArbitraryIndex _\Torus (\Projective^2) \cong \frac{\Integers [\ToricGenerator_0, \ToricGenerator_1, \ToricGenerator_2]}{\langle \ToricGenerator_0 \ToricGenerator_1 \ToricGenerator_2 \rangle}
        \end{equation}
    for equivariant cohomology, with the $\Integers[\formalTorus_1, \formalTorus_2]$-module structure arising from the map $\Integers[\formalTorus_1, \formalTorus_2] \to \Integers [\ToricGenerator_0, \ToricGenerator_1, \ToricGenerator_2]$ given by
        \begin{equation}
        \label{eqn:equivariant-variables-assignment-presentation-projective-2}
            \begin{cases}
                \formalTorus_1 \mapsto \ToricGenerator_0 - \ToricGenerator_1 \\
                \formalTorus_2 \mapsto \ToricGenerator_0 - \ToricGenerator_2,
            \end{cases}
        \end{equation}
    as in \eqref{eqn:character-maps-on-equivariant-cohomology-presentation}.
    Via \autoref{thm:equivariant-quantum-cohomology-presentation-closed-toric-manifold}, the map \eqref{eqn:equivariant-variables-assignment-presentation-projective-2} provides the module structure for the presentations of equivariant quantum cohomology
        \begin{equation}
            \begin{aligned}
                \QuantumCohomology^\ArbitraryIndex_\Torus (\Projective^2) &\cong \frac{\NovikovRing \GradedCompletedTensorProduct \Integers [\ToricGenerator_0, \ToricGenerator_1, \ToricGenerator_2]}{\langle \ToricGenerator_0 \ToricGenerator_1 \ToricGenerator_2 - \NovVariable \rangle},
                & \hspace{2em}
                \QuantumCohomology^\ArbitraryIndex_{\ExtendedTorus} (\Projective^2) &\cong \frac{\NovikovRing \GradedCompletedTensorProduct \Integers [\ToricGenerator_0, \ToricGenerator_1, \ToricGenerator_2, \formalAdditionalCircle]}{\langle \ToricGenerator_0 \ToricGenerator_1 \ToricGenerator_2 - \NovVariable \rangle}.
            \end{aligned}
        \end{equation}
    Only the facet $\Facet_0$ does not contain $\vertexMomentPolytope_0$, so the differential connection is 
        \begin{equation}
            \Connection^{\vertexMomentPolytope_0} _0 = \formalAdditionalCircle \NovVariable \dbyd{\NovVariable} + \ToricGenerator_0.
        \end{equation}
        
    We demonstrate how to compute the shift operator $\ShiftOperator^{\vertexMomentPolytope_0} _{1}$.
    We have $\ShiftOperator^{\vertexMomentPolytope_0} _{1} (1) = \ToricGenerator_1$.
    By applying the flatness equation to $1$, we find $\ShiftOperator^{\vertexMomentPolytope_0} _1 (\ToricGenerator_0) = \ToricGenerator_0 \ToricGenerator_1$:
        \begin{equation}
            \begin{aligned}
            0 &= \ShiftOperator^{\vertexMomentPolytope_0} _1 \Connection^{\vertexMomentPolytope_0}_0 (1) - \Connection^{\vertexMomentPolytope_0}_0 \ShiftOperator^{\vertexMomentPolytope_0} _1 (1) \\
            &= \ShiftOperator^{\vertexMomentPolytope_0} _1 (\ToricGenerator_0) - \Connection^{\vertexMomentPolytope_0}_0 (\ToricGenerator_1) \\
            &= \ShiftOperator^{\vertexMomentPolytope_0} _1 (\ToricGenerator_0) - \ToricGenerator_0 \ToricGenerator_1.
            \end{aligned}
        \end{equation}
    We use \eqref{eqn:equivariant-variables-assignment-presentation-projective-2} to deduce the value of $\ShiftOperator^{\vertexMomentPolytope_0} _{1}$ for the remaining monomials $\ToricGenerator_1$ and $\ToricGenerator_2$:
        \begin{equation}
            \begin{aligned}
                \ShiftOperator^{\vertexMomentPolytope_0} _1 (\ToricGenerator_1) &= \ShiftOperator^{\vertexMomentPolytope_0} _1 (\ToricGenerator_0 - \formalTorus_1)
                &
                \hspace{2em}
                \ShiftOperator^{\vertexMomentPolytope_0} _1 (\ToricGenerator_2) &= \ShiftOperator^{\vertexMomentPolytope_0} _1 (\ToricGenerator_0 - \formalTorus_2)
                \\
                &= \ToricGenerator_0 \ToricGenerator_1 - (\formalTorus_1 + \formalAdditionalCircle) \ToricGenerator_1
                &
                &= \ToricGenerator_0 \ToricGenerator_1 - \formalTorus_2 \ToricGenerator_1
                \\
                &= \ToricGenerator_1^2 - \formalAdditionalCircle \ToricGenerator_1
                &
                &= \ToricGenerator_1 \ToricGenerator_2 \\
                &= \ToricGenerator_1 (\ToricGenerator_1 - \formalAdditionalCircle)
            \end{aligned}
        \end{equation}
    We derive $\ShiftOperator^{\vertexMomentPolytope_0} _{1}$ for the monomials in degree 4 by repeating these steps: first apply the flatness equation to each of the monomials $\ToricGenerator_0$, $\ToricGenerator_1$ and $\ToricGenerator_2$; and second use \eqref{eqn:equivariant-variables-assignment-presentation-projective-2} to deduce $\ShiftOperator^{\vertexMomentPolytope_0} _{1}$ on the remaining monomials.
    This yields
        \begin{equation}
            \begin{aligned}
                \ShiftOperator^{\vertexMomentPolytope_0} _1 (\ToricGenerator_0 ^2) &= \ToricGenerator_0 ^2 \ToricGenerator_1
                &
                \hspace{3em}
                \ShiftOperator^{\vertexMomentPolytope_0} _1 (\ToricGenerator_1 ^2) &= \ToricGenerator_1 (\ToricGenerator_1 - \formalAdditionalCircle) ^2
                \\
                \ShiftOperator^{\vertexMomentPolytope_0} _1 (\ToricGenerator_0 \ToricGenerator_1) &= \ToricGenerator_0 \ToricGenerator_1 (\ToricGenerator_1 - \formalAdditionalCircle)
                &
                \ShiftOperator^{\vertexMomentPolytope_0} _1 (\ToricGenerator_1 \ToricGenerator_2) &= \ToricGenerator_1 (\ToricGenerator_1 - \formalAdditionalCircle) \ToricGenerator_2
                \\
                \ShiftOperator^{\vertexMomentPolytope_0} _1 (\ToricGenerator_0 \ToricGenerator_2) &= \ToricGenerator_0 \ToricGenerator_1 \ToricGenerator_2
                &
                \ShiftOperator^{\vertexMomentPolytope_0} _1 (\ToricGenerator_2^2) &= \ToricGenerator_1 \ToricGenerator_2^2.
            \end{aligned}
        \end{equation}
    The pattern $\ShiftOperator^{\vertexMomentPolytope_0} _1 (\ToricGenerator_0^{k_0} \ToricGenerator_1^{k_1} \ToricGenerator_2^{k_2}) = \ToricGenerator_0^{k_0} \ToricGenerator_1 (\ToricGenerator_1 - \formalAdditionalCircle)^{k_1} \ToricGenerator_2^{k_2}$ holds for $k_0 + k_1 + k_2 \le 2$ because the differentiation operation vanishes for such polynomials.
    It fails in general, however.
    For example, applying the flatness equation to $\ToricGenerator_1 \ToricGenerator_2$ yields
        \begin{align}
            0 &= \ShiftOperator^{\vertexMomentPolytope_0} _1 \Connection^{\vertexMomentPolytope_0}_0 (\ToricGenerator_1 \ToricGenerator_2) - \Connection^{\vertexMomentPolytope_0}_0 \ShiftOperator^{\vertexMomentPolytope_0} _1 (\ToricGenerator_1 \ToricGenerator_2) \\
            &= \ShiftOperator^{\vertexMomentPolytope_0} _1 (\ToricGenerator_0 \ToricGenerator_1 \ToricGenerator_2) - \Connection^{\vertexMomentPolytope_0}_0 (\ToricGenerator_1 (\ToricGenerator_1 - \formalAdditionalCircle) \ToricGenerator_2) \\
            &= \ShiftOperator^{\vertexMomentPolytope_0} _1 (\NovVariable) - \formalAdditionalCircle \NovVariable \dbyd{\NovVariable} (\ToricGenerator_1 (\ToricGenerator_1 - \formalAdditionalCircle) \ToricGenerator_2) - \ToricGenerator_0 (\ToricGenerator_1 (\ToricGenerator_1 - \formalAdditionalCircle) \ToricGenerator_2) \label{part:last-line-before-differentiation-executed}\\
            &= \NovVariable \ToricGenerator_1 - \formalAdditionalCircle \NovVariable - \NovVariable  (\ToricGenerator_1 - \formalAdditionalCircle) \\
            &= 0.
        \end{align}
    The input to the differentiation operation in \eqref{part:last-line-before-differentiation-executed} must be rearranged to the format described in \autoref{remark:evaluating-differentiation-using-toric-presentation},
        \begin{equation}
            \begin{aligned}
                \ToricGenerator_1 (\ToricGenerator_1 - \formalAdditionalCircle) \ToricGenerator_2 &= (\ToricGenerator_0 - \formalTorus_1) (\ToricGenerator_1 - \formalAdditionalCircle) \ToricGenerator_2 \\
                &= \NovVariable + \formalAdditionalCircle \ToricGenerator_0 \ToricGenerator_2 + \formalAdditionalCircle \formalTorus_1 \ToricGenerator_2 - \formalTorus_1 \ToricGenerator_1 \ToricGenerator_2,
            \end{aligned}
        \end{equation}
    before differentiating with respect to the variable $\NovVariable$.
    A naive differentiation `as is' would have missed the $\NovVariable$ term.
    The only monomials $r(\ToricGenerator) \in \Integers [\ToricGenerator_0, \ToricGenerator_1, \ToricGenerator_2]$ which survive under \eqref{eqn:cohomology-presentation-compact-toric-manifold} are the monomials $r(\ToricGenerator) = \ToricGenerator_0^{k_0} \ToricGenerator_1^{k_1} \ToricGenerator_2^{k_2}$ with $k_0 + k_1 + k_2 \le 2$.
    Therefore we have completely described $\ShiftOperator^{\vertexMomentPolytope_0} _1$ since the input can be put in the format of \autoref{remark:evaluating-differentiation-using-toric-presentation}.
    
\subsection{Toric negative line bundle}

    While we discuss toric negative line bundles in this section, the reader may wish to consult \autoref{sec:tautological-line-bundle} for a running example (the line bundle $\LineBundle_{\Projective^1} (-1)$).

    Let $(\BaseToricLineBundle, \TorusAction_\BaseToricLineBundle)$ be a closed monotone toric manifold with symplectic form $\SymplecticForm_\BaseToricLineBundle$.
    Let $\IndexOfMonotonicity_\BaseToricLineBundle$ denote the monotonicity constant\footnote{
        The \define{monotonicity constant} of a monotone symplectic manifold $(\Manifold, \SymplecticForm)$ is the positive number $\IndexOfMonotonicity > 0$ for which $\FirstChernClass(\TangentSpace \Manifold, \SymplecticForm) = \IndexOfMonotonicity [\SymplecticForm]$ holds.
    } of $\BaseToricLineBundle$.
    We use a moment polytope $\MomentPolytope_\BaseToricLineBundle = \Set{\InnerProduct{\InwardNormalVector^\BaseToricLineBundle_i}{\pointDualTorusLieAlgebra} \ge \FacetBoundaryCondition^\BaseToricLineBundle_i} \subset \RealNumbers^\dimManifold$ for $\BaseToricLineBundle$, and we denote the $\Torus^\BaseToricLineBundle$-invariant subsets corresponding to its facets $\Facet^\BaseToricLineBundle_i$ by $\Divisor^\BaseToricLineBundle_i$.

    The complex line bundle $\ProjectionToricLineBundle : \ToricLineBundle \to \BaseToricLineBundle$ is \define{negative} if $\FirstChernClass(\ToricLineBundle) = - \NegativeLineBundleNumber [\SymplecticForm_\BaseToricLineBundle]$ holds for some positive $\NegativeLineBundleNumber > 0$.
    There is a symplectic form $\SymplecticForm_\ToricLineBundle$ on such $\ToricLineBundle$ for which $(\ToricLineBundle, \SymplecticForm_\ToricLineBundle)$ has a convex structure $(\ConvexCoordMap, \SphereBundle \ToricLineBundle)$ \cite[Section~7]{oancea_fibered_2008, ritter_floer_2014}.
    Here, $\SphereBundle \ToricLineBundle$ is the sphere bundle.
    Moreover, $(\ToricLineBundle, \SymplecticForm_\ToricLineBundle)$ is monotone if and only if $\NegativeLineBundleNumber < \IndexOfMonotonicity_\BaseToricLineBundle$ holds, and in this case the monotonicity constant of $\ToricLineBundle$ is $\IndexOfMonotonicity_\ToricLineBundle = \IndexOfMonotonicity_\BaseToricLineBundle - \NegativeLineBundleNumber$ \cite[Section~4.3]{ritter_circle-actions_2016}.
    We assume $0 < \NegativeLineBundleNumber < \IndexOfMonotonicity_\BaseToricLineBundle$.
    
    We follow Ritter's procedure to construct a moment polytope for $\ToricLineBundle$ \cite[Section~7]{ritter_circle-actions_2016}.
    The line bundle $\ToricLineBundle$ is isomorphic to $\LineBundle (\sum_i \DivisorContributionToBundle_i \Divisor^\BaseToricLineBundle_i)$ for integers $\DivisorContributionToBundle_i \in \Integers$.
    Ritter gives a construction to find these integers when $\SymplecticForm_\BaseToricLineBundle$ is primitive\footnote{
        The symplectic form $\SymplecticForm$ on $\Manifold$ is \define{primitive} if $[\SymplecticForm] \in \Cohomology^2 (\Manifold; \Integers)$ is integral and is not a multiple of another class.
    } \cite[Sections~7.4-7.6]{ritter_circle-actions_2016}.
    The moment polytope $\MomentPolytope_\ToricLineBundle$ lies in $\RealNumbers^{\dimManifold + 1} = \RealNumbers^\dimManifold \DirectSum \RealNumbers$.
    With respect to this decomposition, set
        \begin{equation}
            \begin{aligned}
                &\InwardNormalVector^\ToricLineBundle_i = (\InwardNormalVector^\BaseToricLineBundle_i, -\DivisorContributionToBundle_i)
                \hspace{1em}
                &
                &\FacetBoundaryCondition^\ToricLineBundle_i = \FacetBoundaryCondition^\BaseToricLineBundle_i
                \hspace{1em}
                &
                &\text{for $i = 1, \ldots, \numberFacets$}
                \\
                &\InwardNormalVector^\ToricLineBundle_{\numberFacets + 1} = (0, 1)
                &
                &\FacetBoundaryCondition^\ToricLineBundle_{\numberFacets + 1} = 0,
            \end{aligned}
        \end{equation}
    and let $\MomentPolytope_\ToricLineBundle$ be the polytope defined by $\Set{\InnerProduct{\InwardNormalVector^\ToricLineBundle_i}{\pointDualTorusLieAlgebra} \ge \FacetBoundaryCondition^\ToricLineBundle_i} \subset \RealNumbers^{\dimManifold + 1}$.
    The inward normal fan of $\MomentPolytope_\ToricLineBundle$ is not complete because $\ToricLineBundle$ is not closed.
    The first $\numberFacets$ facets of $\MomentPolytope_\ToricLineBundle$ correspond to $\Divisor^\ToricLineBundle_i = \ProjectionToricLineBundle \Inverse \Divisor^\BaseToricLineBundle_i$, while the last corresponds to the zero section $\Divisor^\ToricLineBundle_{\numberFacets + 1} = \BaseToricLineBundle \subset \ToricLineBundle$.
    
    Denote by $\Cocharacter^\ToricLineBundle_i$ the cocharacter of $\Torus^\ToricLineBundle = (\Circle)^{\dimManifold + 1}$ corresponding to the vector $\InwardNormalVector^\ToricLineBundle_i$.
    The $\Circle$-action $\TorusAction^\ToricLineBundle \ComposedWith \Cocharacter^\ToricLineBundle_i$ is not linear in general.
    Instead, its Hamiltonian has the form $\Hamiltonian_i (\ConvexCoordMap (\eltContactManifold, \RadialCoord)) = f_i(\eltContactManifold) \RadialCoord + \text{constant}$ at infinity, where $f_i : \SphereBundle \ToricLineBundle \to \RealNumbers$ is a nonnegative Reeb flow-invariant function.
    Such Hamiltonians satisfy a maximum principle \cite[Appendix~C]{ritter_circle-actions_2016}, and hence our constructions for monotone closed toric manifolds extend to this setup.
    
    \begin{theorem}
    \label{thm:all-presentations-work-for-toric-line-bundles}
        The presentations of cohomology, quantum cohomology, equivariant cohomology and equivariant quantum cohomology from Theorems \ref{thm:cohomology-presentation-toric-compact}, \ref{thm:quantum-cohomology-presentation-closed-toric-manifolds}, \ref{thm:equivariant-cohomology-presentation-closed-toric-manifold} and \ref{thm:equivariant-quantum-cohomology-presentation-closed-toric-manifold} hold for $\ToricLineBundle$ and $\MomentPolytope_\ToricLineBundle$.
        Moreover, given a vertex $\vertexMomentPolytope$ of $\MomentPolytope_\ToricLineBundle$, the methods of \autoref{sec:connection-computation-on-closed-toric-manifold} apply to compute $\Connection^{\FixedPoint_\vertexMomentPolytope}_i$ for $i \notin \NeighbouringFacets(\vertexMomentPolytope)$ and $\ShiftOperator^{\FixedPoint_\vertexMomentPolytope}_i$ for all $i$.
    \end{theorem}
    
    The only exception is that we can only compute $\ShiftOperator^{\FixedPoint_\vertexMomentPolytope}_\Cocharacter$ for those cocharacters $\Cocharacter \in \LatticeCocharacters{\Torus^\ToricLineBundle}$ which are nonnegative sums of the cocharacters $\Cocharacter_i$.
    Such $\Cocharacter$ lie in a submonoid of $\NonnegativeLatticeCocharacters{\TorusAction^\ToricLineBundle}{\Torus^\ToricLineBundle}$ in general.
    We cannot define the shift operator for all characters because the shift operator is not defined for characters with a negative slope.
    
    The cocharacter $\Cocharacter = \Cocharacter_{\numberFacets + 1}$ induces the fibre-wise rotation action about the zero section.
    The Hamiltonian of this action is linear of positive slope \cite[Section~7.6]{ritter_floer_2014}, so the direct limit of the sequence
        \begin{equation}
        \label{eqn:direct-limit-for-symplectic-cohomology-quantum-seidel-toric-line-bundle}
            \begin{tikzcd}[column sep=large]
                \QuantumCohomology^{\ArbitraryIndex} (\ToricLineBundle) \arrow[r, "{\QuantumSeidel(\Cocharacter, \FixedPoint_\vertexMomentPolytope)}"] 
                &
                \QuantumCohomology^{\ArbitraryIndex +  |\Cocharacter, \FixedPoint_\vertexMomentPolytope|} (\ToricLineBundle) \arrow[r, "{\QuantumSeidel(\Cocharacter, \FixedPoint_\vertexMomentPolytope)}"] 
                &
                \QuantumCohomology^{\ArbitraryIndex + 2 |\Cocharacter, \FixedPoint_\vertexMomentPolytope|} (\ToricLineBundle) \arrow[r, "{\QuantumSeidel(\Cocharacter, \FixedPoint_\vertexMomentPolytope)}"] 
                & 
                \cdots.
            \end{tikzcd}
        \end{equation}
    is naturally isomorphic to the symplectic cohomology  $\SymplecticCohomology^\ArbitraryIndex (\ToricLineBundle)$ \cite[Theorem~22]{ritter_floer_2014}.
    All the modules in this sequence are $\QuantumCohomology^{\ArbitraryIndex} (\ToricLineBundle)$ and the maps $\QuantumSeidel(\Cocharacter, \FixedPoint_\vertexMomentPolytope)$ are module maps given by multiplication.
    Therefore, the direct limit of the sequence may be characterised as
        \begin{equation}
        \label{eqn:symplectic-cohomology-as-quotient-of-quantum-cohomology}
            \SymplecticCohomology^\ArbitraryIndex (\ToricLineBundle) = \frac{\QuantumCohomology^{\ArbitraryIndex} (\ToricLineBundle)}{\ker (\QuantumSeidel(\Cocharacter, \FixedPoint_\vertexMomentPolytope) ^\dimManifold)}
        \end{equation}
    because the sequence $\QuantumSeidel(\Cocharacter, \FixedPoint_\vertexMomentPolytope)^p (\QuantumCohomology^{\ArbitraryIndex} (\ToricLineBundle))$ stabilizes by $p = \dimManifold$ \cite[Theorem~1]{ritter_floer_2014}.
    Combining \eqref{eqn:symplectic-cohomology-as-quotient-of-quantum-cohomology} with  \eqref{eqn:presentation-for-quantum-cohomology-of-closed-toric-manifold} yields the following powerful result.
    
    \begin{theorem}
        [{Symplectic cohomology presentation \cite[Theorem~1.5(2)]{ritter_circle-actions_2016}}]
    \label{thm:symplectic-cohomology-presentation-toric-manifolds}
        The map $\ToricGenerator_i \mapsto \CanonicalMap^\ArbitraryIndex (\Divisor_i \Dual)$ induces a ring isomorphism
        \begin{equation}
        \label{eqn:presentation-for-symplectic-cohomology-of-toric-manifold}
            {\NovikovRing [\ToricGenerator_1^{\pm 1}, \ldots, \ToricGenerator_{\numberFacets + 1}^{\pm 1}]} / {(\LinearRelationsIdeal + \QuantumStanleyReisnerIdeal)} \overset{\cong}{\to} \SymplecticCohomology^\ArbitraryIndex (\ToricLineBundle),
        \end{equation}
        where $\CanonicalMap^\ArbitraryIndex : \QuantumCohomology^{\ArbitraryIndex} (\ToricLineBundle) \to \SymplecticCohomology^\ArbitraryIndex (\ToricLineBundle)$ is the canonical map.
    \end{theorem}
    
    Similarly, the direct limit of the sequence
        \begin{equation}
        \label{eqn:direct-limit-for-equivariant-symplectic-cohomology-quantum-seidel-toric-line-bundle}
            \begin{tikzcd}[column sep=huge]
                \QuantumCohomology^{\ArbitraryIndex} _{\ExtendedTorus} (\ToricLineBundle, \Extended{\TorusAction}) \arrow[r, "{\QuantumSeidel_{\ExtendedTorus} (\Cocharacter, \FixedPoint_\vertexMomentPolytope)}"] &
                \QuantumCohomology^{\ArbitraryIndex +  |\Cocharacter, \FixedPoint_\vertexMomentPolytope|} _{\ExtendedTorus} (\ToricLineBundle, \Extended{\Cocharacter} \PullBack \Extended{\TorusAction}) \arrow[r, "{\QuantumSeidel_{\ExtendedTorus} (\Cocharacter, \FixedPoint_\vertexMomentPolytope)}"]
                & \cdots,
            \end{tikzcd}
        \end{equation}
    is the equivariant symplectic cohomology $\SymplecticCohomology^\ArbitraryIndex_{\ExtendedTorus} (\ToricLineBundle, \Extended{\TorusAction})$ \cite[Equation~(7.4)]{liebenschutz-jones_intertwining_2020}.
    Since the map $(\ClassifyingSpace \Extended{\Cocharacter})\PullBack$ is an isomorphism, the sequence \eqref{eqn:direct-limit-for-equivariant-symplectic-cohomology-quantum-seidel-toric-line-bundle} is naturally isomorphic to
        \begin{equation}
        \label{eqn:direct-limit-for-equivariant-symplectic-cohomology-shift-operator-toric-line-bundle}
            \begin{tikzcd}%[column sep=small]
                \QuantumCohomology^{\ArbitraryIndex} _{\ExtendedTorus} (\ToricLineBundle, \Extended{\TorusAction}) \arrow[r, "\ShiftOperator_\Cocharacter ^{\FixedPoint_\vertexMomentPolytope}"] &
                \QuantumCohomology^{\ArbitraryIndex +  |\Cocharacter, \FixedPoint_\vertexMomentPolytope|} _{\ExtendedTorus} (\ToricLineBundle, \Extended{\TorusAction}) \arrow[r, "\ShiftOperator_\Cocharacter ^{\FixedPoint_\vertexMomentPolytope}"] &
                \QuantumCohomology^{\ArbitraryIndex + 2 |\Cocharacter, \FixedPoint_\vertexMomentPolytope|} _{\ExtendedTorus} (\ToricLineBundle,  \Extended{\TorusAction}) \arrow[r, "\ShiftOperator_\Cocharacter ^{\FixedPoint_\vertexMomentPolytope}"]
                & \cdots.
            \end{tikzcd}
        \end{equation}
    This second sequence preserves the difference-differential connection $(\ShiftOperator^{\FixedPoint_\vertexMomentPolytope}, \Connection^{\FixedPoint_\vertexMomentPolytope})$ by \autoref{prop:connection-commutes-with-floer-shift-operator}.
    Unfortunately, both of sequences \eqref{eqn:direct-limit-for-equivariant-symplectic-cohomology-quantum-seidel-toric-line-bundle} and \eqref{eqn:direct-limit-for-equivariant-symplectic-cohomology-shift-operator-toric-line-bundle} do not satisfy the key properties of \eqref{eqn:direct-limit-for-symplectic-cohomology-quantum-seidel-toric-line-bundle} which we used to derive \eqref{eqn:symplectic-cohomology-as-quotient-of-quantum-cohomology}.
    Specifically, the modules are different in \eqref{eqn:direct-limit-for-equivariant-symplectic-cohomology-quantum-seidel-toric-line-bundle}, while the maps in \eqref{eqn:direct-limit-for-equivariant-symplectic-cohomology-shift-operator-toric-line-bundle} are not module maps (they are twisted by $(\ClassifyingSpace \Extended{\Cocharacter})\PullBack$).
    These sequences may nonetheless be used to understand $\SymplecticCohomology^\ArbitraryIndex_{\ExtendedTorus} (\ToricLineBundle, \Extended{\TorusAction})$, as in \autoref{sec:tautological-line-bundle}.
        
\subsection{Tautological line bundle}
\label{sec:tautological-line-bundle}
    
    The tautological line bundle $\LineBundle_{\Projective^1}(-1) \to \Projective^1$ is a negative line bundle with $\NegativeLineBundleNumber = 1$.
    The base $\Projective^1$ together with the $\Circle$-action 
        \begin{equation}
        \label{eqn:circle-action-on-projective-line-for-line-bundle-example}
            \eltTorus \cdot [\eltComplexProjectiveSpace_0, \eltComplexProjectiveSpace_1] = [\eltComplexProjectiveSpace_0, \ExponentialNumber^{2 \PiNumber \ImaginaryNumber \eltTorus} \eltComplexProjectiveSpace_1]
        \end{equation}
    admits the moment polytope from \autoref{fig:tautological-bundle-base-moment-polytope-and-fan}, defined by $\InwardNormalVector_0 = -1$, $\InwardNormalVector_1 = 1$ and $\FacetBoundaryCondition_0 = -1$, $\FacetBoundaryCondition_1 = 0$.
    
\begin{figure}[ht]
	\begin{center}\begin{tikzpicture}
			\node[inner sep=0] at (0,0) {\includegraphics[width=10 cm]{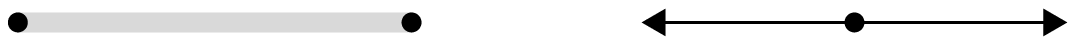}};
% 			\draw[step=1.0,black,thin] (-6,-5) grid (6,5);
			\node at (5.2,0){$\InwardNormalVector_1$};
			\node at (0.7,0){$\InwardNormalVector_0$};
			\node at (-1.15,-0.3){$\vertexMomentPolytope_1 = \Facet_0$};
			\node at (-4.85,-0.3){$\vertexMomentPolytope_0 = \Facet_1$};
		\end{tikzpicture}
	\end{center}
	\caption{The moment polytope and the inner normal fan for $\Projective^1$.}
	\label{fig:tautological-bundle-base-moment-polytope-and-fan}
\end{figure}
    
    The bundle $\LineBundle_{\Projective^1}(-1) \to \Projective^1$ is isomorphic to $\LineBundle (- \Divisor_1 ) \to \Projective^1$ \cite[Example, page~2021]{ritter_circle-actions_2016}.
    Therefore the total space $\LineBundle_{\Projective^1}(-1)$ admits the moment polytope from \autoref{fig:tautological-bundle-moment-polytope-and-fan}, defined by its inward normal vectors $\InwardNormalVector_0 = (-1, 0)$, $\InwardNormalVector_1 = (1, 1)$ and $\InwardNormalVector_2 = (0, 1)$; and $\FacetBoundaryCondition_0 = -1$, $\FacetBoundaryCondition_1 = 0$ and $\FacetBoundaryCondition_2 = 0$.
    The unique primitive set is $\Set{0, 1}$, for which \eqref{eqn:vector-relating-primitive-index-set-with-counterparts-defining-cone} reads $\InwardNormalVector_0 + \InwardNormalVector_1 = 1 \InwardNormalVector_2$, so the quantum Stanley-Reisner ideal is $\QuantumStanleyReisnerIdeal = \langle \ToricGeneratorWithNovikovWeighting_0 \ToricGeneratorWithNovikovWeighting_1 - \ToricGeneratorWithNovikovWeighting_2 \rangle$.
    The Novikov ring is the polynomial ring $\NovikovRing = \Integers[\NovVariable]$ with $\NovVariable$ in degree 2.
    \autoref{thm:all-presentations-work-for-toric-line-bundles} yields the presentation
        \begin{equation}
            \QuantumCohomology^\ArbitraryIndex_{\ExtendedTorus} (\LineBundle_{\Projective^1}(-1)) \cong \frac{\NovikovRing \GradedCompletedTensorProduct \Integers [\ToricGenerator_0, \ToricGenerator_1, \ToricGenerator_2, \formalAdditionalCircle]}{\langle \ToricGenerator_0 \ToricGenerator_1 - \NovVariable \ToricGenerator_2 \rangle},
        \end{equation}
    where the $\NovikovRing \GradedCompletedTensorProduct \Integers [\formalAdditionalCircle, \formalTorus_1, \formalTorus_2]$-module structure is induced by the map
        \begin{equation}
        \label{eqn:equivariant-variables-assignment-presentation-tautological-line-bundle}
            \begin{cases}
                \formalTorus_1 \mapsto \ToricGenerator_0 - \ToricGenerator_1 \\
                \formalTorus_2 \mapsto - \ToricGenerator_1 - \ToricGenerator_2.
            \end{cases}
        \end{equation}
    
\begin{figure}[ht]
	\begin{center}
		\begin{tikzpicture}
			\node[inner sep=0] at (0,0) {\includegraphics[width=10 cm]{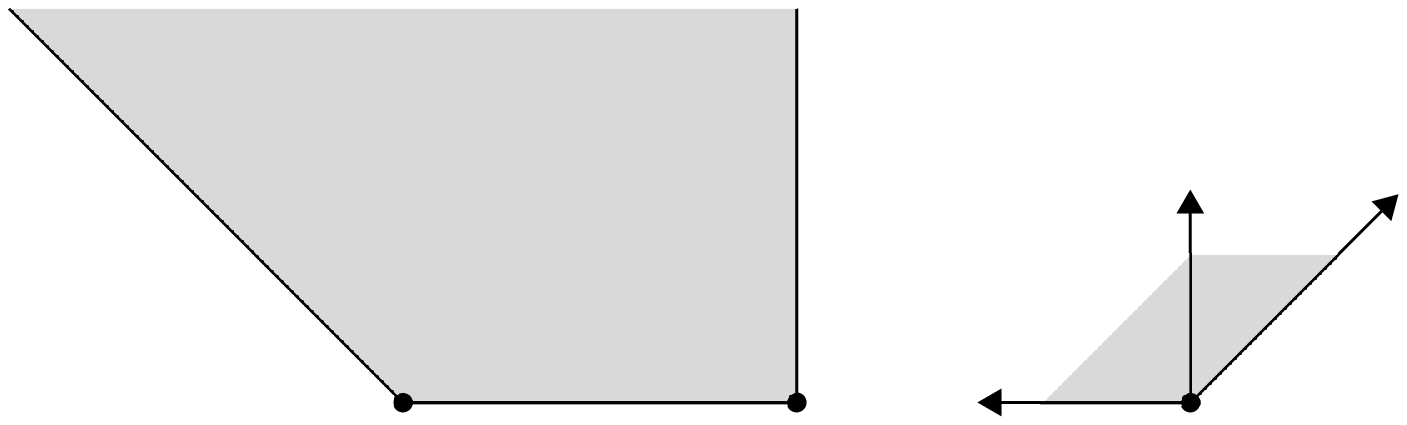}};
% 			\draw[step=1.0,black,thin] (-6,-5) grid (6,5);
% 			\node at (0,0){0};
			\node at (-4.0,0){$\Facet_1$};
			\node at (-0.7,-1.6){$\Facet_2$};
			\node at (1,0){$\Facet_0$};
			\node at (5.0,0.3){$\InwardNormalVector_1$};
			\node at (3.5,0.3){$\InwardNormalVector_2$};
			\node at (1.8,-1.4){$\InwardNormalVector_0$};
			\node at (0.7,-1.6){$\vertexMomentPolytope_1$};
			\node at (-2.1,-1.6){$\vertexMomentPolytope_0$};
		\end{tikzpicture}
	\end{center}
	\caption{The moment polytope and the inner normal fan for $\LineBundle_{\Projective^1} (-1)$.}
	\label{fig:tautological-bundle-moment-polytope-and-fan}
\end{figure}
    
    Fix the vertex $\vertexMomentPolytope_0 = (0, 0)$ with $\FixedPoint_{\vertexMomentPolytope_0} = [1, 0]$.
    The differential connection is $\Connection^{\vertexMomentPolytope_0}_0 = \formalAdditionalCircle \NovVariable \dbyd{\NovVariable} + \ToricGenerator_0$.
    Using the same approach as \autoref{sec:projective-plane-computations}, we find that the shift operator $\ShiftOperator^{\vertexMomentPolytope_0} _2$ satisfies
        \begin{align}
            \ShiftOperator^{\vertexMomentPolytope_0} _2 (1) &= \ToricGenerator_2
            \\
            \ShiftOperator^{\vertexMomentPolytope_0} _2 (\ToricGenerator_0) &= \ToricGenerator_0 \ToricGenerator_2
            \\
            \ShiftOperator^{\vertexMomentPolytope_0} _2 (\ToricGenerator_1) &= \ToricGenerator_1 \ToricGenerator_2
            \\
            \label{eqn:shift-operator-of-divisor-2-in-tautological-line-bundle}
            \ShiftOperator^{\vertexMomentPolytope_0} _2 (\ToricGenerator_2) &= (\ToricGenerator_2 - \formalAdditionalCircle) \ToricGenerator_2.
        \end{align}
    
    Set $A = \NovikovRing \GradedCompletedTensorProduct \Integers [\formalAdditionalCircle, \formalTorus_1, \formalTorus_2]$.
    Consider the sequence \eqref{eqn:direct-limit-for-equivariant-symplectic-cohomology-quantum-seidel-toric-line-bundle}.
    The modules in this sequence are $\QuantumCohomology^\ArbitraryIndex_{\ExtendedTorus} (\LineBundle_{\Projective^1}(-1), (\Extended{\Cocharacter}_2^r) \PullBack \Extended{\TorusAction})$ for $r \ge 0$.
    They are all naturally isomorphic to $A \langle 1, \ToricGenerator_2 \rangle$, the free $A$-module with basis $\langle 1, \ToricGenerator_2 \rangle$.
    The $r$-th map in the sequence,
        \begin{equation}
        \label{eqn:rth-map-in-sequence-of-equivariant-quantum-seidel-maps}
            \QuantumSeidel_{\ExtendedTorus} (\Cocharacter_2, \FixedPoint_{\vertexMomentPolytope_0}) : \QuantumCohomology^\ArbitraryIndex_{\ExtendedTorus} (\LineBundle_{\Projective^1}(-1), (\Extended{\Cocharacter}_2^{r - 1}) \PullBack \Extended{\TorusAction}) \to \QuantumCohomology^\ArbitraryIndex_{\ExtendedTorus} (\LineBundle_{\Projective^1}(-1), (\Extended{\Cocharacter}_2^r) \PullBack \Extended{\TorusAction}),
        \end{equation}
    is equal to the composition $((\ClassifyingSpace \Extended{\Cocharacter}_2) \PullBack)^{-r} \ComposedWith \ShiftOperator^{\vertexMomentPolytope_0}_2 \ComposedWith ((\ClassifyingSpace \Extended{\Cocharacter}_2) \PullBack)^{r - 1}$.
    By rearranging \eqref{eqn:shift-operator-of-divisor-2-in-tautological-line-bundle}, we find that \eqref{eqn:rth-map-in-sequence-of-equivariant-quantum-seidel-maps} is the $A$-module map given by
        \begin{equation}
        \label{eqn:equivariant-quantum-seidel-maps-written-in-basis}
            \left\{
            \begin{aligned}
                1 &\mapsto \ToricGenerator_2 \\
                \ToricGenerator_2 &\mapsto (\formalTorus_2 - r \formalAdditionalCircle)(r \formalAdditionalCircle + \formalTorus_1 - \formalTorus_2) + (\NovVariable + \formalTorus_1 - 2 \formalTorus_2 + (2r - 1) \formalAdditionalCircle) \ToricGenerator_2.
            \end{aligned}
            \right.
        \end{equation}
    with respect to the basis $\langle 1, \ToricGenerator_2 \rangle$.
    Its determinant is
        \begin{equation}
            d_r = (\formalTorus_2 - r \formalAdditionalCircle)(\formalTorus_2 - \formalTorus_1 - r \formalAdditionalCircle).
        \end{equation}
        
    The maps \eqref{eqn:equivariant-quantum-seidel-maps-written-in-basis} are isomorphisms over the fraction field $\FractionField(A)$, and hence we have an isomorphism of $A$-modules
        \begin{equation}
           \FractionField(A) \Tensor_{A} \SymplecticCohomology^\ArbitraryIndex_{\ExtendedTorus} (\LineBundle_{\Projective^1}(-1)) \cong \FractionField(A) \langle 1, \ToricGenerator_2 \rangle.
        \end{equation}
    The induced inclusion $\SymplecticCohomology^\ArbitraryIndex_{\ExtendedTorus} (\LineBundle_{\Projective^1}(-1)) \Inclusion \FractionField(A) \langle 1, \ToricGenerator_2 \rangle$ is injective, and thus every element of $\SymplecticCohomology^\ArbitraryIndex_{\ExtendedTorus} (\LineBundle_{\Projective^1}(-1))$ is of the form\footnote{
        Consider by analogy the sequence $\Integers \overset{\cdot d_1}{\to} \Integers \overset{\cdot d_2}{\to} \Integers \overset{\cdot d_3}{\to} \cdots$ for nonzero integers $d_r \in \Integers \setminus 0$.
        This sequence is isomorphic to the sequence of inclusions $\Integers \Inclusion \frac{1}{d_1} \Integers \Inclusion \frac{1}{d_1 d_2} \Integers \Inclusion \cdots$.
        The direct limit is naturally a $\Integers$-submodule of $\RationalNumbers = \FractionField(\Integers)$ and every element of the direct limit is a quotient of an element $\alpha \in \Integers$ by a product $d_1 \cdot d_2 \cdots d_r$. 
    }
        \begin{equation}
        \label{eqn:fraction-form-for-elements-of-equivariant-symplectic-cohomology-tautological-line-bundle}
            \frac{\alpha}{d_1 \cdots d_r}
        \end{equation}
    for some $r \ge 0$ and some $\alpha \in A \langle 1, \ToricGenerator_2 \rangle$.
    Note that $\SymplecticCohomology^\ArbitraryIndex_{\ExtendedTorus} (\LineBundle_{\Projective^1}(-1))$ does not surject onto the set of elements of the form \eqref{eqn:fraction-form-for-elements-of-equivariant-symplectic-cohomology-tautological-line-bundle}.
    
    The non-equivariant quantum Seidel maps in the analogous sequence \eqref{eqn:direct-limit-for-symplectic-cohomology-quantum-seidel-toric-line-bundle} are all the multiplication-by-$\ToricGenerator_2$ map, giving
        \begin{equation}
        \label{eqn:symplectic-cohomology-presentation-tautological-line-bundle-direct-proof}
            \SymplecticCohomology^\ArbitraryIndex (\LineBundle_{\Projective^1}(-1)) \cong \NovikovRing [\ToricGenerator_0, \ToricGenerator_1, \ToricGenerator_2^{\pm 1}] / (\LinearRelationsIdeal + \StanleyReisnerIdeal).
        \end{equation}
    The linear relations are $\ToricGenerator_1 - \ToricGenerator_0$ and $\ToricGenerator_1 + \ToricGenerator_2$, so \eqref{eqn:symplectic-cohomology-presentation-tautological-line-bundle-direct-proof} is equivalent to \eqref{eqn:presentation-for-symplectic-cohomology-of-toric-manifold}.
    
    \begin{remark}
        [Reeb orbits]
        The contact manifold $\SphereBundle \LineBundle_{\Projective^1}(-1)$ is isomorphic to $\HighDimensionalSphere{3}$, which we write as the unit sphere in $\ComplexNumbers^2$.
        The induced $\Torus$-action on $\HighDimensionalSphere{3}$ is given by $(\eltTorus_1, \eltTorus_2) \cdot (\eltComplexProjectiveSpace_0, \eltComplexProjectiveSpace_1) = (\ExponentialNumber^{2 \PiNumber \ImaginaryNumber (\eltTorus_2 - \eltTorus_1)} \eltComplexProjectiveSpace_0, \ExponentialNumber^{2 \PiNumber \ImaginaryNumber \eltTorus_2} \eltComplexProjectiveSpace_1)$, where the $\eltTorus_1$ action comes from \eqref{eqn:circle-action-on-projective-line-for-line-bundle-example} and the $\eltTorus_2$ action rotates the fibres.
        The flow along the Reeb vector field is given by $\eltCircle \cdot (\eltComplexProjectiveSpace_0, \eltComplexProjectiveSpace_1) = (\ExponentialNumber^{2 \PiNumber \ImaginaryNumber \eltCircle} \eltComplexProjectiveSpace_0, \ExponentialNumber^{2 \PiNumber \ImaginaryNumber \eltCircle} \eltComplexProjectiveSpace_1)$.
        Let $\gamma_r : \Circle \to \HighDimensionalSphere{3}$ be the Reeb orbit $\eltCircle \mapsto (\ExponentialNumber^{2 \PiNumber \ImaginaryNumber r \eltCircle}, 0)$ for $r \in \Integers_{> 0}$, and consider the $\ExtendedTorus$-orbit $\ExtendedTorus \gamma_r$ in $\ContractibleLoopSpace{\LineBundle_{\Projective^1}(-1)}$.
        This $\ExtendedTorus$-orbit is isomorphic to $\Circle$ because the Reeb orbits are determined by their value at 0.
        Therefore the map $\ExtendedTorus \to \ExtendedTorus \gamma_r$ is a map $\ExtendedTorus \to \Circle$.
        As a character, this map $\ExtendedTorus \to \Circle$ corresponds to $\formalTorus_2 - \formalTorus_1 - r \formalAdditionalCircle$.
        Repeating this with the Reeb orbits $\eltCircle \mapsto (0, \ExponentialNumber^{2 \PiNumber \ImaginaryNumber r \eltCircle})$ yields $\formalTorus_2 - r \formalAdditionalCircle$.
    \end{remark}
    
    \begin{conjecture}
        [Significance of the denominators]
    \label{conjecture:significance-of-denominators}
        In general, the denominators in \eqref{eqn:fraction-form-for-elements-of-equivariant-symplectic-cohomology-tautological-line-bundle} will be the product of the characters corresponding to those Reeb orbits whose $\ExtendedTorus$-orbit is isomorphic to $\Circle$.
    \end{conjecture}

\appendix
\section{Topological proofs}
\label{app:topological-proofs}

    Each of these lemmas is clear if $\Cocharacter = \Identity_\Manifold$, so without loss of generality we assume that $\Manifold$ is simply connected.

    \begin{lemma}
    \label{lem:short-exact-sequence-on-degree-2-homology-of-clutching-bundle}
        The sequence
        \begin{equation}
        \begin{tikzcd}
            0 \arrow[r] & \arrow[r, "\Pole^-"] \Homology_2 (\Manifold) & \arrow[r] \Homology_2 (\ClutchingBundle{\Cocharacter}) & \arrow[r] \Homology_2 (\Sphere) \cong \Integers & 0.
        \end{tikzcd}
        \end{equation}
        is a short exact sequence.
    \end{lemma}
    
    \begin{proof}
        The sphere $\Sphere$ is the union of $\ComplexNumbers^+ = \Sphere \setminus \Set{\Pole^-}$ and $\ComplexNumbers^- = \Sphere \setminus \Set{\Pole^+}$, and the intersection of these sets is the cylinder $\RealNumbers \times \Circle$.
        We use the Mayer-Vietoris sequence of the bundle $\ClutchingProjection : \ClutchingBundle{\Cocharacter} \to \Sphere$ associated to this decomposition of the sphere.
        The relevant part of the Mayer-Vietoris long exact sequence is
            \begin{equation}
            \label{eqn:mayer-vietoris}
                \begin{tikzcd}[column sep = tiny]
                    \Homology_2 (\ClutchingProjection \Inverse (\RealNumbers \times \Circle)) \arrow[r]
                    \arrow[d, phantom, ""{coordinate, name=top}]
                    & \Homology_2 (\ClutchingProjection \Inverse (\ComplexNumbers^-)) \DirectSum \Homology_2 (\ClutchingProjection \Inverse (\ComplexNumbers^+))  \arrow[r]
                    & \Homology_2 (\ClutchingProjection \Inverse (\Sphere))
                    \arrow[dll,
                        rounded corners,
                        to path={ -- ([xshift=2ex]\tikztostart.east)
                        |- (top) [near end]\tikztonodes
                        -| ([xshift=-2ex]\tikztotarget.west)
                        -- (\tikztotarget)}]
                        &
                    \\ \Homology_1 (\ClutchingProjection \Inverse (\RealNumbers \times \Circle)).
                    & 
                    &
                \end{tikzcd}
            \end{equation}
            
        We have a homeomorphism $\ClutchingProjection \Inverse (\RealNumbers \times \Circle) \cong \Manifold \times \Circle \times \RealNumbers$ from the trivialisation of $\ClutchingBundle{\Cocharacter}$ over the pole $\Pole^+$.
        This homeomorphism yields $\Homology_2 (\ClutchingProjection \Inverse (\RealNumbers \times \Circle)) \cong \Homology_2 (\Manifold) \DirectSum \Homology_1 (\Manifold) \Tensor \Homology_1 (\Circle)$ and $\Homology_1 (\ClutchingProjection \Inverse (\RealNumbers \times \Circle)) \cong \Homology_1 (\Manifold) \DirectSum \Homology_1 (\Circle)$ by the Kunneth theorem.
        Since $\Manifold$ is simply connected, we have $\Homology_1 (\Manifold) = 0$, which gives $\Homology_2 (\ClutchingProjection \Inverse (\RealNumbers \times \Circle)) \cong \Homology_2 (\Manifold)$ and $\Homology_1 (\ClutchingProjection \Inverse (\RealNumbers \times \Circle)) \cong \Homology_1 (\Circle)$.
        Moreover, the degree-2 isomorphism is independent of the choice of trivialising the bundle over the different poles $\Pole^\pm$ (a chain homotopy between the maps $\Homology_2 (\Manifold) \to \Homology_2 (\ClutchingProjection \Inverse (\RealNumbers \times \Circle))$ is found by considering the inclusions $\Manifold \to \ClutchingBundle{\Cocharacter}$ to the fibres at points $(s, 0) \in \RealNumbers \times \Circle \subset \Sphere$).
        
        Contracting each of $\ComplexNumbers^\pm$ to the pole $\Pole^\pm$ converts \eqref{eqn:mayer-vietoris} into
            \begin{equation}
            \label{eqn:simplified-mayer-vietoris-on-bundle}
                \begin{tikzcd}[row sep=0]
                    \Homology_2 (\Manifold) \arrow[r]
                    & \Homology_2 (\Manifold) \DirectSum \Homology_2 (\Manifold) \arrow[r]
                    & \Homology_2 (\ClutchingBundle{\Cocharacter}) \arrow[r]
                    & \Homology_1 (\Circle)
                    \\
                    x \arrow[r, mapsto]
                    & (x, x).
                \end{tikzcd}
            \end{equation}
        Therefore, one of the copies of $\Homology_2 (\Manifold)$ cancels with the first term of the sequence.
        
        Finally, we have a commutative diagram
            \begin{equation}
            \label{eqn:commutative-diagram-relating-mayer-vietoris-of-bundle-to-base}
                \begin{tikzcd}
                    \Homology_2 (\ClutchingBundle{\Cocharacter}) \arrow[r] \arrow[d, "\ClutchingProjection \PushForward"]
                    & \Homology_1 (\Circle) \arrow[d, equal]
                    \\
                    \Homology_2 (\Sphere) \arrow[r, "\cong"]
                    & \Homology_1 (\Circle)
                \end{tikzcd}
            \end{equation}
        where the isomorphism $\Homology_2 (\Sphere) \to \Homology_1 (\Circle)$ comes from the Mayer-Vietoris sequence for the sphere $\Sphere$.
        The lemma follows from combining \eqref{eqn:simplified-mayer-vietoris-on-bundle} and \eqref{eqn:commutative-diagram-relating-mayer-vietoris-of-bundle-to-base}.
    \end{proof}
    
    \begin{lemma}
    \label{lem:class-existence-for-quantum-flatness-from-intertwining}
        There is a class $\beta \in \Cohomology^2 _{\ExtendedTorus} (\ClutchingBundle{\Cocharacter})$ which satisfies $\beta^+ = \DegreeTwoCoclass^\FixedPoint$, $(\ClassifyingSpace \Extended{\Cocharacter}) \PullBack (\beta^-) = \DegreeTwoCoclass^\FixedPoint$ and $\beta (\DegreeTwoClass^\FixedPoint) = \DegreeTwoCoclass (\DegreeTwoClass)$ for $\DegreeTwoClass \in \Homology_2 (\Manifold)$.
    \end{lemma}
    
    \begin{proof}
    
        Write the sphere $\Sphere$ as the union of $\ComplexNumbers^+ = \Sphere \setminus \Set{\Pole^-}$ and $\ComplexNumbers^- = \Sphere \setminus \Set{\Pole^+}$, as in \autoref{lem:short-exact-sequence-on-degree-2-homology-of-clutching-bundle}.
        We will apply the ($\ExtendedTorus$-equivariant cohomology) Mayer-Vietoris sequence to the bundle $\ClutchingProjection : \ClutchingBundle{\Cocharacter} \to \Sphere$ using this decomposition of the sphere.
        The relevant part of this Mayer-Vietoris long exact sequence is
            \begin{equation}
            \label{eqn:mayer-vietoris-equivariant}
            \begin{tikzcd}[column sep = tiny]
                &
                \arrow[d, phantom, ""{coordinate, name=top}]
                & \Cohomology^1 _{\ExtendedTorus} (\ClutchingProjection \Inverse (\RealNumbers \times \Circle))
                \arrow[dll,
                    rounded corners,
                    to path={ -- ([xshift=2ex]\tikztostart.east)
                    |- (top) [near end]\tikztonodes
                    -| ([xshift=-2ex]\tikztotarget.west)
                    -- (\tikztotarget)}]
                    &
                \\ \Cohomology^2 _{\ExtendedTorus} (\ClutchingProjection \Inverse (\Sphere))
                \arrow[r]
                & \Cohomology^2 _{\ExtendedTorus} (\ClutchingProjection \Inverse (\ComplexNumbers^-)) \DirectSum \Cohomology^2 _{\ExtendedTorus} (\ClutchingProjection \Inverse (\ComplexNumbers^+))
                \arrow[r]
                & \Cohomology^2 _{\ExtendedTorus} (\ClutchingProjection \Inverse (\RealNumbers \times \Circle)).
            \end{tikzcd}
            \end{equation}
        The second group in the sequence is $\Cohomology^2_{\ExtendedTorus} (\ClutchingBundle{\Cocharacter})$.
        We will show $\Cohomology^1 _{\ExtendedTorus} (\ClutchingProjection \Inverse (\RealNumbers \times \Circle)) = 0$, and then use the rest of the sequence to construct the element $\beta \in \Cohomology^2_{\ExtendedTorus} (\ClutchingBundle{\Cocharacter})$.
            
        The inclusion $\Manifold \to \ClutchingProjection \Inverse (\RealNumbers \times \Circle)$ along the equator at $\eltCircle = 0$ is independent of the choice of trivialisation between $\Hemisphere^\pm \times \Manifold$.
        As per \eqref{eqn:functorial-map-on-equivariant-cohomology-general}, this inclusion yields a natural map $\Cohomology^\ArbitraryIndex _{\ExtendedTorus} (\ClutchingProjection \Inverse (\RealNumbers \times \Circle)) \to \Cohomology^\ArbitraryIndex _\Torus (\Manifold)$, with the corresponding group map $\Torus \to \ExtendedTorus$.
        The inclusion fits into a fibre bundle $\Manifold \to \ClutchingProjection \Inverse (\RealNumbers \times \Circle) \to \RealNumbers \times \Circle$ which pairs with the groups $\Torus \to \ExtendedTorus \to \AdditionalCircle$.
        That is, $\Torus$ acts on $\Manifold$, $\ExtendedTorus$ acts on $\ClutchingProjection \Inverse (\RealNumbers \times \Circle)$ and $\AdditionalCircle$ acts on $\RealNumbers \times \Circle$, and the maps satisfy \eqref{eqn:equivariant-function-definition} with respect to these actions.
        The $\AdditionalCircle$-action on $\RealNumbers \times \Circle$ is free, which gives $\Cohomology^\ArbitraryIndex _{\AdditionalCircle} (\RealNumbers \times \Circle) = \Integers$.
        Therefore the map $\Cohomology^\ArbitraryIndex _{\ExtendedTorus} (\ClutchingProjection \Inverse (\RealNumbers \times \Circle)) \to \Cohomology^\ArbitraryIndex _\Torus (\Manifold)$ is a canonical isomorphism (because the Leray-Serre spectral sequence degenerates on the first page).
        But $\Cohomology^1 _\Torus (\Manifold) = 0$ vanishes when $\Manifold$ is simply connected ($\HomotopyGroup_1 (\UniversalSpace \Torus \times_{\Torus} \Manifold) = 0$ follows from the long exact sequence of homotopy groups of the Borel quotient $\UniversalSpace \Torus \times_{\Torus} \Manifold$).
        Therefore $\Cohomology^1 _{\ExtendedTorus} (\ClutchingProjection \Inverse (\RealNumbers \times \Circle)) = 0$.
        
        From the above argument, we also get a canonical isomorphism $\Cohomology^2 _{\ExtendedTorus} (\ClutchingProjection \Inverse (\RealNumbers \times \Circle)) \to \Cohomology^2 _\Torus (\Manifold, \TorusAction)$.
        As per the definition of $\DegreeTwoCoclass^\FixedPoint$ in \autoref{sec:lifting-ordinary-to-equivariant-degree-two-coclasses}, define $\DegreeTwoCoclass^\FixedPoint_\Torus \in \Cohomology^2 _\Torus (\Manifold, \TorusAction)$ via the analogous split short exact sequence to \eqref{eqn:split-short-exact-sequence-borel-quotient}.
            
        Just as $\ComplexNumbers$ is equivariantly contractible to $0$, the manifold $\ClutchingProjection \Inverse (\ComplexNumbers^+)$ will  $\ExtendedTorus$-equiv\-ari\-ant\-ly retract onto the fibre $\ClutchingProjection \Inverse (\Pole^+) = \Manifold$ with action $\Extended{\TorusAction}$.
        We have the following commutative diagram (the splitting maps commute only in the obvious way).
            \begin{equation}
            \label{eqn:split-exact-sequences-for-torus-into-extended-torus}
            \begin{tikzcd}
                \Cohomology^2 (\ClassifyingSpace \AdditionalCircle)
                \arrow[r]
                \arrow[dd, equal]
                & \Cohomology^2 (\ClassifyingSpace \ExtendedTorus)
                \arrow[dd]
                \arrow[rrr]
                &[-25pt] & &[-25pt] \Cohomology^2 (\ClassifyingSpace \Torus)
                \arrow[dd]
                \\[-25pt]
                && 0
                \arrow[lu, "\ni" description]
                \arrow[dd, maps to]
                \arrow[r, equal]
                & 0
                \arrow[dd, maps to]
                \arrow[ru, "\in" description]
                \\
                \Cohomology^2 (\ClassifyingSpace \AdditionalCircle)
                \arrow[ddd]
                \arrow[r]
                & \Cohomology^2 _{\ExtendedTorus} (\Manifold, \Extended{\TorusAction})
                \arrow[ddd]
                \arrow[rrr]
                \arrow[uu, bend right, "\FixedPoint \PullBack"']
                &&& \Cohomology^2 _\Torus (\Manifold, \TorusAction)
                \arrow[ddd]
                \arrow[uu, bend left, "\FixedPoint \PullBack"]
                \\[-25pt]
                && \DegreeTwoCoclass^\FixedPoint
                \arrow[lu, "\ni" description]
                \arrow[r, maps to]
                \arrow[d, maps to]
                \arrow[uu, bend right, maps to]
                & \DegreeTwoCoclass^\FixedPoint _\Torus
                \arrow[d, maps to]
                \arrow[ur, "\in" description]
                \arrow[uu, bend left, maps to]
                \\
                && \DegreeTwoCoclass
                \arrow[r, equal]
                \arrow[u, bend right, dashed, maps to]
                \arrow[ld, "\ni" description]
                & \DegreeTwoCoclass
                \arrow[u, bend left, dashed, maps to]
                \arrow[dr, "\in" description]
                \\[-25pt]
                0
                \arrow[r]
                & \Cohomology^2 (\Manifold)
                \arrow[rrr, equal]
                \arrow[uuu, bend right, dashed]
                &&& \Cohomology^2 (\Manifold)
                \arrow[uuu, bend left, dashed]
            \end{tikzcd}
            \end{equation}
        The classes $\DegreeTwoCoclass$, $\DegreeTwoCoclass^\FixedPoint$ and $\DegreeTwoCoclass^\FixedPoint _\Torus$ are drawn in the diagram.
        The vertical short exact sequences are precisely the sequences which define $\DegreeTwoCoclass^\FixedPoint$ and $\DegreeTwoCoclass^\FixedPoint _\Torus$.
        From the diagram, we \emph{deduce} $\DegreeTwoCoclass^\FixedPoint \mapsto \DegreeTwoCoclass^\FixedPoint _{\Torus}$ on the second row.
        
        Similarly, $\ClutchingProjection \Inverse (\ComplexNumbers^-)$ contracts onto the fibre $\ClutchingProjection \Inverse (\Pole^-) = \Manifold$ with action $\Cocharacter \cdot \Extended{\TorusAction}$.
        The same diagram as \eqref{eqn:split-exact-sequences-for-torus-into-extended-torus} holds for $\Cohomology^2 _{\ExtendedTorus} (\Manifold, \Cocharacter \cdot \Extended{\TorusAction})$ giving $((\ClassifyingSpace \Extended{\Cocharacter}) \PullBack )\Inverse (\DegreeTwoCoclass^\FixedPoint) \mapsto \DegreeTwoCoclass^\FixedPoint _\Torus$.
        
        Thus the Mayer-Vietoris sequence \eqref{eqn:mayer-vietoris-equivariant} becomes
            \begin{equation}
            \label{eqn:mayer-vietoris-equivariant-manipulated}
            \begin{tikzcd}
                0
                \arrow[r]
                & \Cohomology^2 _{\ExtendedTorus} (\ClutchingBundle{\Cocharacter})
                \arrow[r]
                & \Cohomology^2 _{\ExtendedTorus} (\Manifold, \Cocharacter \cdot \Extended{\TorusAction}) \DirectSum \Cohomology^2 _{\ExtendedTorus} (\Manifold, \Extended{\TorusAction})
                \arrow[r]
                & \Cohomology^2 _\Torus (\Manifold, \TorusAction)
                \\[-25pt]
                && ((\ClassifyingSpace \Extended{\Cocharacter}) \PullBack) \Inverse (\DegreeTwoCoclass^\FixedPoint) \DirectSum 0
                \arrow[r, maps to]
                & \DegreeTwoCoclass^\FixedPoint _\Torus
                \\[-25pt]
                && 0 \DirectSum \DegreeTwoCoclass^\FixedPoint
                \arrow[r, maps to]
                & -\DegreeTwoCoclass^\FixedPoint _\Torus.
            \end{tikzcd}
            \end{equation}
        The minus sign in $-\DegreeTwoCoclass^\FixedPoint _\Torus$ simply comes from our Mayer-Vietoris conventions.
        The desired class $\beta$ is the preimage of $((\ClassifyingSpace \Extended{\Cocharacter}) \PullBack) \Inverse (\DegreeTwoCoclass^\FixedPoint) \DirectSum \DegreeTwoCoclass^\FixedPoint _\Torus$.
        The classes $\beta^\pm$ are, by definition, the restrictions of $\beta$ to the fibres above the poles $\Pole^\pm$, so $\beta^+ = \DegreeTwoCoclass^\FixedPoint$ and $(\ClassifyingSpace \Extended{\Cocharacter}) \PullBack (\beta^-) = \DegreeTwoCoclass^\FixedPoint$ are automatically satisfied.
        
        The fibre inclusion $\ClutchingBundle{\Cocharacter} \to \UniversalSpace \ExtendedTorus \times_{\ExtendedTorus} \ClutchingBundle{\Cocharacter}$ induces a map $\Homology_2 (\ClutchingBundle{\Cocharacter}) \to \Homology_2 ^{\ExtendedTorus} (\ClutchingBundle{\Cocharacter})$.
        Let $\DegreeTwoClass^\FixedPoint_{\eqnt}$ denote the image of $\DegreeTwoClass^\FixedPoint = \Pole^- \PushForward (\DegreeTwoClass) + \ClutchingSection^\FixedPoint \PushForward [\Sphere]$ under this map.
        The claim $\beta (\DegreeTwoClass^\FixedPoint) = \DegreeTwoCoclass (\DegreeTwoClass)$ precisely means $\beta (\DegreeTwoClass^\FixedPoint _{\eqnt}) = \DegreeTwoCoclass (\DegreeTwoClass)$.
        To prove this claim, we consider the two terms of $\DegreeTwoClass^\FixedPoint$ separately.
        Diagram chasing using the fixed section $\ClutchingSection^\FixedPoint : \Sphere \to \ClutchingBundle{\Cocharacter}$ with \eqref{eqn:mayer-vietoris-equivariant} yields $(\ClutchingSection^\FixedPoint) \PullBack \beta = 0 \in \Cohomology^2 _{\ExtendedTorus} (\Sphere)$, and hence $\beta ((\ClutchingSection^\FixedPoint \PushForward [\Sphere])_{\eqnt}) = 0$.
        Our definition of $\beta$ readily yields $\beta \mapsto \DegreeTwoCoclass$ under $\Cohomology^2 _{\ExtendedTorus} (\ClutchingBundle{\Cocharacter}) \to \Cohomology^2_{\ExtendedTorus} (\Manifold) \to \Cohomology^2 (\Manifold)$, and hence we have $\beta (\DegreeTwoClass^\FixedPoint _{\eqnt}) = \beta ((\Pole^- \PushForward (\DegreeTwoClass))_{\eqnt}) = \DegreeTwoCoclass (\DegreeTwoClass)$ as required.
        \end{proof}

\renewcommand{\doitext}{}
\bibliography{main}

\begin{thebibliography}{BMO11}
\expandafter\ifx\csname url\endcsname\relax
  \def\url#1{\texttt{#1}}\fi
\expandafter\ifx\csname doi\endcsname\relax
  \def\doi#1{\burlalt{doi:#1}{http://dx.doi.org/#1}}\fi
\expandafter\ifx\csname
  urlprefix\endcsname\relax\def\urlprefix{\textsc{url}:}\fi
\expandafter\ifx\csname href\endcsname\relax
  \def\href#1#2{#2}\fi
\expandafter\ifx\csname burlalt\endcsname\relax
  \def\burlalt#1#2{\href{#2}{#1}}\fi

\bibitem[AD14]{audin_morse_2014}
M.~Audin and M.~Damian.
\newblock {\em Morse {Theory} and {Floer} {Homology}}.
\newblock Universitext, {Virtual} {Series} on {Symplectic} {Geometry}.
  Springer-Verlag, London, 2014.
\newblock \textsc{doi}:\doi{10.1007/978-1-4471-5496-9}.
\newblock \textsc{isbn}:\texttt{978-1-4471-5495-2}.

\bibitem[Bat93]{batyrev_quantum_1993}
V.~Batyrev.
\newblock Quantum {Cohomology} {Rings} of {Toric} {Manifolds}, October 1993,
  \textsc{arXiv:}\burlalt{9310004}{http://arxiv.org/abs/9310004}.

\bibitem[BMO11]{braverman_quantum_2011}
A.~Braverman, D.~Maulik, and A.~Okounkov.
\newblock Quantum cohomology of the {Springer} resolution.
\newblock {\em Advances in Mathematics}, 227(1):421--458, 2011.
\newblock \textsc{doi}:\doi{10.1016/j.aim.2011.01.021}.

\bibitem[BO17]{bourgeois_s1-equivariant_2017}
F.~Bourgeois and A.~Oancea.
\newblock {$S^1$}-{Equivariant} {Symplectic} {Homology} and {Linearized}
  {Contact} {Homology}.
\newblock {\em International Mathematics Research Notices},
  2017(13):3849--3937, July 2017.
\newblock \textsc{doi}:\doi{10.1093/imrn/rnw029}.

\bibitem[Cox11]{cox_toric_2011}
D.~A. Cox.
\newblock {\em Toric varieties}.
\newblock Graduate studies in mathematics ; v. 124. American Mathematical
  Society, Providence, R.I., 2011.
\newblock \textsc{isbn}:\texttt{978-0-8218-4819-7}.

\bibitem[Dub92]{dubrovin_integrable_1992}
B~Dubrovin.
\newblock Integrable systems in topological field theory.
\newblock {\em Nuclear Physics B}, 379(3):627--689, July 1992.
\newblock \textsc{doi}:\doi{10.1016/0550-3213(92)90137-Z}.

\bibitem[Fra10]{franz_describing_2010}
M.~Franz.
\newblock Describing toric varieties and their equivariant cohomology.
\newblock {\em Colloquium Mathematicum}, 121(1):1--16, 2010.
\newblock \textsc{doi}:\doi{10.4064/cm121-1-1}.
\newblock arXiv: 0909.0057.

\bibitem[Gut17]{gutt_positive_2017}
J.~Gutt.
\newblock The positive equivariant symplectic homology as an invariant for some
  contact manifolds.
\newblock {\em Journal of Symplectic Geometry}, 15(4):1019--1069, 2017.
\newblock \textsc{doi}:\doi{10.4310/JSG.2017.v15.n4.a3}.

\bibitem[HS95]{hofer_floer_1995}
H.~Hofer and D.~Salamon.
\newblock Floer homology and {Novikov} rings.
\newblock In H.~Hofer, C.~Taubes, A.~Weinstein, and E.~Zehnder, editors, {\em
  The {Floer} {Memorial} {Volume}}, Progress in {Mathematics}, pages 483--524.
  Birkhäuser Basel, Basel, 1995.
\newblock \textsc{doi}:\doi{10.1007/978-3-0348-9217-9_20}.
\newblock \textsc{isbn}:\texttt{978-3-0348-9217-9}.

\bibitem[Iri17]{iritani_shift_2017}
H.~Iritani.
\newblock Shift operators and toric mirror theorem.
\newblock {\em Geometry \& Topology}, 21(1):315--343, February 2017.
\newblock \textsc{doi}:\doi{10.2140/gt.2017.21.315}.

\bibitem[Kos86]{koszul_lectures_1986}
J.~L. Koszul.
\newblock {\em Lectures on {Fibre} {Bundles} and {Differential} {Geometry}}.
\newblock Springer Berlin Heidelberg, Berlin, Heidelberg, 1986.
\newblock \textsc{doi}:\doi{10.1007/978-3-662-02503-1}.
\newblock \textsc{isbn}:\texttt{978-3-540-12876-2 978-3-662-02503-1}.

\bibitem[LJ20]{liebenschutz-jones_intertwining_2020}
T.~Liebenschutz-Jones.
\newblock An intertwining relation for equivariant {Seidel} maps.
\newblock Preprint, October 2020,
  \textsc{arXiv:}\burlalt{2010.03342}{http://arxiv.org/abs/2010.03342}.

\bibitem[MO19]{maulik_quantum_2019}
D.~Maulik and A.~Okounkov.
\newblock Quantum groups and quantum cohomology.
\newblock {\em Astérisque}, 408:1--212, 2019.
\newblock \textsc{doi}:\doi{10.24033/ast.1074}.

\bibitem[MR18]{mclean_mckay_2018}
M.~McLean and A.~Ritter.
\newblock The {McKay} correspondence via {Floer} theory.
\newblock Preprint, February 2018,
  \textsc{arXiv:}\burlalt{1802.01534}{http://arxiv.org/abs/1802.01534}.

\bibitem[MS04]{mcduff_j-holomorphic_2004}
D.~McDuff and D.~Salamon.
\newblock {\em J-holomorphic curves and symplectic topology}.
\newblock Colloquium publications ({American} {Mathematical} {Society}) ; v.
  52. American Mathematical Society, Providence, R.I., 2004.
\newblock \textsc{isbn}:\texttt{978-0-8218-3485-5}.

\bibitem[MT06]{mcduff_topological_2006}
D.~McDuff and S.~Tolman.
\newblock Topological properties of {Hamiltonian} circle actions.
\newblock {\em International Mathematics Research Papers}, 2006, January 2006.
\newblock \textsc{doi}:\doi{10.1155/IMRP/2006/72826}.

\bibitem[Oan08]{oancea_fibered_2008}
A.~Oancea.
\newblock Fibered symplectic cohomology and the {Leray}-{Serre} spectral
  sequence.
\newblock {\em Journal of Symplectic Geometry}, 6(3):267--351, September 2008.
\newblock \textsc{doi}:\doi{10.4310/JSG.2008.v6.n3.a3}.
\newblock Publisher: International Press of Boston.

\bibitem[Rit13]{ritter_topological_2013}
A.~Ritter.
\newblock Topological quantum field theory structure on symplectic cohomology.
\newblock {\em Journal of Topology}, 6(2):391--489, June 2013.
\newblock \textsc{doi}:\doi{10.1112/jtopol/jts038}.

\bibitem[Rit14]{ritter_floer_2014}
A.~Ritter.
\newblock Floer theory for negative line bundles via {Gromov}–{Witten}
  invariants.
\newblock {\em Advances in Mathematics}, 262:1035--1106, September 2014.
\newblock \textsc{doi}:\doi{10.1016/j.aim.2014.06.009}.

\bibitem[Rit16]{ritter_circle-actions_2016}
A.~Ritter.
\newblock Circle-actions, quantum cohomology, and the {Fukaya} category of
  {Fano} toric varieties.
\newblock {\em Geometry \& Topology}, 20(4):1941--2052, September 2016.
\newblock \textsc{doi}:\doi{10.2140/gt.2016.20.1941}.

\bibitem[RV14]{rot_functoriality_2014}
T.~O. Rot and R.~C. A.~M. Vandervorst.
\newblock Functoriality and duality in {Morse}–{Conley}–{Floer} homology.
\newblock {\em Journal of Fixed Point Theory and Applications}, 16(1):437--476,
  December 2014.
\newblock \textsc{doi}:\doi{10.1007/s11784-015-0223-6}.

\bibitem[Sal97]{salamon_lectures_1997}
D.~Salamon.
\newblock Lectures on {Floer} homology, December 1997.
\newblock
  \urlprefix\url{https://people.math.ethz.ch/~salamon/PREPRINTS/floer.pdf}.

\bibitem[Sei97]{seidel_$_1997}
P.~Seidel.
\newblock $ \pi_1$ of {Symplectic} {Automorphism} {Groups} and {Invertibles} in
  {Quantum} {Homology} {Rings}.
\newblock {\em Geometric \& Functional Analysis GAFA}, 7(6):1046--1096,
  December 1997.
\newblock \textsc{doi}:\doi{10.1007/s000390050037}.

\bibitem[Sei08]{seidel_biased_2007}
P.~Seidel.
\newblock A biased view of symplectic cohomology.
\newblock In {\em Current developments in mathematics, 2006}, pages 211--253.
  Int. Press, Somerville, MA, 2008.
\newblock \textsc{doi}:\doi{10.4310/CDM.2006.v2006.n1.a4}.

\bibitem[Sei18]{seidel_connections_2018}
P.~Seidel.
\newblock Connections on equivariant {Hamiltonian} {Floer} cohomology.
\newblock {\em Commentarii Mathematici Helvetici}, 93(3):587--644, September
  2018.
\newblock \textsc{doi}:\doi{10.4171/CMH/445}.

\bibitem[Vit96]{viterbo_functors_1996}
C.~Viterbo.
\newblock Functors and {Computations} in {Floer} homology with {Applications}
  {Part} {II}, October 1996,
  \textsc{arXiv:}\burlalt{1805.01316}{http://arxiv.org/abs/1805.01316}.

\bibitem[Zha19]{zhao_periodic_2019}
J.~Zhao.
\newblock Periodic symplectic cohomologies.
\newblock {\em Journal of Symplectic Geometry}, 17(5):1513--1578, September
  2019.
\newblock \textsc{doi}:\doi{10.4310/JSG.2019.v17.n5.a9}.
\newblock Publisher: International Press of Boston.

\end{thebibliography}
\bibliographystyle{halpha}

\end{document}